\documentclass[11pt]{article}
\usepackage[utf8]{inputenc} 
\usepackage[T1]{fontenc}    
\usepackage{geometry}
\usepackage{amsmath, amsfonts, bm, mathrsfs, amssymb}
\usepackage{amsthm,graphicx}
\usepackage[dvipsnames]{xcolor}

\usepackage[unicode=true,
 bookmarks=false,
 breaklinks=false,pdfborder={0 0 1},colorlinks=false]
 {hyperref}
\hypersetup{
 colorlinks,citecolor=blue,filecolor=blue,linkcolor=blue,urlcolor=blue}

\allowdisplaybreaks

\usepackage{float}
\usepackage{multirow}
\usepackage{footnote}
\usepackage{dsfont}
\usepackage{booktabs}

\usepackage[linesnumbered,ruled,vlined]{algorithm2e}
\usepackage{algorithmic}

\definecolor{cc}{RGB}{0,127,0}

\definecolor{bluegray}{rgb}{0.1, 0.1, 0.7}

\usepackage{dsfont}

\title{Dimension free ridge regression}

\author{Chen Cheng\thanks{Department of Statistics, Stanford University} \and 
Andrea Montanari\thanks{Department of Electrical Engineering
    and Department of Statistics, Stanford University; School of Mathematics,
    Institute for Advanced Studies, Princeton}}

\newcommand{\opnorm}[1]{{\left\vert\kern-0.25ex\left\vert\kern-0.25ex\left\vert #1 
		\right\vert\kern-0.25ex\right\vert\kern-0.25ex\right\vert}}

\newcommand{\bigO}{\mathcal{O}}
\newcommand{\bigOmg}{\Omega}
\newcommand{\bigTht}{\Theta}
\newcommand{\smin}{s_{\sf min}}
\newcommand{\proj}{\boldsymbol{P}}
\newcommand{\bzeta}{\boldsymbol{\zeta}}
\newcommand{\intdim}{\mathsf{intdim}}
\newcommand{\constant}{\mathsf{C}}
\newcommand{\constantx}{\constant_{\bx}}
\newcommand{\constantsig}{\mathsf{d}_{\bSigma}}
\newcommand{\tconstantsig}{\tilde{\mathsf{d}}_{\bSigma}}
\newcommand{\constantlambdaefct}{\mathsf{C}_{\bSigma}}

\newcommand{\half}{\frac{1}{2}}
\newcommand{\Riter}{\mathsf{R}}
\newcommand{\Siter}{\mathsf{S}}
\newcommand{\lambdaprop}{\lambda_{\mathsf{p}}}
\newcommand{\lambdabv}{\lambda_{\mathsf{bv}}}
\newcommand{\vareps}{\tau}
\def\bvarphi{{\boldsymbol \varphi}}

\newcommand{\mc}[1]{\mathcal{#1}}

\newcommand{\Ffct}{\mathscr{F}}
\newcommand{\Rfct}{\mathscr{R}}

\newcommand{\bb}{\boldsymbol{b}}

\newcommand{\integersp}{\integers_{\geq 0}}
\newcommand{\wb}[1]{\overline{#1}}
\newcommand{\ind}{\mathds{1}}
\newcommand{\Ep}{\mathbb{E}}
\newcommand{\Prb}{\mathbb{P}}
\newcommand{\real}{\mathbb{R}} 
\newcommand{\prn}[1]{\left({#1}\right)} 
\newcommand{\brk}[1]{\left[{#1}\right]} 
\newcommand{\brc}[1]{\left\{{#1}\right\}} 
\newcommand{\norm}[1]{\left\|{#1}\right\|} 

\newcommand{\normop}[1]{\norm{#1}}

\newcommand{\est}[1]{\widehat{#1}}

\newcommand{\bxnew}{\bx_{\mathsf{new}}}
\newcommand{\ynew}{y_{\mathsf{new}}}
\newcommand{\risk}{\mathscr{R}}
\newcommand{\bias}{\mathscr{B}}
\newcommand{\var}{\mathscr{V}}

\newcommand{\estboldbeta}{\est{\boldbeta}}

\newcommand{\ridgebeta}{\estboldbeta_\lambda}

\newcommand{\covX}{\what{\bSigma}}

\def\naturals{{\mathbb N}}
\def\reals{{\mathbb R}}

\def\spn{{\rm span}}

\DeclareMathOperator*{\argmin}{arg\,min}

\newcommand{\VAR}{\mathsf{V}}
\newcommand{\BIAS}{\mathsf{B}}
\newcommand{\RISK}{\mathsf{R}}
\newcommand{\lambdaefct}{\lambda_\star}
\newcommand{\muefct}{\mu_\star}

\newtheoremstyle{myexample} 
    {\topsep}                    
    {\topsep}                    
    {\rm\small }                   
    {}                           
    {\bf }                   
    {.}                          
    {.5em}                       
    {}  

\newtheoremstyle{myremark} 
    {\topsep}                    
    {\topsep}                    
    {\rm}                        
    {}                           
    {\bf}                        
    {.}                          
    {.5em}                       
    {}  

\newtheorem{claim}{Claim}[section]
\newtheorem{lemma}[claim]{Lemma}

\newtheorem{assumption}{Assumption}

\newtheorem{theorem}{Theorem}
\newtheorem{proposition}[claim]{Proposition}
\newtheorem{corollary}[claim]{Corollary}

\theoremstyle{myremark}
\newtheorem{remark}{Remark}[section]

\theoremstyle{myremark}

\theoremstyle{myexample}

\def\<{\langle}
\def\>{\rangle}

\def\argmin{{\rm argmin}}

\def\c{{\sf c }}

\def\eps{{\varepsilon}}

\def\id{{\boldsymbol I}}

\def\sT{{\sf T}}

\def\cH{{\cal H}}

\def\bQ{{\boldsymbol Q}}

\def\const{{\kappa}}

\def\normal{{\sf N}}

\def\Cov{{\rm Cov}}
\def\Var{{\sf Var}}

\def\Tr{{\rm {Tr}}}

\def\de{{\rm d}}

\def\boldbeta{{\boldsymbol \beta}}

\def\Var{{\sf Var}}

\def\ind{\mathbb{I}}

\def\bfzero{{\bf 0}}

\def\bb{{\boldsymbol b}}

\def\cF{{\cal F}}

\def\Var{{\rm Var}}

\def\naturals{\mathbb{N}}
\def\integers{\mathbb{Z}}
\def\reals{\mathbb{R}}

\def\realsp{\mathbb{R}_{\ge 0}}

\def\bA{{\boldsymbol A}}

\def\bB{{\boldsymbol B}}
\def\bC{{\boldsymbol C}}

\def\bI{{\boldsymbol I}}
\def\bW{{\boldsymbol W}}
\def\bU{{\boldsymbol U}}
\def\bV{{\boldsymbol V}}

\def\bM{{\boldsymbol M}}
\def\bD{{\boldsymbol D}}

\def\bX{{\boldsymbol X}}

\def\bZ{{\boldsymbol Z}}

\def\balpha{{\boldsymbol \alpha}}

\def\bu{{\boldsymbol u}}

\def\bx{{\boldsymbol x}}

\def\by{{\boldsymbol y}}
\def\bz{{\boldsymbol z}}

\def\bSigma{{\boldsymbol \Sigma}}

\def\btheta{{\boldsymbol \theta}}

\def\bg{{\boldsymbol g}}
\def\bv{{\boldsymbol v}}

\def\bS{{\boldsymbol S}}

\def\beps{{\boldsymbol \varepsilon}}

\def\spn{{\rm span}}

\def\prob{{\mathbb P}}
\def\E{{\mathbb E}}

\def\Treg[#1]{T^{{\rm reg},#1}}
\def\GW[#1]{{\rm GW}(#1)}
\def\MGW[#1]{{\rm MGW}(#1)}

\def\btheta{{\boldsymbol \theta}}

\def\rank{{\rm rank}}

\def\spn{{\rm span}}

\def\constu{\varepsilon}

\def\tby{\tilde{\boldsymbol{y}}}
\def\tbg{\tilde{\boldsymbol{g}}}
\def\tomega{\tilde{\omega}}

\newcommand{\normF}[1]{\norm{#1}_F}

\date{\today}

\begin{document}

\maketitle
	
	\begin{abstract}
	Random matrix theory has become a widely useful tool in high-dimensional statistics and theoretical
	machine learning. However, random matrix theory is largely focused on the proportional
	asymptotics in which the number of columns grows proportionally to the number of rows of the data matrix.
	This is not always the most natural setting in statistics where columns correspond
	to covariates and rows to samples.
	 
	With the objective to move beyond the proportional asymptotics,
	we revisit ridge regression ($\ell_2$-penalized least squares) on i.i.d. data
	$(\bx_i,y_i)$, $i\le n$, where $\bx_i$ is a feature vector and $y_i = \<\boldbeta,\bx_i\>+\eps_i
	\in\reals$
	is a response. We allow the feature vector to be high-dimensional, or even infinite-dimensional,
	in which case it belongs to a separable Hilbert space, and assume either
	$\bz_i := \bSigma^{-1/2}\bx_i$ to have i.i.d. entries, or to satisfy a certain convex
	concentration property.

	Within this setting, we establish  non-asymptotic bounds that
	approximate the bias and variance of ridge regression in terms of the bias and
	variance of an `equivalent' sequence model (a regression model with diagonal design matrix). 
	The approximation is up to multiplicative factors bounded by 
	$(1\pm \Delta)$ for some explicitly small $\Delta$. 
	
	Previously, such an approximation result was known only in the proportional regime
	and only up to additive errors: in particular, it did not allow to characterize the behavior
	of the excess risk when this converges to $0$.  
	Our general theory recovers earlier results in the proportional regime (with better error rates).
	As a new application, we obtain a completely explicit and sharp 
	characterization of ridge regression for Hilbert covariates with regularly varying spectrum.
	Finally, we analyze the overparametrized near-interpolation setting and obtain sharp 
	`benign overfitting' guarantees. 
	\end{abstract}

	\tableofcontents
	
	\newcommand*{\horzbar}{\rule[.5ex]{2.5ex}{0.5pt}}
\newcommand*{\what}[1]{\widehat{#1}}

\section{Introduction}
\label{sec:Introduction}

In regression modeling, we typically assume to be given data $(\bx_i,y_i)$, $i\le n$
that are i.i.d. samples from a common distribution $\prob$, with $\bx_i$ a feature vector,
and $y_i\in\reals$ a scalar response. We would like to estimate a model $f:\bx\mapsto f(\bx)$
to predict $\ynew$
from $\bxnew$, where  $(\bxnew,\ynew)\sim\prob$ is a new sample from the same 
distribution.
In this paper, we will focus on linear models whereby $f(\bx) = \<\estboldbeta,\bx\>$,
and use ridge regression for the estimator $\estboldbeta$. Denoting by $\bX$ the matrix with
rows $\bx_1,\dots,\bx_n$, we have 
\begin{align}
	\ridgebeta &:=\arg\min_{\bb}\Big\{\frac{1}{n}\|\by-\bX\bb\|^2+\lambda\|\bb\|^2\Big\}\\
	& = (\bX^\sT \bX + n\lambda \bI)^{-1} \bX^\sT \by \, .\label{eq:RidgeDef2}
\end{align}
We will also be interested in the $\lambda\to 0+$ limit of this estimator which
(in the overparametrized case) corresponds to the minimum norm interpolator of the data,
and refer to it as `ridgeless regression.'
We will denote by $\boldbeta := \arg\min_{\bb}\E\{(y-\bb^\sT\bx)^2\}$ the population regressor.

Statistical theory studies this and similar estimators in three 
different regimes:
\begin{enumerate}
\item The classical low-dimensional setting in which $\bx_i,\boldbeta\in\reals^d$
with $d$ fixed and $n\to\infty$. In this regime, the empirical covariance $\covX:=\bX^\sT \bX/n$
converges to the population covariance $\bSigma := \E\{\bx_1\bx_1^\sT\}$ (provided the latter
exists) and $\ridgebeta$ is asymptotically normal \cite{van2000asymptotic}. 
\item The (by now) classical high-dimensional regime in which $\bx_i,\boldbeta\in\reals^d$
with $d\gg n$ but: $(i)$~the population covariance $\bSigma$ is well conditioned,
and $(ii)$~the population regressor $\boldbeta$ is sparse. In this case it is advised to 
replace the $\ell_2$ penalty $\|\bb\|^2$ by a sparsity promoting penalty, e.g. 
 $\|\bb\|_1$ \cite{tibshirani1996regression,donoho2005stable}. In many ways, this regime
  is similar to the previous one, provided $n\gg s\log d$.
While $\covX$ does not concentrate, its restrictions to subsets of $O(s)$ 
 coordinates do \cite{candes2005decoding}.
\item The proportional regime in which $n\asymp d$. In this case $\covX$ does not 
concentrate, and $\ridgebeta$ is not consistent, and indeed consistent estimation is 
generally impossible. 
However, accurate characterizations of the ridge estimator and its risk can be derived 
using random matrix theory \cite{dicker2016ridge,dobriban2018high,hastie2022surprises, wu2020optimal, richards2021asymptotics}. 
Such characterizations answers the question of 
$\eps$-consistency: for what sample size, and what data distributions does the ridge estimator achieves
error $\E\{\|\ridgebeta-\boldbeta\|^2\}\le \eps$? Similar characterizations hold for other
estimators such as the Lasso \cite{bayati2011lasso,miolane2021distribution,celentano2020lasso}, 
robust M-estimators \cite{bean2013optimal,el2013robust,el2018impact,donoho2016high}, and so on 
\cite{barbier2019optimal,thrampoulidis2018precise,taheri2021fundamental,celentano2022fundamental}.
\end{enumerate}
Despite the wealth of fascinating technical results in this area, this state of affairs
leaves open many important questions.

\vspace{0.1cm}

 \noindent\emph{First,} it would be important have a unified theoretical framework
 that does not require the statistician to decide which asymptotics to use. 
For instance, in order to apply sharp asymptotics in the classical or proportional
regimes, it is often assumed that a given pair $(n,d)$ is in fact an element of
a sequence  $(n,d(n))$ with, respectively, either $d(n) \asymp 1$, or $d(n) \asymp n$.

In practice we are given a single pair, say  $(n,d) = (1000,50)$: should we interpret this
as $d \asymp 1$, $d \asymp n$, or yet another regime that is not covered by current theory
(e.g., $d\asymp n^{2/3}$)?

In fact, the distinction between three types of asymptotics outlined above 
is rather the consequence of
the technical tools used to derive them, rather than a fundamental statistical phenomenon.

\vspace{0.1cm}

\noindent\emph{Second,} the restriction $d = O(n)$ (or $s = O(n)$ in sparse regression)
which is implied both by the proportional and by the classical asymptotics is artificial.
 While this condition might seem necessary for consistency  at first sight
 (it might seem that at least $d$ observations are required to estimate $d$ parameters), as shown
in \cite{bartlett2020benign,tsigler2020benign} this is in fact not the case.
Further, it is not even clear how to check in practice $d = O(n)$ for a given pair $n,d$.

\vspace{0.1cm}

\noindent\emph{Third,} it would be important to remove the assumption of a well conditioned $\bSigma$, 
 and derive precise asymptotics for general covariances.
 We would argue that the ill-conditioned case is most important in practice, since 
high-dimensional data  have often low-dimensional structures.

\vspace{0.1cm}

\noindent\emph{Fourth,} the proportional asymptotics is 
somewhat un-natural from a statistical viewpoint.
Most statisticians are used to think of the data distribution is fixed (in particular, $d$ is fixed),
while we sample size $n$ increases. In a standard proportional setting, one instead assumes
$n,d\to\infty$ together with $n/d\to\delta$: the data distribution changes with the sample size.
%

\vspace{0.2cm}

Recent progress on several of these issues was achieved in the context of ridge regression.
Among others, \cite{hastie2022surprises} derived a characterization for bias and variance in 
the proportional regime
that is \emph{non-asymptotic}, i.e. holds up to an approximation error that is explicit and 
vanishes for large $n$, $d$.  Using a different approach,  \cite{bartlett2020benign,tsigler2020benign} 
obtained bounds 
on bias and variance that hold for arbitrary (possibly infinite) dimension $d$, in terms of 
of the decay of eigenvalue of $\bSigma$. These bounds allow to demonstrate `benign overfitting,'
i.e. choices of $\bSigma,\boldbeta$ (i.e. data distributions) such that minimum norm
interpolator is consistent.

The results \cite{hastie2022surprises,bartlett2020benign,tsigler2020benign} have limitations. 
The characterization of the risk proved in 
 \cite{hastie2022surprises} has sharp leading constants, but only holds for 
  $C^{-1}\le n/d\le C$ with $C$ a constant,
 and holds up to an additive error. However, this error terms
  can be larger than the actual excess risk when
 the latter vanishes.
  The bounds of  \cite{bartlett2020benign,tsigler2020benign}, 
 on the other hand, hold up to unspecified
  multiplicative constants. The proof techniques in these two sets of results are 
  furthermore very different.

In this paper we attempt to provide a unified picture that covers 
 these gaps, by extending the sharp characterization of ridge regression 
of \cite{hastie2022surprises} beyond the proportional regime. This will allows to recover the 
benign overfitting
results of \cite{bartlett2020benign,tsigler2020benign} (in several cases) with sharp constants.
 In doing so, we will extend random 
matrix theory analysis to cases with $d\gg n$ or $d=\infty$, without restrictions on the
condition number of $\bSigma$. In the case $d=\infty$,  the feature vectors $\bx_i$ 
are  random elements in a separable Hilbert space, whose distribution is fixed (does not change 
with $n$), and whose covariance $\bSigma$ is a trace class self-adjoint operator.

The rest of the paper is organized as follows. 
The next section describes the setting for our analysis,
the main assumptions and the 
resulting asymptotic characterization. It also provides some 
intuition and connects our results to earlier work. 
Section \ref{sec:Statement} contains the formal statement of our general results, 
while Section \ref{sec:Applications} specializes our
theorem  to regimes of interest and develops tools to check its assumptions. 
Section \ref{sec:Numerical} evaluates our characterization for certain choices of 
$\bSigma,\boldbeta$, and compare the predictions with simulations.
Finally, proof are presented in Sections \ref{sec:proof-main}
and \ref{proof:R-approximation}, with most technical steps deferred to the appendices.

	\section{Setting and characterization}
\label{sec:Setting}
%

\paragraph{Ridge regression in Hilbert space}

We consider the simple linear model
\begin{align}
	y_i = \bx_i^\sT \boldbeta + \eps_i\, , \label{eq:LinearModel}
\end{align}
where $\boldbeta \in \reals^d$ is the ground truth signal. The random features 
$\bx_i \in \reals^d$ and noise $\eps_i$ are independent, and the 
$(\bx_i, \eps_i)$ are i.i.d. samples with $1 \leq i \leq n$. We assume $\bx_i, \eps_i$ are
 mean zero with covariances $\Cov(\bx_i) = \bSigma$ and $\Var(\eps_i) = \vareps^2$. Defining
  the data matrix
\begin{align*}
	\bX = \begin{bmatrix}
		\horzbar & \bx_1^\sT & \horzbar \\
		\horzbar & \bx_2^\sT & \horzbar \\
		& \vdots &  \\
		\horzbar& \bx_n^\sT &\horzbar
	\end{bmatrix} \in \reals^{n \times d}\, ,
\end{align*} 
the response vector $\by = (y_1, \cdots, y_n)^\sT$ and the noise vector $\beps = (\eps_1, \cdots, \eps_n)^\sT$, we can write in matrix form
\begin{align}
	\by = \bX \boldbeta + \beps \, . \label{eq:LinearModel-Matrix}
\end{align}

In this paper, we assume the dimension $d \in \integersp \cup \{\infty\}$. When $d < \infty$,
 we are in the usual setup of linear model with finite dimensional features. In the case
  $d = \infty$, we assume that the $\bx_i$'s' are i.i.d.\ random vectors from a real, 
  separable Hilbert space $\cH$.
  We will use $\|\bx\|$ to denote the norm and  $\<\bx_1,\bx_2\>$ or
  $\bx_1^{\sT}\bx_2$ to denote the scalar product in this
  space. We understand the infinite dimensional matrix $\bx \bx^\sT$ as an operator $\cH \to \cH: \boldbeta \to \<\bx, \boldbeta\> \bx$. Given a linear operator  $\bA:\cH\to\cH$, we denote by $\|\bA\|$ the associated operator norm.
  
We will assume the covariance operator $\bSigma =\E[\bx\bx^{\sT}]$  to be trace-class, namely
\begin{align*}
	\Tr (\bSigma) =\E\{\|\bx_i\|^2\}< \infty \, ,
\end{align*}
and, without loss of generality, we also  assume $\norm{\bSigma} = 1$. Recall that,
 without loss of generality, one can always assume $\cH$ to be 
 $\ell_2 := \{\bx = (x_1, x_2, \cdots, ) : \sum_{i=1}^\infty x_i^2 < \infty \}$
 \cite{brezis2011functional}.
 
For an estimator $\estboldbeta =\estboldbeta(\bX,\by)$ we define the excess 
risk as 
\begin{align*}
	\risk_\bX(\estboldbeta; \boldbeta) = \Ep_{\bxnew, \by} \brk{(\bxnew^\sT \estboldbeta - \bxnew^\sT \boldbeta )^2 \mid \bX} = \Ep_{\by} \brk{\|\estboldbeta - \boldbeta\|_{\bSigma}^2 \mid \bX},
\end{align*}  
where $\bxnew$ is an independent copy of 
 $\bx_1, \cdots, \bx_n$ and $\|\bx\|_\bSigma^2 := \bx^\sT \bSigma \bx$.
 We will also refer to this as the `test error' or the `generalization error'
 (although the latter is actually given by the difference between $\risk_\bX$
 and ts empirical version.) Let us emphasize that in this definition, 
 $\risk_\bX(\estboldbeta; \boldbeta)$ is a random quantity because it depends on
 the data $\bX$: however, as we will prove, it concentrates around a non-random value.
 
 The generalization error admits a variance-bias decomposition $\risk_\bX(\estboldbeta; \boldbeta) = \var_\bX(\estboldbeta; \boldbeta) +\bias_\bX(\estboldbeta; \boldbeta)$, with
\begin{align*}
	\var_\bX(\estboldbeta; \boldbeta)  = \Tr \prn{\bSigma\Cov (\estboldbeta \mid \bX) } \, , \qquad
	\bias_\bX(\estboldbeta; \boldbeta)  = \norm{\Ep_{\by} [ \estboldbeta \mid \bX ] - \boldbeta  }_{\bSigma}^2  \, .
\end{align*}
For ridge regression, we can write 
explicit forms of variance and bias:
\begin{subequations}
\begin{align}
	\var_\bX(\lambda) &
	 = \frac{\vareps^2}{n} \Tr \prn{\bSigma \cdot \covX (\covX + \lambda \bI)^{-2}} \, , \label{eq:var-X} \\ 
	\bias_\bX(\lambda) &  
	 = \lambda^2 \< \boldbeta, (\covX + \lambda \bI)^{-1} \bSigma (\covX + \lambda \bI)^{-1}\boldbeta\> \label{eq:bias-X} \, .
\end{align}
\end{subequations}

\paragraph{Assumptions on the covariates distribution} We impose
  the following assumptions on the covariates $\bx_i$ throughout the paper.
\begin{assumption} \label{asmp:data-dstrb}
We assume $\E[\bx_i]=\bfzero$, $\bSigma :=\E[\bx_i\bx_i^{\sT}]$ is a trace class operator:
 $\Tr (\bSigma) < \infty$ and (without loss of generality) $\norm{\bSigma}= 1$. 
We denote its eigenvalues by $1 = \sigma_1 \geq \sigma_2 \geq \cdots$ in non-increasing order. 
We assume $\|\boldbeta\|_{\bSigma^{-1}} := \|\bSigma^{-1/2}\boldbeta\|<\infty$.

We further assume  $\bx_i = \bSigma^{1/2} \bz_i$ where the following hold.
 
 \noindent{{\bf I.}} There exist  $\constantsig := \constantsig(n) \geq n$ such that,
  for all $1 \leq k \leq \min\{n, d\}$ 
	\begin{align}
		\sum_{l=k}^d \sigma_l \leq \constantsig \sigma_k \,.\label{eq:DefConstantsig}
	\end{align} 

\noindent\emph{{\bf II.}} There exist $\constantx >0$, such that  one of the following condition holds:
	\begin{enumerate}
		\item[$(a)$] \textbf{Independent sub-Gaussian coordinates}:  $\bz_i$ has independent but not necessarily identically
		 distributed coordinates with uniformly bounded sub-Gaussian norm. Namely: each 
		  coordinate $z_{ij}$ of $\bz_i$ satisfies $\Ep [z_{ij}]= 0$, $\Var (z_{ij}) = 1$ and $\norm{z_{ij}}_{\psi_2}:= \sup_{p \geq 1} 
		  p^{-\frac 1 2} \prn{\Ep \brk{|z_{ij}|^p}}^{\frac 1 p} \leq \constantx$.
		\item[$(b)$] \textbf{Convex concentration}: allowing $\bz_i$ to have dependent coordinates, 
		the following holds for any $1-$Lipschitz convex function $\varphi : \real^d \to \real$, and for every $t > 0$
		\begin{align*}
			\Prb \prn{| \varphi(\bz_i) - \Ep \varphi(\bz_i)| \geq t} \leq 2 \exp \prn{-t^2/ \constantx^2}\, .
		\end{align*}
	\end{enumerate}
\end{assumption}

The technical motivation for assumption II is to establish concentration of quadratic forms of $\bz_i$, 
via Hanson-Wright inequality. We notice that the convex concentration property is implied 
by any of the following. $(i)$~By Talagrand inequality, convex concentration holds for
random vectors $\bz_i$ with independent bounded entries \cite[Theorem 7.12]{boucheron2013concentration}.
$(ii)$~By Herbst's argument, concentration of Lipschitz functions (and hence in particular 
convex concentration) holds for random vectors $\bz_i$ that satisfy a log-Sobolev inequality
\cite[Proposition 5.4.1]{bakry2014analysis}. $(iii)$~Finally, as a special case of the last point,
vectors $\bz_i$ with strongly log-concave probability density function satisfy this condition 
\cite[Corollary 5.7.2]{bakry2014analysis}.

The form of Hanson-Wright inequality that we will use is given below.
\begin{lemma}[Hanson-Wright inequality \cite{adamczak2015note, rudelson2013hanson}] \label{lem:hanson-wright} Suppose 
$\bx \in \mathbb{R}^d$ is a random copy of the features vector $\bx_i$ satisfying Assumption~\ref{asmp:data-dstrb}.
Then there exists a universal constant $c_0>0$ such that,
 for any matrix $\bM \in \mathbb{R}^{d \times d}$ with $\Tr (\bSigma^{\frac 1 2} \bM \bSigma^{\frac 1 2} )
 <\infty$, we have
	\begin{equation*}
		\Prb \prn{\left|\bx^\sT \bM \bx - \Tr \prn{\bSigma \bM} \right| \geq t} \leq 2 \exp \Big\{ - c_0 \, \min 
		\Big(
		\frac{t^2} {\constantx^4 \|\bSigma^{\frac 1 2} \bM \bSigma^{\frac 1 2} \|_F^2}, 
		 \frac{t}{\constantx^2 \|\bSigma^{\frac 1 2} \bM \bSigma^{\frac 1 2} \| }
		 \Big) \Big\} \, .
	\end{equation*}
\end{lemma}
\begin{remark}
	The results of~\cite{adamczak2015note, rudelson2013hanson} are stated for finite $d$. However, 
	the inequality also holds for $d =\infty$ on the Hilbert space $\ell_2$ by a standard
	approximation argument. Namely, one can project the vector $\bx$ on the span
	of  the top $k$-eigenvectors of $\bSigma$, establish concentration, and take $k\to\infty$ at the end.
\end{remark}

\paragraph{Effective variance and bias} An important observation of~\cite{hastie2022surprises} 
is that variance $\var_\bX$ and bias $\bias_\bX$ concentrate around some non-random quantities,
that can be interpreted in terms of an `effective' regression problem. While \cite{hastie2022surprises}
proves such characterization in the proportional regime $n\asymp d$, here we will extend its validity 
and prove stronger guarantees. 

Define the effective regularization $\lambdaefct$ as the unique non-negative solution of
\begin{align}
	n \cdot \prn{1 - \frac{\lambda}{\lambdaefct}} = \Tr \prn{\bSigma(\bSigma + \lambdaefct \bI)^{-1}} \, , \label{eq:lambda-fixed-point}
\end{align}
we then define the effective variance and bias as
	\begin{align}
		\VAR_n(\lambda) & := \frac{\vareps^2 \Tr \prn{\bSigma^2 (\bSigma + \lambdaefct \bI)^{-2}}}{n -\Tr \prn{\bSigma^2 (\bSigma + \lambdaefct \bI)^{-2}} } \, , \label{eq:VAR-n} \\
		\BIAS_n(\lambda) & := \frac{\lambdaefct^2 \<\boldbeta, \prn{\bSigma + \lambdaefct \bI}^{-2} \bSigma \boldbeta\>}{1 - n^{-1} \Tr \prn{\bSigma^2 (\bSigma +\lambdaefct \bI)^{-2}}} \, , \label{eq:BIAS-n}\\
		\RISK_n(\lambda) & := \BIAS_n(\lambda)+\VAR_n(\lambda)\, . \label{eq:RISK-n}
	\end{align}
Our main result ---stated in the next section--- will establish dimension-free guarantees of
the form
\begin{align}
 \var_\bX = (1+o_n(1)) \VAR_n\, , \;\;\;\;\; \bias_\bX = (1+o_n(1)))\BIAS_n\, .
 \end{align}
 These improve over earlier work in two important directions. 
First, they are \emph{dimension free}, and in particular do not assume $n\asymp d$.
Second, they provide \emph{multiplicative approximations}, and hence retain their utility 
when the risk is small.

 \paragraph{Bounds, interpretation, benign overfitting}
 Before stating our formal results relating  
  $\var_\bX$ to $\VAR_n$ and $\bias_\bX$ to $\BIAS_n$, it is useful to develop some intuition 
  about the expressions  \eqref{eq:VAR-n}, \eqref{eq:BIAS-n} and their immediate 
  consequences.
  Note that, by Eq.~\eqref{eq:lambda-fixed-point}, we necessarily have 
  \begin{align}
  \Tr \prn{\bSigma^2(\bSigma + \lambdaefct \bI)^{-2}}< \Tr \prn{\bSigma(\bSigma + \lambdaefct \bI)^{-1}}\le n\, .
  \end{align}
  If we assume that inequality between the first and last term holds 
  with a constant multiplicative factor, i.e. 
  $\Tr \prn{\bSigma^2(\bSigma + \lambdaefct \bI)^{-2}}\le n(1-c_\star^{-1})$ for some constant
  $c_\star\in (0,\infty)$, then we get
	\begin{align}
\VAR_n(\lambda) &\le  \frac{c_\star\vareps^2}{n} \Tr \prn{\bSigma^2 (\bSigma + \lambdaefct \bI)^{-2}}\, , \label{eq:VAR-n-Simplified} \\
\BIAS_n(\lambda) & \le  c_\star\lambdaefct^2 \<\boldbeta, \prn{\bSigma + \lambdaefct \bI}^{-2} \bSigma \boldbeta\>\, .
\label{eq:BIAS-n-Simplified} 
	\end{align}
 Comparing these bounds with the bias and variance of general ridge regression
 in Eqs.~\eqref{eq:var-X}, \eqref{eq:var-X}, we observe that the right hand sides are 
 (modulo the factor $c_\star$) the bias and variance of a modified ridge regression in which:
 \begin{itemize}
\item The design matrix is non-random and given by $\bSigma^{1/2}$ instead of $\bX$.
\item The regularization parameter is $\lambda_\star$ instead of $\lambda$.
\item The noise level is $\vareps/\sqrt{n}$ instead of $\vareps$.
\end{itemize}
 Even more explicit expressions can be obtained by writing the right-hand side of 
 Eqs.~\eqref{eq:VAR-n-Simplified}, \eqref{eq:BIAS-n-Simplified} in the basis that diagonalizes 
 $\bSigma$ as in the next proposition. A proof of this statement
is in Appendix~\ref{proof:effective-quantities-bound}.  
%
\begin{proposition} \label{prop:effective-quantities-bound}
Assume $\Tr \prn{\bSigma^2(\bSigma + \lambdaefct \bI)^{-2}}\le n(1-c_\star^{-1})$, for 
$c_\star\in (1,\infty)$. Let $\bSigma :=\sum_{i\ge 1}\sigma_i\bv_i\bv_i^{\sT}$ be the eigendecomposition
of of $\bSigma$, and denote by $\boldbeta_{\le k}:=\sum_{i\le k}\<\boldbeta,\bv_i\>\bv_i$ the orthogonal projection of $\boldbeta$
onto the span of $\bv_1,\dots,\bv_k$, and by $\boldbeta_{> k}:=\boldbeta-\boldbeta_{\le k}$ its complement.
 Finally, let $k_\star := \max\{k\, :\; \sigma_k \ge \lambdaefct\}$, and define 
the tail effective rank parameters by
\begin{align}
r_q(k) := \sum_{\ell>k}\Big(\frac{\sigma_\ell}{\sigma_{k+1}}\Big)^q\, ,\;\;\;\;
\overline{r}(k) := \frac{r_1(k)^2}{r_2(k)}\, .\label{eq:EffRank}
\end{align}
 Then, defining $b_k:=\sigma_k/\sigma_{k+1}$,
 we have $2n \ge k_\star+ r_1(k_\star)/b_{k_{\star}}$  and
\begin{align}
\VAR_n(\lambda) &\le c_\star\vareps^2\Big(\frac{k_\star}{n}+\frac{r_2(k_\star)}{n}\Big)
\le c_\star\vareps^2\Big(\frac{k_\star}{n}+\frac{4b^2_{k_{\star}}n }{\overline{r}(k_\star)}\Big)\, ,\label{eq:CrudeBoundVar}\\
\BIAS_n(\lambda) &\le c_\star 
\Big(\sigma_{k_\star}^2 \|\boldbeta_{\le k_\star}\|_{\bSigma^{-1}}^2 +\|\boldbeta_{>k_\star}\|_{\bSigma}^2\Big)\, .\label{eq:CrudeBoundBias}
\end{align}
\end{proposition}
(We notice that if the singular values $\sigma_k$ do not decay faster than exponentially, 
then $b_k$ is of order one.)
While these are only bounds on the theoretical characterization $\BIAS_n(\lambda)$, $\VAR_n(\lambda)$
for bias and variance, our main resuls (Theorem \ref{thm:main} and Theorem \ref{thm:main-ridgeless-underparameterized}) will allow to transfer them 
to the actual bias and variance $\bias_\bX(\lambda)$, $\var_\bX(\lambda)$ (modulo additional error terms).
\begin{remark}
These bounds (more precisely, the bounds  on $\bias_\bX(\lambda)$, $\var_\bX(\lambda)$
that follow from these and Theorem \ref{thm:main}) are closely related to the ones in  \cite{bartlett2020benign,tsigler2020benign},
see in particular  \cite[Theorem 1]{tsigler2020benign}.
It is worth pointing out two important differences. \emph{First,} the bounds in 
 Eqs.~\eqref{eq:CrudeBoundVar}, \eqref{eq:CrudeBoundBias}
 are somewhat more precise/explicit: there is no unspecified constant 
 factor\footnote{The factor $c_\star$ is explicit and, if useful, can be replaced by the original expression.}, no dependence on 
the condition number of $\sigma_1/\sigma_{k_\star}$, and no multiplicative factor depending on the probability.
\emph{Second,}  Eqs.~\eqref{eq:CrudeBoundVar}, \eqref{eq:CrudeBoundBias} are only proved for the specific
value of $k_\star$ defined there.
\end{remark}

\begin{remark}
The bounds of Eqs.~\eqref{eq:CrudeBoundVar}, \eqref{eq:CrudeBoundBias} allow to characterize settings
in which the excess test error (as predicted by our theory) vanishes. 
Indeed, for $\VAR_n(\lambda)$ to vanish, it is sufficient that $k_\star/n\to 0$  and
$\overline{r}(k_\star)/n\to \infty$. A simple sufficient condition for $\BIAS_n(\lambda)\to 0$
is that $\boldbeta\in \spn(\bv_1,\dots,\bv_k)$ with $\sigma_k/\sigma_{k_\star}\to\infty$. 

We will discuss special examples in Section \ref{sec:Applications}, and show how our
general results allow to derive more precise estimates of the risk in those cases.
\end{remark}

 \paragraph{Equivalent sequence model} The discussion
 above relies on the assumption  $\Tr \prn{\bSigma^2(\bSigma + \lambdaefct \bI)^{-2}}\le n(1-c_\star^{-1})$,
which implies the simple bounds \eqref{eq:VAR-n-Simplified},  \eqref{eq:BIAS-n-Simplified}.
However the interpretation in terms of a modified ridge regression problem holds for 
the exact formulas of Eqs.~\eqref{eq:VAR-n}, \eqref{eq:BIAS-n}.
This interpretation was developed in the context of earlier work on the proportional 
asymptotics \cite{donoho2013accurate,celentano2022fundamental},
but it is useful to spell it out here for the present context.

In the modified model, we observe $\by^s$ that is related to $\boldbeta$ according to
\begin{align} \label{Eq:equivalent-sequence-model}
\by^s = \bSigma^{1/2} \boldbeta+\frac{\omega}{\sqrt{n}}\bg\, ,\;\;\;\bg \sim\normal(0,\id_d)\, ,
\end{align}
Without loss of generality, we can work in the basis in which $\bSigma$ is diagonal,
and therefore rewrite the above as $y^s_i = \sigma_i^{1/2}\beta_i+(\omega/\sqrt{n})g_i$,
which coincides with the definition of the classical sequence model \cite{tsybakov2009nonparametric}.

We use ridge regression at regularization level $\lambda_\star$ as defined in 
Eq.~\eqref{eq:lambda-fixed-point}:
\begin{align}\label{eq:SeqEstimator}
 \ridgebeta^s:= 
 \argmin_{\bb}\big\{\|\by^s-\bSigma^{1/2}\bb\|^2+\lambda_\star\|\bb\|^2\big\}\, .
\end{align}
Finally, choose the noise level 
$\omega$ to be the unique positive solution of
\begin{align}
\omega^2 = \vareps^2+\E_{\bg}\big\{\|\ridgebeta^s-\boldbeta\|_{\bSigma}^2
\big\}\, .
\end{align}
Then our theoretical prediction for the excess test error $\RISK_n(\lambda)$   
coincides with the excess test error of the sequence model:
\begin{align}
\RISK_n(\lambda) = \E_{\bg}\big\{\|\ridgebeta^s-\boldbeta\|_{\bSigma}^2
\big\}\, .
\end{align}
Summarizing, the predicted test error for the original model is equal to the test
error
in the sequence model, albeit at a different 
value of the ridge regularization parameter and of the noise level. 
Needless to say, studying the sequence model is significantly simpler than the original model
\eqref{eq:LinearModel}.
%

\paragraph{A naive explanation} The emergence of the equivalent sequence model is 
somewhat surprising: and one might be tempted to give a simple explanation as 
follows\footnote{This construction is related to the 
debiasing without `degrees-of-freedom' correction, see e.g. 
\cite{javanmard2014confidence,celentano2023lasso}.}.
Defining $\tby^s = n^{-1}\bSigma^{-1/2}\bX^{\sT}\by$, Eq.~\eqref{eq:LinearModel-Matrix}
yields:
\begin{align}
\tby^s &= \bSigma^{1/2}  \boldbeta+\frac{\tomega}{\sqrt{n}}\tbg\, ,\\
\tbg &:= \frac{\sqrt{n}}{\tomega}\left(\frac{1}{n}\bSigma^{-1/2}\bX^{\sT}\bX\bSigma^{-1/2}-
\id\right)\bSigma^{1/2}\boldbeta+
\frac{1}{\tomega \sqrt{n}} \bSigma^{-1/2}\bX^{\sT} \beps\, ,
\end{align}
where we choose $\tomega$ so that $\|\tbg\|\approx\sqrt{n}$.
This way of rewriting the original model \eqref{eq:LinearModel-Matrix} is suggestively
similar to  Eq.~\eqref{Eq:equivalent-sequence-model}. However, it falls 
short of capturing the actual structure of the equivalent sequence model for several reasons:
$(i)$~It is unclear why $\tbg$ defined above should be approximately isotropic;
$(ii)$~The effective noise level $\tomega$ does not match the actual effective noise level 
$\omega$ (the latter depends on $\lambda$);
$(iii)$~Most importantly, the above representation does not clarify why 
the behavior of the ridge estimator \eqref{eq:RidgeDef2} should be related to the one of the sequence model estimator
\eqref{eq:SeqEstimator}.

	\section{Statement of main results}
\label{sec:Statement}

\paragraph{Big-Oh notation} For two functions $f(\bx)$ and $g(\bx)$ (where $\bx$ can be a scalar or a vector),
 we 
write $f(\bx) = \bigO_{\balpha}(g(\bx))$ if there exists
 a constant $\constant_{\balpha}$ depending only on the value of
  $\balpha$ (also $\balpha$ can be either a scalar or a vector) such that $|f(\bx)| \leq 
  \constant_{\balpha} |g(\bx)|$ for all $\bx$.
   In particular, if the constant is universal we write $f(\bx) = \bigO(g(\bx))$.
    Similarly, we write
     $f(\bx) = \bigOmg_{\balpha}(g(\bx))$ if
      $|f(\bx)| \geq \constant_{\balpha} |g(\bx)|$ 
      for all $\bx$ and some constant $\constant_{\balpha}  > 0$. 
      Finally, we write $f(\bx) = \bigTht_{\balpha}(g(\bx))$ 
      if we have both $f(\bx) = \bigO_{\balpha}(g(\bx))$ 
      and $f(\bx) = \bigOmg_{\balpha}(g(\bx))$.

\vspace{0.2cm}

We will state four theorems. 
The first two concern ridge regression with positive regularization
$\lambda>0$: Theorem \ref{thm:main} is our most general result that forms the basis for
all of other ones; Theorem \ref{thm:main-Bis} is a simplified version of the previous one, 
and covers values of ridge regularization $\lambda$ that we expect to include the optimal $\lambda$.
The other two theorems apply to  the ridgeless case $\lambda=0+$: 
Theorem \ref{thm:main-ridgeless} applies to  overparametrized case,
and Theorem \ref{thm:main-ridgeless-underparameterized} to the underparametrized one.

\subsection{Ridge regression}
Our approximation guarantees will depend on the pair $\bSigma$, $\boldbeta$ through the following 
three quantities  (in the case $\lambda=0+$, these quantities will be modified later):
\begin{enumerate}
	\item The ratio between effective dimension and regularization 
	parameter:
	\begin{align} \label{eq:def-chi-n}
		\chi_n (\lambda) := 1 + \frac{\sigma_{\lfloor \eta n\rfloor} \constantsig  \log^2 (\constantsig)}{n\lambda} \, . 
	\end{align}
	Here $\eta$ is a constant that only depends on $\constantx$, and hence we will leave it implicit. 
	\item The ratio between regularization and effective regularization 
	\begin{align}
		\const := \min\Big(\frac{\lambda}{\lambdaefct}; 1 - \frac{\lambda}{\lambdaefct}\Big) >0\, .
		\label{asmp:lambda}
	\end{align}
	\item For a positive semi-definite operator $\bQ$, define
	the modified population resolvent:
	\begin{align}
		\Rfct_0(\zeta, \mu; \bQ) := \Tr \prn{\bSigma^{\frac 1 2} \bQ \bSigma^{\frac 1 2} (\zeta \bI + \mu \bSigma)^{-1}}\, .\label{eq:R-F-func-0}
	\end{align}
	Letting $\boldbeta = \bSigma^{1/2}\btheta$, $\|\btheta\|<\infty$, we consider the ratio
	\begin{align}\label{eq:RhoLambda}
		\rho(\lambda) := 
		\frac{\Rfct_0(\lambdaefct, 1; \btheta \btheta^\sT/\norm{\btheta}^2)}{\Rfct_0(\lambdaefct,1; \bI)} \in (0, 1] \, .
	\end{align}
\end{enumerate}

We next present our master theorem for ridge regression:
 its proof is postponed to Section~\ref{sec:proof-main}.
\begin{theorem}[Ridge regression] \label{thm:main}
	Under Assumption~\ref{asmp:data-dstrb}, 
	for any positive integers $k$ and $D$, there exist constants 
	$\eta = \eta(\constantx) \in (0, 1/2)$
	 and $\constant = \constant(\constantx, D) > 0$ such that the following hold.
	 Define  $\chi_n (\lambda), \const,
	\rho(\lambda)$  as above (with $\eta = \eta(\constantx)$ in Eq.~\eqref{eq:def-chi-n}).

	If it holds that
	\begin{align*}
		\chi_n(\lambda)^3 \log^2 n \leq \constant n \const^{4.5} \, , \qquad  n^{-2D + 1} = \bigO \prn{ \sqrt{\frac{\const^3 \log^2 n}{n \max \brc{1, \lambda}}}} \, ,
	\end{align*}
	then for all $n = \Omega_{k,D}(1)$,  with probability $1-\bigO_k(n^{-D+1})$ we have:
	\begin{enumerate}
		\item \textbf{Variance approximation.} 
		\begin{align}
			\left|\var_\bX(\lambda) - \VAR_n(\lambda)\right| &  
			= \bigO_{k, \constantx, D} \prn{\frac{\chi_n (\lambda)^3 \log^2 n}{n^{1- \frac{1}{k}} \const^{9.5}}} \cdot \VAR_n(\lambda) \, .
		\end{align}
		\item \textbf{Bias approximation.} 
		If we additionally have $\chi_n(\lambda)^3 \log^2 n \leq \constant n \const^{4.5} \sqrt{\rho(\lambda)}$ and $\lambda kn^{-\frac{1}{k}} \leq \const/2$, for all  $n = \Omega_{k,D}(1)$,
		we have
		\begin{align}\label{eq:MainBiasApprox}
			\left|\bias_\bX(\lambda) - \BIAS_n(\lambda)\right| & 
			= \bigO_{k, \constantx, D} \prn{\frac{\lambdaefct(\lambda)^{k+1}}{n \const^3} + \frac{\chi_n (\lambda)^3 \log^2 n}{\sqrt{\rho(\lambda)} n^{1- \frac{1}{k} } \const^{8.5}}}  \cdot \BIAS_n(\lambda)\, .
		\end{align}
	\end{enumerate}
\end{theorem}
\begin{remark}
The condition $\|\boldbeta\|_{\bSigma^{-1}}<\infty$  in Assumption~\ref{asmp:data-dstrb}
amounts to requiring that the coefficients of $\boldbeta$ in the basis of eigenvectors $\bv_i$
of $\bSigma$ decay fast enough. Namely, it is equivalent to 
 $\sum_i\<\bv_i,\boldbeta\>^2/\sigma_i<\infty$.
This condition appears to be a proof artifact and we would expect that the conclusion
of the theorem should hold under the weaker condition
 $\|\boldbeta\|_{\bSigma} < \infty$, which is required in the equivalent sequence model 
 in Eq.~\eqref{Eq:equivalent-sequence-model}. 
 This condition cannot be eliminated by an approximation argument, because
 it appears (implicitly) in the definition of $\rho(\lambda)$,
 via $\|\btheta\|^2=\|\boldbeta\|_{\bSigma^{-1}}^2$. In particular, if 
 $\|\boldbeta\|_{\bSigma^{-1}}^2\to\infty$, then $\rho(\lambda)\to 0$ and the bias
 approximation bound \eqref{eq:MainBiasApprox} becomes vacuous.
\end{remark}

\begin{remark}
As mentioned above, the conditions on the isotropic random vectors $\bz_i$
in Assumption \ref{asmp:data-dstrb} are mainly imposed to be able to apply  
Hanson-Wright inequality (Lemma \ref{lem:hanson-wright}). It is an interesting research question to 
analyze ridge regression for covariates which do not satisfy this inequality.
\end{remark}

\subsection{The non-negligible regularization regime}

Theorem~\ref{thm:main} is our master result in the most general form.
In order to simplify it, we consider two different regimes, depending on the value of 
the regularization $\lambda$: the non-negligible regularization regime in this section and the
min-norm limit in the next section.

Note that our predictions for the 
variance and bias $\VAR_n(\lambda)$, $\BIAS_n(\lambda)$
depend on $\lambda$
only through the solution $\lambdaefct(\lambda)$ of Eq.~\eqref{eq:lambda-fixed-point}, and therefore
through the ratio $\nu := \lambda/\lambdaefct(\lambda) \in [0,1]$. If $\nu\to 0$,
then ridge regression is effectively equivalent to min-norm regression, a case that we 
analyze in greater detail in the next section. If $\nu\to 1$, then the regularization
dominates, which is of course suboptimal. In this subsection, we analyze the most interesting
case $0 < \nu < 1$ (and bounded away from $0$ and $1$). The next proposition gives sufficient conditions for this to 
be the case. We say that $f:\integers_{\ge 0}\to \reals_{>0}$ is polynomially varying
if, for any $\delta\in (0, 1)$ there exist constants $0<c_1(\delta)<c_2(\delta)<\infty$
such that, for all $k$, $c_1(\delta)f(k)\le f(\lfloor \delta k\rfloor ) \le c_2(\delta)f(k)$
(see also Section \ref{sec:RegularSpectrum}.)

\begin{proposition}\label{propo:Nu}
Define the effective rank $\constantsig(k)$ as in Eq.~\eqref{eq:DefConstantsig},
and  $\constantsig^{-1}(m) := \max \{k:\,\constantsig(k) \leq m\} \leq n$.
If $\sigma_{\constantsig^{-1}(n/2)}/\sigma_{\constantsig^{-1}(2n)}\le C$ for a constant 
$C$, then setting $\lambda \in[\sigma_{\constantsig^{-1}(n)}/C', C'\sigma_{\constantsig^{-1}(n)}]$ yields
$\nu = \lambda/\lambdaefct(\lambda)\in [\nu_1,\nu_2]$ for some constants $0<\nu_1<\nu_2 < 1$ depending
on $C$ and $C'$.

Further, the above conditions hold, provided $k\mapsto \sigma_k$
and $k\mapsto \constantsig(k)$ are polynomially varying.
\end{proposition}
The proof of this proposition is presented in Appendix~\ref{proof:Nu}.

When $\lambda$ is chosen in this optimal regime so that $\nu = \Theta(1)$, it is relatively easy
to characterize the behavior of $\lambdaefct$ and other constants. 
In particular, fixing $\nu$ and substituting into Eq.~\eqref{eq:lambda-fixed-point},
we get
$\lambdaefct = \lambda_0(n(1-\nu))$, where $\lambda_0$ is defined by
\begin{align} \label{eq:def-lambda-zero}
	\Tr \prn{\bSigma(\bSigma + \lambda_0(m) \bI)^{-1}} = m \, .
\end{align}
The behavior of $\lambda_0(m)$ is characterized below.

\begin{proposition} \label{prop:control-lambda-zero}
	For $\lambda_0(m)$ in Eq.~\eqref{eq:def-lambda-zero} and the effective rank parameter 
	$\constantsig(k)$ in Assumption~\ref{asmp:data-dstrb}, we have
	\begin{align*}
		  \sigma_{2m} \le \lambda_0(m) \le \sigma_{\constantsig^{-1}(m/2)} .
	\end{align*}
	Further, under the conditions of Proposition \ref{propo:Nu}, we have
	$\lambdaefct(\lambda) =\lambda_0(n(1-\nu))\in [\sigma_{\constantsig^{-1}(n)}/C,\sigma_{\constantsig^{-1}(n)}C]$
	for some constant $C<\infty$.
\end{proposition}
We can then simplify Theorem \ref{thm:main} to the following.

\begin{theorem}\label{thm:main-Bis}
	Under Assumption~\ref{asmp:data-dstrb}, 
	further assume the `non-negligible regularization' condition: namely $\lambda$ is chosen 
	so that
	$\nu = \lambda/\lambdaefct(\lambda)\in [1/C,1-1/C]$.  
	Define $\tconstantsig(n) :=\constantsig(n)(\log\constantsig(n))^2$.

	There exists a constant $\eta$ such that, for some $\epsilon > 0$ if
	$\tconstantsig(n)\le (\sigma_{2n}/\sigma_{\lfloor\eta n\rfloor})n^{4/3}(\log n)^{-2/3-\epsilon}$, then  with probability $1-\bigO(n^{-10})$ we have (suppressing the dependence on $C, C'$ and $\epsilon$ in the big-Oh notation):
	\begin{enumerate}
		\item \textbf{Variance approximation.} 
		\begin{align*}
			\left|\var_\bX(\lambda) - \VAR_n(\lambda)\right| &  
		= \bigO\prn{\frac{1}{n^{0.99}}
		\left(\frac{\tconstantsig(n)\sigma_{\lfloor\eta n\rfloor }}
			{n \sigma_{2n}}\right)^3 }\cdot \VAR_n(\lambda) \, .
		\end{align*}
		\item \textbf{Bias approximation.} 
		Additionally if  $\norm{\boldbeta}_{\bSigma^{-1}}^2\le C''$,
		$\tconstantsig(n)\le (\sigma_{2n}/\sigma_{\lfloor\eta n\rfloor})n^{7/6}(\log n)^{-2/3-\epsilon}$, then
		we have 
		\begin{align*}
			\left|\bias_\bX(\lambda) - \BIAS_n(\lambda)\right| & 
			= \bigO\prn{\frac{1}{n^{0.49}}
		\left(\frac{\tconstantsig(n)\sigma_{\lfloor\eta n\rfloor }}
			{n \sigma_{2n}}\right)^3  }   \cdot \BIAS_n(\lambda)\, .
		\end{align*}
	\end{enumerate}
\end{theorem}
The proof of this Theorem is presented in Appendix~\ref{proof:main-Bis}.
As an example, if the eigenvalues $\sigma_i$ decrease polynomially, then 
$\sigma_{\lfloor \eta n\rfloor}/\sigma_{2n} = \bigO(1)$. If this is the case, the last theorem
yields $\var_\bX(\lambda) = (1+o_n(1)) \cdot \VAR_n(\lambda)$ as soon as 
$\constantsig(n)=\bigO(n^{1.32})$ and $\bias_\bX(\lambda) = (1+o_n(1)) \cdot \BIAS_n(\lambda)$ as soon as 
$\tconstantsig(n)=\bigO(n^{1.16})$.

Section \ref{sec:Applications} will discuss in greater detail 
applications to the proportional regime  $d=\bigO(n)$
and the high-dimensional regime $d=\infty$ under the assumption of polynomially decaying spectrum. 
In these cases, we will prove more precise estimates implying in particular
 $\lambda_0(n) = \sigma_n \log^{\bigTht(1)}(n)$. Note that in these cases
 we also get $\constantsig(n)= n \log^{\bigTht(1)}(n)$, and therefore the above conditions
 for $(1+o_n(1))$ approximation are easily met.

\subsection{Ridgeless regression}

We next consider the ridgeless limit for in the overparametrized case:
recall that $\ridgebeta$  coincides in this case with the minimum norm interpolator.
In this case we need to modify the quantities defined above to measure the quality of our approximation.
We begin by noting that Eq.~\eqref{eq:lambda-fixed-point} makes perfect sense in the case $\lambda=0$
and we have $\lim_{\lambda \downarrow 0} \lambdaefct(\lambda)
  = \lambdaefct(0) > 0$. We then use the following definitions.
\begin{enumerate}
\item We replace $\chi_n(\lambda)$ of Eq.~\eqref{eq:def-chi-n} by:
\begin{align} 
\label{eq:def-chi-n-prime}
	\chi_n'(\const) & := 1 + \frac{\sigma_{\lfloor \eta n\rfloor} \constantsig  \log^2 (\constantsig)}{\const n \lambdaefct(0)} \, ,
\end{align}
where $\const$ will be introduced in the theorem statement.
\item The quantity $\rho(\lambda)$ defined in Eq.~\eqref{eq:RhoLambda} has a well defined limit
as $\lambda\downarrow 0$, given by 
\begin{align*}
		\rho(0) := 
\frac{\Rfct_0(\lambdaefct(0), 1; \btheta \btheta^\sT/\norm{\btheta}^2)}{\Rfct_0(\lambdaefct(0),1; \bI)} \in (0, 1] \, .
	\end{align*}
	\item Finally we define
	\begin{align*}
	\constantlambdaefct:= 
	1 - \frac{1}{n}\Tr \prn{\bSigma^2(\bSigma + \lambdaefct(0) \bI )^{-2}}  \in (0, 1)\, .
	\end{align*}
	\end{enumerate}
 
 It is worth noticing that $\lambdaefct(0) = \lambda_0(n)$ as is discussed in the previous section. The control by Proposition~\ref{prop:control-lambda-zero} applies for $\lambdaefct(0)$. Before giving the statement, we introduce a piece of terminology.
 We say that $A$ happens on the event $E$ with probability at least $1-\Delta$
 if $\prob(A^c \mbox{ and } E) \le \Delta$  (and, as a consequence,
 $\prob(A) \ge 1-\Delta-\prob(E^c)$).
\begin{theorem}[Ridgeless regression in the overparameterized regime] \label{thm:main-ridgeless}
	Suppose Assumption~\ref{asmp:data-dstrb} holds with $n<d$.
	Further assume  $\sigma_n > 0$, and
	 let $\smin$ be the minimum nonzero eigenvalue of the sample covariance $\covX = \bX^\sT \bX/n$. 
	For any positive integers $k$ and $D$, there exist constants 
	$\eta = \eta(\constantx) \in (0, 1/2)$
	 and $\constant_1 = \constant_1(\constantx, D) > 0$,
	   $\constant_i = \constant_i(k,\constantx, D) > 0$, $i\in\{2,3\}$, such that the following hold,
	   for $\chi_n'(\const)$, $\rho(0)$, $\constantlambdaefct$ as above. 
	
	Let $\kappa>0$ be such that the following hold
	\begin{align*}
		\const \leq \constantlambdaefct^2/8 \, , \qquad \chi_n'(\const)^3 \log^2n \leq \constant_1 n \const^{4.5} \, , \qquad  n^{-2D + 1} = \bigO \prn{ \sqrt{\frac{\const^3 \log^2 n}{n\max \brc{1, \const  \lambdaefct(0)}}}} \, .
	\end{align*}
	Then, on the event
	$\{\smin \geq 8 \lambdaefct(0) \kappa \}$, the following hold with probability $1-\bigO_k(n^{-D+1})$:
	\begin{enumerate}
		\item \textbf{Variance approximation.} If in addition $\chi_n' (\const)^3 \log^2 n \leq  \constant_2n^{1- \frac{1}{k}} \const^{9.5}$, then
		\begin{align*}
			\left|\var_\bX(0) - \VAR_n(0)\right| & = \bigO_{k, \constantx, D} \prn{ \const \cdot \prn{\frac{ \lambdaefct(0)}{\smin} + \frac{1}{\constantlambdaefct^2}} + \frac{\chi_n' (\const)^3 \log^2 n}{n^{1- \frac{1}{k}} \const^{9.5}}} \cdot \VAR_n(0) \, .
		\end{align*}
		\item \textbf{Bias approximation.} 
		If in addition $\chi_n'(\const)^3 \log^2 n \leq \constant_1 n \const^{4.5} \sqrt{\rho(0)}$, $\lambdaefct(0) k n^{-\frac{1}{k}} \leq 1/4$ and 
		\begin{align*}
			\frac{\lambdaefct(0)^{k+1}}{n \const^3}  +  \frac{\chi_n' (\const)^3 \log^2 n}{\sqrt{\rho(0)} n^{1- \frac{1}{k}} \const^{8.5}} \leq \constant_3 \, ,
		\end{align*}
		then
		\begin{align*}
			& \left|\bias_\bX(0) - \BIAS_n(0)\right| \nonumber \\
			& = \bigO_{k, \constantx, D} \prn{ \frac{\const}{\constantlambdaefct^2} + \frac{\lambdaefct(0)^{k+1}}{n \const^3}  +  \frac{\chi_n' (\const)^3 \log^2 n}{\sqrt{\rho(0)} n^{1- \frac{1}{k}} \const^{8.5}}} \cdot \BIAS_n(0) \nonumber \\
			& \qquad +  \min \brc{\bigO\prn{\frac{\const \lambdaefct(0)  \norm{\boldbeta}^2}{\smin}}, \bigO_{\constantx, D}(\const^2 \lambdaefct(0)^2 \chi_n'(\const)^2)\norm{\btheta_{\leq n}}^2 +  \bigO_{\constantx, D}(\const \lambdaefct(0) \chi_n'(\const)) \norm{\boldbeta_{>n}}^2} \, .
		\end{align*}
	\end{enumerate}
	Finally, for any $\eps>0$, $\constantx<\infty$
	there exist constants $\constant_4=\constant_4(\constantx, \eps,D)$, $\constant_5=\constant_5(\constantx)$, such that,
	for $\min \{|d/n-1|, d/n\} \ge \eps$, the following holds with probability
	$1-\bigO(n^{-D+1})$ for $ n = \bigOmg_{\constantx, \eps, D}(1)$:
	\begin{align}
	\label{eq:SigmaMin_LB}
	\smin \geq \max\{\constant_4\sigma_d, \sigma_{\constant_5n } \}\, .
	\end{align}
\end{theorem}
The proof of this theorem is presented in 
 Appendix~\ref{proof:main-ridgeless}. We note that it is possible to
 derive a simplified version of this theorem (in analogy with Theorem 
 \ref{thm:main-Bis}) under polynomially varying spectrum.
  We refrain from doing so for brevity, and defer further study of this setting to Section \ref{sec:Applications}.

\begin{remark}
Our approach to proving Theorem~\ref{thm:main-ridgeless} consists in reducing the ridgeless 
case $\lambda=0+$ to the case $\lambda>0$, and appealing to
Theorem~\ref{thm:main}.
 For instance, when controlling the variance, we will use triangular inequality
\begin{align*}
	\left|\var_\bX(0) - \VAR_n(0)\right| \leq \left|\var_\bX(\lambda) - \VAR_n(\lambda) \right| +
	 \left|\var_\bX(0) - \var_\bX(\lambda) \right| + \left|\VAR_n(0) - \VAR_n(\lambda) \right| \, .
\end{align*}
We then use Theorem~\ref{thm:main} to bound the first term by a quantity that diverges as
 $\lambda\downarrow 0$,
and the main technical challenge is in bounding the other two terms by a quantity that
vanishes faster than any polynomial as  $\lambda\downarrow 0$.  
\end{remark}

\begin{remark}
In Theorem~\ref{thm:main-ridgeless} we use the (random) minimum nonzero eigenvalue 
$\smin$ of the sample covariance $\covX$. To apply the theorem,
 we need to choose $\kappa$ such that $\{\smin \geq 8 \lambdaefct(0) \kappa\}$ holds
 with high probability, and therefore we need a lower bound on $\smin$ 
 that holds with high probability. 
 
 Equation \eqref{eq:SigmaMin_LB} provides such a lower bounds under general
 conditions.
 In Section \ref{sec:Applications}, we will show that this lower bound implies
 optimal results in two cases: 
 $(i)$~proportional regime and $(ii)$~polynomially varying spectrum.
 In general \eqref{eq:SigmaMin_LB} might not be strong enough in certain cases.
 Nevertheless, Theorem~\ref{thm:main-ridgeless} allow us to use case-specific lower bounds
 as needed.
\end{remark}

In the underparameterized regime $d < n$, we have $\lim_{\lambda \downarrow 0} 
\lambdaefct(\lambda) = 0$ and therefore the previous bounds do not apply. 
In this case, we trivially have $\bias_\bX(0) = \BIAS_n(0) = 0$. 
The proof 
for the variance approximation requires a different proof, which is presented in 
Appendix~\ref{proof:main-ridgeless-underparameterized}.
%
\begin{theorem}[Ridgeless regression in the underparameterized regime] \label{thm:main-ridgeless-underparameterized}
Suppose Assumption~\ref{asmp:data-dstrb} holds with $n>d$, and further assume
	\begin{align*}
		\nu \:= \min\Big(\frac{d}{n},1-\frac{d}{n}\Big)\in (0,1)\, .
	\end{align*}
%
	\begin{enumerate}
		\item \textbf{Variance approximation.} 
 There exist constants $\eta$ and $C$ (depending on $k, \constantx$ and $D$) such that, for some $\epsilon > 0$ if
 $n^{-\prn{\frac{1}{4} -\epsilon}\prn{1 - \frac{1}{k}}} \log^{8}n \le C \nu^{15.5}   $, with probability $1-\bigO_k(n^{-D+1})$:
		\begin{align*}
			\left|\var_\bX(0) - \VAR_n(0)\right| & = 
			\bigO_{k, \constantx, D} \prn{ \frac{\log^8 n}{n^{\prn{\frac{1}{4} - \epsilon} \prn{1- \frac{1}{k}}} \nu^{15.5}  }} \cdot \VAR_n(0) \, .
		\end{align*}
		\item \textbf{Bias approximation.} $\bias_\bX(0) = \BIAS_n(0) = 0$ (this holds deterministically on the event
		$\rank(\bX) = d$).
	\end{enumerate}
\end{theorem}

\begin{remark}
 Theorem~\ref{thm:main-ridgeless-underparameterized} allows polynomial dependence of 
	$n$ and $d$, in contrast to the vast literature on the proportional regime when
	 $n \asymp d$. In particular  the condition for the
	 variance approximation holds provided $n^{-\frac{1}{4} +\eps'} \le C(\nu)^{15.5}$
	 for some constant $\eps'>0$. If we assume $n\asymp d^{1+\alpha}$, $\alpha\ge 0$,
	 this will hold for all $n$ large enough provided $\alpha < 0.25/15.25 \approx 0.016$.
\end{remark}

	\section{Applications} 
\label{sec:Applications}

\subsection{Proportional regime}

As a first application, we revisit the proportional regime that is defined by the following 
condition.
\begin{assumption} \label{asmp:Sigma-proportional}
	There exists a constant $M > 1$ such that $M^{-1} \leq d/n \leq M$ and $\sigma_d \geq M^{-1}$.
\end{assumption}
This case is well studied and is not the main motivation of the present paper, but it
is nevertheless important to compare our results to earlier work.   We refer the reader to
\cite{dicker2016ridge,advani2017high,dobriban2018high,wu2020optimal,richards2020asymptotics}
for background.

Among others, the results of \cite{hastie2022surprises} are more directly comparable to ours
because they establish nonasymptotic bounds comparing variance and bias 
to the effective variance and bias of Eqs.~\eqref{eq:VAR-n} and \eqref{eq:BIAS-n}, for both ridge and ridgeless regression.
 The proofs of \cite{hastie2022surprises} build on recent advances in random matrix theory,
 and in particular the anisotropic local law of \cite{knowles2017anisotropic}.

 Here we apply Theorems~\ref{thm:main}, \ref{thm:main-ridgeless} and \ref{thm:main-ridgeless-underparameterized} to 
 the proportional regime.
 We note that, under assumption \ref{asmp:Sigma-proportional}, the minimum eigenvalue
 of  $\bX^\sT \bX$ is, with high probability, of order  $n$. In order for the ridge regularization
 to have a non-trivial effect, we need to choose $\lambda \asymp 1$ as well,
 cf.~\eqref{eq:var-X} and \eqref{eq:bias-X}. We will therefore assume $\lambda$ bounded above and below 
 (there is no  loss of generality in using the same constant as in Eq.~\eqref{asmp:Sigma-proportional}).
 We will address the case $\lambda=0+$ in a separate statement below.
\begin{proposition} \label{prop:proportional-ridge}
	Let Assumptions~\ref{asmp:data-dstrb} and \ref{asmp:Sigma-proportional} hold, and further
	assume $\lambda \in [1/M,M]$. Then for any positive integers $k$ and $D$, 
	if $n = \bigOmg_{k, M, \constantx, D}(1)$, with probability 
	$1 - \bigO_{k}(n^{-D+1})$ we have
	\begin{align*}
		\left|\var_\bX(\lambda) - \VAR_n(\lambda)\right| &  =  \bigO_{k,  M, \constantx, D} \prn{\frac{\log^8 n}{n^{1 - \frac{1}{k}}}} \cdot \VAR_n(\lambda) \, , \\
		\left|\bias_\bX(\lambda) - \BIAS_n(\lambda)\right| &  = \bigO_{k,  M, \constantx, D} \prn{\frac{\log^8 n}{n^{\half - \frac{1}{k}}}}  \cdot \BIAS_n(\lambda) \, .
	\end{align*}
\end{proposition}
The proof of this result is presented in Appendix~\ref{sec:proportional}.

We note that the rates $\bigO(n^{-1})$ and $\bigO(n^{-1/2})$ are optimal for variance and bias
 approximation---corresponding to fluctuations of the average law and local law for the 
 resolvent~\cite{alex2014isotropic, knowles2017anisotropic}. 

 Note that \cite{hastie2022surprises} informally claimed that $n^{-1/2}$ is the 
 optimal rate in the above estimates. While this is correct for the bias, for the variance
 Proposition~\ref{prop:proportional-ridge} yields a faster rate.
 As related phenomenon arises for linear eigenvalue statistics of random matrices
 (i.e. statistics of the form $n^{-1}\sum_{i=1}^n\varphi(\lambda_i)$). While naively such 
 statistics would have normal deviations of order $n^{-1/2}$, the actual deviations
 are of order $n^{-1}$ because of eigenvalues correlations \cite{lytova2009central}.

We finally consider the ridgeless case.
\begin{proposition} \label{prop:proportional-ridgeless}
	Let Assumptions~\ref{asmp:data-dstrb} and \ref{asmp:Sigma-proportional} hold for
	$\bx_i = \bSigma^{1/2}\bz_i$, where $\bz_i$ has i.i.d.\  sub-Gaussian coordinates. 
	\begin{enumerate}
		\item \textbf{Overparameterized regime.} If additionally $d/n \geq 1 + M^{-1}$, then for all $n = \bigOmg_{M, \constantx, D}(1)$, with probability $1 - \bigO(n^{-D+1})$ we have
		\begin{align*}
			\left|\var_\bX(0) - \VAR_n(0)\right| & = \bigO_{ M, \constantx, D} \prn{ n^{-1/14}} \cdot \VAR_n(0) \, , \nonumber \\
			\left|\bias_\bX(0) - \BIAS_n(0)\right| & = \bigO_{M, \constantx, D} \prn{n^{-1/28}} \cdot \BIAS_n(0)\, .
		\end{align*}
		\item \textbf{Underparameterized regime.} If additionally $d/n \leq 1 - M^{-1}$, then for all $n = \bigOmg_{M, \constantx, D}(1)$, with probability $1 - \bigO(n^{-D+1})$ we have
		\begin{align*}
			\left|\var_\bX(0) - \VAR_n(0)\right| &  =  \bigO_{M, \constantx, D} \prn{n^{-1/5}} \cdot \VAR_n(0) \, , \\
			\bias_\bX(0) = \BIAS_n(0)  & = 0 \, .
		\end{align*}
	\end{enumerate}
\end{proposition}
We do not expect the exponent $1/14$, $1/28$, $1/5$ in this statement to be tight. However,
as in the positive $\lambda$ case, also in this case the error is multiplicative.

 The most direct comparison of results in this section are
Theorem~2 and Theorem~5 in \cite{hastie2022surprises}. Let us point out two
 ways in which the present result improves over the earlier \cite{hastie2022surprises}.
 \begin{itemize}
 \item Consider the case $\lambda \in [1/M,M]$. 
 In \cite[Theorem 5]{hastie2022surprises} the rate for variance approximation of
  ridge regression is $\bigO(n^{-1/2})$, while here we obtain the faster rate  $\bigO(n^{-1})$.
\item Consider the overparametrized case $\lambda=0+$. 
In \cite[Theorem 2]{hastie2022surprises} the error terms are additive, while
Proposition \ref{prop:proportional-ridgeless} provides multiplicative error terms.
In this regime, the variance is bounded below, but the bias is not.
The quality of
 approximation of our theorem does not deteriorate in the interesting case in 
 which the bias becomes small, unlike in  \cite{hastie2022surprises}.
 \end{itemize}

\subsection{Polynomially varying spectrum}
\label{sec:RegularSpectrum}

We next consider the highly overparametrized case $d \gg n$.
Overparametrized ridge (or minimum norm) regression attracted significant attention
recently because of the realization that many deep learning models are overparametrized and 
overfit the training data. This connection is reviewed in \cite{bartlett2021deep,belkin2021fit}.

Here we consider covariate vectors $\bx_i$ taking values in a general Hilbert space with $d = \infty$,
under Assumption \ref{asmp:data-dstrb} on the covariates distribution.
This is most closely related to  \cite{bartlett2020benign,tsigler2020benign}, and
\cite{koehler2021uniform}. The last paper derives refined upper bounds using Gaussian 
width techniques, but is limited to the case of Gaussian covariates and, as for earlier
results, is only accurate up to constant factors.

We will show that our general theory yields excess risk estimates that are accurate up 
to $1+o_n(1)$ multiplicative errors.
We impose the following condition on the spectrum of $\bSigma$.
\begin{assumption}[Polynomially varying spectrum] \label{asmp:Sigma-bounded-varying}
	There exists a monotone decreasing
	function $\psi:(0,1]\to [1,\infty)$ 
	with  $\lim_{\delta \downarrow 0} \psi(\delta) = \infty$, such that
	$\sigma_{\lfloor \delta i \rfloor} / \sigma_i \leq \psi(\delta)$ for all 
	$\delta \in (0, 1]$, $i \in \naturals$ and $\delta i \geq 1$.
\end{assumption}
Recall that, by definition, for any $j\le i$, $\sigma_{j}/\sigma_i\ge 1$. 
The polynomially varying condition requires that, if $i,j$ diverge proportionally,
then the eigenvalue ratio $\sigma_{j}/\sigma_i$  stays bounded.
Note that this assumption is equivalent to 
$\sup_{i\ge 1}\sigma_{\lfloor \delta i \rfloor} / \sigma_i<\infty$ for every $\delta\in(0,1]$,
which is in turn equivalent to 
\begin{align}
\lim\sup_{i \to \infty} \frac{\sigma_{\lfloor \delta i \rfloor}}{\sigma_i} <\infty\, .
\end{align}

As special case, Assumption~\ref{asmp:Sigma-bounded-varying} holds if
the sorted eigenvalues 
$(\sigma_1, \sigma_2, \cdots )$ forms a so-called \emph{regularly varying sequence}, namely 
for any $\delta \in (0, \infty)$, 
\begin{align*}
	\lim_{i \to \infty} \frac{\sigma_{\lfloor \delta i \rfloor}}{\sigma_i} = \psi(\delta) \, ,
\end{align*}
where $\psi(\delta)$ is positive and finite for any $\delta$. 
In other words, in the regularly varying case, the ratio $\sigma_{j}/\sigma_i$ 
converges when  $i,j$ diverge proportionally.

A special case of regularly varying spectrum is given by Zipf's law whereby
$\sigma_i = i^{-\alpha}$ for some $\alpha >1$ (in this case $\psi(\delta) = \delta^{-\alpha}$). 
Regularly varying functions were characterized by \cite{karamata1933mode} 
(for functions on the positive real line), and by
\cite{galambos1973regularly} (for the sequences, i.e. functions defined on the naturals).
 Namely all such sequences take the form
\begin{align*}
	\sigma_i = i^{-\alpha} a_i \exp \Big\{\sum_{j=1}^i b_j / j \Big\} \, ,
\end{align*}
where $a_i$ are arbitrary and converge to a positive limit as $i \to \infty$ and 
$b_i \to 0$. 

It is easy to see that Assumption~\ref{asmp:Sigma-bounded-varying} holds 
beyond the case of regularly varying sequences.
Consider for instance $\sigma_i = 3^{-s}$ for all $2^s \leq i < 2^{s+1}$, $s=0,1,\cdots$. 

Applying Theorems~\ref{thm:main} and \ref{thm:main-ridgeless} to $\bSigma$ with polynomially varying 
spectrum, we obtain the following result, whose  proofs are detailed 
in Appendix~\ref{sec:bounded-varying}.
\begin{proposition} \label{prop:bounded-varying-ridge}
	Let Assumptions~\ref{asmp:data-dstrb} and \ref{asmp:Sigma-bounded-varying} hold. 
	For any constants $M>0$, $\gamma\in(0,1/3)$, and positive integers $k$, $D$ the following holds.
	  If $\constantsig \le M n^{1+ \gamma}$ and $\lambda / \lambdaefct(\lambda) \in  [1/M,
	  1-1/M]$, then
	   for $n = \bigOmg_{k, M, \psi, \gamma, \constantx, D}(1)$, with probability $1 - \bigO_{k}(n^{-D+1})$
	\begin{align*}
		\left|\var_\bX(\lambda) - \VAR_n(\lambda)\right| &  =  \bigO_{k, M, \psi, \constantx, D} \prn{\frac{(\constantsig/n)^3 \log^8 n}{n^{1- \frac{1}{k}} }} \cdot \VAR_n(\lambda) \, .
	\end{align*}
	If additionally $\constantsig = \bigO_{M, \psi, \constantx}(n^{1 + \gamma} \prn{\rho(\lambda)}^{1/6})$ and 
	$\lambdaefct(0) = \bigO(1)$, 
	with the same probability we have (cf.~Theorem~\ref{thm:main} for the function $\rho$)
	\begin{align*}
		\left|\bias_\bX(\lambda) - \BIAS_n(\lambda)\right| &  
		= \bigO_{k, M, \psi, \constantx, D} \prn{\frac{(\constantsig / n)^3 \log^8 n}{\sqrt{\rho(\lambda)} n^{1- \frac{1}{k} }}}   \cdot \BIAS_n(\lambda) \, .
	\end{align*}
\end{proposition}

Applying Theorem~\ref{thm:main-ridgeless}, we have the following conclusion for ridgeless regression. 
\begin{proposition} \label{prop:bounded-varying-ridgeless}
	Let Assumptions~\ref{asmp:data-dstrb} and \ref{asmp:Sigma-bounded-varying} hold. Suppose $\boldbeta = \bSigma^{1/2} \btheta$ with $\norm{\btheta} < \infty$. If we have $\lambdaefct(0) / \sigma_n = \bigO(\log^{\bigO(1)}n)$ and $\constantsig(n) = \bigO (n\log^{\bigO(1)}n))$, for any $n = \bigOmg_{\psi, \constantx, D}(1)$, it holds with probability $1 - \bigO(n^{-D+1})$ that
	\begin{align*}
		\left|\var_\bX(0) - \VAR_n(0)\right| & = \bigO_{\psi, \constantx, D} \prn{ n^{-1/15}} \cdot \VAR_n(0) \, .
	\end{align*}
\end{proposition}
\begin{remark}
The assumptions $\lambdaefct(0) / \sigma_n = \bigO(\log^{\bigO(1)}n)$ and 
$\constantsig = \bigO (n\log^{\bigO(1)}n)$ are primarily introduced to simplify 
the form of the statement. These two conditions can be relaxed to $\lambdaefct(0) / \sigma_n = \bigO(n^{\bar\gamma})$
 and $\constantsig =\bigO(n^{1+\bar\gamma})$ for a sufficiently small $\bar\gamma$, but we do not pursue 
 this generalization here.
\end{remark}
\begin{remark}
It is possible to apply the upper/lower bounds on the bias
of Theorem~\ref{thm:main-ridgeless} to prove bounds on the bias in the setting of 
Proposition   \ref{prop:bounded-varying-ridgeless}. 
However the resulting error term is larger than
$(\sigma_n^2 \norm{\btheta_{\leq n}}^2 + \sigma_n  \norm{\boldbeta_{> n}}^2)$, 
which is  the size  of the upper bound  on $\BIAS_n(0)$ in 
Proposition~\ref{prop:effective-quantities-bound}.
\end{remark}

In order to illustrate the accuracy of our general framework, 
we apply Proposition~\ref{prop:bounded-varying-ridge}  to derive sharp asymptotics
for bias and variance in a number cases.
In each of the case below, we scale the regularization parameter $\lambda$ 
as $\lambda = \tilde{\lambda}_0(n)\cdot \nu$ for a certain explicit function $\tilde{\lambda}_0(n)$.
The scaling $\tilde{\lambda}_0(n)$ is chosen so that the bias and variance retain a non-trivial dependence on 
$\nu$ for large $n$. We expect that the excess risk achieved by optimal regularization is also covered
by this scaling (up to negligible corrections), but do not prove it formally here.
\begin{theorem} \label{thm:bounded-varying-ridge-check-assumptions}
	Let Assumption~\ref{asmp:data-dstrb} hold. Then,
	for a fixed constant $\nu > 0$ and any positive integer $D$, the following events hold with probability $1 - \bigO(n^{-D})$
	(the $o_n(1)$ errors may depend on $D$):
	\begin{enumerate}
		\item \textbf{Regularly varying spectrum with $\alpha > 1$.} Assume $(\sigma_i)_{i\ge 1}$	
		is a regularly varying sequence with exponent $\alpha>1$. As a consequence,
		$\sigma_i = i^{-\alpha} a_i \exp \brc{\sum_{j=1}^i b_j / j }$  with $a_i$ converging
		to a positive limit and $b_i \to 0$. Define $\c_\star = \c_\star(\nu)>0$ 
		as the unique positive solution of
		\begin{align*}
			1 = \nu \c_{\star}^{-1} + \frac{\pi/\alpha}{\sin(\pi / \alpha)} \c_{\star}^{-1/\alpha}\, .
		\end{align*} 
		Then we have
		\begin{align}
			\lambdaefct(\nu n^{-\alpha}) & = \c_{\star}\sigma_n (1+o_n(1)) \, ,\\
			\var_\bX(\nu n^{-\alpha}) & = \frac{\tau^2  (1- \nu \c_{\star}^{-1})(\alpha - 1)}{1 + \nu \c_{\star}^{-1} (\alpha - 1)}   (1 +o_n(1)) \, .\label{eq:VarianceRegDecay}
		\end{align}
		Let $F_\boldbeta(x) = \sum_{k=1}^{\lfloor nx\rfloor }\<\boldbeta,\bv_k\>^2$. If additionally
		 $\boldbeta$ satisfies the following ``polynomial-decay'' property: for some $0<\theta \leq 1$ that
		\begin{align*}
			\int_0^\infty x^\alpha\,  \de F_{\boldbeta}(x) = \bigO \prn{n^{1- \theta} \int_0^\infty x^\alpha \prn{1 + \c_{\star}x^\alpha}^{-1}  \de F_{\boldbeta}(x)} \, ,
		\end{align*}
		we further have
		\begin{align}
			\bias_{\bX}(\nu n^{-\alpha})  & = \frac{\sigma_n \c_{\star}^2 \alpha}{1 + \nu \c_{\star}^{-1} (\alpha - 1)} \int_{0}^{\infty}\frac{x^{\alpha}}
			{(1+\c_{\star}x^{\alpha})^2}\de F_{\boldbeta}(x)\big(1+o_n(1)\big)\, . \label{eq:BiasRegDecay}
		\end{align} 
		\item \textbf{Regularly varying spectrum with $\alpha = 1$.} 
		Next consider the case 	$\sigma_i = i^{-1} a_i (1 + \log i)^{-\alpha'}$ for some 
		$\alpha' > 1$ with $a_i$ converging to a positive limit. Define $\c_\star = \c_\star(\nu) > 0$ as
		\begin{align*}
			\c_\star = \nu + \frac{1}{\alpha'-1} \, .
		\end{align*}
		We have
		\begin{align}
			\lambdaefct(\nu n^{-1} \log^{1-\alpha'} n) &= \c_{\star} \sigma_n \log n ( 1 + o_n(1)) \, , \\
			\var_\bX(\nu n^{-1} \log^{1-\alpha'} n) & =\frac{\vareps^2}{\c_\star \log n}\, \big(1+o_n(1)\big)  \, .
		\end{align}
		Let $F_\boldbeta(x) = \sum_{k=1}^{\lfloor(n/\log n) x\rfloor }\<\boldbeta,\bv_k\>^2$. If additionally $\boldbeta$ satisfies the following ``rapid-decay'' property: for some $0 < \theta \leq 1$ that
		\begin{align*}
			\int_0^\infty x \,  \de F_{\boldbeta}(x) = \bigO \prn{n^{1-\theta} \int_0^\infty x \prn{1 + \c_{\star}x}^{-1} \,  \de F_{\boldbeta}(x)} \, .
		\end{align*}
		then we further have
		\begin{align}
			\bias_{\bX}(\nu n^{-1} \log^{1-\alpha'} n)  & = \c_{\star}^2  \sigma_n \log n \int_0^\infty	\frac{x}{(1+\c_{\star} x)^2}\, \de F_{\boldbeta}(x)\, \big(1+o_n(1)\big)\, .
		\end{align}
		\item \textbf{A non-regularly varying spectrum.} 
		$\sigma_i = p^{-s}$ for all $q^s \leq i < q^{s+1}$, with $1 < q < p$ and $s =0 ,1, \ldots$ 
		Define $s_\star$ such that $q^{s_\star} \leq n < q^{s_\star + 1}$, and for positive integer $r$ the following decreasing function in $t > 0$,
		\begin{align*}
			G_{p,q,r}(t) = \sum_{k=-\infty}^\infty \frac{q^k}{(1 + tp^k)^r} \, .
		\end{align*}
		Let $\rho_\star = n/(q^{s_\star + 1} - q^{s_\star}) \in [1/(q-1), q/(q-1))$. Then there exists a unique solution $\c_{\star} = \c_{\star}(\nu)$ to the following equation
		\begin{align*}
			1 = \nu \c_{\star}^{-1} + \rho_\star^{-1} \cdot G_{p,q,1}(\c_{\star}) \, .
		\end{align*} 
		Then we have
		\begin{align}
			\lambdaefct(\nu  p^{-s_\star}) & =\c_{\star} \sigma_n  \, \big(1+o_n(1)\big) \, , \\
			\var_\bX(\nu  p^{-s_\star}) & = \frac{G_{p,q,2}(\c_\star) \tau^2 }{\rho_\star - G_{p,q,2}(\c_\star)} \big(1+o_n(1)\big) \, .
		\end{align}
		Let $F_{\boldbeta}(x) = \sum_{k = 1}^{q^{\lceil x \rceil} - 1} \< \boldbeta, \bv_k \>^2$. If additionally $\boldbeta$ satisfies the following ``rapid-decay'' property: for some $0<\theta \leq 1$ that
		\begin{align*}
		\int_0^\infty p^{x-s_\star}  \,  \de F_{\boldbeta}(x) = \bigO \prn{n^{1-\theta} \int_0^\infty p^{x-s_\star} (1 + \c_{\star} p^{x-s_\star} )^{-1}  \,  \de F_{\boldbeta}(x)} \, ,
		\end{align*}
		we further have
		\begin{align}
			\bias_{\bX}(\nu  p^{-s_\star})  & = \frac{\c_{\star}^2 \sigma_n}{1 - \rho_\star^{-1} G_{p,q,2}(\c_{\star})} \int_0^\infty \frac{p^{x-s_\star}}{(1 +  \c_{\star} p^{x-s_\star})^2}\, \de F_{\boldbeta}(x)\, \big(1+o_n(1)\big) \, .
		\end{align}
	\end{enumerate}
\end{theorem}
The proof of this theorem is presented in 
Appendix~\ref{proof:bounded-varying-ridge-check-assumptions}.

\begin{remark}
In the case of a regularly varying spectrum with $\alpha > 1$, the 
bias vanishes with the sample size as $n^{-\alpha+o(1)}$ but the variance stays bounded
away from zero as long as $\tau>0$, cf.~Eq.~\eqref{eq:VarianceRegDecay}. 
In other words in this case overfitting is not benign and 
Theorem~\ref{thm:bounded-varying-ridge-check-assumptions} quantifies precisely this claim.

On the other hand, in the case $\alpha=1$, both bias and variance vanish for large $n$,
an therefore we achieve benign overfitting. We must emphasize however that the  variance decay 
is very slow, namely 	$\var_\bX(\lambda) \asymp (\log n)^{-1}$, and hence the decay of the excess risk is at least as slow.
\end{remark}

	\section{Numerical illustrations}
\label{sec:Numerical}

In this section we evaluate numerically the 
theoretical prediction for variance and bias, cf. Eqs.~\eqref{eq:VAR-n},
\eqref{eq:BIAS-n} and compare them with the results of numerical simulations with synthetic data.
We carry out the simulations in the ridgeless limit $\lambda=0+$
(corresponding to min-norm interpolation). This case is interesting because
it is not covered by some of our theorems. Our numerical experiments suggest that
the theoretical predictions of Eqs.~\eqref{eq:VAR-n},
\eqref{eq:BIAS-n} hold in a broader domain of validity than the one that we are able to
control rigorously.

We use Gaussian covariates $\bx_i$. By rotational invariance, we can limit ourselves
to diagonal covariance $\bSigma$. We will consider two eigenvalue structures:
\begin{description}
\item[$(I)$ Regularly varying with $\alpha>1$.] This is defined by $\sigma_i = i^{-\alpha}$
for all $i\ge 1$. This fits within the first case of  Theorem 
\ref{thm:bounded-varying-ridge-check-assumptions}.
\item[$(II)$ Regularly varying with $\alpha=1$.] This model is defined by
$\sigma_i = i^{-1}(1+\log i)^{-\alpha'}$, with $\alpha' > 1$. This fits within the 
second case of  Theorem 
\ref{thm:bounded-varying-ridge-check-assumptions}.
\end{description}
 In all numerical experiments, we generate data according to the model \eqref{eq:LinearModel}
 with a true parameters vector $\boldbeta$ concentrated on the top $d_0=100$ eigenvectors
 of $\bSigma$. More precisely, we will use
 $\boldbeta = (1,1,\dots,1, 0, 0, \ldots)$ where
  $\|\boldbeta\|_0=d_0=100$.

In Figure~\ref{fig:TheoryPrediction_n}, we plot our theoretical predictions $\VAR_n$, $\BIAS_n$,
 $\RISK_n$  for variance, bias and as a function of the sample size $n$, 
 for the two models $(I)$ and $(II)$ defined above. We use $\lambda=0+$.
  In each case, we consider several values of 
 the exponents $\alpha$, $\alpha'$ that control the decay of eigenvalues of $\bSigma$. 

In Figure~\ref{fig:TheoryPrediction_lambda}, we plot the same quantities at fixed sample
size $n=500$ and vary  the regularization parameter $\lambda$.
A few facts emerge from these figures:
\begin{itemize}
\item For both models, the bias of the minimum norm interpolator is a decreasing function
of the sample size $n$, and appears to vanish as $n\to\infty$, see second row 
of Figure~\ref{fig:TheoryPrediction_n}. 
\item In contrast, the variance exhibits a strikingly different behavior in the two
covariance models, see first row 
of Figure~\ref{fig:TheoryPrediction_n}. For model $(I)$ (polynomial eigenvalue decay, 
with exponent $\alpha>1$), the variance increases with $n$, and eventually stabilizes to a limit 
value.
For model $(II)$ (exponent $\alpha=1$), the variance decreases with $n$,
and appears to vanish, albeit very slowly, as $n\to\infty$.
\item As a consequence of these points, the excess test error of minimum norm interpolation vanishes
with sample size in model $(II)$ but does not vanish in model $(I)$. 
This behavior (and the one at previous points) is precisely quantified by 
Theorem \ref{thm:bounded-varying-ridge-check-assumptions} for $\lambda>0$.
\item Finally the dependence of bias and variance on $\lambda$ is the expected one. 
As $\lambda$ increases, bias increases but variance decreases. However,
the balance between these two factors is non-trivial:
\begin{itemize}
\item For the slowest eigenvalue decay (large $\alpha$ 
in model $(I)$ or large $\alpha'$ in model $(II)$),   the optimal $\lambda$ is strictly positive. 
\item On the other hand, for the fastest eigenvalue decay, the optimal $\lambda$ vanishes.
In these case interpolation is superior to ridge regression: we need to overfit to achieve the best
test error.
\end{itemize}
\end{itemize}

\begin{figure}[htbp!]
	\centering
	\begin{tabular}{cc}
		\includegraphics[width=0.45\textwidth]{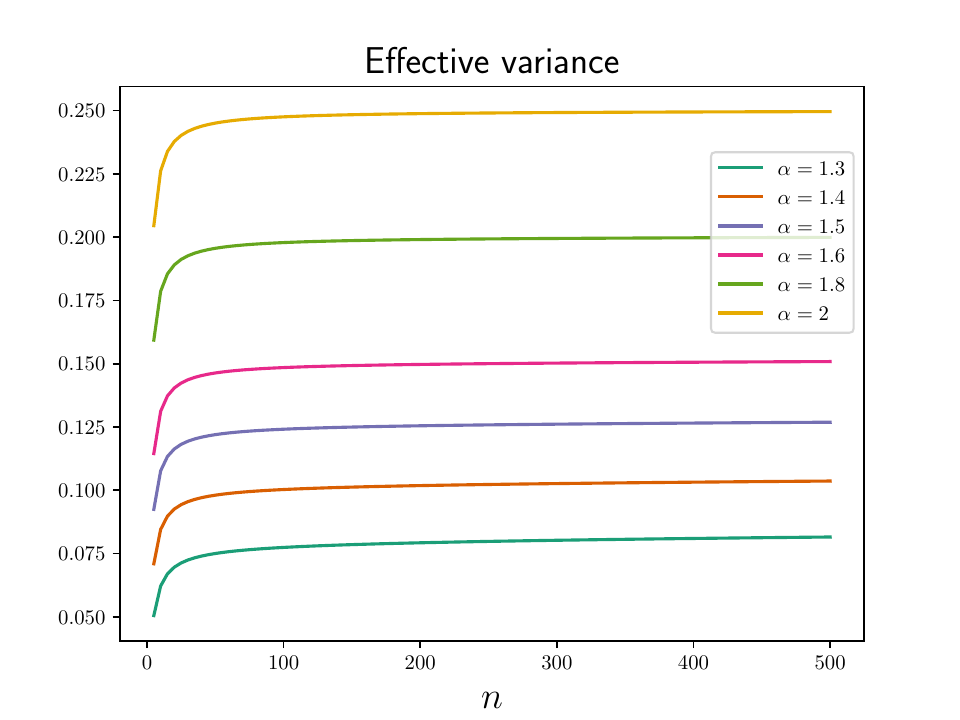} & \includegraphics[width=0.45\textwidth]{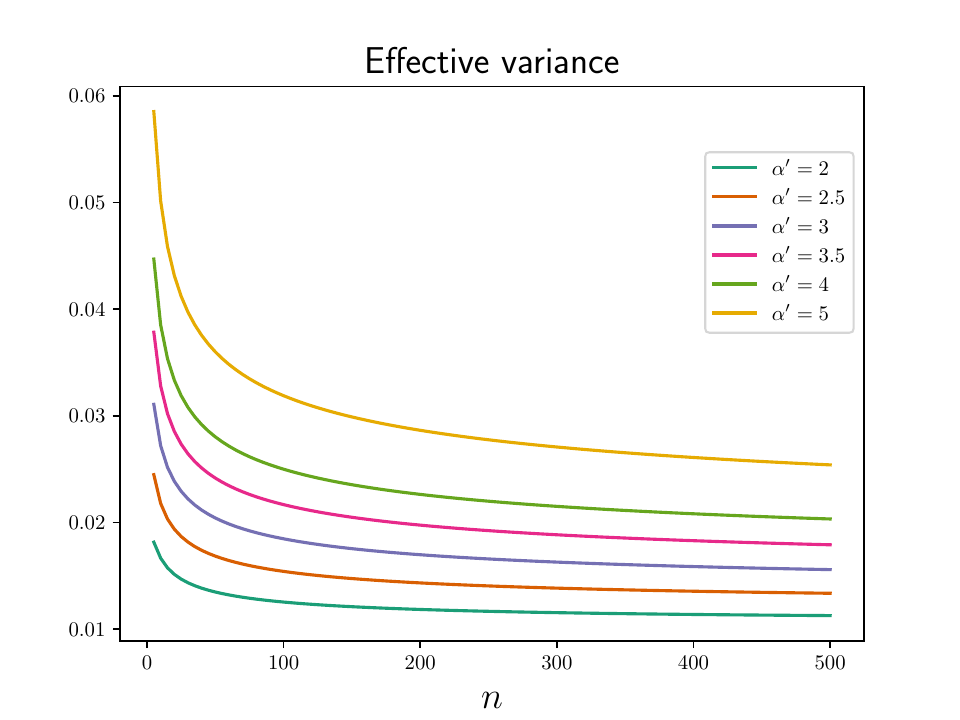} \tabularnewline
		\includegraphics[width=0.45\textwidth]{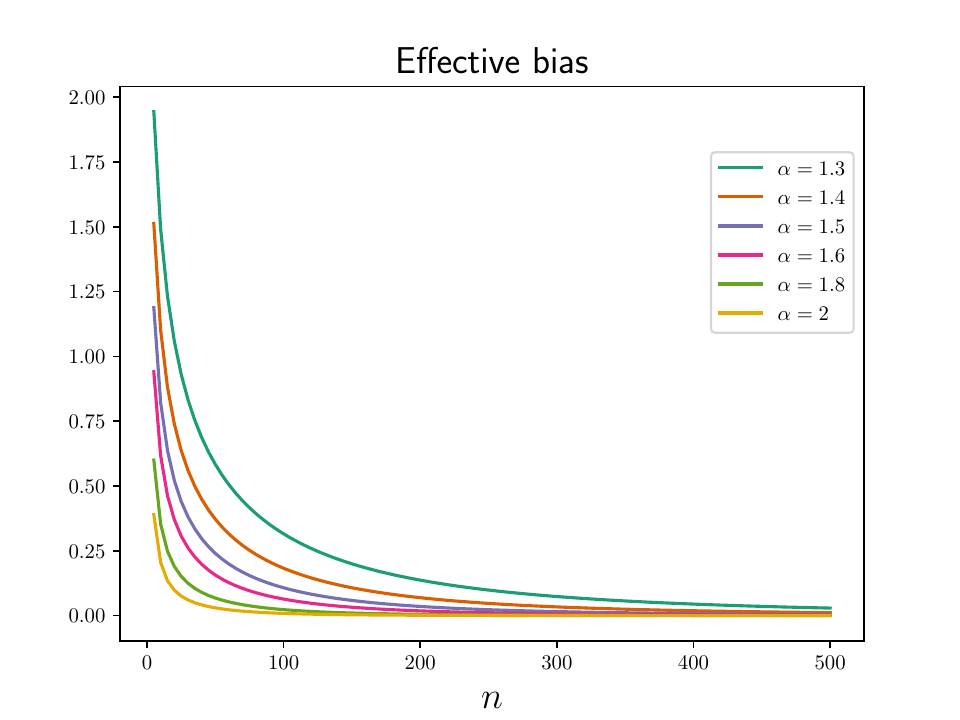} & \includegraphics[width=0.45\textwidth]{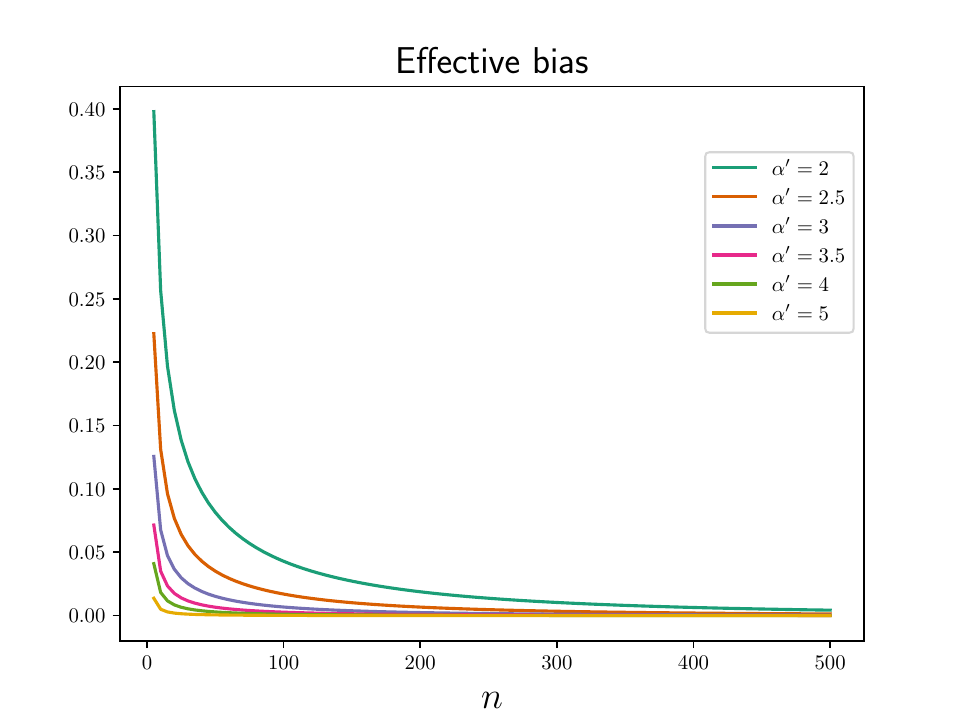} \tabularnewline \includegraphics[width=0.45\textwidth]{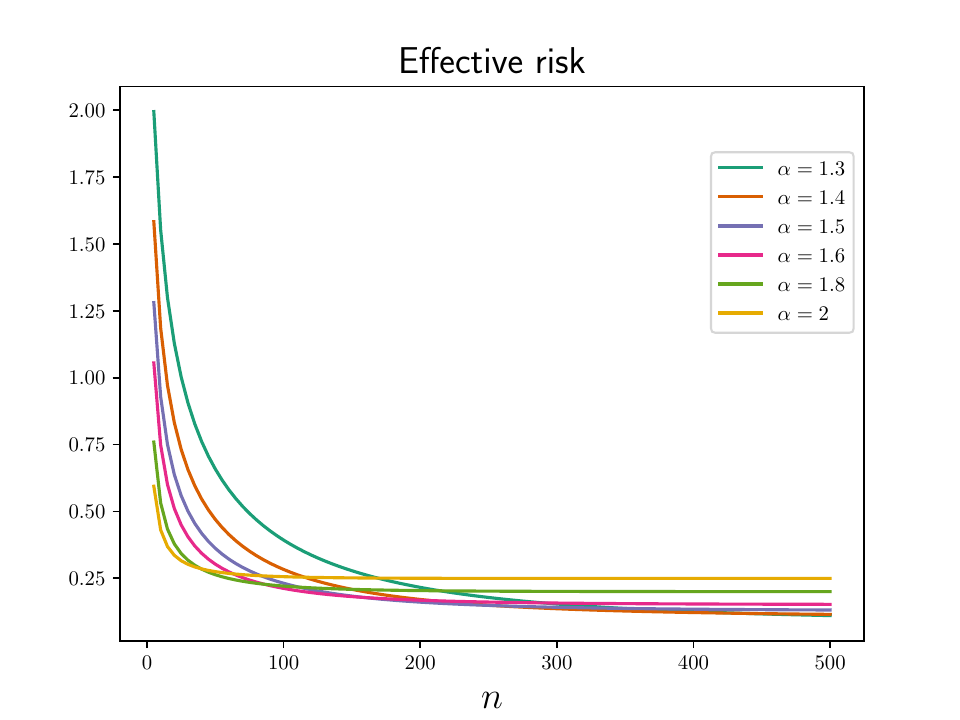} &  \includegraphics[width=0.45\textwidth]{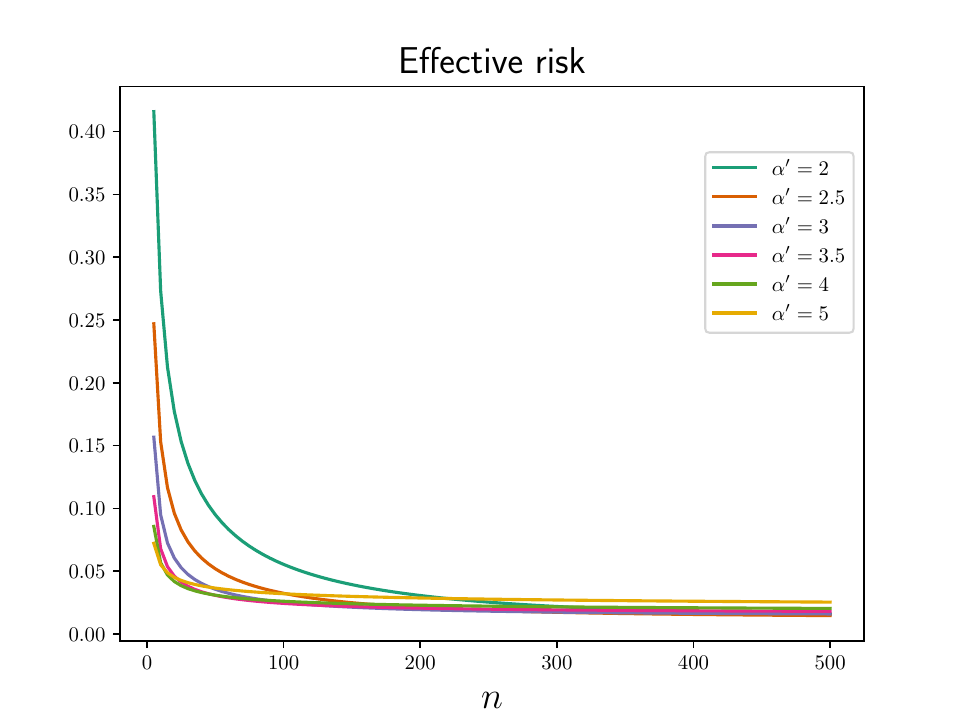} \tabularnewline
		\multicolumn{1}{c}{Model $(I)$: $\sigma_i = i^{-\alpha}$. } &
		\multicolumn{1}{c}{Model $(II)$: $\sigma_i = i^{-1}(1+\log i)^{-\alpha'}$. } \tabularnewline
	\end{tabular}
	\caption{Effective variance, bias, and risk of minimum norm interpolation
	(a.k.a. ridgeless regression) for two covariance structures defined as models $(I)$
	and $(II)$ (power law decay of the eigenvalues with exponents $\alpha>1$ and $\alpha=1$),
	as a function of the sample size $n$. 
	In model $(I)$ we 
	let the noise level to be $\vareps = 0.5$, and in $(II)$ we take $\vareps = 0.2$.} \label{fig:TheoryPrediction_n}
\end{figure}

\begin{figure}[htbp!]
	\centering
	\begin{tabular}{cc}
		\includegraphics[width=0.45\textwidth]{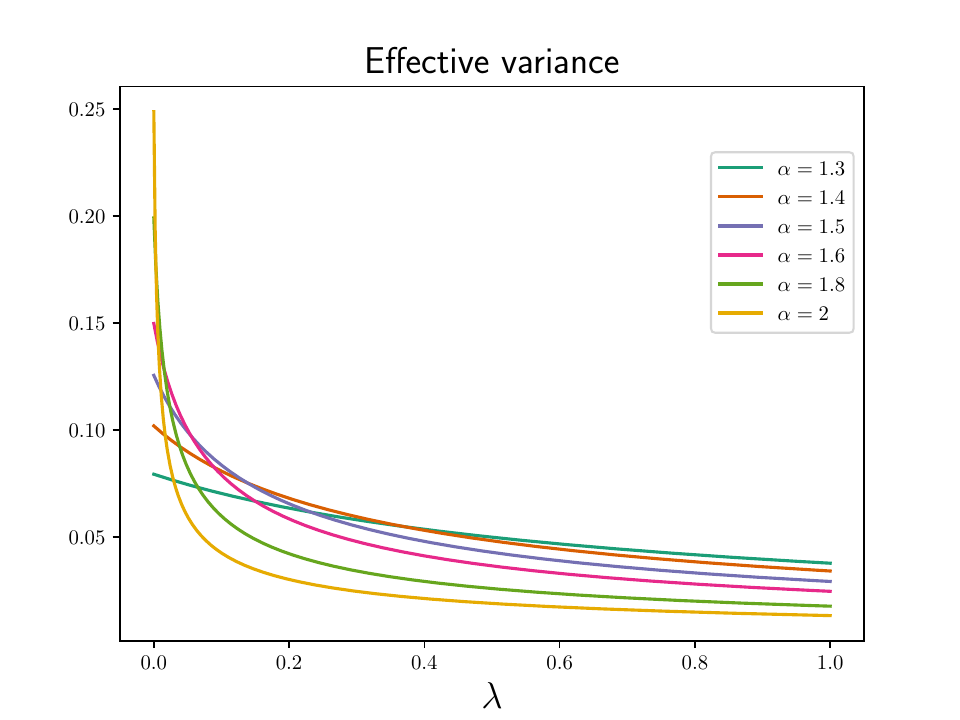} & \includegraphics[width=0.45\textwidth]{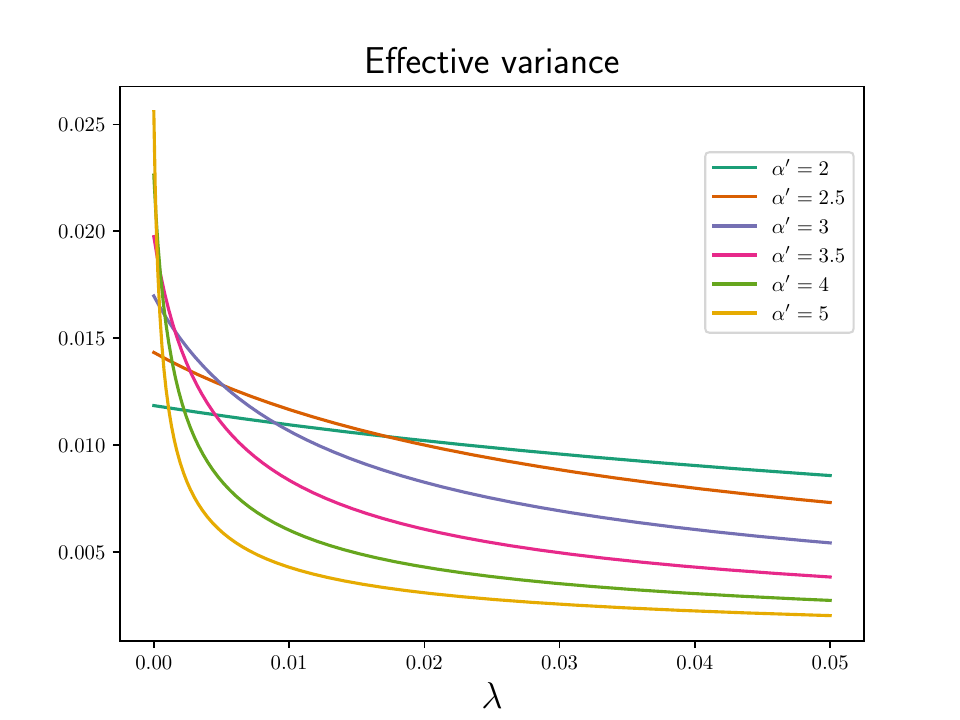} \tabularnewline
		\includegraphics[width=0.45\textwidth]{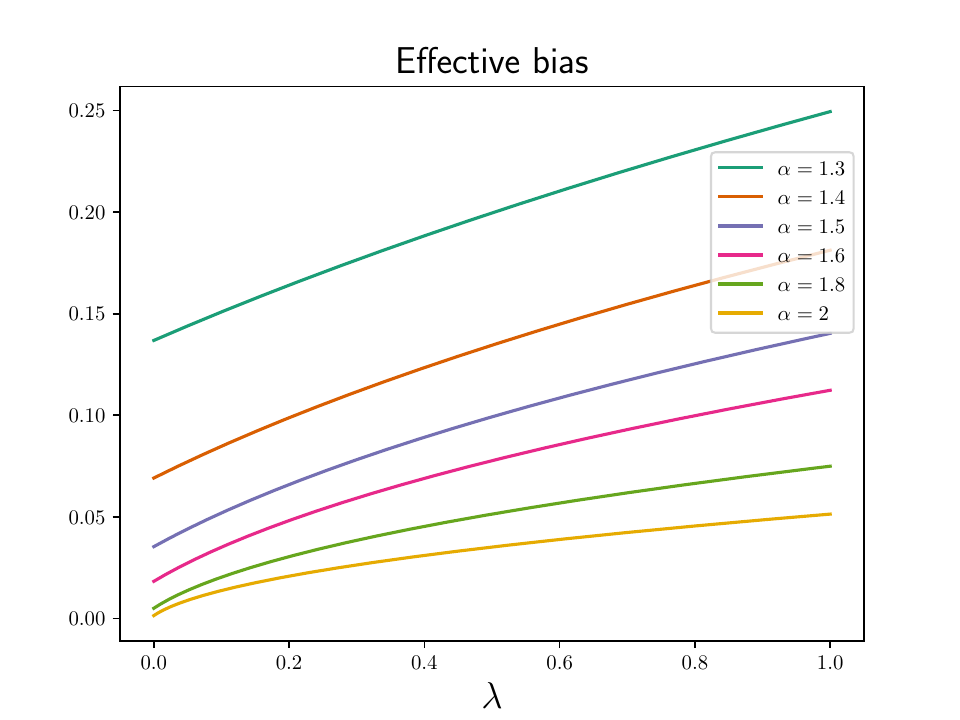} & \includegraphics[width=0.45\textwidth]{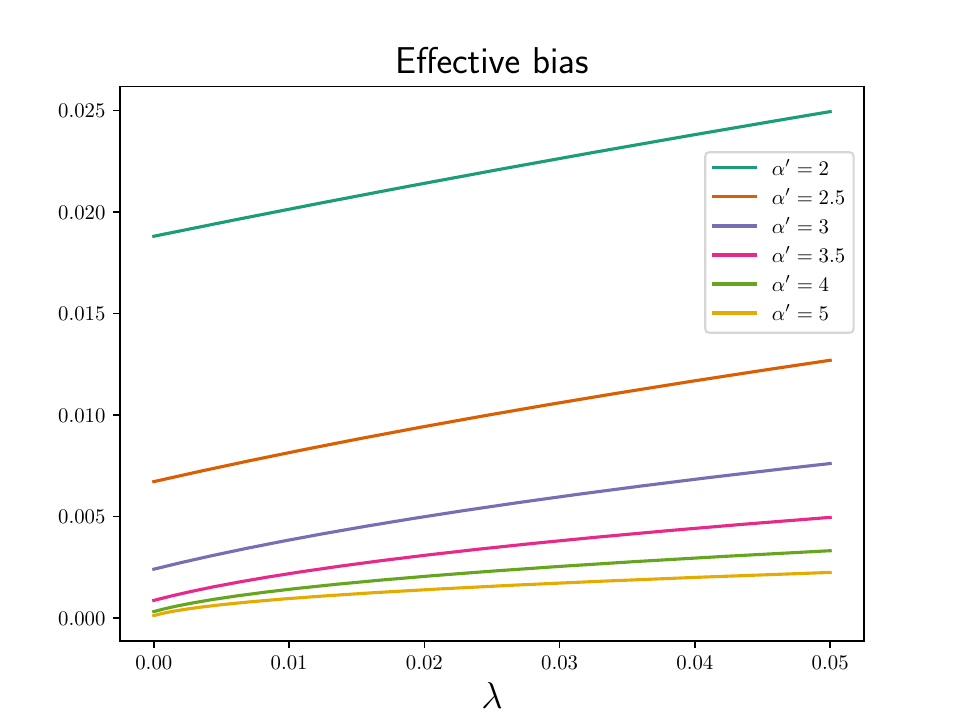} \tabularnewline \includegraphics[width=0.45\textwidth]{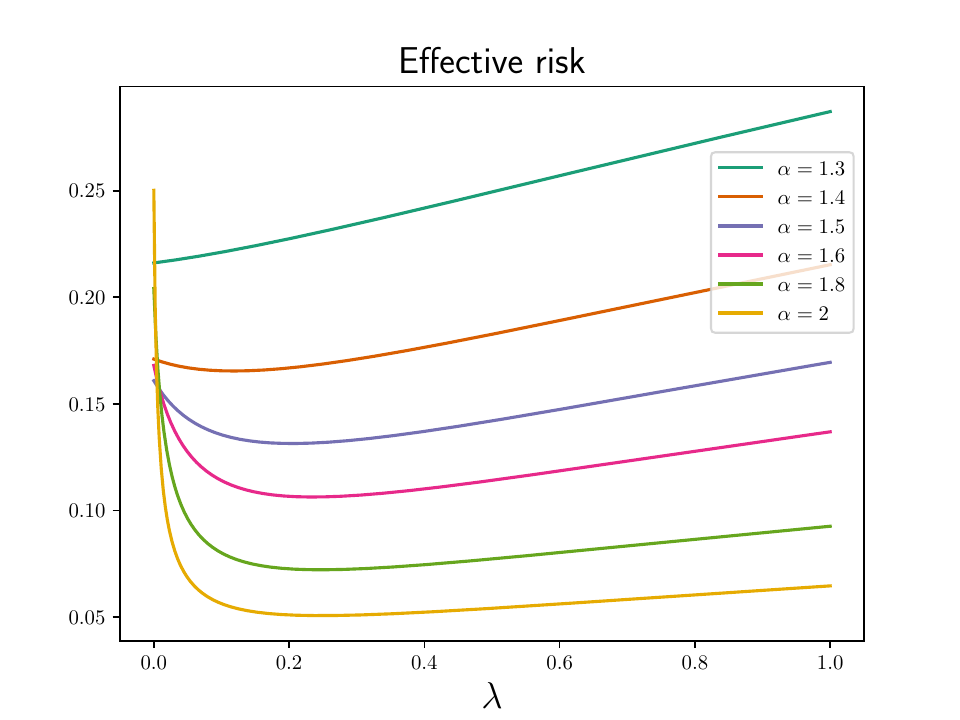} &  \includegraphics[width=0.45\textwidth]{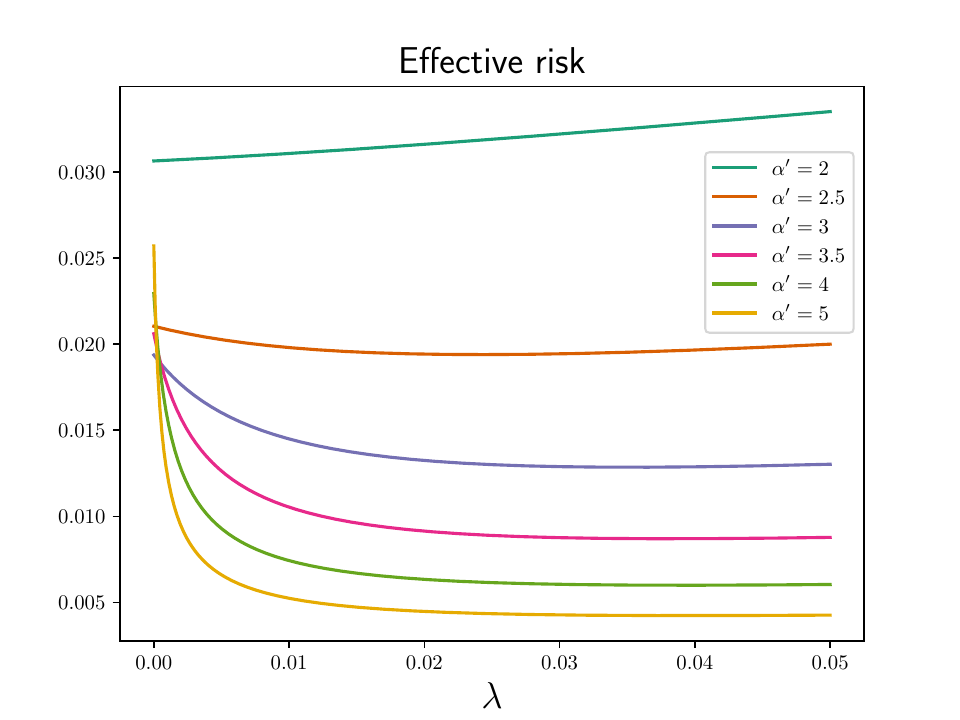} \tabularnewline
		\multicolumn{1}{c}{Model $(I)$: $\sigma_i = i^{-\alpha}$. } &
		\multicolumn{1}{c}{Model $(II)$: $\sigma_i = i^{-1}(1+\log i)^{-\alpha'}$. } \tabularnewline
	\end{tabular}
	\caption{Effective variance, bias, and risk of minimum norm interpolation
	 for two covariance structures defined as models $(I)$
	and $(II)$. Here we fix $n = 500$ and vary the regularization parameter. In model $(I)$ we 
	let the noise size to be $\vareps = 0.5$, and in $(II)$ we take $\vareps = 0.2$.} 
	 \label{fig:TheoryPrediction_lambda}
\end{figure}

\begin{figure}[htbp!]
	\centering
	\begin{tabular}{cc}
		\includegraphics[width=0.45\textwidth]{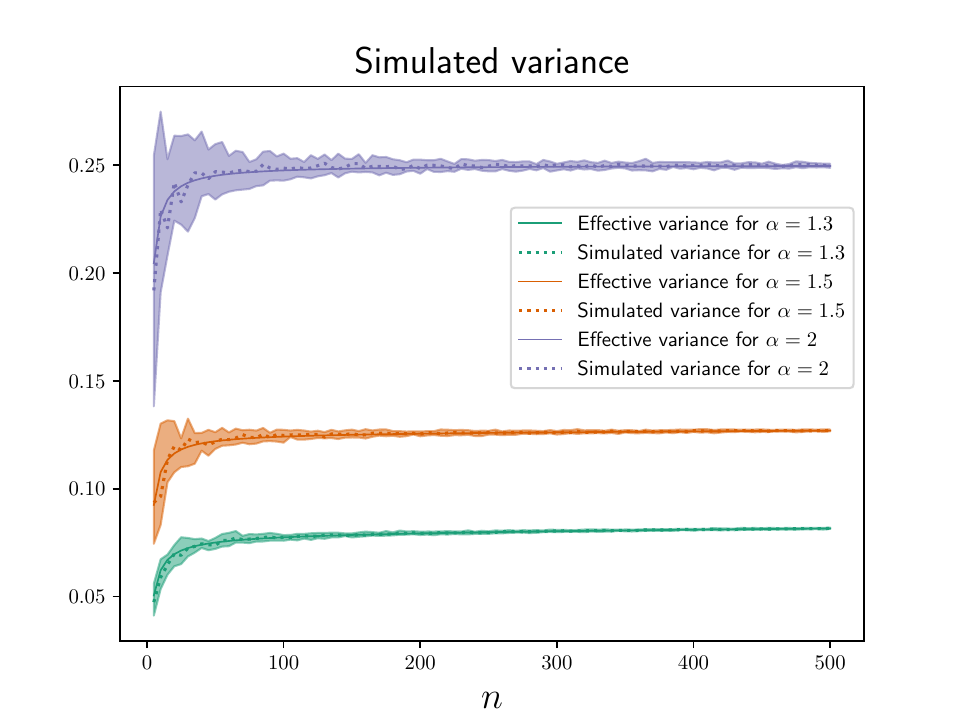} & \includegraphics[width=0.45\textwidth]{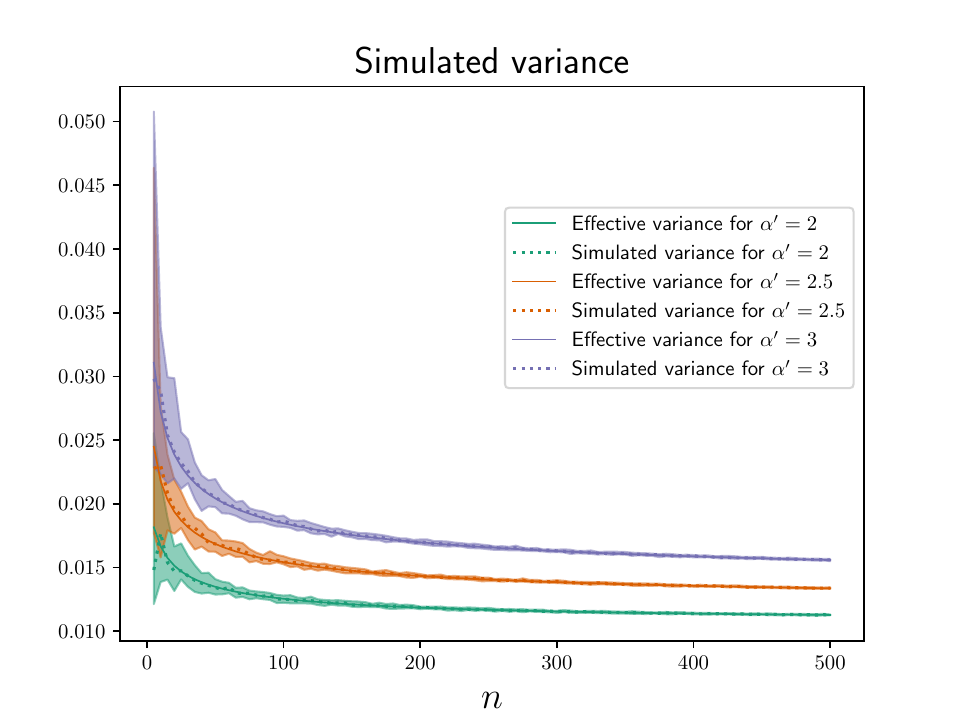} \tabularnewline
		\includegraphics[width=0.45\textwidth]{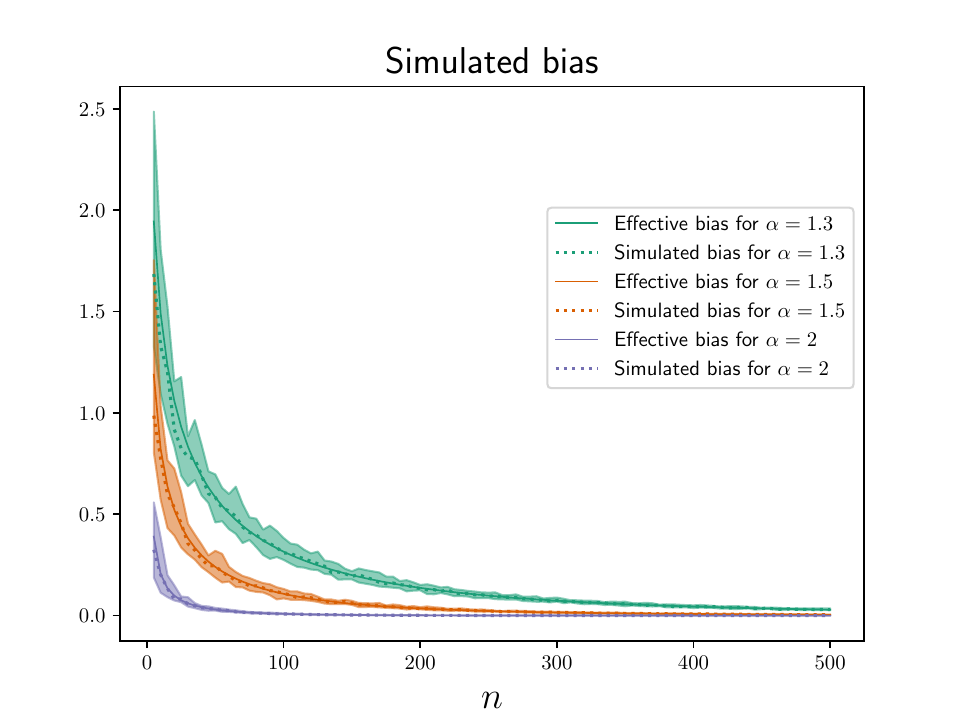} & \includegraphics[width=0.45\textwidth]{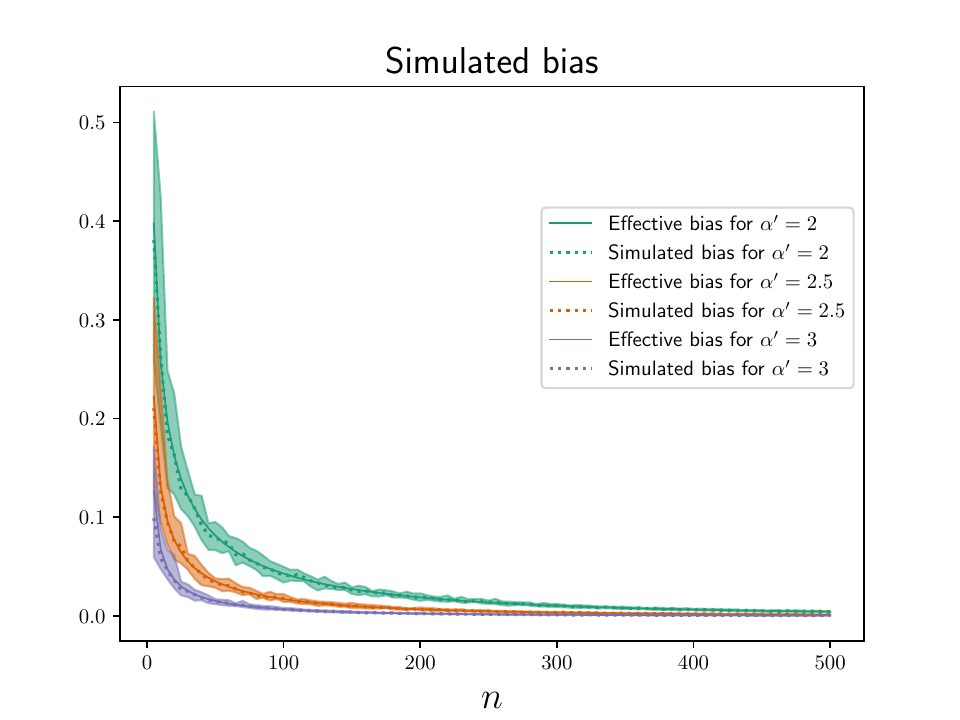} \tabularnewline \includegraphics[width=0.45\textwidth]{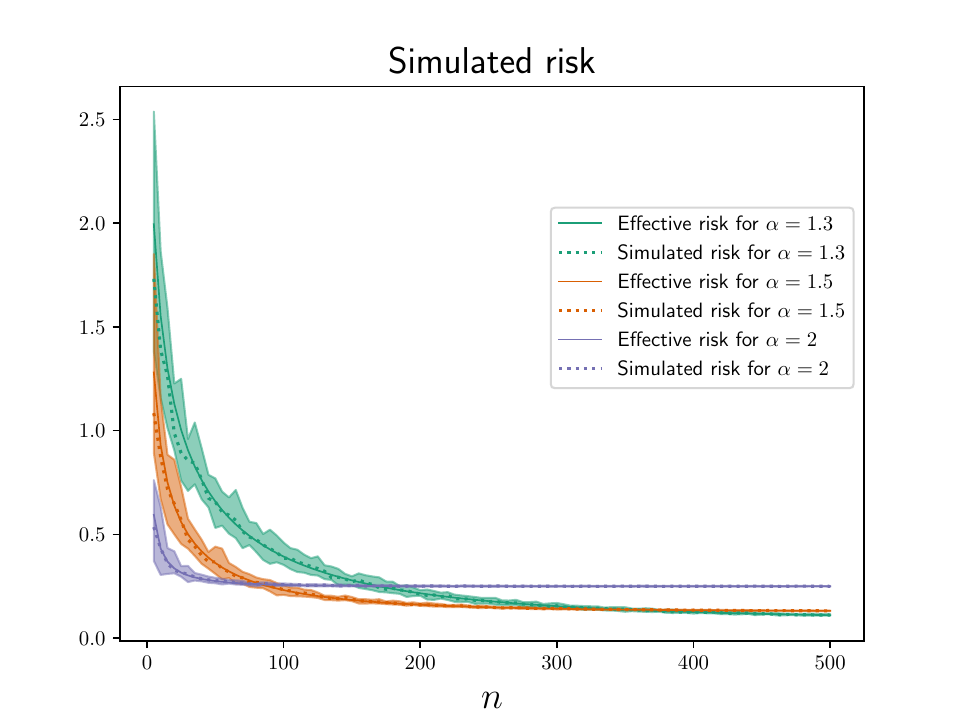} &  \includegraphics[width=0.45\textwidth]{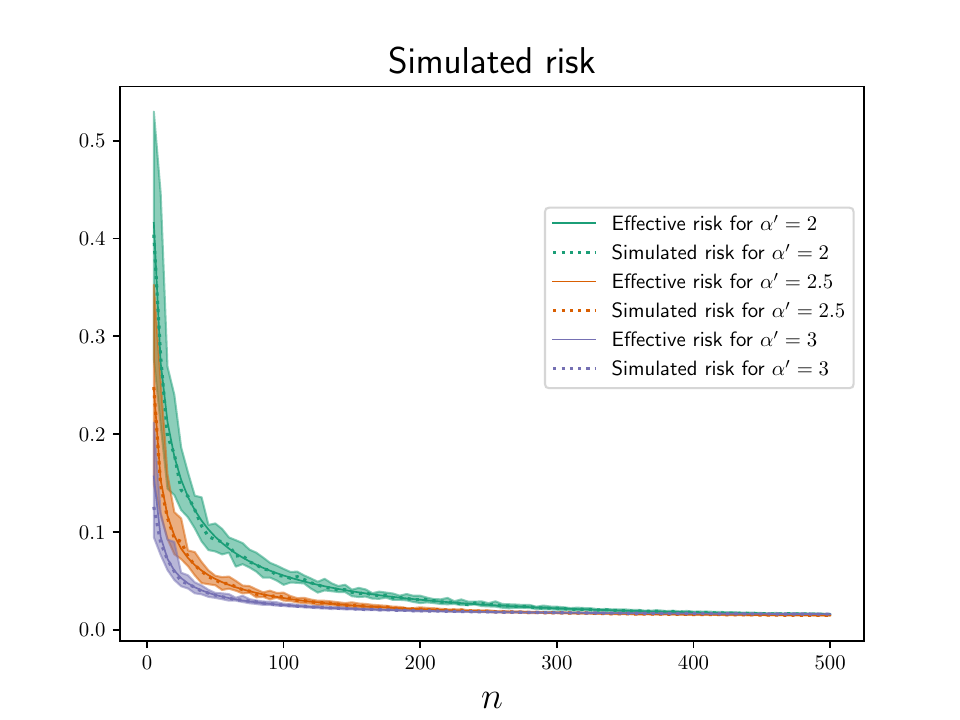} \tabularnewline
		\multicolumn{1}{c}{Model $(I)$: $\sigma_i = i^{-\alpha}$. } &
		\multicolumn{1}{c}{Model $(II)$: $\sigma_i = i^{-1}(1+\log i)^{-\alpha'}$. } \tabularnewline
	\end{tabular}
	\caption{Simulation results for variance, bias, and risk of minimum norm interpolation
	 for two covariance structures defined as models $(I)$
	and $(II)$.
	 In model $(I)$ we let the noise size to be $\vareps = 0.5$, and in $(II)$ we take $\vareps = 0.2$. 
	 For each $n$, we run $20$ independent trials and take the median for the dotted lines. 
	 We also show the shaded areas between $10\%$ and $90\%$ quantiles.} \label{fig:Simulations}
\end{figure}

The above discussion is based on evaluating the theoretical formulas
for bias and variance, as given in Eqs.~\eqref{eq:VAR-n},
\eqref{eq:BIAS-n}. While our main result, Theorems \ref{thm:main}, \ref{thm:main-ridgeless}
guarantee that these formulas are accurate, it is important how accurate they are at small
or moderate $n$, and whether random deviations modify the picture.

 In Figure \ref{fig:Simulations} we plot numerical simulations corroborating
 that $\var_\bX, \bias_\bX$ do 
 concentrate around $\VAR_n, \BIAS_n$ in models $(I)$ and $(II)$.
 As mentioned above, the predictions $\VAR_n, \BIAS_n$ appear to be accurate beyond what 
 is guaranteed by Theorem \ref{thm:main-ridgeless}, and the error appears to be 
 a $(1+o_n(1))$ multiplicative factor.
 
	\section{Proof of Theorem~\ref{thm:main}} \label{sec:proof-main}
Let $\mathcal{F}_k := \sigma(\bx_1, \cdots, \bx_k)$ be the $\sigma$-field generated by the 
first $k$ data points for $1 \leq k \leq n$, and $\mathcal{F}_0$ the trivial $\sigma$-field.
 We then have $\var_\bX, \bias_\bX \in \mathcal{F}_n$ and $\VAR_n, \BIAS_n \in \mathcal{F}_0$. 
 Extending the previous notation of $\Rfct_0$ in Eq.~\eqref{eq:R-F-func-0} to $\Rfct_k$, we let
\begin{align} \label{eq:R-F-func}
	\Rfct_k(\zeta, \mu; \bQ) = \Tr \prn{\bSigma^\half \bQ \bSigma^\half (\zeta \bI + \mu \bSigma + \bX_k^\sT \bX_k )^{-1}} \, , \qquad \Ffct_k(\zeta, \mu; \bQ) = \zeta \Rfct_k(\zeta, \mu; \bQ) \, ,
\end{align}
where $\zeta > 0, \mu \geq 0$, $\bQ$ is a p.s.d.\ matrix with bounded spectral norm, and 
$\bX_k = [\bx_1, \cdots, \bx_k]^\sT \in \reals^{k \times d}$ is the partial data matrix
comprising the first $k$ rows of $\bX$. By convention we set
 $\bX_0^\sT \bX_0 := \boldsymbol{0}$ when $k=0$. An immediate consequence is that 
 $\Rfct_k, \Ffct_k \in \mathcal{F}_k$. Define $\muefct := \muefct(\zeta, \mu)$ as 
 the unique  solution on of the following equation on $(\mu,\infty)$
\begin{align} \label{eq:mu-fixed-point}
	\muefct = \mu + \frac{n}{1 + \Rfct_0(\zeta, \muefct; \bI)} \, .
\end{align}
 For $\zeta = n \lambda$ and $\mu=0$, this equation reduces to Eq.~\eqref{eq:lambda-fixed-point}, via
 the change of variables $\muefct = n\lambda/ \lambdaefct$. For $\mu>0$ existence and
 uniqueness follows by a similar argument to the case $\mu=0$. 
 Indeed, setting $\xi:=(\muefct-\mu)^{-1}$,
 the  equation is equivalent to $n\xi = 1+\Tr(\bSigma(\bA+\xi^{-1}\bSigma)^{-1})$, 
 where $\bA:= n\lambda\bI+\mu\bSigma$. It is further equivalent to $n - \xi^{-1} = \Tr(\bSigma(\xi\bA+\bSigma)^{-1})$. Existence and uniqueness follow since the left-hand side is monotone 
 increasing and the right-hand side monotone decreasing in $\xi$. 

In order to quantify the approximation errors $|\var_\bX - \VAR_n|$ and $|\bias_\bX - \BIAS_n|$,
 we will apply the following lemma (Lemma~\ref{lem:var-bias-formula}), which 
 expresses the bias and variance $\bias_\bX$, $\var_\bX$
 in terms of derivatives of $\Ffct_n$ and $\Ffct_0$ w.r.t. $\lambda$ and $\mu$. 
\begin{lemma} \label{lem:var-bias-formula}
	For any $\lambda > 0, \mu \geq 0$, the quantity $\muefct > \mu$ is uniquely determined and we have at the point $(\zeta, \mu) = (n\lambda, 0)$,
	\begin{align*}
		\var_\bX(\lambda) & = \vareps^2 \cdot \frac{\partial}{\partial \zeta} \Ffct_n(\zeta, \mu; \bI) \, , &  \bias_\bX(\lambda) & = - \zeta \cdot  \frac{\partial}{\partial \mu} \Ffct_n(\zeta, \mu; \btheta \btheta^\sT) \, ; \\
		\VAR_n(\lambda) & = \vareps^2 \cdot \frac{\partial}{\partial \zeta} \Ffct_0(\zeta, \muefct(\zeta, \mu); \bI)  \, , &  \BIAS_n(\lambda) & = - \zeta \cdot  \frac{\partial}{\partial \mu} \Ffct_0(\zeta, \muefct(\zeta ,\mu); \btheta \btheta^\sT) \, ,
	\end{align*}
	and $\muefct(n\lambda, 0) = n\lambda / \lambdaefct$.
\end{lemma}
The proof of this lemma follows by differentiation of  
the definition \eqref{eq:R-F-func} and using Eqs.~\eqref{eq:var-X} and  \eqref{eq:bias-X}.
We refer to Appendix  \ref{proof:var-bias-formula} for details.

Our proof strategy proceeds in four parts: (I)~We  show that ---due to the regularity properties
of $\Ffct_0$ and $\Ffct_n$--- a bound on  $|\Ffct_0 - \Ffct_n|$
implies a bound on the difference of their derivatives, and hence
(via Lemma \ref{lem:var-bias-formula}) on the error in approximating bias and variance;
(II)~We prove a bound on $|\Ffct_0 - \Ffct_n|$
interpolating between $\Ffct_0$ and $\Ffct_n$ by adding one row at the time to $\bX$;
(III) and (IV)~We apply these general bounds to controlling variance and bias,  respectively .

Recall that we defined $\btheta:= \bSigma^{-1/2}\boldbeta$, and assumed $\|\btheta\|<\infty$.
 By homogeneity, we can and will assume $\norm{\btheta} = 1$ throughout the proof.

\paragraph{Part I: Reduction to function values approximation} The following 
lemma reduces controlling the difference of derivatives of $\Ffct_0$ and $\Ffct_n$ to the 
less arduous task of bounding the difference in function values. 
Its proof is presented in Appendix~\ref{proof:function-value-bound-to-derivative}.
\begin{lemma} \label{lem:function-value-bound-to-derivative}
	For any fixed $k \in \naturals, \delta \in \realsp$ and a $(k+1)$-times continuously differentiable function $f(t)$ on $[0, k\delta]$, we have
	\begin{align*}
		\left|f'(0)\right| \leq \bigO_k \prn{\frac{\max_{0 \leq j \leq k} |f(j\delta)|}{\delta} + \sup_{t \in [0, k \delta]} |f^{(k+1)}(t)| \cdot \delta^{k}}. 
	\end{align*}
\end{lemma}

With the help of Lemmas~\ref{lem:var-bias-formula} and \ref{lem:function-value-bound-to-derivative}, for any $\delta \in \realsp$, we can upper bound the variance and bias approximations by
\begin{align}
	\left|\var_\bX(\lambda) - \VAR_n(\lambda)\right| & = \bigO_k \Bigg(\vareps^2 \cdot \max_{0 \leq j \leq k} \frac{\left| \Ffct_n(n\lambda + j \delta, 0; \bI) -  \Ffct_0(n\lambda + j \delta, \muefct(n\lambda + j \delta, 0); \bI)\right|}{\delta} \nonumber \\
	&\qquad  + \vareps^2 \delta^{k}\cdot \sup_{\zeta' \in [n\lambda, n\lambda + k\delta]} \left|\frac{\partial^{k+1}}{\partial {\zeta'}^{k+1}} \Ffct_n(\zeta', 0; \bI) - \frac{\partial^{k+1}}{\partial {\zeta'}^{k+1}} \Ffct_0(\zeta', \muefct(\zeta', 0); \bI)\right|\Bigg) \, , \label{eq:variance-approx-bound-mid-1}
\end{align}
and
\begin{align}
	\left|\bias_\bX(\lambda) - \BIAS_n(\lambda)\right| & = \bigO_k \Bigg(n\lambda \cdot \max_{0 \leq j \leq k} \frac{\left| \Ffct_n(n\lambda, j \delta; \btheta \btheta^\sT) -  \Ffct_0(n\lambda , \muefct(n\lambda, j \delta); \btheta \btheta^\sT)\right|}{\delta} \nonumber \\
	&\qquad  + n\lambda \delta^{k}\cdot \sup_{\mu' \in [0, k\delta]} \left|\frac{\partial^{k+1}}{\partial {\mu'}^{k+1}} \Ffct_n(n\lambda, \mu'; \btheta \btheta^\sT) - \frac{\partial^{k+1}}{\partial {\mu'}^{k+1}} \Ffct_0(n\lambda, \muefct(n\lambda, \mu'); \btheta \btheta^\sT)\right|\Bigg) \, . \label{eq:bias-approx-bound-mid-1}
\end{align}
Before passing to bounding errors in function values, we provide upper bounds 
for higher order derivatives in Eqs.~\eqref{eq:variance-approx-bound-mid-1} and
 \eqref{eq:bias-approx-bound-mid-1}. Bounding the derivatives of
  $\Ffct_n$ is easier as we can easily write an explicit formula 
  for the $k$-th derivative for any $k$. 
(The proof of this lemma is presented in Appendix~\ref{proof:derivative-F-n}).
\begin{lemma}\label{lem:derivative-F-n}
	For any fixed $k \in \naturals$, we have for all $\zeta > 0$ and $\mu \geq 0$,
	\begin{align*}
		\left|	\frac{\partial^k}{\partial \zeta^k}\Ffct_n(\zeta, 0; \bI) \right|  =\bigO_k \prn{\frac{\Ffct_n(\zeta, 0; \bI) }{\zeta^k}} \, , \qquad \left|	\frac{\partial^k}{\partial \mu^k}\Ffct_n(\zeta, \mu; \btheta \btheta^\sT) \right|  = \bigO_k \prn{\frac{\Ffct_n(\zeta, \mu; \btheta \btheta^\sT) }{\zeta^k}}\, .
	\end{align*}
\end{lemma}
\noindent Computing higher order derivatives of $\Ffct_0$ is less straightforward 
because $\Ffct_0$ depends on $\muefct$ which itself depends implicitly depending on $(\zeta, \mu)$. 
 We postpone this proof to Appendix~\ref{proof:derivative-F-0}.
\begin{lemma} \label{lem:derivative-F-0}
	Let Eq.~\eqref{asmp:lambda} hold. Then, for any fixed $k \in \naturals$, we have for all $\zeta > 0$,
	\begin{align*}
		\left|	\frac{\partial^k}{\partial \zeta^k}\Ffct_0(\zeta, \muefct(\zeta, 0); \bI) \right|  = \bigO_{k} \prn{\frac{\Ffct_0(\zeta, \muefct(\zeta, 0); \bI)}{\zeta^k \const^{2k}}} \, ,
	\end{align*}
	and for all $\mu$ such that $0 \leq \mu \leq \muefct(\zeta, \mu)/2$,
	\begin{align*}
		\left|\frac{\partial^k}{\partial \mu^k}\Ffct_0(\zeta, \muefct(\zeta, \mu); \btheta \btheta^\sT)\right| = \bigO_{k} \prn{\frac{\Ffct_0(\zeta, \muefct(\zeta, \mu); \btheta \btheta^\sT)}{\muefct(\zeta, \mu)^k \const^{2k}}}\, .
	\end{align*}
\end{lemma}

\paragraph{Part II: Bounding errors in function values}   We next proceed to bounding
 $|\Ffct_n(\zeta, \mu; \bQ) - \Ffct_0(\zeta, \muefct(\zeta, \mu); \bQ)|$ for a p.s.d.\ matrix 
 $\bQ$, which appears in Eqs.~\eqref{eq:variance-approx-bound-mid-1} and 
 \eqref{eq:bias-approx-bound-mid-1}. Recall that $\Ffct_i(\zeta, \mu; \bQ) =
  \zeta \Rfct_i(\zeta, \mu; \bQ)$.

  The next theorem bounds $|\Rfct_n(\zeta, \mu; \bQ) - \Rfct_0(\zeta, \muefct(\zeta, \mu); \bQ)|$
and   is the most important technical step in the proof of our main theorems.
Its proof is outlined in Section \ref{proof:R-approximation},
with several technical lemmas deferred to the appendices
\begin{theorem} \label{thm:R-approximation}
Introduce the shorthand $\Riter_0(\bQ) := 
   \Rfct_0(\zeta, \muefct(\zeta, \mu); \bQ)$.
	Under Assumption~\ref{asmp:data-dstrb}, for any $\zeta >0, \mu \geq 0$, p.s.d.\ matrix $\bQ$ with $\|\bQ\| = 1$ and positive integer $D$, there exists constants $\eta = \eta(\constantx) \in (0, 1/2)$, $\constant_\alpha = \constant_\alpha(\constantx, D) > 0$, $\constant_\beta = \constant_\beta(\constantx, D) > 0$ and $\constant_\gamma = \constant_\gamma(\constantx, D)$ such that for
	\begin{align*}
		\gamma  &:= \min \brc{\frac{2}{n} \prn{1 + \frac{\constant_\gamma\constantsig \sigma_{\lfloor \eta n\rfloor} \cdot \log n \log (\constantsig n)}{\zeta}} + \frac{2}{\muefct(\zeta, \mu)} , \frac{1}{\zeta}} \, , \\
		\alpha_1 &:=  \constant_\alpha \log n \cdot \sqrt{\gamma \Riter_0(\bI)}\, , \\
		\alpha_2&  :=  \constant_\alpha \log n \cdot \sqrt{\gamma^3 \Riter_0(\bQ)} \, , \\
		\beta_1 & := \constant_\beta \prn{\sqrt{n \log n} \cdot \frac{\alpha_1 \gamma \Riter_0(\bQ) + \alpha_2 (1 + \Riter_0(\bI))}{1 + \Riter_0(\bI)^2} + n \cdot \brc{\frac{\gamma^2 \Riter_0(\bQ) + \alpha_1 \alpha_2}{1 + \Riter_0(\bI)^2} + \frac{\alpha_1^2  \gamma \Riter_0(\bQ)}{1 + \Riter_0(\bI)^3}} + \frac{\gamma \Riter_0(\bQ)}{1 + \Riter_0(\bI)}}  \, , \\
		\beta_2 & := \frac{\constant_\beta n \beta_1}{1 + \Riter_0(\bI)^2}\, ,
	\end{align*}
	if $\alpha_1 \leq \Riter_0(\bI)/8$, $\beta_1 \leq \Riter_0(\bQ)/64$, $\gamma \beta_2(1 + \Riter_0(\bI)) \leq 1/64$ and $n^{-D} = \bigO(\alpha_1/(1 + \Riter_0(\bI)))$, for all $n = \Omega_D(1)$ with probability $1-\bigO(n^{-D+1})$ we have
	\begin{align*}
		|\Rfct_n(\zeta, \mu; \bQ) - \Rfct_0(\zeta, \muefct(\zeta, \mu); \bQ)|
		& = \bigO \prn{\gamma \beta_2 \prn{1 +   \Rfct_0(\zeta, \muefct(\zeta, \mu); \bI)}  \Rfct_0(\zeta, \muefct(\zeta, \mu); \bQ) + \beta_1 }  \, .
	\end{align*}
\end{theorem}

Let us emphasize that this theorem holds under weaker assumptions
than Theorem \ref{thm:main}, but the error bounds it provides are quite implicit.
We can obtain more explicit bounds by imposing the assumptions of  Theorem~\ref{thm:main}. We first define the generalized version of $\rho(\lambda)$ in Eq.~\eqref{eq:RhoLambda} for any p.s.d.\ matrix $\bQ$ as
\begin{align} \label{eq:RhoLambda-General}
	\rho(\lambda) := 
	\frac{\Rfct_0(\lambdaefct, 1; \bQ/\norm{\bQ})}{\Rfct_0(\lambdaefct,1; \bI)} \in (0, 1] \, .
\end{align} 
The proof of this corollary is given in Appendix~\ref{proof:R-approximation-simplified}.

\begin{corollary} \label{cor:R-approximation-simplified}
Under Assumption~\ref{asmp:data-dstrb}, 
	for any positive integers $k$, $D$ and p.s.d.\ matrix $\bQ$ with $\|\bQ\| = 1$, there exist constants 
	$\eta = \eta(\constantx) \in (0, 1/2)$
	 and $\constant = \constant(\constantx, D) > 0$, such that the following hold.
	 Define  $\chi_n (\lambda), \const,
	\rho(\lambda)$  as per Eqs.~\eqref{eq:def-chi-n}, \eqref{asmp:lambda},
	\eqref{eq:RhoLambda-General} (those quantities are defined for $\mu = 0$).
	If 
	it holds that $\muefct(n\lambda, \mu)
	\leq (1 - \const/2)^{-1} \muefct(n\lambda, 0)$, and
	\begin{align*}
		&\chi_n (\lambda)^3  \log^2 n \le \constant n \const^{4.5} \sqrt{\rho(\lambda)} \, , \qquad  n^{-2D + 1} = \bigO \prn{ \sqrt{\frac{\const^3 \log^2 n}{n\max \brc{1, \lambda}}}} \, ,
	\end{align*}
	we then have for all $n = \Omega_D(1)$ with probability $1-\bigO(n^{-D+1})$ that
	\begin{align*}
		& |\Rfct_n(n\lambda, \mu; \bQ) - \Rfct_0(n\lambda, \muefct(n\lambda, \mu); \bQ)|  \leq \mc{E}_n \cdot  \Rfct_0(n\lambda, \muefct(n\lambda, \mu); \bQ),
	\end{align*}
	where
	\begin{align} \label{eq:def-E-n}
		\mc{E}_n & = \bigO_{\constantx, D} \prn{\frac{\chi_n (\lambda)^3 \log^2 n}{n \const^{6.5}} \cdot \sqrt{\frac{\Riter_0(\bI)}{\Riter_0(\bQ)}}} \, .
	\end{align}
\end{corollary}

To further simplify the assumption  $\muefct(n\lambda, \mu) \leq (1 - \const/2)^{-1} \muefct(n\lambda, 0)$ in 
Corollary~\ref{cor:R-approximation-simplified}, the next lemma will be helpful. 
We defer its proof to Appendix~\ref{proof:muefct-small-mu-bound}.
\begin{lemma} \label{lem:muefct-small-mu-bound}
	For any fixed $\zeta = n\lambda > 0$, the function $\muefct(\zeta, \mu)$ is increasing in $\mu$ for all $\mu \geq 0$. Assuming Eq.~\eqref{asmp:lambda}, if $0 \leq \mu \leq n \const^3 /2$, then
	\begin{align*}
		\muefct(\zeta, \mu) \leq (1 - \const/2)^{-1} \muefct(\zeta, 0).
	\end{align*}
\end{lemma}

\paragraph{Part III: Approximation error for variance} We are now ready to combine our results 
in Part I and Part II to obtain approximation errors $|\var_\bX - \VAR_n|$ and $|\bias_\bX - \BIAS_n|$. 
For the variance, we want to take $\delta$ in Eq.~\eqref{eq:variance-approx-bound-mid-1} such that
 $k \delta \leq n\lambda$. In this case, $[n\lambda, n\lambda + k\delta] \subset [n\lambda, 2n\lambda]$.
  Note that $\lambdaefct(\lambda)$ is an increasing function of $\lambda$. Further, by
\begin{align*}
	n \cdot \prn{1 - \frac{\lambda}{\lambdaefct}} =  \Tr \prn{\bSigma(\bSigma + \lambdaefct \bI)^{-1}} \, ,
\end{align*}
we know $\lambda\mapsto \muefct(n\lambda,0) =  n\lambda/\lambdaefct$ is an increasing function. Therefore
 $\lambda/\lambdaefct(\lambda) \leq 2\lambda /\lambdaefct(2\lambda)$, which implies $\lambdaefct(\lambda) \leq \lambdaefct(2\lambda) \leq 2 \lambdaefct(\lambda)$. For $\lambda$ that satisfies 
Eq.~\eqref{asmp:lambda}, this guarantees that for any $\lambda' \in [\lambda, 2\lambda]$,
\begin{align*}
	\frac{\lambda'}{ \lambdaefct(\lambda')} \ge    \frac{\lambda}{\lambdaefct(\lambda)} \geq \const \, ,
\end{align*}
and
\begin{align*}
	1 - \frac{\lambda'}{\lambdaefct(\lambda')} & = \frac{1}{n} \Tr \prn{\bSigma(\bSigma + \lambdaefct(\lambda') \bI)^{-1}} \geq \frac{1}{n} \Tr \prn{\bSigma(\bSigma + 2\lambdaefct(\lambda) \bI)^{-1}} \geq \frac{1}{2} \cdot  \frac{1}{n} \Tr \prn{\bSigma(\bSigma + \lambdaefct(\lambda) \bI)^{-1}} \nonumber \\
	& = \frac{1}{2} \prn{1 -\frac{\lambda}{ \lambdaefct(\lambda)}  } \geq \const/2 \, .
\end{align*}
Hence, for any $\lambda' \in [\lambda, 2\lambda]$, Eq.~\eqref{asmp:lambda} still holds but 
with constant $\const' \ge \const/2$. Therefore, we can apply
Corollary~\ref{cor:R-approximation-simplified} for any
 $\lambda' \in [\lambda, 2 \lambda]$ for $\bQ = \bI$ and $\mu = 0$, provided the
 following conditions hold
\begin{align*}
	 \chi_n (\lambda')^3  \log^2 n \le \constant n (\const/2)^{4.5} \sqrt{\rho(\lambda')} = \constant n (\const/2)^{4.5} \, ,
\end{align*} 
where the last equality used the fact that $\rho(\lambda') = 1$ when $\bQ = \bI$. Finally, setting $\constant' := 2^{-4.5} \constant$ and using the fact that $\chi_n(\lambda')$ is decreasing in $\lambda$, it suffices to require
\begin{align*}
	\chi_n (\lambda)^3  \log^2 n \le \constant' n \const^{4.5} \, ,
\end{align*}
which holds by the theorem's assumptions.

Hence, we can now apply Corollary~\ref{cor:R-approximation-simplified} with $\bQ = \bI$, and it follows that with probability $1 - \bigO_k(n^{-D+1})$,
\begin{align*}
	& \max_{0 \leq j \leq k} \frac{\left| \Ffct_n(n\lambda + j \delta, 0; \bI) -  \Ffct_0(n\lambda + j \delta, \muefct(n\lambda + j \delta, 0); \bI)\right|}{\delta} \nonumber \\
	& \leq \frac{2 n\lambda \mc{E}_n}{\delta} \cdot  \max_{0 \leq j \leq k } \Rfct_0(n\lambda + j \delta, \muefct(n\lambda + j \delta, 0); \bI)  \nonumber \\
	&\leq \frac{2\mc{E}_n}{\delta} \Ffct_0(n\lambda, \muefct(n\lambda, 0); \bI)  \, ,
\end{align*}
where in the last inequality we use that $\Rfct_0(\zeta , \muefct(\zeta, 0); \bI)  = n/ \muefct(\zeta, 0) - 1$ 
is a decreasing function in $\zeta$ as $\muefct(\zeta, 0)$ is increasing in $\zeta$. 
Next by Lemmas~\ref{lem:derivative-F-n} and \ref{lem:derivative-F-0} we obtain
\begin{align*}
	& \sup_{\zeta' \in [n\lambda, n\lambda + k\delta]} \left|\frac{\partial^{k+1}}{\partial {\zeta'}^{k+1}} \Ffct_n(\zeta', 0; \bI) - \frac{\partial^{k+1}}{\partial {\zeta'}^{k+1}} \Ffct_0(\zeta', \muefct(\zeta', 0); \bI)\right| \nonumber \\
	& = \bigO_{k} \prn{ \sup_{\lambda' \in [\lambda, 2\lambda]}  \frac{\Ffct_n(n\lambda', 0; \bI)  + \Ffct_0(n\lambda', \muefct(n\lambda', 0); \bI)}{n^{k+1}{\lambda'}^{k+1} \const^{2k+2}}}  \nonumber \\
	& = \bigO_{k} \prn{ \sup_{\lambda' \in [\lambda, 2\lambda]}  \frac{n\lambda' \Rfct_n(n\lambda', 0; \bI)  + n\lambda' \Rfct_0(n\lambda', \muefct(n\lambda', 0); \bI)}{n^{k+1}\lambda^{k+1} \const^{2k+2}}} \nonumber \\
	& \stackrel{\mathrm{(i)}}{=} \bigO_{k} \prn{ \frac{n\lambda \Rfct_n(n\lambda, 0; \bI)  + n\lambda \Rfct_0(n\lambda, \muefct(n\lambda, 0); \bI)}{n^{k+1}\lambda^{k+1} \const^{2k+2}}} \nonumber \\
	& \stackrel{\mathrm{(ii)}}{=}  \bigO_{k} \prn{\frac{(1 + \mc{E}_n) \Ffct_0(n\lambda, \muefct(n\lambda, 0); \bI)}{n^{k+1}\lambda^{k+1} \const^{2k+2}} } \, ,
\end{align*}
where in (i) we use again that  $\Rfct_n(\zeta , \mu; \bI)$ and $\Rfct_0(\zeta , \muefct(\zeta, 0); \bI)$
 are decreasing in $\zeta$ and in (ii) we apply Corollary~\ref{cor:R-approximation-simplified}.
 Substituting the above displays into Eq.~\eqref{eq:variance-approx-bound-mid-1}, we have
\begin{align}\label{eq:AlmostVar}
	\left|\var_\bX(\lambda) - \VAR_n(\lambda)\right| & = \bigO_{k} \prn{\prn{\frac{\mc{E}_n}{\delta}  +   \frac{(1 + \mc{E}_n) \delta^{k} }{n^{k+1}\lambda^{k+1} \const^{2k+2}}} \cdot \vareps^2   \Ffct_0(n\lambda, \muefct(n\lambda, 0); \bI)} \, .
\end{align}

Finally, we use the fact that
\begin{align*}
	\Ffct_0(n\lambda, \muefct(n\lambda, 0); \bI) &=  n\lambda \Tr \prn{\bSigma\prn{\frac{n\lambda}{\lambdaefct}\bSigma + n\lambda \bI}^{-1}} = \lambdaefct  \Tr \prn{\bSigma(\bSigma + \lambdaefct \bI)^{-1}} \nonumber \\
	& = n\lambda \cdot \frac{\Tr \prn{\bSigma(\bSigma + \lambdaefct \bI)^{-1}}}{n - \Tr \prn{\bSigma(\bSigma + \lambdaefct \bI)^{-1}}} \leq n\lambda \cdot \frac{\Tr \prn{\bSigma^2(\bSigma + \lambdaefct \bI)^{-2}}}{n - \Tr \prn{\bSigma(\bSigma + \lambdaefct \bI)^{-1}}} \nonumber \\
	& \leq n\lambda \cdot \frac{n - \Tr \prn{\bSigma^2(\bSigma + \lambdaefct \bI)^{-2}}}{n - \Tr \prn{\bSigma(\bSigma + \lambdaefct \bI)^{-1}}} \cdot \frac{\Tr \prn{\bSigma^2(\bSigma + \lambdaefct \bI)^{-2}}}{n - \Tr \prn{\bSigma^2(\bSigma + \lambdaefct \bI)^{-2}}} \nonumber \\
	& = n\lambda \cdot \frac{n - \Tr \prn{\bSigma^2(\bSigma + \lambdaefct \bI)^{-2}}}{n - \Tr \prn{\bSigma(\bSigma + \lambdaefct \bI)^{-1}}} \cdot \left. \frac{\partial}{\partial \zeta}\Ffct_0(n\lambda, \muefct(n\lambda, \mu); \bI)\right|_{\mu = 0} \, ,
\end{align*}
and by Eq.~\eqref{asmp:lambda},
\begin{align*}
	\frac{n - \Tr \prn{\bSigma^2(\bSigma + \lambdaefct \bI)^{-2}}}{n - \Tr \prn{\bSigma(\bSigma + \lambdaefct \bI)^{-1}}} \leq \frac{n}{n - \Tr \prn{\bSigma(\bSigma + \lambdaefct \bI)^{-1}}}
	=\frac{\lambdaefct}{\lambda} \leq \const^{-1} \, .
\end{align*}
We therefore have, by Lemma~\ref{lem:var-bias-formula},
$\tau^2\Ffct_0(n\lambda, \muefct(n\lambda, 0); \bI)\le n\lambda\const^{-1}\VAR_n(\lambda)$.
 Substituting in Eq.~\eqref{eq:AlmostVar}, we obtain
\begin{align*}
	\left|\var_\bX(\lambda) - \VAR_n(\lambda)\right| & = \bigO_{k} \prn{\frac{n\lambda \mc{E}_n}{\delta \const} \cdot  +   \frac{(1 + \mc{E}_n) \delta^{k} }{n^k\lambda^{k} \const^{2k+3}}}  \cdot \VAR_n(\lambda) \, .
\end{align*}
By setting $\delta = \lambda \const^2 n^{1-1/k}$, the condition $\delta k \leq n\lambda$ 
is satisfied for all $n = \bigOmg_k(1)$, which completes the proof for variance approximation with
\begin{align*}
	\left|\var_\bX(\lambda) - \VAR_n(\lambda)\right| & = \bigO_{k} \prn{\mc{E}_n  \cdot n^{-\frac{1}{k}} \const^{-3} +  n^{-1} \const^{-3}}  \cdot \VAR_n(\lambda) = \bigO_{k, \constantx, D} \prn{\frac{\chi_n (\lambda)^3 \log^2 n}{n^{1- \frac{1}{k}} \const^{9.5}}} \cdot \VAR_n(\lambda) \, ,
\end{align*}
where we use $\chi_n(\lambda) \geq 1$ in the final bound.

\paragraph{Part IV: Approximation error for bias} 
Note that all the terms on the right-hand side of Eq.~\eqref{eq:bias-approx-bound-mid-1}
are evaluated at the same value of $\lambda$. Hence, Eq.~\eqref{asmp:lambda} 
applies to each of these terms.
We claim that the assumptions of Corollary~\ref{cor:R-approximation-simplified}
apply to all of these terms, provided the following conditions hold
\begin{align}
	\muefct(n\lambda, k\delta) & \leq (1 - \const/2)^{-1} \muefct(n\lambda, 0) \, , 
	\label{eq:CondBias2}\\
	 \chi_n (\lambda)^3  \log^2 n & \le \constant n \const^{4.5} \sqrt{\rho(\lambda)} \, .\label{eq:CondBias3}
\end{align}
Indeed, 
condition \eqref{eq:CondBias2} implies 
$\muefct(n\lambda, \mu) \leq (1 - \const/2)^{-1} \muefct(n\lambda, 0)$ for all $\mu\in [0,k\delta]$
since $\mu\mapsto \muefct(n\lambda, \mu)$  is monotone decreasing;
finally, condition \eqref{eq:CondBias3} is independent of $\mu$. 

then we can apply Lemmas~\ref{lem:derivative-F-n} and \ref{lem:derivative-F-0} and invoke 
Corollary~\ref{cor:R-approximation-simplified} with $\bQ = \btheta \btheta^\sT$.
 To be specific, by Corollary~\ref{cor:R-approximation-simplified}, we have with probability $1 - \bigO_k(n^{-D+1})$,
\begin{align*}
	\max_{0 \leq j \leq k} \frac{\left| \Ffct_n(n\lambda, j \delta; \btheta \btheta^\sT) -  \Ffct_0(n\lambda , \muefct(n\lambda, j \delta); \btheta \btheta^\sT)\right|}{\delta} 
&\leq \frac{\mc{E}_n}{\delta} \max_{j\in\{0,\dots,k\}}\Ffct_0(n\lambda, \muefct(n\lambda, j\delta); \btheta \btheta^\sT)\\
&\leq \frac{\mc{E}_n}{\delta} \Ffct_0(n\lambda, \muefct(n\lambda, 0); \btheta \btheta^\sT) \, , 
\end{align*}
as $	\Ffct_0(n\lambda, \muefct(n\lambda, \mu); \btheta \btheta^\sT) $ decreases with $\mu$. By Lemmas~\ref{lem:derivative-F-n} and \ref{lem:derivative-F-0} we obtain
\begin{align*}
	& \sup_{\mu' \in [0, k\delta]} \left|\frac{\partial^{k+1}}{\partial {\mu'}^{k+1}} \Ffct_n(n\lambda, \mu'; \btheta \btheta^\sT) - \frac{\partial^{k+1}}{\partial {\mu'}^{k+1}} \Ffct_0(n\lambda, \muefct(n\lambda, \mu'); \btheta \btheta^\sT)\right| \nonumber \\
	&  = \bigO_{k} \prn{ \sup_{\mu' \in [0, k \delta]}  \frac{\Ffct_n(n\lambda, \mu'; \btheta \btheta^\sT)}{n^{k+1}{\lambda}^{k+1}} + \sup_{\mu' \in [0, k \delta]}  \frac{\Ffct_0(n\lambda, \muefct(n\lambda, \mu'); \btheta \btheta^\sT)}{{\muefct(n\lambda, \mu')}^{k+1} \const^{2k+2}} }  \nonumber  \\
	& \stackrel{\mathrm{(i)}}{=} \bigO_{k} \prn{ \frac{\Ffct_n(n\lambda, 0; \btheta \btheta^\sT)}{n^{k+1}{\lambda}^{k+1}} +  \frac{\Ffct_0(n\lambda, \muefct(n\lambda, 0); \btheta \btheta^\sT)}{{\muefct(n\lambda, 0)}^{k+1} \const^{2k+2}} } \nonumber \\
	& \stackrel{\mathrm{(ii)}}{=}  \bigO_{k} \prn{ \frac{(1 + \mc{E}_n + \lambdaefct^{k+1} \const^{-2k-2})\Ffct_0(n\lambda, \muefct(n\lambda, 0); \btheta \btheta^\sT)}{n^{k+1}{\lambda}^{k+1}}} \, ,
\end{align*}
where in the bound (i) we use the fact that $\muefct(n\lambda, \mu)$ is increasing in $\mu$ 
 (cf.~Lemma~\ref{lem:muefct-small-mu-bound}) and $\Ffct_k(n\lambda, \mu; \bQ)$ is decreasing in $\mu$
 when $\mu \geq 0$; in (ii) we use that $\muefct(n\lambda, 0) = n\lambda / \lambdaefct$. Combining the calculations above, we have from Eq.~\eqref{eq:bias-approx-bound-mid-1}
\begin{align*}
	\left|\bias_\bX(\lambda) - \BIAS_n(\lambda)\right| & = \bigO_k \prn{ \prn{\frac{n\lambda\mc{E}_n}{\delta}   + \frac{\delta^k\prn{1 + \mc{E}_n + \lambdaefct^{k+1} \const^{-2k-2}}}{n^k{\lambda}^{k}}} \cdot \Ffct_0(n\lambda, \muefct(n\lambda, 0); \btheta \btheta^\sT)} \, .
\end{align*}

Then we make use of the following bound 
\begin{align*}
	\Ffct_0(n\lambda, \muefct(n\lambda, 0); \btheta \btheta^\sT) &=  n\lambda \Tr \prn{\bSigma^{\half} \btheta \btheta^\sT \bSigma^{\half}\prn{\frac{n\lambda}{\lambdaefct}\bSigma + n\lambda \bI}^{-1}} = \lambdaefct  \Tr \prn{\bSigma^{\half} \btheta \btheta^\sT \bSigma^{\half}(\bSigma + \lambdaefct \bI)^{-1}} \nonumber \\
	& = n\lambda \cdot \frac{\btheta^\sT \prn{\bSigma + \lambdaefct \bI}^{-1} \bSigma \btheta}{n - \Tr \prn{\bSigma(\bSigma + \lambdaefct \bI)^{-1}}} \leq \frac{\lambda}{ \lambdaefct} \cdot \frac{\lambdaefct^2 \btheta^\sT \prn{\bSigma + \lambdaefct \bI}^{-2} \bSigma^2 \btheta}{1 - n^{-1}\Tr \prn{\bSigma(\bSigma + \lambdaefct \bI)^{-1}}} \nonumber \\
	& = \frac{\lambda}{ \lambdaefct}  \cdot \frac{n - \Tr \prn{\bSigma^2(\bSigma + \lambdaefct \bI)^{-2}}}{n - \Tr \prn{\bSigma(\bSigma + \lambdaefct \bI)^{-1}}} \cdot \frac{\lambdaefct^2 \btheta^\sT \prn{\bSigma + \lambdaefct \bI}^{-2} \bSigma^2 \btheta}{1 - n^{-1}\Tr \prn{\bSigma^2(\bSigma + \lambdaefct \bI)^{-2}}} \nonumber \\
	& =  \frac{\lambda}{ \lambdaefct}  \cdot \frac{n - \Tr \prn{\bSigma^2(\bSigma + \lambdaefct \bI)^{-2}}}{n - \Tr \prn{\bSigma(\bSigma + \lambdaefct \bI)^{-1}}} \cdot \BIAS_n(\lambda) \, ,
\end{align*}
where in the last line we used the definition of  $\BIAS_n(\lambda)$ in Eq.~\eqref{eq:BIAS-n}.
 By Eq.~\eqref{asmp:lambda}, we have 
\begin{align*}
	\frac{\lambda}{ \lambdaefct}  \cdot \frac{n - \Tr \prn{\bSigma^2(\bSigma + \lambdaefct \bI)^{-2}}}{n - \Tr \prn{\bSigma(\bSigma + \lambdaefct \bI)^{-1}}}
	=1-\frac{1}{n}\Tr \prn{\bSigma^2(\bSigma + \lambdaefct \bI)^{-2}}  \leq1 \, ,
\end{align*}
which reduces the approximation bound for bias to
\begin{align*}
	\left|\bias_\bX(\lambda) - \BIAS_n(\lambda)\right| & = \bigO_{k} \prn{ \frac{n\lambda\mc{E}_n}{\delta}   + \frac{\delta^k\prn{1 + \mc{E}_n + \lambdaefct^{k+1} \const^{-2k-2}}}{n^k{\lambda}^{k} }} \cdot \BIAS_n(\lambda) \, .
\end{align*}
We again take $\delta = \lambda \const^2 n^{1-\frac{1}{k}}$ and the bound becomes
\begin{align*}
	\left|\bias_\bX(\lambda) - \BIAS_n(\lambda)\right| & = \bigO_{k, \constantx, D} \prn{\frac{\lambdaefct^{k+1}}{n \const^2} + \frac{\chi_n (\lambda)^3 \log^2 n}{n^{1- \frac{1}{k} } \const^{8.5}} \cdot \sqrt{\frac{\Riter_0(\bI)}{\Riter_0(\btheta \btheta^\sT)}} } \cdot \BIAS_n(\lambda) \, .
\end{align*}
This bounds hold under the conditions  \eqref{eq:CondBias2} to \eqref{eq:CondBias3},
which are implied by the following:
\begin{align*}
	\muefct(n\lambda, 	k \lambda \const^2 n^{1-\frac{1}{k}} ) & \leq (1 - \const/2)^{-1} \muefct(n\lambda, 0) \, . \\
\chi_n (\lambda)^3  \log^2 n & \le \constant n \const^{4.5} \sqrt{\rho(\lambda)} \, .
\end{align*}
For the first condition, we invoke Lemma~\ref{lem:muefct-small-mu-bound} to obtain a sufficient requirement $\lambda k n^{1-\frac{1}{k}} \leq n \const / 2$. For the last condition, it suffices to have $\chi_n(\lambda)^3 \log^2 n \leq \constant' n \const^{4.5} \sqrt{\rho(\lambda)}$ for the same $\constant' $ defined in Part III.
	\section{Proof of Theorem~\ref{thm:R-approximation}} \label{proof:R-approximation}

\paragraph{Part I: The iterative sequence} The proof 
is based on the following interpolating construction.
We will construct a sequence of random variables
$\mu_i \in \mathcal{F}_{i-1}$ for $i=0,1,\cdots, n+1$ (where, by convention,
 $\mathcal{F}_{-1}=\mathcal{F}_0$ is the trivial $\sigma$-algebra) such that,
 defining
\begin{align}\label{eq:RiterDef}
	\Riter_i (\bQ) := \Tr \prn{\bSigma^{\half}\bQ \bSigma^{\half} \prn{\zeta \bI + \mu_i \bSigma + \bX_i^\sT \bX_i}^{-1}} = \Rfct_i(\zeta, \mu_i ; \bQ) \, ,
\end{align}
we obtain that $\Riter_i(\bQ)$ is approximately a martingale and,
as a consequence,  $\Riter_0(\bQ) \approx \Riter_n(\bQ)$.
We will further have $\mu_0 =  \muefct(\zeta, \mu)$ and $\mu_n \approx \mu$, 
asd therefore we obtain the desired claim 
$\Rfct_0(\zeta, \muefct(\zeta, \mu); \bQ) \approx \Rfct_n(\zeta, \mu; \bQ)$.

\begin{remark} 
Note that
\begin{align*}
\Riter_{i+1} (\bQ) := \Tr \prn{\bSigma^{\half}\bQ \bSigma^{\half} \prn{\zeta \bI + \mu_{i+1} \bSigma + 
\bX_i^\sT \bX_i+\bx_{i+1}\bx_{i+1}^{\sT}}^{-1}}\, ,
\end{align*}
Hence, the difference between $\Riter_{i+1} (\bQ)$ and $\Riter_{i} (\bQ)$ results from two
effects: the rank one update $\bx_{i+1}\bx_{i+1}^{\sT}$, and the change in the 
coefficients $\mu_{i+1}-\mu_i$. Each of these effects can be estimated using matrix inversion,
cf. Eqs.~\eqref{eq:I-term} and \eqref{eq:II-term}: we will choose $\mu_{i+1}-\mu_i$
as to cancel the conditional expectation of the overall change vanish approximately.
\end{remark}

\begin{remark}
   The fact that $\Riter_i(\bQ)$  is nearly constant
   gives rise to the connection between the random design model~\eqref{eq:LinearModel} 
    and the equivalent sequence model~\eqref{Eq:equivalent-sequence-model}.
Indeed, if we further set $\zeta = n \lambda$ and $\mu = 0$, we  recover
 $\Tr \prn{\bSigma^{\half}\bQ \bSigma^{\half} \prn{n\lambda \bI + \muefct \bSigma}^{-1}}
  \approx \Tr \prn{\bSigma^{\half}\bQ \bSigma^{\half} \prn{n \lambda \bI + \bX_n^\sT \bX_n}^{-1}}$.
   Recall $\muefct = n \lambda / \lambdaefct$, and we see that the effect of the sample covariance 
   $\bX_n^\sT \bX_n / n$ is equivalent to the deterministic factor $\muefct \bSigma / n$.
In the classical asymptotics where $d$ is fixed and
     $n \to \infty$, we have $\muefct/n \to 1$, and thus recover the law of large numbers.
\end{remark}

Before formally defining the sequence $\{\mu_0, \cdots, \mu_{n+1}\}$, we introduce some helpful notations.
 We first define the matrices $\bA_i, \bB_i \in \mathcal{F}_i$ for $0 \leq i \leq n$ as
\begin{subequations}
	\begin{align}
		\bA_i & := \bSigma^{\half} \prn{\zeta \bI + \mu_i \bSigma + \bX_i^\sT \bX_i}^{-1} \bSigma^{\half} \, , \label{eq:def-A-k} \\ 
		\bB_i & := \bSigma^{\half} \prn{\zeta \bI + \mu_{i+1} \bSigma + \bX_i^\sT \bX_i}^{-1} \bSigma^{\half} \, . \label{eq:def-B-k}
	\end{align}
\end{subequations}
Then we can write $\Riter_i(\bQ) = \Tr \prn{\bQ \bA_i}$. Similarly we define another 
sequence of functions by $\Siter_i(\bQ) := \Tr \prn{\bQ \bB_i}$. 

Now we are ready to define the sequence $\mu_i \in \mathcal{F}_{i-1}$. We set the initial value 
$\mu_0 = \muefct(\zeta, \mu) \in \mathcal{F}_{-1}$ and thus 
$\Riter_0(\bQ) = \Rfct_0(\zeta, \muefct(\zeta, \mu); \bQ)$. The sequence $(\mu_j)_{j\ge 1}$
is iteratively determined through the following equation
\begin{align} \label{eq:iteration-mu}
	\mu_{i+1} = \mu_i - \frac{1}{1 + \Siter_i(\bI)} = \mu_i -  \frac{1}{1 + \Tr \prn{\bSigma^{\half} \prn{\zeta \bI + \mu_{i+1} \bSigma + \bX_i^\sT \bX_i}^{-1}\bSigma^{\half}}} \qquad \mathrm{s.t.} \  \bB_i  \succ 0\, .
\end{align}
It is evident that if the solution $\mu_{i+1}$ exists and is unique  
(almost surely with respect to the random choice of $\mu_i$), since
$\mu_i \in \mathcal{F}_{i-1}$ and $\bX_i \in \mathcal{F}_i$, it follows that $\mu_{i+1} 
\in \mathcal{F}_i$. The next lemma shows that the iteration via~\eqref{eq:iteration-mu} is indeed 
well-defined. Its proof is in Appendix~\ref{proof:update-rule-guarantee}.
\begin{lemma} \label{lem:update-rule-guarantee}
	There exists a unique strictly decreasing sequence 
	$\mu_0 > \mu_1 > \mu_2 > \cdots > \mu_n > \mu_{n+1}$ satisfying the update rule~\eqref{eq:iteration-mu}.
\end{lemma}

\paragraph{Part II: Approximation to a martingale}  We next explain 
 what is the rationale for the iterative definition of
  Eq.~\eqref{eq:iteration-mu}, and  how it will help us prove the theorem claim.

Since we want to upper bound  $|\Riter_n(\bQ) - \Riter_{0}(\bQ)|$, 
it makes sense to compute the difference $\Riter_i(\bQ) - \Riter_{i-1}(\bQ)$,
\begin{align}
	\Riter_i(\bQ) - \Riter_{i-1}(\bQ) & = \prn{\Riter_i(\bQ) - \Siter_{i-1}(\bQ)} + \prn{\Siter_{i-1}(\bQ) - \Riter_{i-1}(\bQ)} \nonumber \\
	& = \underbrace{\Tr \prn{\bQ \prn{\bA_i - \bB_{i-1}}}}_{\mathrm{(I)}} + \underbrace{\Tr \prn{\bQ \prn{\bB_{i-1} -\bA_{i-1}}}}_{\mathrm{(II)}} \, . \label{eq:R-difference}
\end{align}
Using rgw definitions in Eqs.~\eqref{eq:def-A-k} and \eqref{eq:def-B-k}, we can further expand (I) by Sherman-Morrison formula
\begin{align*}
	\bA_i - \bB_{i-1} & = \bSigma^{\frac 1 2} \brc{\prn{\zeta \bI + \mu_i \bSigma + \bX_{i-1}^\sT \bX_{i-1} + \bx_i \bx_i^\sT}^{-1} - \prn{\zeta \bI + \mu_i \bSigma + \bX_{i-1}^\sT \bX_{i-1}}^{-1}} \bSigma^{\frac 1 2} \nonumber \\
	&= - \frac{ \bSigma^{\frac 1 2} \prn{\zeta \bI + \mu_i \bSigma + \bX_{i-1}^\sT \bX_{i-1}}^{-1} \bx_i \bx_i^\sT  \prn{\zeta \bI + \mu_i \bSigma + \bX_{i-1}^\sT \bX_{i-1}}^{-1}  \bSigma^{\frac 1 2}}{1 + \bx_i^\sT  \prn{\zeta \bI + \mu_i \bSigma + \bX_{i-1}^\sT \bX_{i-1}}^{-1} \bx_i} \,  \nonumber \\
	& = - \frac{ \bSigma^{\frac 1 2} \prn{\zeta \bI + \mu_i \bSigma + \bX_{i-1}^\sT \bX_{i-1}}^{-1} \bSigma^{\frac 1 2} \bz_i \bz_i^\sT  \bSigma^{\frac 1 2} \prn{\zeta \bI + \mu_i \bSigma + \bX_{i-1}^\sT \bX_{i-1}}^{-1}  \bSigma^{\frac 1 2}}{1 + \bz_i^\sT \bSigma^{\frac 1 2}  \prn{\zeta \bI + \mu_i \bSigma + \bX_{i-1}^\sT \bX_{i-1}}^{-1} \bSigma^{\frac 1 2} \bz_i} \nonumber \\
	& = -\frac{\bB_{i-1} \bz_i\bz_i^\sT \bB_{i-1}}{1 + \bz_i^\sT \bB_{i-1} \bz_i} \, , 
\end{align*}
and thus write
\begin{align}\label{eq:I-term}
	\mathrm{(I)} = -\frac{\Tr \prn{\bQ\bB_{i-1} \bz_i\bz_i^\sT \bB_{i-1}}}{1 + \bz_i^\sT \bB_{i-1} \bz_i} \, .
\end{align}
We can also compute (II) by  noting that
\begin{align*}
	\bB_{i-1} - \bA_{i-1} & =  \bSigma^{\frac 1 2} \brc{\prn{\zeta \bI + \mu_i \bSigma + \bX_{i-1}^\sT \bX_{i-1}}^{-1} - \prn{\zeta \bI + \mu_{i-1} \bSigma + \bX_{i-1}^\sT \bX_{i-1}}^{-1}} \bSigma^{\frac 1 2} \nonumber \\
	& = \bSigma^{\frac 1 2} \brc{\prn{\zeta \bI + \mu_i \bSigma + \bX_{i-1}^\sT \bX_{i-1}}^{-1} \cdot \prn{\mu_{i-1} - \mu_i} \bSigma \cdot \prn{\zeta \bI + \mu_{i-1} \bSigma + \bX_{i-1}^\sT \bX_{i-1}}^{-1}} \bSigma^{\frac 1 2} \, ,
\end{align*}
and therefore
\begin{align}\label{eq:II-term}
	\mathrm{(II)} = (\mu_{i-1} - \mu_i) \cdot \Tr \prn{\bQ \bB_{i-1} \bA_{i-1}} \, .
\end{align}

It is now clear what is the motivation for defining $\mu_{i+1}$ 
as per Eq.~\eqref{eq:iteration-mu}.
We hope to have $\mathrm{(I)}+\mathrm{(II)} \approx 0$.
 Under the approximation $\Tr \prn{\bQ\bB_{i-1} \bz_i\bz_i^\sT \bB_{i-1}} \approx \Ep \brk{\Tr \prn{\bQ\bB_{i-1} \bz_i\bz_i^\sT \bB_{i-1}} \mid \mathcal{F}_{i-1}} =\Tr \prn{\bQ \bB_{i-1}^2} \approx  \Tr \prn{\bQ \bB_{i-1} \bA_{i-1}}$, this
 is achieved  when 
\begin{align*}
	\mu_i - \mu_{i-1} = - \frac{1}{1 + \Ep \brk{\bz_i^\sT \bB_{i-1} \bz_i \mid \mathcal{F}_{i-1}}}= - \frac{1}{1 + \Siter_{i-1}(\bI)} \, ,
\end{align*}
which recovers the iteration in Eq.~\eqref{eq:iteration-mu}.

\paragraph{Part III: Proof via stopping times}
We next make the previous argument rigorous.
 For any scalars $\alpha_1, \alpha_2, \beta_1, \beta_2, \gamma > 0$
 (in what follows, we'll use the notation $\Delta := (\alpha_1, \alpha_2, \beta_1, \beta_2, \gamma)$)
  we consider the events
\begin{subequations}
	\begin{align}
		E_i(\bQ) & := \brc{\left|\bz_i^\sT \bB_{i-1} \bz_i - \Siter_{i-1}(\bI) \right| \leq \alpha_1, \left|\bz_i^\sT \bB_{i-1} \bQ \bB_{i-1} \bz_i - \Tr \prn{\bQ\bB_{i-1}^2}\right| \leq \alpha_2, \norm{\bA_i} \leq \gamma }  \, , \label{eq:def-E-i} \\
		F_i(\bQ) & := \brc{\max \{\left|\Riter_i(\bQ) - \Riter_0(\bQ) \right|, \left|\Siter_i(\bQ) - \Riter_0(\bQ) \right|\} \leq \beta_1, \left|\mu_{i+1} - \wb{\mu}_{i+1}  \right| \leq \beta_2, \norm{\bB_i} \leq \gamma} \, , \label{eq:def-F-i}
	\end{align}
\end{subequations}
where $\wb{\mu}_{i+1} = \mu \cdot (i+1)/n + \muefct(\zeta, \mu) \cdot (1-(i+1)/n)$ is nonrandom.
 In particular we set $E_0(\bQ) = \Omega$ so that $E_i(\bQ)$ and $F_i(\bQ)$ are well-defined for
  $0 \leq i \leq n$. It follows then $E_i(\bQ), F_i(\bQ) \in \mathcal{F}_i$. Next we can proceed to 
  define two stopping times via
\begin{subequations}
	\begin{align}
		\brc{T_E(\bQ) \geq k+1} := \prn{\bigcap_{i=0}^{k} E_i(\bQ)} \cap \prn{\bigcap_{i=0}^{k-1} F_i(\bQ)}\, , \\
		\brc{T_F(\bQ) \geq k+1} := \prn{\bigcap_{i=0}^{k} E_i(\bQ)} \cap \prn{\bigcap_{i=0}^{k} F_i(\bQ)}\, ,
	\end{align}
\end{subequations}
for $k = 0, 1, \cdots, n$, with $T_E(\bQ), T_F(\bQ) \in \{0, 1, \cdots, n+1\}$. One can easily check that $T_E(\bQ)$ and $T_F(\bQ)$ are indeed stopping times since the sets in the above displays are in $\mathcal{F}_k$, and another immediate consequence is that $T_E(\bQ) \geq T_F(\bQ)$. These stopping times are helpful since
 the event $\{T_F(\bQ) = n+1\}$ implies
\begin{align*}
	\max_{0 \leq i \leq n} \left\{\left|\Riter_i(\bQ) - \Riter_0(\bQ) \right|, \left|\Siter_i(\bQ) - \Riter_0(\bQ) \right|  \right\} \leq \beta_1 \, ,
\end{align*}
and thus if $\beta_1$ is much smaller than $\Riter_0(\bQ)$, we can show $\Riter_n(\bQ) \approx \Riter_0(\bQ)$ as desired. Therefore, we want to lower bound the probability for the event $\{T_F(\bQ) = n+1\}$. We use the shorthand $p_{i,j}(T_1, T_2, \bQ) := \Prb(T_1(\bQ) \geq i, T_2(\bI) \geq j)$ for $T_1, T_2 \in \{T_E, T_F\}$. By telescoping sum, we have
\begin{align}
	& \Prb(T_F(\bQ) \geq 0, T_F(\bI) \geq 0) - \Prb(T_F(\bQ) = n+1, T_F(\bI) = n+1) \nonumber \\
	&  = p_{0,0}(T_F, T_F, \bQ) - p_{n+1, n+1}(T_F, T_F, \bQ) \nonumber \\
	& = \sum_{k=0}^n (p_{k,k}(T_F, T_F, \bQ) - p_{k+1, k+1}(T_E, T_E, \bQ)) + \sum_{k=1}^{n+1} (p_{k,k}(T_E, T_E, \bQ) - p_{k,k}(T_F, T_F, \bQ)) \nonumber \\
	& \leq \sum_{k=0}^n (p_{k, k}(T_F, T_F, \bQ) - p_{k+1, k}(T_E, T_F, \bQ) + p_{k, k}(T_F, T_F, \bQ) - p_{k, k+1}(T_F, T_E, \bQ)) \nonumber \\
	& \qquad + \sum_{k=1}^{n+1} (p_{k,k}(T_E, T_E, \bQ) - p_{k, k}(T_F, T_E, \bQ) + p_{k,k}(T_E, T_E, \bQ) - p_{k, k}(T_E, T_F, \bQ)) \nonumber \\
	& \leq \sum_{k=0}^n (p_{k, k}(T_F, T_F, \bQ) - p_{k+1, k}(T_E, T_F, \bQ) + p_{k, k}(T_F, T_F, \bI) - p_{k+1, k}(T_E, T_F, \bI)) \nonumber \\
	& \qquad + \sum_{k=1}^{n+1} (p_{k,k}(T_E, T_E, \bQ) - p_{k, k}(T_F, T_E, \bQ) + p_{k,k}(T_E, T_E, \bI) - p_{k, k}(T_F, T_E, \bI)) \, , \label{eq:stopping-times-telescoping}
\end{align}
where in the last inequality we use $\Prb(A \cap B) - \Prb(A \cap B') \leq \Prb(B)- \Prb(B')$ for $B' \subset B$.

We are left with the task of bounding the two terms 
$p_{k,k}(T_F, T_F, \bQ) - p_{k+1, k}(T_E, T_F, \bQ)$ and $p_{k,k}(T_E, T_E, \bQ) - p_{k, k}(T_F, T_E, \bQ)$ for any p.s.d.\ $\bQ$, and showing that they are small. 
Before doing this, we show that, by appropriately choosing $\gamma$, 
we have $\|\bA_i\| \leq \gamma$ and $\|\bB_i\| \leq \gamma$ with high probability. 
The proof of the next lemma is in Appendix~\ref{proof:A-norm-bound}.
\begin{lemma} \label{lem:A-norm-bound}
	Under Assumption~\ref{asmp:data-dstrb}, for any positive integer $D$, there exists a fixed $\eta = \eta(\constantx) \in (0, 1/2)$, such that for all $n = \Omega_D(1)$, it holds with probability $1 -\bigO(n^{-D})$ that
	\begin{align*}
		\norm{\bSigma^{\frac 1 2} \prn{\zeta \bI + \bX^\sT \bX}^{-1} \bSigma^{\frac 1 2}} \leq \frac{2}{n} \prn{1 + \frac{\bigO_{\constantx, D} \prn{\constantsig \sigma_{\lfloor \eta n\rfloor} \cdot \log n \log (\constantsig n)}}{\zeta}}  \, , \qquad \text{for all }\zeta > 0 \, ;
	\end{align*}
	additionally, under the same notations of Proposition~\ref{prop:effective-quantities-bound}, letting $\btheta_{\leq k}:= \sum_{i \leq k} \< \btheta, \bv_i\> \bv_i$ and $\btheta_{> k}:= \btheta - \btheta_{\leq k}$, we have for all $\zeta > 0$,
	\begin{align*}
		\btheta^\sT \bSigma^{\frac 1 2} \prn{\zeta \bI + \bX^\sT \bX}^{-1} \bSigma^{\frac 1 2} \btheta  \leq \frac{2}{n}\prn{1 + \frac{\bigO_{\constantx, D} \prn{\constantsig \sigma_{\lfloor \eta n\rfloor} \cdot \log n \log (\constantsig n)}}{\zeta}}\norm{\btheta_{\leq n}}^2 + \frac{2 \norm{\boldbeta_{>n}}^2}{\zeta} \, .
	\end{align*}
\end{lemma}

The next lemma---upper bounding the first term (I)---uses Hanson-Wright inequality to show concentration for events $E_i(\bQ)$ in~\eqref{eq:def-E-i}. A proof is in Appendix~\ref{proof:from-F-to-E}.

\begin{lemma} \label{lem:from-F-to-E}
	Under Assumption~\ref{asmp:data-dstrb}, choose $\beta_1,\beta_2$ in
	Eq.~\eqref{eq:def-F-i} so that  $\beta_1 \leq \Riter_0(\bQ)/4$ and $\beta_2 \leq \mu/2$. 
	Then for any positive integer $D$, there exists constants $\eta = \eta(\constantx) \in (0, 1/2)$, $\constant_\alpha = \constant_\alpha(\constantx, D)$ and $\constant_\gamma = \constant_\gamma(\constantx, D)$ such that if we take
	\begin{align*}
		\gamma &= \min \brc{\frac{2}{n} \prn{1 + \frac{\constant_\gamma \constantsig \sigma_{\lfloor \eta n\rfloor} \cdot \log n \log (\constantsig n)}{\zeta}} + \frac{2}{\muefct(\zeta, \mu)} , \frac{1}{\zeta}} \, , \\
		\alpha_1 &=  \constant_\alpha \log n \cdot \sqrt{\gamma \Riter_0(\bI)}\, , \\
		\alpha_2 &=  \constant_\alpha \log n \cdot \sqrt{\gamma^3 \Riter_0(\bQ)} \, ,
	\end{align*}
	it holds for all $n = \Omega_D(1)$ that
	\begin{align*}
		p_{k,k}(T_F, T_F, \bQ) - p_{k+1, k}(T_E, T_F, \bQ) & = \bigO(n^{-D}) \, .
	\end{align*}
	In addition, on the event $\{T_F(\bQ) \geq k, T_F(\bI) \geq k\} \in \mathcal{F}_{k-1}$ we have
	(using the shorthand $\Ep_{k-1}\{\,\cdot\,\} := \E\{\,\cdot\,|\cF_{k-1}\}$)
	\begin{align*}
		& \Ep_{k-1} \brk{\left| \bz_k^\sT \bB_{k-1} \bz_k - \Siter_{k-1}(\bI) \right| \ind\brc{\left| \bz_k^\sT \bB_{k-1} \bz_k - \Siter_{k-1}(\bI)\right| \geq  \alpha_1} } \nonumber \\
		& \qquad  = \bigO_{\constantx} \prn{n^{-D} \cdot \sqrt{ \gamma \Riter_0(\bI)}} = \bigO_{\constantx, D} \prn{n^{-D} \cdot \alpha_1}\, , \\
		& \Ep_{k-1} \brk{\left| \bz_{k}^\sT \bB_{k-1} \bQ \bB_{k-1} \bz_{k} - \Tr \prn{\bQ\bB_{k-1}^2}  \right| \ind\brc{\left| \bz_{k}^\sT \bB_{k-1} \bQ \bB_{k-1} \bz_{k} - \Tr \prn{\bQ\bB_{k-1}^2}  \right| \geq  \alpha_2 }} \nonumber \\
		& \qquad =\bigO_{\constantx} \prn{n^{-D} \cdot \sqrt{ \gamma^3 \Riter_0(\bQ)}} = \bigO_{\constantx, D} \prn{n^{-D} \cdot \alpha_2} \, .
	\end{align*}
\end{lemma}

We then proceed to bound the term $p_{k,k}(T_E,T_E,\bQ)-p_{k,k}(T_F,T_E,\bQ)$. 
The  proof of the next lemma is in Appendix~\ref{proof:from-E-to-F}.
\begin{lemma} \label{lem:from-E-to-F}
	Under Assumption~\ref{asmp:data-dstrb}, for any positive integer $D$, there exists a 
	constant $\constant_\beta = \constant_\beta(\constantx, D) > 0$ such that the following holds.
	Consider $\alpha_1, \alpha_2, \gamma$  as defined in Lemma~\ref{lem:from-F-to-E}, and set
	$\beta_1, \beta_2$ by 
	\begin{align*}
		\beta_1 & = \constant_\beta \prn{\sqrt{n \log n} \cdot \frac{\alpha_1 \gamma \Riter_0(\bQ) + \alpha_2 (1 + \Riter_0(\bI))}{1 + \Riter_0(\bI)^2} + n \cdot \brc{\frac{\gamma^2 \Riter_0(\bQ) + \alpha_1 \alpha_2}{1 + \Riter_0(\bI)^2} + \frac{\alpha_1^2  \gamma \Riter_0(\bQ)}{1 + \Riter_0(\bI)^3}} + \frac{\gamma \Riter_0(\bQ)}{1 + \Riter_0(\bI)}}  \, , \\
		\beta_2 & = \frac{\constant_\beta n \beta_1}{1 + \Riter_0(\bI)^2}\, .
	\end{align*}
%
If $\alpha_1\le \Riter_0(\bI)/4$, $\beta_1 \leq \Riter_0(\bQ)/4$, $\beta_2 \leq \mu/2$ and $n^{-D} = \bigO \prn{ \alpha_1 / \prn{1 + \Riter_0(\bI)}}$,
then for all $1 \leq k \leq n+1$ and $n = \Omega_D(1)$,
	\begin{align*}
	p_{k,k}(T_E, T_E, \bQ) - p_{k, k}(T_F, T_E, \bQ)  = \bigO(n^{-D}) \, .
	\end{align*}
\end{lemma}
Applying Lemmas~\ref{lem:from-F-to-E} and \ref{lem:from-E-to-F} to Eq.~\eqref{eq:stopping-times-telescoping} (note that we can take $\bQ = \bI$), we have shown that
\begin{align*}
	1 -\Prb(T_F(\bQ) = n+1, T_F(\bI) = n+1) &=  \Prb(T_F(\bQ) \geq 0, T_F(\bI) \geq 0) - \Prb(T_F(\bQ) = n+1, T_F(\bI) = n+1)  \nonumber \\
	&= \bigO(n^{-D+1}) \, , 
\end{align*}
which implies by choosing the parameter $\Delta$ given by the above lemmas, with probability 
$1 - \bigO(n^{-D+1})$
\begin{align*}
	|\Riter_n(\bQ) - \Riter_0(\bQ)| \leq \beta_1 \, ,  \qquad |\mu_n - \wb{\mu}_n| \leq \beta_2 \, .
\end{align*}
Therefore, since $\wb{\mu}_n = \mu$,
$\mu_0 = \muefct(\zeta, \mu)$, and recalling the definition of $\Riter_k(\bQ)$,
cf. Eq.~\eqref{eq:RiterDef}, we have 
%
\begin{align*}
	|\Rfct_n(\zeta, \mu; \bQ) -& \Rfct_0(\zeta, \muefct(\zeta, \mu); \bQ)| \leq |\Rfct_n(\zeta, \mu; \bQ) - \Rfct_n(\zeta, \mu_n; \bQ)| + |\Rfct_n(\zeta, \mu_n; \bQ) -  \Rfct_0(\zeta, \muefct(\zeta, \mu);\bQ)| \nonumber \\
	& = \left|(\mu_n - \mu) \cdot \Tr \prn{\bQ \bSigma^{\half} (\zeta \bI + \mu \bSigma + \bX^\sT \bX)^{-1} \bSigma^{\half} \bA_n} \right| + \left| \Riter_n(\bQ) - \Riter_0(\bQ) \right| \nonumber \\
	& \stackrel{\mathrm{(i)}}{\leq} \beta_2 \normop{\bSigma^{\half} (\zeta \bI + \mu \bSigma + \bX^\sT \bX)^{-1} \bSigma^{\half}} \Riter_n(\bQ) + \left| \Riter_n(\bQ) - \Riter_0(\bQ) \right|\, \nonumber \\
	& \stackrel{\mathrm{(ii)}}{\leq} \gamma \beta_2 \Riter_n(\bQ) +  \left| \Riter_n(\bQ) - \Riter_0(\bQ) \right|  \leq \gamma \beta_2 \Riter_0(\bQ) + (1 + \gamma \beta_2) \left| \Riter_n(\bQ) - \Riter_0(\bQ) \right|  \nonumber \\
	& \leq  \gamma \beta_2 \Riter_0(\bQ) + \beta_1 (1 + \gamma \beta_2) \stackrel{\mathrm{(ii)}}{\leq}  \frac{5}{4} \gamma \beta_2 \Riter_0(\bQ) + \beta_1  \, ,
\end{align*}
where in (ii) we used Lemma \ref{lem:A-norm-bound};  in (iii) we used the fact that $\beta_1 \leq \Riter_0(\bQ)/4$
by assumption. We explain the inequality in (i) more carefully as it is less evident. Denoting by $\bB = \bSigma^{\half} (\zeta \bI + \mu \bSigma + \bX^\sT \bX)^{-1} \bSigma^{\half}$, we first show $\bB$ and $\bA_n$ commute. Clearly commutativity holds if $\mu = \mu_n$, otherwise we have
\begin{align*}
	\bB \bA_n & = \bSigma^{\half} (\zeta \bI + \mu \bSigma + \bX^\sT \bX)^{-1} \bSigma^{\half} \cdot \bSigma^{\half} (\zeta \bI + \mu_n \bSigma + \bX^\sT \bX)^{-1} \bSigma^{\half} \nonumber \\
	& = (\mu_n - \mu)^{-1}  \bSigma^{\half} \brc{(\zeta \bI + \mu \bSigma + \bX^\sT \bX)^{-1} -   (\zeta \bI + \mu_n \bSigma + \bX^\sT \bX)^{-1}} \bSigma^{\half} \nonumber \\
	& = (\mu_n - \mu)^{-1} \prn{\bB - \bA_n} = \bA_n \bB \, .
\end{align*}
Noting that $\bB$ and $\bA_n$ are both p.s.d.\ compact self-adjoint operators in Hilbert space, commutativity implies they can be simultaneously orthogonally diagonalized and that $\bB^{\half}$ and $\bA_n^{\half}$ also commute. Consequently, combined with the fact that $\Tr(\bA_n\bC)\le \|\bA_n\|\Tr(\bC)$ for any p.s.d. matrix $\bC$, we have (i) from
\begin{align*}
	 \Tr \prn{\bQ \bSigma^{\half} (\zeta \bI + \mu \bSigma + \bX^\sT \bX)^{-1} \bSigma^{\half} \bA_n} & = \Tr  \prn{\bQ \bB \bA_n} = \Tr \prn{\bB \cdot \bA_n^{\half} \bQ \bA_n^{\half}} \leq \norm{\bB} \Tr (\bQ \bA_n) \nonumber \\
	 & = \normop{\bSigma^{\half} (\zeta \bI + \mu \bSigma + \bX^\sT \bX)^{-1} \bSigma^{\half}} \Riter_n(\bQ) \, .
\end{align*}

We therefore proved
 the following. If  $\beta_1 \leq \Riter_0(\bQ)/4$ and $n^{-D} = \bigO \prn{ \alpha_1 / \prn{1 + \Riter_0(\bI)}}$,
 then
 \begin{align}
 \beta_2 \leq \mu/2\;\;\; \Rightarrow\;\;\;
 |\Rfct_n(\zeta, \mu; \bQ) - \Rfct_0(\zeta, \muefct(\zeta, \mu); \bQ)| = \bigO \prn{ \gamma \beta_2 \Riter_0(\bQ) + \beta_1 }\, .
 \label{eq:muLarger} 
 \end{align}
 
To remove the condition  $\beta_2 \leq \mu/2$, we use the following  estimate,
 proven in Appendix~\ref{proof:R-fct-approximation}.
\begin{lemma} \label{lem:R-fct-approximation}
	Under Assumption~\ref{asmp:data-dstrb}, consider the parameter tuple 
	$\Delta = (\alpha_1, \alpha_2, \beta_1, \beta_2, \gamma)$ defined in 
	Lemmas~\ref{lem:from-F-to-E} and \ref{lem:from-E-to-F}. If $\alpha_1 \leq \Riter_0(\bI)/8$, 
	$\beta_1 \leq \Riter_0(\bQ)/64$, $\gamma \beta_2 (1 + \Riter_0(\bI)) \leq 1/ 64$, $n^{-D} = \bigO(\alpha_1/(1 + \Riter_0(\bI)))$ and
	$\beta_2>\mu/2$, then
	 we have
	\begin{align}
	|\Rfct_n(\zeta, \mu; \bQ) - \Rfct_0(\zeta, \muefct(\zeta, \mu); \bQ)| = 	\bigO \prn{\gamma \beta_2 \prn{1 +  \Riter_0(\bI)} \Riter_0(\bQ) + \beta_1 }  \label{eq:muSmaller} 
	\end{align}
\end{lemma}
Combining Eqs.~\eqref{eq:muLarger} and \eqref{eq:muSmaller}, the proof is complete.

 
		
	\section*{Acknowledgements}
	This work was supported by the NSF through award DMS-2031883, the Simons Foundation through 
	Award 814639 for the Collaboration on the Theoretical Foundations of Deep Learning, the NSF grant
	 CCF-2006489, the ONR grant N00014-18-1-2729, and a grant from Eric and Wendy Schmidt
	 at  the Institute for Advanced Studies. C. Cheng is supported
	 by the William R. Hewlett Stanford graduate fellowship.
	 
	 Part of this work was carried out while A. Montanari was on partial leave from Stanford and a 
	 Chief Scientist at Ndata Inc dba Project N. The present research is unrelated to A. Montanari’s 
	 activity while on leave. 
	
	\bibliographystyle{amsalpha}
	\bibliography{all-bibliography}
	
	\newpage
	\appendix
	
	\section{Proof of Proposition~\ref{prop:effective-quantities-bound}} \label{proof:effective-quantities-bound}
Since $\sigma_{k_\star} \geq \lambdaefct \geq \sigma_{k_\star +1}$, we have
\begin{align*}
	k_\star+ \frac{r_1(k_\star)}{b_{k_\star}} & = \sum_{l=1}^{k_\star} \frac{\sigma_l}{\sigma_l} + \sum_{l=k_\star +1}^d \frac{\sigma_l}{\sigma_{k_\star }}  \leq \sum_{l=1}^{k_\star} \frac{\sigma_l + \lambdaefct}{\sigma_l + \lambdaefct} + \sum_{l=k_\star +1}^d \frac{\sigma_l}{\lambdaefct}  \nonumber \\
	& \leq \sum_{l=1}^{k_\star} \frac{2\sigma_l}{\sigma_l + \lambdaefct} + \sum_{l=k_\star +1}^d \frac{2\sigma_l}{\sigma_l + \lambdaefct} = 2 \Tr \prn{\bSigma(\bSigma + \lambdaefct \bI)^{-1}} \leq 2 n \, .
\end{align*}
Next we bound $\VAR_n(\lambda)$. Recalling that  $\Tr \prn{\bSigma^2(\bSigma + \lambdaefct \bI)^{-2}}\le n(1-c_\star^{-1})$, it then follows
\begin{align*}
	\VAR_n(\lambda) & = \frac{\vareps^2 \Tr \prn{\bSigma^2(\bSigma + \lambdaefct \bI)^{-2}}}{n - \Tr \prn{\bSigma^2(\bSigma + \lambdaefct \bI)^{-2}}} \leq  \frac{c_\star \tau^2}{n} \cdot \prn{ \sum_{l=1}^{k_\star} \frac{\sigma_l^2}{(\sigma_l + \lambdaefct)^2} + \sum_{l=k_\star +1}^d \frac{\sigma_l^2}{(\sigma_l + \lambdaefct)^2}} \\
	& \leq \frac{c_\star \tau^2}{n} \cdot \prn{k_\star + \sum_{l=k_\star +1}^d \frac{\sigma_l^2}{\lambdaefct^2}} \leq c_\star\vareps^2\Big(\frac{k_\star}{n}+\frac{r_2(k_\star)}{n}\Big) \stackrel{\mathrm{(i)}}{\leq} c_\star\vareps^2\Big(\frac{k_\star}{n}+\frac{4 b_{k_\star}^2 n }{\overline{r}(k_\star)}\Big)  \, ,
\end{align*}
where in (i) we use the previous bound $r_1(k_\star) \leq 2 b_{k_\star} n$. Finally, for the bias term, we have
\begin{align*}
	\BIAS_n(\lambda) & = \frac{\lambdaefct^2 \<\boldbeta, \prn{\bSigma + \lambdaefct \bI}^{-2} \bSigma \boldbeta\>}{1 - n^{-1} \Tr \prn{\bSigma^2 (\bSigma +\lambdaefct \bI)^{-2}}} \leq c_\star \sum_{l=1}^d \frac{\lambdaefct^2 \sigma_l}{(\sigma_l + \lambdaefct)^2} \<\boldbeta,\bv_l\>^2 \nonumber \\
	& \leq c_\star \prn{\sum_{l=1}^{k_\star} \lambdaefct^2 \sigma_l^{-1} \<\boldbeta,\bv_l\>^2 + \sum_{l=k_\star + 1}^{d} \sigma_l  \<\boldbeta,\bv_l\>^2} \le c_\star 
	\Big(\sigma_{k_\star}^2 \|\boldbeta_{\le k_\star}\|_{\bSigma^{-1}}^2 +\|\boldbeta_{>k_\star}\|_{\bSigma}^2\Big)\, .
\end{align*}
	\section{Auxiliary lemmas}

\subsection{Proof of Lemma~\ref{lem:var-bias-formula}} \label{proof:var-bias-formula}
The lemma follows by pure calculations.
\paragraph{Identities for \texorpdfstring{$\var_\bX(\lambda)$}{TEXT} and \texorpdfstring{$\bias_\bX(\lambda)$}{TEXT}}
Substitute in Eq.~\eqref{eq:R-F-func}, we have
\begin{align*}
	\vareps^2 \cdot \frac{\partial}{\partial \zeta} \Ffct_n(n\lambda, 0; \bI) & = \vareps^2 \cdot \frac{\partial }{\partial \zeta} \left. \prn{\zeta \Tr \prn{\bSigma (\zeta \bI +  \bX^\sT \bX)^{-1}}} \right|_{\zeta=n\lambda}\nonumber \\
	& = \vareps^2 \brc{\Tr \prn{\bSigma (n\lambda \bI +  \bX^\sT \bX)^{-1}} - n\lambda\Tr \prn{\bSigma (n\lambda \bI +  \bX^\sT \bX)^{-2}}} \nonumber \\
	& = \vareps^2 \Tr \prn{\bSigma \bX^\sT \bX (n\lambda \bI + \bX^\sT \bX)^{-2}} = \var_\bX(\lambda)\, ,
\end{align*}
and similarly for the bias term
\begin{align*}
	- n\lambda \cdot \frac{\partial}{\partial \mu}\Ffct_n(n\lambda, 0; \btheta \btheta^\sT)  & =  - n\lambda \cdot \frac{\partial }{\partial \mu} \left.\prn{n\lambda \Tr \prn{\bSigma^{\half} \btheta \btheta^\sT \bSigma^{\half} (n\lambda \bI +  \mu \bSigma +  \bX^\sT \bX)^{-1}}} \right|_{\mu = 0} \nonumber \\
	& = n^2\lambda^2 \left. \Tr \prn{\bSigma^{\half} \btheta \btheta^\sT \bSigma^{\half} (n\lambda \bI +  \mu \bSigma +  \bX^\sT \bX)^{-1} \bSigma  (n\lambda \bI +  \mu \bSigma +  \bX^\sT \bX)^{-1}} \right|_{\mu = 0} \nonumber \\
	& = n^2\lambda^2 \Tr \prn{\boldbeta \boldbeta^\sT (n\lambda \bI +  \bX^\sT \bX)^{-1} \bSigma  (n\lambda \bI  +  \bX^\sT \bX)^{-1}} = \bias_\bX(\lambda)\, .
\end{align*}

\paragraph{Identities for $\VAR_n(\lambda)$ and $\BIAS_n(\lambda)$} First we verify that $\muefct(n\lambda, 0) = n\lambda/\lambdaefct$. Set $\zeta = n\lambda$ and $\mu = 0$ in Eq.~\eqref{eq:mu-fixed-point}, we obtain 
\begin{align*}
	\muefct = \frac{n}{1 + \Rfct_0(n\lambda, \muefct; \bI)} = \frac{n}{1 + \Tr \prn{\bSigma (\muefct \bSigma + n\lambda \bI)^{-1} }} =  \frac{n}{1 + \muefct^{-1} \Tr \prn{\bSigma ( \bSigma + \frac{n\lambda}{\muefct} \bI)^{-1} }}\, , \nonumber 
\end{align*}
and thus
\begin{align*}
	n - \muefct = \Tr \prn{\bSigma \prn{ \bSigma + \frac{n\lambda}{\muefct} \bI}^{-1} }\, ,
\end{align*}
which proves the claim comparing to Eq.~\eqref{eq:lambda-fixed-point}. Further by~\eqref{eq:mu-fixed-point}, we can compute the derivatives
\begin{subequations}
\begin{align}
	\frac{\partial}{\partial \zeta} \muefct(n\lambda, 0) & = \frac{\Tr \prn{\bSigma (\bSigma + \lambdaefct \bI)^{-2}}}{n - \Tr \prn{\bSigma^2(\bSigma + \lambdaefct \bI)^{-2}}}\, ,   \\
	\frac{\partial}{\partial \mu} \muefct(n\lambda, 0) & = \frac{n}{n - \Tr \prn{\bSigma^2(\bSigma + \lambdaefct \bI)^{-2}}} \, .
\end{align}
\end{subequations}

We can then proceed to write
\begin{align*}
	& \vareps^2 \cdot \frac{\partial}{\partial \zeta} \Ffct_0(n\lambda, \muefct(n\lambda, 0); \bI) \nonumber \\
	& = \vareps^2 \cdot \frac{\partial }{\partial \zeta} \left.\prn{\zeta \Tr \prn{\bSigma (\zeta \bI + \muefct \bSigma)^{-1}}} \right|_{\zeta=n\lambda}  \nonumber \\
	& = \vareps^2  \brc{ \Tr \prn{\bSigma (n\lambda \bI + \muefct \bSigma)^{-1}} - n\lambda\Tr \prn{\bSigma (n\lambda \bI + \muefct \bSigma)^{-2}}  -  n\lambda\Tr \prn{\bSigma^2 (n\lambda \bI + \muefct \bSigma)^{-2}} \cdot \frac{\partial}{\partial \zeta} \muefct(n\lambda, 0) } \nonumber \\
	& = \frac{\vareps^2}{\muefct} \Tr \prn{\bSigma^2 (\bSigma + \lambdaefct \bI)^{-2}} \cdot \prn{1 - \frac{\lambdaefct \Tr \prn{\bSigma (\bSigma + \lambdaefct \bI)^{-2}}}{n - \Tr \prn{\bSigma^2(\bSigma + \lambdaefct \bI)^{-2}}}} \nonumber \\
	& = \frac{\vareps^2}{\muefct} \Tr \prn{\bSigma^2 (\bSigma + \lambdaefct \bI)^{-2}} \cdot \frac{n - \Tr \prn{\bSigma(\bSigma + \lambdaefct \bI)^{-1}}}{n - \Tr \prn{\bSigma^2(\bSigma + \lambdaefct \bI)^{-2}}} \nonumber \\
	& \stackrel{\mathrm{(i)}}{=}  \frac{\vareps^2 \Tr \prn{\bSigma^2 (\bSigma + \lambdaefct \bI)^{-2}}}{n -\Tr \prn{\bSigma^2 (\bSigma + \lambdaefct \bI)^{-2}} } = \VAR_n(\lambda) \, ,
\end{align*}
where in (i) we use Eq.~\eqref{eq:lambda-fixed-point} which implies $\muefct = n -  \Tr \prn{\bSigma(\bSigma + \lambdaefct \bI)^{-1}}$. For the bias we can compute
\begin{align*}
	& - n\lambda \cdot \frac{\partial}{\partial \mu}\Ffct_0(n\lambda, \muefct(n\lambda, 0); \btheta \btheta^\sT) \nonumber \\
	& = n^2\lambda^2 \Tr \prn{\bSigma^{\half} \btheta \btheta^\sT \bSigma^{\half} (n\lambda \bI +  \muefct \bSigma )^{-1} \bSigma  (n\lambda \bI +  \muefct \bSigma )^{-1}}  \cdot \frac{n}{n - \Tr \prn{\bSigma^2(\bSigma + \lambdaefct \bI)^{-2}}} \nonumber \\
	& = \frac{\lambdaefct^2 \boldbeta^\sT \prn{\bSigma + \lambdaefct \bI}^{-2} \bSigma \boldbeta}{1 - n^{-1} \Tr \prn{\bSigma^2 (\bSigma +\lambdaefct \bI)^{-2}}} = \BIAS_n(\lambda) \, .
\end{align*}
The proof is complete.

\subsection{Proof of Lemma~\ref{lem:function-value-bound-to-derivative}} \label{proof:function-value-bound-to-derivative}
The lemma is an analogue of \cite[Lemma.~5]{hastie2022surprises}, which requires two-sided differentiability around $0$ and makes use of higher order central difference operators from numerical analysis. Here we apply a more straightforward argument. For any $0 \leq j \leq k$, by Taylor expansion with Lagrange remainder, we can write
\begin{align*}
	f(j\delta) = \sum_{l=0}^k j^l \cdot \frac{\delta^l}{j!} f^{(l)}(0) + j^{k+1} \cdot \frac{\delta^{k+1}}{(k+1)!} f^{(k+1)}(t_j) \, ,
\end{align*}
for some $t_j \in [0, j\delta]$. We can write the $k+1$ equations in matrix form,
\begin{align*}
	\underbrace{\begin{bmatrix}
		1 & 0 & 0 & \cdots & 0 \\
		1 & 1 & 1 & \cdots & 1 \\
		1 & 2 & 4 & \cdots & 2^k \\
		\vdots & \vdots & \vdots & \ddots & \vdots \\
		1 & k & k^2 & \cdots & k^k 
	\end{bmatrix}}_{ := \bV_k} \begin{bmatrix}
	f(0) \\ f'(0) \delta \\ f''(0)\delta^2/2 \\ \vdots \\ f^{(k)}(0) \delta^k /k! 
	\end{bmatrix} + \frac{\delta^{k+1}}{(k+1)!}  \begin{bmatrix}
	0 \\ f^{(k+1)}(t_1) \\ 2^{k+1}f^{(k+1)}(t_2) \\ \vdots \\ k^{k+1} f^{(k+1)}(t_k)
	\end{bmatrix} = \begin{bmatrix}
	f(0) \\ f(\delta) \\ f(2\delta) \\ \vdots \\ f(k\delta) 
	\end{bmatrix} \, .
\end{align*}
The Vandermonde matrix $\bV_k$ is invertible, and therefore we can write
\begin{align*}
	\begin{bmatrix}
		f(0) \\ f'(0) \delta \\ f''(0)\delta^2/2 \\ \vdots \\ f^{(k)}(0) \delta^k /k! 
	\end{bmatrix} = \bV_k^{-1} \begin{bmatrix}
	f(0) \\ f(\delta) \\ f(2\delta) \\ \vdots \\ f(k\delta) 
\end{bmatrix} - \frac{\delta^{k+1}}{(k+1)!}   \bV_k^{-1} \begin{bmatrix}
0 \\ f^{(k+1)}(t_1) \\ 2^{k+1}f^{(k+1)}(t_2) \\ \vdots \\ k^{k+1} f^{(k+1)}(t_k)
\end{bmatrix} \, .
\end{align*}
Denote by $\| \bM \|_\infty$ the $\ell_\infty$-induced operator norm, we thus have
\begin{align*}
	|f'(0) \delta| & \leq \left\|\begin{bmatrix}
		f(0) \\ f'(0) \delta \\ f''(0)\delta^2/2 \\ \vdots \\ f^{(k)}(0) \delta^k /k! 
	\end{bmatrix} \right\|_\infty \leq \left\| \bV_k^{-1}\right\|_\infty \cdot \prn{ \left\| \begin{bmatrix}
	f(0) \\ f(\delta) \\ f(2\delta) \\ \vdots \\ f(k\delta) 
\end{bmatrix}\right\|_\infty + \frac{\delta^{k+1}}{(k+1)!} \left\| \begin{bmatrix}
0 \\ f^{(k+1)}(t_1) \\ 2^{k+1}f^{(k+1)}(t_2) \\ \vdots \\ k^{k+1} f^{(k+1)}(t_k)
\end{bmatrix}\right\|_\infty} \nonumber \\
& = \bigO_k \prn{\max_{0 \leq j \leq k} |f(j\delta)|+ \sup_{t \in [0, k \delta]} |f^{(k+1)}(t)| \cdot \delta^{k+1}}.
\end{align*}
Dividing $\delta$ from both sides completes the proof.

\subsection{Proof of Lemma~\ref{lem:derivative-F-n}} \label{proof:derivative-F-n}
\paragraph{Part I: Derivative w.r.t. $\zeta$} By Lemma~\ref{lem:var-bias-formula}, for $\lambda = \zeta/n$,
\begin{align*}
	\frac{\partial}{\partial \zeta} \Ffct_n(\zeta, 0; \bI) = \var_\bX(\lambda) / \vareps^2 = \Tr \prn{\bSigma \bX^\sT \bX (\bX^\sT \bX + \zeta \bI)^{-2}} \, ,
\end{align*}
we can easily write out derivatives with respect to $\zeta$ up to any order $k \geq 1$ as
\begin{align*}
	\frac{\partial^k}{\partial \zeta^k} \Ffct_n(\zeta, 0; \bI)  = \bigO_k \prn{ \Tr \prn{\bSigma \bX^\sT \bX (\bX^\sT \bX + \zeta \bI)^{-1-k}} }\, , 
\end{align*}
and therefore
\begin{align*}
	\left|\frac{\partial^k}{\partial \zeta^k} \Ffct_n(\zeta, 0; \bI)  \right| & \leq \bigO_k \prn{\frac{1}{\zeta^{k-1}} \Tr \prn{\bSigma \bX^\sT \bX (\bX^\sT \bX + \zeta \bI)^{-2}} }\leq \bigO_k \prn{\frac{1}{\zeta^{k-1}} \Tr \prn{\bSigma (\bX^\sT \bX + \zeta \bI)^{-1}}} \nonumber \\
	&= \bigO_k \prn{\frac{\Ffct_n(\zeta, 0; \bI) }{\zeta^k}} \, .
\end{align*}
\paragraph{Part II: Derivative w.r.t. $\mu$} We can directly compute that
\begin{align*}
	\left|\frac{\partial^k}{\partial \mu^k} \Ffct_n(\zeta, \mu; \btheta \btheta^\sT)  \right|  & = \left| \frac{\partial^k}{\partial \mu^k}  \cdot \zeta  \Tr \prn{\bSigma^{\frac 1 2} \btheta \btheta^\sT \bSigma^{\frac 1 2} (\zeta \bI + \mu \bSigma + \bX^\sT \bX)^{-1}} \right| \nonumber \\
	& = \bigO_k \prn{ \zeta  \Tr \prn{\bSigma^{\frac 1 2} \btheta \btheta^\sT \bSigma^{\frac 1 2} \prn{(\zeta \bI + \mu \bSigma + \bX^\sT \bX)^{-1} \bSigma}^k(\zeta \bI + \mu \bSigma + \bX^\sT \bX)^{-1}} } \nonumber \\
	& = \bigO_k \prn{ \zeta  \Tr \prn{\btheta \btheta^\sT  \prn{\bSigma^{\frac 1 2}(\zeta \bI + \mu \bSigma + \bX^\sT \bX)^{-1} \bSigma^{\frac 1 2}}^{k+1}} } \nonumber \\
	& \stackrel{\mathrm{(i)}}{=} \bigO_k \prn{\zeta^{1-k}  \Tr \prn{\btheta \btheta^\sT  \bSigma^{\frac 1 2}(\zeta \bI + \mu \bSigma + \bX^\sT \bX)^{-1} \bSigma^{\frac 1 2}}} = \bigO_k \prn{\frac{\Ffct_n(\zeta, \mu; \btheta \btheta^\sT) }{\zeta^k}} \, ,
\end{align*}
where in (i) we use $\norm{\bSigma} = 1$.

\subsection{Proof of Lemma~\ref{lem:derivative-F-0}} \label{proof:derivative-F-0}

\paragraph{Part I: Derivative w.r.t. $\zeta$} Note that for $\lambda = \zeta/n$,
\begin{align*}
	\Ffct_0(\zeta, \muefct(\zeta, 0); \bI) = \zeta \Tr \prn{\bSigma (\zeta \bI + \muefct(\zeta, 0) \bSigma)^{-1} } = \lambdaefct \Tr \prn{\bSigma (\bSigma + \lambdaefct \bI)^{-1} }\, .
\end{align*}
Combining with the fixed-point equation~\eqref{eq:lambda-fixed-point} that determines $\lambdaefct$, we further get
\begin{align*}
	\Ffct_0(\zeta, \muefct(\zeta, 0); \bI) =  n \lambdaefct - \zeta\,. 
\end{align*}
Therefore, for all $k \geq 1$.
\begin{align} \label{eq:mid-derivative-lambda-1}
	\frac{\partial}{\partial \zeta^k}\Ffct_0(\zeta, \muefct(\zeta, 0); \bI) = n \cdot \frac{\partial^k \lambdaefct}{\partial \zeta^k} - \ind \{k = 1\}\, ,
\end{align}
and it boils down to controlling higher order derivatives of $\lambdaefct$ w.r.t. $\zeta$. Of course, we need to first show that we can actually write $\lambdaefct = \lambdaefct(\zeta)$ locally by implicit function theorem. Since
\begin{align*}
	\zeta = \lambdaefct \cdot \prn{n - \Tr \prn{\bSigma(\bSigma + \lambdaefct \bI)^{-1}}}
\end{align*}
which is clearly a increasing function of $\lambdaefct$ on the right hand side, and thus $\partial \zeta / \partial \lambdaefct > 0$ and the implicit function theorem applies. To calculate the higher order derivative of the inverse function, we apply the formula for higher order derivatives of inverse function~\cite{apostol2000calculating}
\begin{align} \label{eq:mid-derivative-lambda-2}
	\frac{\partial^k \lambdaefct}{\partial \zeta^k} & = \left|\frac{\partial \zeta}{\partial \lambdaefct} \right|^{1-2k} \cdot \sum_{\substack{m_1 + m_2 + \cdots + m_p = k-1  \\ m_1 + 2m_2 + \cdots + pm_p= 2k-2}} \bigO_k \prn{\prod_{l=1}^p \prn{\frac{\partial^l \zeta}{\partial \lambdaefct^l}}^{m_l}} \, .
\end{align} 

To further upper bound the above display, we need a lower bound for the derivative $\partial \zeta / \partial \lambdaefct$ and upper bounds for higher order derivatives $\partial^l \zeta / \partial \lambdaefct^l$. Using the Leibniz rule, we can compute that
\begin{align*}
	\frac{\partial^l \zeta}{\partial \lambdaefct^l} & = \sum_{r = 0}^l \binom{l}{r} \frac{\partial^r \lambdaefct}{\partial \lambdaefct^r} \cdot \frac{\partial^{l-r} \prn{n - \Tr \prn{\bSigma(\bSigma + \lambdaefct \bI)^{-1}}}}{\partial \lambdaefct^{l-r}} \\
	& = \lambdaefct \cdot \frac{\partial^{l} \prn{n - \Tr \prn{\bSigma(\bSigma + \lambdaefct \bI)^{-1}}}}{\partial \lambdaefct^{l}} + l \cdot \frac{\partial^{l-1} \prn{n - \Tr \prn{\bSigma(\bSigma + \lambdaefct \bI)^{-1}}}}{\partial \lambdaefct^{l-1}} \, .
\end{align*}
For $l=1$, since
\begin{align*}
	\frac{\partial \zeta}{\partial \lambdaefct} & = \lambdaefct \cdot \Tr \prn{\bSigma(\bSigma + \lambdaefct \bI)^{-2}} + n - \Tr \prn{\bSigma(\bSigma + \lambdaefct \bI)^{-1}} = n - \Tr \prn{\bSigma^2(\bSigma + \lambdaefct \bI)^{-2}} \nonumber \\
	& \geq n - \Tr \prn{\bSigma(\bSigma + \lambdaefct \bI)^{-1}} = \frac{\zeta}{\lambdaefct} \geq n \const \, ,
\end{align*}
we have $n \const \leq \partial \zeta / \partial \lambdaefct \leq n$. When $l \geq 2$, we get
\begin{align*}
	\frac{\partial^l \zeta}{\partial \lambdaefct^l} & = (-1)^{l-1} l! \cdot \lambdaefct \Tr \prn{\bSigma(\bSigma + \lambdaefct \bI)^{-l-1}} + (-1)^{l-2} l! \cdot  \Tr \prn{\bSigma(\bSigma + \lambdaefct \bI)^{-l}} \nonumber \\
	&	= (-1)^{l-2} l! \cdot \Tr \prn{\bSigma^2(\bSigma + \lambdaefct \bI)^{-l-1}} \nonumber \\
	&= \bigO_l \prn{\norm{(\bSigma + \lambdaefct \bI)^{-l+1}} \cdot \Tr \prn{\bSigma^2 (\bSigma + \lambdaefct \bI)^{-2}}} = \bigO_l \prn{\frac{n}{\lambdaefct^{l-1}}} \, .
\end{align*}
Substituting the above displays into Eq.~\eqref{eq:mid-derivative-lambda-2} yields
\begin{align*}
	\frac{\partial^k \lambdaefct}{\partial \zeta^k} & = \prn{\frac{1}{(n \const)^{2k-1}}} \cdot \sum_{\substack{m_1 + m_2 + \cdots + m_p = k-1  \\ m_1 + 2m_2 + \cdots + pm_p= 2k-2}} \bigO_k \prn{\prod_{l=1}^p \bigO_l \prn{\frac{n^{m_l}}{\lambdaefct^{lm_l -m_l}}}} = \bigO_{k} \prn{\frac{1}{n^k \lambdaefct^{k-1} \cdot \const^{2k-1}}} 
\end{align*}
Taken collectively with Eq.~\eqref{eq:mid-derivative-lambda-1} and $\Ffct_0(\zeta, \muefct(\zeta, 0); \bI) = \lambdaefct \Tr \prn{\bSigma (\bSigma + \lambdaefct \bI)^{-1} } \geq \const n \lambdaefct $, we obtain for all $k \geq 2$,
\begin{align*}
	\left|	\frac{\partial^k}{\partial \zeta^k}\Ffct_0(\zeta, \muefct(\zeta, 0); \bI) \right| = n 	\left|\frac{\partial^k \lambdaefct}{\partial \zeta^k}\right| = \bigO_{k} \prn{\frac{\Ffct_0(\zeta, \muefct(\zeta, 0); \bI)}{n^{k}\lambdaefct^{k} \const^{2k}}} = \bigO_{k} \prn{\frac{\Ffct_0(\zeta, \muefct(\zeta, 0); \bI)}{\zeta^k \const^{2k}}}\, ,
\end{align*}
where we use Assumption~\eqref{asmp:lambda} again for the final bound. This is also valid for $k=1$ as
\begin{align*}
	\left|	\frac{\partial}{\partial \zeta}\Ffct_0(\zeta, \muefct(\zeta, 0); \bI) \right| & = n 	\left|\frac{\partial \lambdaefct}{\partial \zeta}\right| + 1= \bigO \prn{\frac{\Ffct_0(\zeta, \muefct(\zeta, 0); \bI)}{\zeta \const^{2}}} + 1 = \bigO \prn{\frac{n \lambdaefct - \zeta}{\zeta \const^{2}}} + 1 \nonumber \\
	& = \bigO \prn{\frac{n \lambdaefct - \zeta}{\zeta \const^{2}}} = \bigO \prn{\frac{\Ffct_0(\zeta, \muefct(\zeta, 0); \bI)}{\zeta  \const^{2}}} \, ,
\end{align*}
where we use $\Ffct_0(\zeta, \muefct(\zeta, 0); \bI)/ \zeta = n\lambdaefct/\zeta -1$ and
\begin{align*}
	\frac{n \lambdaefct - \zeta}{\zeta \const^{2}} \geq \prn{(1-\const)^{-1} - 1} \const^{-2} \geq \const^{-1} \geq 1 \, .
\end{align*}

\paragraph{Part II: Derivative w.r.t. $\mu$} Now we fix $\zeta$ and allow $\mu$ be take nonzero values. We will also use the shorthand $\muefct = \muefct(\zeta, \mu)$. Similar to the previous part, we apply Fa\`a di Bruno's formula to $\Ffct_0$ and bound
\begin{align} \label{eq:mid-derivative-mu-1}
	\left|\frac{\partial^k}{\partial \mu}\Ffct_0(\zeta, \muefct; \btheta \btheta^\sT)\right|  = \sum_{m_1 + 2m_2 + \cdots + pm_p =k} \bigO_k \prn{\frac{\partial^{m_1 + \cdots + m_p} }{\partial \muefct^{m_1 + \cdots + m_p}}\Ffct_0(\zeta, \muefct; \btheta \btheta^\sT)  \cdot \prod_{l=1}^p \prn{\frac{\partial^l \muefct}{\partial \mu^l }}^{m_l}}\, .
\end{align}
For any $1 \leq l \leq k-1$, we have
\begin{align*}
	\left|\frac{\partial^l}{\partial \muefct^l} \Ffct_0(\zeta, \muefct; \btheta \btheta^\sT)\right| & = \bigO_l \prn{\zeta \Tr \prn{\bSigma^{\half} \btheta \btheta^\sT \bSigma^{\half} \cdot \bSigma^{l} (\zeta \bI + \muefct \bSigma)^{-1-l} }} = \bigO_l \prn{\frac{\Ffct_0(\zeta, \muefct; \btheta \btheta^\sT)}{\muefct^l}}.
\end{align*}
To bound higher order derivatives $\partial^l \muefct/\partial \mu^l$, we apply again the formula for higher order derivatives of inverse function. Of course, this would first require showing the existence of inverse function by implicit function theorem, which will be evident as we will provide a lower bound for $|\partial \mu/ \partial \muefct|$ below. By~\cite{apostol2000calculating}, we have for all $1 \leq l \leq k-1$,
\begin{align} \label{eq:mid-derivative-mu-2}
	\left|\frac{\partial^l \muefct}{\partial \mu^l }\right| & = \left|\frac{\partial \mu}{\partial \muefct} \right|^{1-2l} \cdot \sum_{\substack{m_1 + m_2 + \cdots + m_p = l-1  \\ m_1 + 2m_2 + \cdots + pm_p= 2l-2}} \bigO_l \prn{\prod_{r=1}^p \prn{\frac{\partial^r \mu}{\partial \muefct^r}}^{m_r}}\,. 
\end{align}

This is a more manageable formula as we can explicitly write $\mu$ as a function of $\muefct$
\begin{align*}
	\mu = \muefct - \frac{n}{1 + \Rfct_0(\zeta, \muefct; \bI)} = \muefct - \frac{n}{1 + \Tr(\bSigma(\zeta \bI + \muefct \bSigma)^{-1})}\, .
\end{align*}
We can compute the first order derivative as
\begin{align*}
	\frac{\partial \mu}{\partial \muefct} = 1- \frac{n \Tr \prn{\bSigma^2 (\zeta \bI + \muefct \bSigma)^{-2}}}{\prn{1 + \Tr \prn{\bSigma (\zeta \bI + \muefct \bSigma)^{-1}}}^2} = 1 - \frac{\prn{\muefct - \mu} \cdot\Tr \prn{\bSigma^2 (\zeta \bI + \muefct \bSigma)^{-2}} }{1 + \Tr \prn{\bSigma (\zeta \bI + \muefct \bSigma)^{-1}}}\, , 
\end{align*}
which, together with $0 \leq \mu \leq \muefct/2$, implies a lower bound
\begin{align*}
	\frac{\partial \mu}{\partial \muefct} &  \geq 1 - \frac{\muefct\Tr \prn{\bSigma^2 (\zeta \bI + \muefct \bSigma)^{-2}} }{1 + \Tr \prn{\bSigma (\zeta \bI + \muefct \bSigma)^{-1}}} = \frac{1 + \zeta \Tr \prn{\bSigma (\zeta \bI + \muefct \bSigma)^{-2}}}{1 + \Tr \prn{\bSigma (\zeta \bI + \muefct \bSigma)^{-1}}} \nonumber \\
	& \geq \frac{1}{1 + \Tr \prn{\bSigma (\zeta \bI + \muefct \bSigma)^{-1}}}  = \frac{\muefct - \mu}{n} \geq \frac{\muefct}{2n}\, .
\end{align*}
To further bound higher order derivatives, we again appeal to Fa\`a di Bruno's formula. Use the shorthand $\Rfct_0 = \Rfct_0(\zeta, \muefct; \bI)$,  we have for all $r \geq 1$,
\begin{align*}
	\left|\frac{\partial^r \mu}{\partial \muefct^r} \right|  = \sum_{m_1 + 2m_2 + \cdots + pm_p =r} \bigO_r \prn{\frac{\partial^{m_1 + \cdots + m_p} }{\partial \Rfct_0^{m_1 + \cdots + m_p}} \frac{n}{1 + \Rfct_0}  \cdot \prod_{s=1}^p \prn{\frac{\partial^s \Rfct_0}{\partial \muefct^s }}^{m_s}}\, .
\end{align*}
Making use of the following two bounds,
\begin{align*}
	\frac{\partial^{s} }{\partial \Rfct_0^{s}} \frac{n}{1 + \Rfct_0} &= \bigO_s \prn{\frac{n}{(1 + \Rfct_0)^{s+1}}} = \bigO_s \prn{\frac{\muefct}{(1 + \Rfct_0)^s}}\, ,\\
	\frac{\partial^s \Rfct_0}{\partial \muefct^s } & = \bigO_s \prn{\Tr \prn{\bSigma^{s+1} (\zeta \bI + \muefct \bSigma)^{-s-1}}} = \bigO_s \prn{\frac{\Rfct_0}{\muefct^s}} \, ,
\end{align*}
we can further obtain
\begin{align*}
	\left|\frac{\partial^r \mu}{\partial \muefct^r} \right|  = \sum_{m_1 + 2m_2 + \cdots + pm_p =r} \bigO_r \prn{\frac{\muefct}{(1 + \Rfct_0)^{m_1 + \cdots + m_p}}  \cdot \prod_{s=1}^p \frac{\Rfct_0^{m_s}}{\muefct^{s m_s}}} = \bigO_r \prn{\frac{1}{\muefct^{r-1}}}\,. 
\end{align*}
Taking the above displays into Eq.~\eqref{eq:mid-derivative-mu-2} and use the condition $\muefct / n \geq \const$, we have
\begin{align*}
	\left|\frac{\partial^l \muefct}{\partial \mu^l }\right| = \bigO_{l}\prn{\frac{1}{\const^{2l-1}}} \cdot  \sum_{\substack{m_1 + m_2 + \cdots + m_p = l-1  \\ m_1 + 2m_2 + \cdots + pm_p= 2l-2}} \bigO_l \prn{\prod_{r=1}^p \bigO_r \prn{\frac{1}{\muefct^{rm_r-m_r}}}} = \bigO_{l}\prn{\frac{1}{\muefct^{l-1} \const^{2l-1}}}\, .
\end{align*}
Finally, taking the above display back into Eq.~\eqref{eq:mid-derivative-mu-1} yields
\begin{align*}
		\left|\frac{\partial^k}{\partial \mu^k}\Ffct_0(\zeta, \muefct; \btheta \btheta^\sT)\right| & = \sum_{m_1 + 2m_2 + \cdots + pm_p =k} \bigO_k \prn{\frac{\Ffct_0(\zeta, \muefct; \btheta \btheta^\sT)}{\muefct^{m_1+\cdots +m_p}}  \cdot \prod_{l=1}^p \bigO_{l}\prn{\frac{1}{\muefct^{lm_l-m_l} \const^{2lm_l - m_l}}}} \nonumber \\
		&= \bigO_{k} \prn{\frac{\Ffct_0(\zeta, \muefct; \btheta \btheta^\sT)}{\muefct^k \const^{2k}}}\, .
\end{align*}

\subsection{Proof of Lemma~\ref{lem:muefct-small-mu-bound}} \label{proof:muefct-small-mu-bound}
First we show $\muefct(\zeta, \mu)$ is increasing in $\mu$ when $\mu \geq 0$. To this end, we consider the function
\begin{align*}
	f(t) = t  - \frac{n}{1 + \Rfct_0(\zeta, t; \bI)} \, .
\end{align*}
By Eq.~\eqref{eq:mu-fixed-point}, we have $f(\muefct(\zeta, \mu)) = \mu$ for all $\mu \geq 0$. Further, we prove $f(t)$ is increasing in $[\muefct(\zeta, 0), \infty)$. We write
\begin{align*}
	f'(t) = 1- \frac{n \Tr \prn{\bSigma^2 (\zeta \bI + t \bSigma)^{-2}}}{\prn{1 + \Tr \prn{\bSigma (\zeta \bI + t \bSigma)^{-1}}}^2} \stackrel{\mathrm{(i)}}{=} 1 - \frac{\prn{t - f(t)} \cdot\Tr \prn{\bSigma^2 (\zeta \bI + t \bSigma)^{-2}} }{1 + \Tr \prn{\bSigma (\zeta \bI + t \bSigma)^{-1}}}\, , 
\end{align*}
where in (i) we use that $t - f(t) = n/(1 + \Rfct_0(\zeta, t; \bI))$. Define
\begin{align*}
	g(t) :=  \frac{\Tr \prn{\bSigma^2 (\zeta \bI + t \bSigma)^{-2}} }{1 + \Tr \prn{\bSigma (\zeta \bI + t \bSigma)^{-1}}} \, ,
\end{align*}
we have
\begin{align*}
	f'(t) - f(t) g(t) = 1 - \frac{t \Tr \prn{\bSigma^2 (\zeta \bI + t \bSigma)^{-2}}}{1 + \Tr \prn{\bSigma (\zeta \bI + t \bSigma)^{-1}}} = \frac{1 + \zeta \Tr \prn{\bSigma (\zeta \bI + t \bSigma)^{-2}}}{1 + \Tr \prn{\bSigma (\zeta \bI + t \bSigma)^{-1}}} > 0 \, ,
\end{align*}
and therefore $e^{-g(t)}f(t)$ is increasing. As $f(\muefct(\zeta, 0)) = 0$ (cf.~Eq.~\eqref{eq:mu-fixed-point}), we must have $f(t) \geq 0$ for all $t \geq \muefct(\zeta, 0)$. Substituting back into the above display with $g(t) \geq 0$ yields
\begin{align*}
	f'(t) \geq f'(t) - f(t) g(t) > 0 \, .
\end{align*}

We then proceed to show a sufficient condition for $\muefct(\zeta, \mu) \leq (1 - \const/2)^{-1} \muefct(\zeta, 0)$ is $0 \leq \mu \leq n \const^3/2$ under Assumption~\eqref{asmp:lambda}. Provided with monotonicity of $f(t)$, the desired condition $\muefct(\zeta, \mu) \leq (1 - \const/2)^{-1} \muefct(\zeta, 0)$ is essentially equivalent to $\mu = f(\muefct(\zeta, \mu)) \leq f((1 - \const/2)^{-1} \muefct(\zeta, 0))$. Together with $\muefct(\zeta, 0) = n/(1 + \Rfct_0(\zeta, \muefct(\zeta, 0); \bI))$, we obtain a lower bound for the right hand side
\begin{align*}
	& f((1 - \const/2)^{-1} \muefct(\zeta, 0)) \\
	& = (1 - \const/2)^{-1} \muefct(\zeta, 0) - \frac{n}{1 + \Tr \prn{\bSigma (\zeta \bI + (1 - \const/2)^{-1} \muefct(\zeta, 0) \bSigma)^{-1}}} \\
	& = \frac{n (1 - \const/2)^{-1} }{1 + \Tr \prn{\bSigma (\zeta \bI +  \muefct(\zeta, 0) \bSigma)^{-1}} } - \frac{n (1 - \const/2)^{-1}}{(1 - \const/2)^{-1} + \Tr \prn{\bSigma ((1 - \const/2)\zeta \bI +  \muefct(\zeta, 0) \bSigma)^{-1}}} \\
	& \stackrel{\mathrm{(i)}}{\geq} \frac{n (1 - \const/2)^{-1} }{1 + \Tr \prn{\bSigma (\zeta \bI +  \muefct(\zeta, 0) \bSigma)^{-1}} } - \frac{n (1 - \const/2)^{-1}}{(1 - \const/2)^{-1} + \Tr \prn{\bSigma (\zeta \bI +  \muefct(\zeta, 0) \bSigma)^{-1}}} \nonumber \\
	& = \frac{n  \prn{(1 - \const/2)^{-1} - 1}}{\prn{1 + \Tr \prn{\bSigma (\zeta \bI +  \muefct(\zeta, 0) \bSigma)^{-1}}} \cdot \prn{1 + (1 - \const/2) \Tr \prn{\bSigma (\zeta \bI +  \muefct(\zeta, 0) \bSigma)^{-1}}} } \\
	& \stackrel{\mathrm{(ii)}}{\geq} \frac{n  \prn{(1 - \const/2)^{-1} - 1}}{\prn{1 + \Tr \prn{\bSigma (\zeta \bI +  \muefct(\zeta, 0) \bSigma)^{-1}}}^2} \\
	& = \frac{ \prn{(1 - \const/2)^{-1} - 1} \muefct(\zeta,0)^2}{n} \, ,
\end{align*}
where in (i) and (ii) we use two times the trivial bound $1 - \const/2 \leq 1$. By Assumption~\eqref{asmp:lambda},
\begin{align*}
	\frac{\muefct(\zeta,0)}{n} = \frac{\zeta}{n\lambdaefct} = \frac{\lambda}{\lambdaefct} \geq \const, 
\end{align*}
we know 
\begin{align*}
	f((1 - \const/2)^{-1} \muefct(\zeta, 0)) \geq n \cdot \const^2 \prn{(1 - \const/2)^{-1} - 1} \geq n \cdot \const^3 / 2 \, ,
\end{align*}
and thus a sufficient condition for $\muefct(\zeta, \mu) \leq (1 - \const/2)^{-1}\muefct(\zeta, 0)$ is $\mu \leq n \const^3 /2$.

	\section{Proofs for Theorem~\ref{thm:R-approximation}}
\subsection{Proof of Lemma~\ref{lem:update-rule-guarantee}} \label{proof:update-rule-guarantee}
We define $\wb{\mu}_i := \inf \left\{\mu \mid \bSigma^{\half} \prn{\zeta \bI + \mu \bSigma + \bX_i^\sT \bX_i}^{-1} \bSigma^{\half} \succ 0 \right\}$.
Note that, by construction $\wb{\mu}_{i+1}\le \wb{\mu}_i$.
 Let $\bvarphi \in \real^d$ be the leading normalized eigenvector of $\bSigma$. If $\mu \leq - (\zeta + \norm{\bX_i \bvarphi}^2) / \normop{\bSigma}$, it follows that
\begin{align*}
	\bvarphi^\sT \prn{\zeta \bI + \mu \bSigma + \bX_i^\sT \bX_i} \bvarphi = \zeta +\mu \normop{\bSigma} + \norm{\bX_i \bvarphi}^2 \leq 0\, ,
\end{align*}
which implies $\wb{\mu}_i \geq - (\zeta + \norm{\bX_i \bvarphi}^2) / \normop{\bSigma} > -\infty$. The update rule is equivalent to solving the equation
\begin{align*}
	\mu_{i+1} + \frac{1}{1 + \Tr \prn{\bSigma \prn{\zeta \bI + \mu_{i+1} \bSigma + \bX_i^\sT \bX_i}^{-1}}} = \mu_i \, , \qquad \mu_{i+1} \in (\wb{\mu}_i, \infty)\, .
\end{align*}
For all $t \in (\wb{\mu}_i, \infty)$, let
\begin{align*}
	f(t) = t +\frac{1}{1 + \Tr \prn{\bSigma \prn{\zeta \bI + t \bSigma + \bX_i^\sT \bX_i}^{-1}}}\, .
\end{align*}
In this given domain, $\bSigma^{\frac 1 2} \prn{\zeta \bI + t \bSigma + \bX_i^\sT \bX_i}^{-1} \bSigma^{\frac 1 2} \succ 0$ and thus
 $\Tr \prn{\bSigma \prn{\zeta \bI + t \bSigma + \bX_i^\sT \bX_i}^{-1}}$ is decreasing in $t$
 (this can be seen by computing its derivative with respect to $t$), which further implies $f(t)$ is strictly increasing in $t$. Since
$$\lim_{t\downarrow \wb{\mu}_i} \Tr \prn{\bSigma \prn{\zeta \bI + t \bSigma + \bX_i^\sT \bX_i}^{-1}} = \infty \, ,$$ and we have
\begin{align*}
	\lim_{t \downarrow \wb{\mu}_i} f(t) & = \wb{\mu}_i < \mu_i\, , \qquad f(\mu_i) > \mu_i\, .
\end{align*}
(The first inequality follows since $\mu_i\in (\wb{\mu}_{i-1},\infty)$ and $\wb{\mu}_i\le\wb{\mu}_{i-1}$.)
Thus, there must be a unique $\mu_{i+1} \in (\wb{\mu}_i, \mu_i)$ that solves $f(\mu_{i+1}) = \mu_i$, proving the lemma.

\subsection{Proof of Lemma~\ref{lem:A-norm-bound}} \label{proof:A-norm-bound}
Without loss of generality, we can always assume $d \geq n$ or simply $d =\infty$ by embedding $\real^d$ into the Hilbert space $\ell_2$ since we always have
\begin{align*}
	\sum_{l=k}^d \sigma_l \leq \constantsig \sigma_k \, , 
\end{align*} 
when $\sigma_k = 0$. We write the spectral decomposition of $\bSigma$ as
\begin{align*}
	\bSigma = \sum_{i=1}^d \sigma_i \bv_i \bv_i^\sT \, ,
\end{align*}
with $\sigma_1 \geq \sigma_2 \geq \cdots \geq \sigma_n \geq \cdots$ where $\brc{\bv_i}$ form an orthogonal basis of eigenvectors. For any $k \leq n$, define the projection operators
\begin{align*}
	\proj_k := \sum_{i=1}^k \bv_i \bv_i^\sT \, , \qquad \proj_k^\perp := \bI - \proj_k  = \sum_{i=k+1}^{d} \bv_i \bv_i^\sT \, ,
\end{align*}
and we write
\begin{align*}
	\bX = \bX \proj_k + \bX \proj_k^\perp := \bU_k + \bW_k \, .
\end{align*}
\paragraph{Part I: Decomposing into the top and lower eigenspaces} By writing $\bX = \bU_k + \bW_k$, we can have the following inequality:
\begin{lemma} \label{lem:cov-lower-bound-eigenspace} For any $1 \leq k \leq n-1$, 
\begin{align*}
	\zeta \bI + \bX^\sT \bX & \succeq \frac{\zeta}{2} \bI + \prn{1 + \frac{2\norm{\bW_k^\sT \bW_k}}{\zeta}}^{-1} \bU_k^\sT \bU_k \, .
\end{align*}
\end{lemma}
\begin{proof}
	Note that
	\begin{align}
		\zeta \bI + \bX^\sT \bX  
		&= \zeta \bI + \bU_k^\sT \bU_k + \bU_k^\sT \bW_k + \bW_k^\sT \bU_k + \bW_k^\sT \bW_k \nonumber \\
		& \succeq \frac{\zeta}{2} \bI + \bU_k^\sT \bU_k + \bU_k^\sT \bW_k + \bW_k^\sT \bU_k + \prn{1 + \frac{\zeta}{2\norm{\bW_k^\sT \bW_k}}} \bW_k^\sT \bW_k \nonumber \\
		& = \frac{\zeta}{2} \bI + \prn{1 + \frac{2\norm{\bW_k^\sT \bW_k}}{\zeta}}^{-1}  \bU_k^\sT \bU_k + \bC_k \, , \nonumber
	\end{align}
	where
	\begin{align*}
		\bC_k = \prn{1 + \frac{\zeta}{2\norm{\bW_k^\sT \bW_k}}}^{-1} \bU_k^\sT \bU_k + \bU_k^\sT \bW_k + \bW_k^\sT \bU_k + \prn{1 + \frac{\zeta}{2\norm{\bW_k^\sT \bW_k}}} \bW_k^\sT \bW_k  = \bD_k^\sT \bD_k \succeq 0 \, ,
	\end{align*}
	with
	\begin{align*}
		\bD_k = \prn{1 + \frac{\zeta}{2\norm{\bW_k^\sT \bW_k}}}^{-\frac 1 2} \bU_k +  \prn{1 + \frac{\zeta}{2\norm{\bW_k^\sT \bW_k}}}^{\frac 1 2} \bW_k \, .
	\end{align*}

\end{proof}
\noindent To apply the above lemma, we need to further provide an upper bound on $\|\bW_k^\sT \bW_k\|$, which we summarize as the following result.

\begin{lemma} \label{lem:W-k-norm-upper-bound}
	Let Assumption~\ref{asmp:data-dstrb} holds, we have for any $1 \leq k \leq n-1$ with probability $1 - \bigO(n^{-D})$ that,
	\begin{align*}
		\|\bW_k \bW_k^\sT \| = \bigO_{\constantx, D} \prn{\constantsig \sigma_k \cdot \log n \log (\constantsig n)} \, .
	\end{align*}
\end{lemma}
\begin{proof}
Let $\bzeta_i = \proj_k^\perp \bx_i \bx_i^\sT \proj_k^\perp \in \real^{d \times d}$, we can write
\begin{align*}
	\bS_k := \bW_k^\sT \bW_k = \sum_{i=1}^n \proj_k^\perp \bx_i \bx_i^\sT \proj_k^\perp = \sum_{i=1}^n \bzeta_i \, .
\end{align*}
Since $\norm{\bzeta_i} = \norm{\proj_k^\perp \bx_i}^2$, we can apply Hanson-Wright inequality (cf.~Lemma~\ref{lem:hanson-wright}) and conclude that
\begin{align*}
	\Prb \prn{\left|\norm{\bzeta_i} - \Tr \prn{\proj_k^\perp \bSigma} \right| \geq t} & = \Prb \prn{\left|\bx_i^\sT \proj_k^\perp \bx_i - \Tr \prn{\proj_k^\perp \bSigma} \right| \geq t} \nonumber \\
	& \leq 2 \exp \left\{ - \Omega \prn{ \min \brc{\frac{t^2}{\constantx^4 \normF{\bSigma^{\frac 1 2} \proj_k^\perp \bSigma^{\frac 1 2} }^2}, \frac{t}{\constantx^2 \normop{\bSigma^{\frac 1 2} \proj_k^\perp \bSigma^{\frac 1 2} }}}} \right\} \, .
\end{align*}
For $t = \bigTht_{\constantx, D} (\|\bSigma^{\frac 1 2} \proj_k^\perp \bSigma^{\frac 1 2}\|_F \log n)$ we have with probability $1-\bigO(n^{-D})$ that for all $i=1,2,\cdots, n$ 
\begin{align*}
	\norm{\bzeta_i} \leq  \Tr \prn{\proj_k^\perp \bSigma} + \bigTht_{\constantx, D}  \prn{ \normF{\bSigma^{\frac 1 2} \proj_k^\perp \bSigma^{\frac 1 2}} \log n} = \bigO_{\constantx, D} \prn{ \Tr \prn{\proj_k^\perp \Sigma} \log n} \, ,
\end{align*}
where the last inequality follows from $\norm{\bSigma} = 1$ and 
\begin{align*}
	\normF{\bSigma^{\frac 1 2} \proj_k^\perp \bSigma^{\frac 1 2}} & = \sqrt{\Tr \prn{\proj_k^\perp \bSigma\proj_k^\perp \bSigma}}\leq \Tr \prn{\proj_k^\perp \bSigma} \, .
\end{align*}

In the next step, we will adopt a standard truncation argument and apply a matrix concentration inequality. By setting $L_k := \bigTht_{\constantx, D} (\Tr \prn{\proj_k^\perp \bSigma} \log n )$, $\tilde{\bzeta}_i :=   \bzeta_i \ind \{\|\bzeta_i \| \leq L_k\}$ and considering
\begin{align*}
	\tilde{\bS}_k := \sum_{i=1}^k \tilde{\bzeta}_i = \sum_{i=1}^k \bzeta_i \ind \brc{\norm{\bzeta_i} \leq L_k} \, ,
\end{align*}
we have $\tilde{\bS}_k = \bS_k$ with probability $1-\bigO(n^{-D})$. It order to
bound $\|\tilde{\bS}_k\|$, we will use matrix Bernstein inequality. Since we 
know $\|\tilde{\bzeta}_i\|\le L_k$ by construction, we only need to upper bound the matrix variance. 
The $\tilde{\bzeta}_i$'s' are independent symmetric random matrices and therefore we have
\begin{align*}
	\Var (\tilde{\bS}_k) & \preceq \sum_{i=1}^n  \Ep [\tilde{\bzeta}_i^2] \stackrel{\mathrm{(i)}}{\preceq} \sum_{i=1}^n L_k \Ep [\tilde{\bzeta}_i] \stackrel{\mathrm{(ii)}}{\preceq}  \sum_{i=1}^n L_k \Ep \brk{\bzeta_i} \stackrel{\mathrm{(iii)}}{\preceq}  n L_k \cdot \proj_k^\perp \bSigma \proj_k^\perp =: \bV_k \, ,
\end{align*}
where in (i) we use $\|\tilde{\bzeta}_i\| \leq L_k$, in (ii) we apply 
$\tilde{\bzeta}_i \preceq \bzeta_i$ and lastly in (iii) we use 
$\Ep[\bzeta_i] = \Ep [\proj_k^\perp \bx_i \bx_i^\sT \proj_k^\perp]  = 
\proj_k^\perp \bSigma \proj_k^\perp$. It then follows that 
$\|\bV_k\| \leq nL_k \|\proj_k^\perp \bSigma \proj_k^\perp\| = n \sigma_{k+1} L_k \leq n \sigma_k L_k =: v_k$. Combine with the bound on the intrinsic dimension under Assumption~\ref{asmp:data-dstrb},
\begin{align*}
	\intdim \prn{\bV_k} & = \frac{\Tr \prn{\bV_k}}{\norm{\bV_k}} = \frac{\sum_{l=k+1}^\infty \sigma_l}{\sigma_{k+1}} \leq \constantsig \, ,
\end{align*}
we can thus deduce from the Bernstein inequality with intrinsic dimension~\cite[Theorem~7.3.1]{tropp2015introduction} that for $t \geq \sqrt{v_k} + L_k/3$        
\begin{align*}
	\Prb (\|\tilde{\bS}_k - \Ep [\tilde{\bS}_k]\| \geq t) & \leq 4 \constantsig \cdot \exp \prn{\frac{-t^2/2}{v_k + L_k t/3}} \, .
\end{align*}
Finally, by further bounding the mean
\begin{align*}
	\|\Ep [\tilde{\bS}_k]\| & \leq \|n\Ep [\tilde{\bzeta}_i]\|  \leq n \cdot \norm{\Ep \brk{\bzeta_i}} = n \cdot \norm{\proj_k^\perp \bSigma \proj_k^\perp} \leq n\sigma_k  \, ,
\end{align*}
we can obtain with probability $1-\bigO(n^{-D})$,
\begin{align}
	\|{\tilde{\bS}_k}\| & = \bigO_D \prn{ \prn{\sqrt{v_k} + L_k} \log (\constantsig n)}+ \|\Ep [\tilde{\bS}_k]\| \nonumber \\
	& \stackrel{\mathrm{(i)}}{=} \bigO_{\constantx, D} \prn{\brc{\sqrt{n\sigma_k \cdot \Tr \prn{\proj_k^\perp \bSigma} \log n } + \Tr \prn{\proj_k^\perp \bSigma} \log n} \cdot \log (\constantsig n) + n \sigma_k} \nonumber \\
	& \stackrel{\mathrm{(ii)}}{=} \bigO_{\constantx, D} \prn{\brc{\sqrt{n\sigma_k \cdot \constantsig \sigma_k  \log n } + \constantsig \sigma_k \log n} \cdot \log (\constantsig n) + n \sigma_k} \nonumber 
\end{align}
where in (i) we make use of $v_k = n\sigma_k L_k$, and apply Assumption~\ref{asmp:data-dstrb} for the spectrum in (ii). Next by the fact that $\constantsig \geq n$, we can further write
\begin{align*}
		\|{\tilde{\bS}_k}\| & = \bigO_{\constantx, D} \prn{\constantsig \sigma_k \cdot \log n \log (\constantsig n)}\, .
\end{align*}
The proof is complete as $\bW_k \bW_k^\sT =  \bS_k = \tilde{\bS}_k =$ holds with probability $1-\bigO(n^{-D})$.
\end{proof}
To bound the norm of $\bSigma^{\frac 1 2} \prn{\zeta \bI + \bX^\sT \bX}^{-1} \bSigma^{\frac 1 2}$, we apply Lemmas~\ref{lem:cov-lower-bound-eigenspace} and~\ref{lem:W-k-norm-upper-bound} and obtain
\begin{align*}
	\zeta \bI + \bX^\sT \bX & \succeq \frac{\zeta}{2} \bI + \prn{1 + \frac{2\norm{\bW_k^\sT \bW_k}}{\zeta}}^{-1} \bU_k^\sT \bU_k \nonumber \\
	& \succeq \frac{\zeta}{2} \bI + \prn{1 + \frac{\bigO_{\constantx, D} \prn{\constantsig \sigma_k \cdot \log n \log (\constantsig n)}}{\zeta}}^{-1} \bU_k^\sT \bU_k \, .
\end{align*}
Therefore by block matrix inverse, we can further get
\begin{align}
	&\bSigma^{\frac 1 2} \prn{\zeta \bI + \bX^\sT \bX}^{-1} \bSigma^{\frac 1 2} \nonumber \\
	& \preceq \bSigma^{\frac 1 2} \prn{\frac{\zeta}{2} \bI + \prn{1 + \frac{\bigO_{\constantx, D} \prn{\constantsig \sigma_k \cdot \log n \log (\constantsig n)}}{\zeta}}^{-1} \bU_k^\sT \bU_k}^{-1} \bSigma^{\frac 1 2} \nonumber \\
	& = \bSigma^{\frac 1 2} \prn{\frac{\zeta}{2} \bI + \prn{1 + \frac{\bigO_{\constantx, D} \prn{\constantsig \sigma_k \cdot \log n \log (\constantsig n)}}{\zeta}}^{-1} \proj_k \bSigma^{\half} \bZ^\sT \bZ \bSigma^{\half} \proj_k}^{-1} \bSigma^{\frac 1 2} \nonumber \\
	& \preceq \prn{1 + \frac{\bigO_{\constantx, D} \prn{\constantsig \sigma_k \cdot \log n \log (\constantsig n)}}{\zeta}} \prn{\proj_k \bZ^\sT \bZ \proj_k}^\dagger + \frac{2\proj_k^\perp \bSigma \proj_k^\perp}{\zeta} \, , \label{eq:resolvent-norm-bound}
\end{align}
where $\bX = \bZ \bSigma^{\frac 1 2}$. Define the matrix $\bV_k = \begin{bmatrix} \bv_1 & \cdots & \bv_k \end{bmatrix} \in \real^{d \times k}$, we can then write $\proj_k = \bV_k \bV_k^\sT$. Thus by exploiting the block matrix structure, it follows that
\begin{align}
	\norm{\bSigma^{\frac 1 2} \prn{\zeta \bI + \bX^\sT \bX}^{-1} \bSigma^{\frac 1 2}} \leq \prn{1 + \frac{\bigO_{\constantx, D} \prn{\constantsig \sigma_k \cdot \log n \log (\constantsig n)}}{\zeta}} \lambda_{\min} \prn{\bV_k^\sT \bZ^\sT \bZ \bV_k}^{-1} + \frac{2 \sigma_k}{\zeta} \label{eq:mid-A-norm-bound-1} \, .
\end{align}
Substituting $\btheta = \bSigma^{-1/2} \boldbeta$ into Eq.~\eqref{eq:resolvent-norm-bound}, we also obtain
\begin{align}
	&\btheta^\sT \bSigma^{\frac 1 2} \prn{\zeta \bI + \bX^\sT \bX}^{-1} \bSigma^{\frac 1 2} \btheta \nonumber \\
	&\leq  \prn{1 + \frac{\bigO_{\constantx, D} \prn{\constantsig \sigma_k \cdot \log n \log (\constantsig n)}}{\zeta}} \lambda_{\min} \prn{\bV_k^\sT \bZ^\sT \bZ \bV_k}^{-1} \norm{\btheta_{\leq k}}^2 + \frac{2 \norm{\boldbeta_{>k}}^2}{\zeta} \, .\label{eq:mid-theta-resolvent-bound-1} 
\end{align}
\paragraph{Part II: Lower bounding the smallest eigenvalue \texorpdfstring{$\lambda_{\min} \prn{\bV_k^\sT \bZ^\sT \bZ \bV_k}$}{TEXT}}
The last step is then to provide a lower bound for the smallest eigenvalue of $\bV_k^\sT \bZ^\sT \bZ \bV_k$. 
 Consider $\tilde{\bZ} = \begin{bmatrix} \tilde{\bz}_1 & \cdots & \tilde{\bz}_n \end{bmatrix}^\sT \in \mathbb{R}^{n \times k}$ with 
\begin{align*}
	\tilde{\bz}_i = \bV_k^\sT \bz_i =  \begin{bmatrix} \langle \bz, \bv_1\rangle \\ \vdots \\ \langle \bz, \bv_{k-1} \rangle \end{bmatrix} \, .
\end{align*}
We therefore need to lower bound $\lambda_{\min}(\tilde{\bZ}^\sT \tilde{\bZ})$ 
where $\tilde{\bZ}$ has  i.i.d. rows $\tilde{\bz}_i$ in $\real^k$. 
An immediate consequence is that $\Ep [\tilde{\bz}_i] = \boldsymbol{0}$ and
 $\Var (\tilde{\bz}_i) = \bI_k$. Moreover, for any unit vector $\bvarphi \in \real^k$, 
 we can apply Hanson-Wright (cf.~Lemma~\ref{lem:hanson-wright}) and deduce that for any $t \geq 0$,
\begin{align}
	\Prb \prn{\left|\langle \tilde{\bz}_i, \bvarphi\rangle^2 - 1\right| \geq t} & = \Prb \prn{\left|\bz_i^\sT \bV_{k} \bvarphi \bvarphi^\sT \bV_{k}^\sT \bz_i - \Tr \prn{\bV_{k} \bvarphi \bvarphi^\sT \bV_{k}^\sT} \right| \geq t} \nonumber \\
	& \leq 2 \exp \left\{ - \bigOmg \prn{ \min \brc{\frac{t^2}{\constantx^4 \normF{\bV_{k} \bvarphi \bvarphi^\sT \bV_{k}^\sT }^2}, \frac{t}{\constantx^2 \normop{\bV_{k} \bvarphi \bvarphi^\sT \bV_{k}^\sT}}} } \right\} \nonumber \\
	& = 2 \exp \prn{-  \bigOmg_{\constantx} \prn{ \min \brc{t^2, t}}}, \label{eq:hanson-wright-for-z-top-k}
\end{align}
where we use the fact that $\normF{\bV_{k} \bvarphi \bvarphi^\sT \bV_{k}^\sT } = \normop{\bV_{k} \bvarphi\bvarphi^\sT \bV_{k}^\sT } \leq 1$. Thus we can bound the fourth moment of $\langle \tilde{\bz}_i, \bvarphi \rangle$ by
\begin{align*}
	\Ep \brk{\langle \tilde{\bz}_i, \bvarphi\rangle^4} &= \int_0^\infty 2t 	\Prb \prn{\langle \tilde{\bz}_i, \bvarphi\rangle^2 \geq t}  \de t  \nonumber \\
	 &\leq 1 + \int_0^\infty 2 (t + 1) 	\Prb \prn{\langle \tilde{\bz}_i, \bvarphi\rangle^2 \geq t + 1}  \de t \nonumber \\
	 & \leq 1 + 4 \int_0^\infty (t+1) \exp \prn{-  \bigOmg_{\constantx} \prn{ \min \brc{t^2, t}}} \de t = \bigO_{\constantx}(1) \, .
\end{align*}
Clearly the above bound holds uniformly for all $\bvarphi \in \mathbb{S}^{k-1}$ from the unit sphere in $\real^k$. Since the upper bound $\bigO_{\constantx}(1)$ does not depend on $k$, we can appeal to~\cite[Theorem~2.2]{yaskov2014lower} and obtain that with probability $1-\bigO(n^{-D})$
\begin{align*}
	\lambda_{\min} \prn{n^{-1} \tilde{\bZ}^\sT \tilde{\bZ}} \geq 1 - \bigO_{\constantx} \prn{\sqrt{\frac{k}{n}}} - \bigO_{D} \prn{\sqrt{\frac{ \log n}{n}}} \, .
\end{align*}
Therefore, if we choose $k = \lfloor \eta n\rfloor$ for some fixed $\eta$ such that $\bigO_{\constantx} (\sqrt{\eta}) \leq 1/4$,  it holds for $n = \Omega_D(1)$ that
\begin{align*}
	\lambda_{\min} \prn{n^{-1} \tilde{\bZ}^\sT \tilde{\bZ}} \geq 1 - \frac{1}{4} - \frac{1}{4} = \half \, ,
\end{align*}
and we therefore conclude the proof by taking $k = \lfloor \eta n \rfloor$ as above and substituting into Eq.~\eqref{eq:mid-A-norm-bound-1}
\begin{align*}
	\norm{\bSigma^{\frac 1 2} \prn{\zeta \bI + \bX^\sT \bX}^{-1} \bSigma^{\frac 1 2}} & \leq \frac{2}{n} \prn{1 + \frac{\bigO_{\constantx, D} \prn{\constantsig \sigma_k \cdot \log n \log (\constantsig n)}}{\zeta}}  + \frac{2 \sigma_{k}}{\zeta} \nonumber \\
	&= \frac{2}{n} \prn{1 + \frac{\bigO_{\constantx, D} \prn{\constantsig \sigma_k \cdot \log n \log (\constantsig n)}}{\zeta}} \, ,
\end{align*}
where in the last line we use the fact that $\constantsig \geq n$ and therefore $\sigma_{k} = \bigO (\constantsig \sigma_k \cdot \log n \log (\constantsig n) / n)$. 

Similarly for Eq.~\eqref{eq:mid-theta-resolvent-bound-1}, we have
\begin{align*}
	 \btheta^\sT \bSigma^{\frac 1 2} \prn{\zeta \bI + \bX^\sT \bX}^{-1} \bSigma^{\frac 1 2} \btheta  & \leq \frac{2}{n}\prn{1 + \frac{\bigO_{\constantx, D} \prn{\constantsig \sigma_k \cdot \log n \log (\constantsig n)}}{\zeta}}\norm{\btheta_{\leq k}}^2 + \frac{2 \norm{\boldbeta_{>k}}^2}{\zeta} \\
	 & \leq \frac{2}{n}\prn{1 + \frac{\bigO_{\constantx, D} \prn{\constantsig \sigma_k \cdot \log n \log (\constantsig n)}}{\zeta}}\norm{\btheta_{\leq n}}^2 + \frac{2 \norm{\boldbeta_{>n}}^2}{\zeta} \, ,
\end{align*}
where in the last line we use the fact that for all $k+1 \leq i \leq n$,
\begin{align*}
	\<\boldbeta, \bv_i \>^2 = \sigma_i \<\btheta, \bv_i \>^2 \leq \sigma_k \<\btheta, \bv_i \>^2 = \bigO (\constantsig \sigma_k \cdot \log n \log (\constantsig n) / n)  \<\btheta, \bv_i \>^2 \, .
\end{align*}
The proof is complete.

\subsection{Proof of Lemma~\ref{lem:from-F-to-E}} \label{proof:from-F-to-E}
We apply Hanson-Wright inequality in Lemma~\ref{lem:hanson-wright} and get
\begin{align*}
	\Prb \prn{\left|\bz_k^\sT \bB_{k-1} \bz_k - \Siter_{k-1}(\bI) \right| \geq t \mid \bB_{k-1}} & \leq 2 \exp \brc{-\Omega \prn{\min \brc{\frac{t^2}{\constantx^4 \normF{\bB_{k-1}}^2}, \frac{t}{\constantx^2 \norm{\bB_{k-1}}}}}}
\end{align*}
and
\begin{align*} 
	& \Prb \prn{\left|\bz_{k}^\sT \bB_{k-1} \bQ \bB_{k-1} \bz_{k} - \Tr \prn{\bQ\bB_{k-1}^2} \right| \geq t \mid \bB_{k-1}} \nonumber \\
	& \leq 2 \exp \brc{- \Omega \prn{ \min \brc{\frac{t^2}{\constantx^4 \normF{\bB_{k-1} \bQ \bB_{k-1}}^2}, \frac{t}{\constantx^2 \norm{\bB_{k-1} \bQ \bB_{k-1}}}}}}  \, . \nonumber 
\end{align*}
In particular, on the event $\{T_F(\bQ) \geq k, T_F(\bI) \geq k\}$, we have
\begin{align*}
	\left|\Siter_{k-1}(\bI) - \Riter_0(\bI) \right| \leq \beta_1 \leq \frac 1 4 \Riter_0(\bI) \, , \qquad  \left|\Siter_{k-1}(\bQ) - \Riter_0(\bQ) \right| \leq \beta_1 \leq \frac 1 4 R_0(\bQ)\, , \qquad \norm{\bB_{k-1}} \leq \gamma \, ,
\end{align*} 
which further implies that
\begin{align*}
	\normop{\bB_{k-1}} \leq \normF{\bB_{k-1}} & = \sqrt{\Tr \prn{\bB_{k-1}^2}}  \leq \sqrt{\norm{\bB_{k-1}} \cdot \Siter_{k-1}(\bI)} = \bigO \prn{\sqrt{\gamma \Riter_0(\bI)}} \, , 
\end{align*}
and
\begin{align*}
	\normop{\bB_{k-1} \bQ\bB_{k-1}} \le \normF{\bB_{k-1} \bQ\bB_{k-1}} & = \sqrt{\Tr \prn{\bB_{k-1} \bQ\bB_{k-1}^2 \bQ\bB_{k-1}}} \nonumber \\
	&  \leq \sqrt{\norm{\bQ^{\frac 1 2}\bB_{k-1}^2 \bQ^{\frac 1 2}} \cdot \Tr \prn{\bB_{k-1} \bQ\bB_{k-1}}} \nonumber \\
	&  \leq \sqrt{\norm{\bB_{k-1}}^2 \cdot \Tr \prn{\bQ^{\frac 1 2}\bB_{k-1}^2 \bQ^{\frac 1 2}}} \nonumber \\
	& \stackrel{\mathrm{(i)}}{\leq} \sqrt{ \norm{\bB_{k-1}}^3 \cdot \Siter_{k-1}(\bQ)} = \bigO \prn{\sqrt{\gamma^3 \Riter_0(\bQ)}} \, ,
\end{align*}
Substituting the above bounds into the Hanson-Wright inequalities, we have conditioning on $H_k: = \{T_F(\bQ) \geq k, T_F(\bI) \geq k\}$ for some constant $\constant = \constant(\constantx, D)$ that
\begin{align*}
	\exp \brc{- \Omega \prn{\frac{\constant  \log n \cdot \sqrt{\gamma \Riter_0(\bI)}}{\constantx^2 \norm{\bB_{k-1}}}}} & = \bigO(n^{-D}) \, , \\
	\exp \brc{- \Omega \prn{ \frac{\constant \log n \cdot \sqrt{\gamma^3 \Riter_0(\bQ)}}{\constantx^2 \norm{\bB_{k-1} \bQ \bB_{k-1}}}}} & = \bigO(n^{-D}) \, ,
\end{align*} 
and therefore it holds with probability $1-\bigO(n^{-D})$ that
\begin{align*}
	\left|\bz_k^\sT \bB_{k-1} \bz_k - \Siter_{k-1}(\bI) \right| & \leq \constant  \log n \cdot \sqrt{\gamma \Riter_0(\bI)} = : \alpha_1 \, , \\
	\left|\bz_{k}^\sT \bB_{k-1} \bQ \bB_{k-1} \bz_{k} - \Tr \prn{\bQ\bB_{k-1}^2} \right|  & \leq \constant \log n \cdot \sqrt{\gamma^3 \Riter_0(\bQ)}  = : \alpha_2 \, .
\end{align*}
The Hanson-Wright inequalities also give the following upper bounds on the expectations conditioning on the tail event when $\|\bB_{k-1}\| \leq \gamma$. In particular, we would have
\begin{align}
	& \Ep_{k-1} \brk{\left| \bz_k^\sT \bB_{k-1} \bz_k - \Siter_{k-1}(\bI) \right| \ind\brc{\left| \bz_k^\sT \bB_{k-1} \bz_k - \Siter_{k-1}(\bI)\right| \geq  \alpha_1} } \ind (H_k) \nonumber \\
	& = \int_{ \constant \log n \cdot \sqrt{\gamma \Riter_0(\bI)}}^\infty \Prb \prn{\left|\bz_k^\sT \bB_{k-1} \bz_k - \Siter_{k-1}(\bI) \right| \geq t \mid \bB_{k-1}}\ind (H_k) \de t \nonumber \\
	& \leq \int_{ \constant \log n \cdot \sqrt{\gamma \Riter_0(\bI)}}^\infty 2 \exp \brc{-\Omega \prn{\min \brc{\frac{t^2}{\constantx^4 \normF{\bB_{k-1}}^2}, \frac{t}{\constantx^2 \norm{\bB_{k-1}}}}}} \ind (H_k)\de t \nonumber \\
	& \leq \int_{ \constant \log n \cdot \sqrt{\gamma R_0(I)}}^\infty 2 \exp \brc{- \Omega \prn{\frac{t}{\constantx^2 \normF{\bB_{k-1}}}}} \ind (H_k)\de t \nonumber \\
	& \leq \bigO\prn{\constantx^2 \normF{\bB_{k-1}}} \cdot \bigO(n^{-D})  = \bigO_{\constantx} \prn{n^{-D} \cdot \sqrt{ \gamma \Riter_0(\bI)}} \, . \nonumber
\end{align}
Similarly it also holds that
\begin{align}
	& \Ep_{k-1} \brk{\left| \bz_{k}^\sT \bB_{k-1} \bQ \bB_{k-1} \bz_{k} - \Tr \prn{\bQ\bB_{k-1}^2}  \right| \ind\brc{\left| \bz_{k}^\sT \bB_{k-1} \bQ \bB_{k-1} \bz_{k} - \Tr \prn{\bQ\bB_{k-1}^2}  \right| \geq  \alpha_2 }}\ind (H_k)  \nonumber \\
	& = \bigO_{\constantx} \prn{n^{-D} \cdot \sqrt{ \gamma^3 \Riter_0(\bQ)}} \, . \nonumber 
\end{align}

To finish the proof, we now only need to show $\normop{\bA_k} \leq \gamma$ holds with probability $1-\bigO(n^{-D})$. We provide upper bounds for small $k \leq n/2$ and large $k > n/2$ separately. Under the assumption $\beta_2 \leq \mu/2$, we make use of the fact that $H_k \subset F_{k-1}(\bQ)$ which enables us to derive
\begin{align}
	\left|\mu_k - \muefct(\zeta, \mu) + \frac{k}{1 +\Riter_0(\bI)}\right| & \leq \beta_2 \leq \frac{\mu}{2} \, , \label{eq:from-F-to-E-mid-1}
\end{align}
and thus for all $k \leq n/2$,
\begin{align*}
	\mu_k \geq \muefct(\zeta, \mu)  - \frac{k}{1 + \Riter_0(\bI)} - \frac{\mu}{2} = \frac{\mu}{2} + \frac{n-k}{1 + \Riter_0(\bI)} \geq \frac{\muefct(\zeta, \mu)}{2} \, ,
\end{align*}
which in particular implies for $k \leq n/2$ that
\begin{align*}
	\normop{\bA_k} \leq \normop{ \bSigma^{\frac 1 2} \prn{\zeta \bI + \mu_k \bSigma}^{-1} \bSigma^{\frac 1 2}} \leq \frac{2}{\muefct(\zeta, \mu)} \, .
\end{align*}

On the other hand, if $k > n/2$, we can still deduce from Eq.~\eqref{eq:from-F-to-E-mid-1} that $\mu_k \geq \mu/2 > 0$ and thus
\begin{align*}
	\bA_k & = \bSigma^{\frac 1 2} \prn{\zeta \bI + \mu_k \bSigma + \bX_k^\sT \bX_k}^{-1} \bSigma^{\frac 1 2} \preceq  \bSigma^{\frac 1 2} \prn{\zeta \bI + \bX_k^\sT \bX_k}^{-1} \bSigma^{\frac 1 2}  \, .
\end{align*}
Applying Lemma~\ref{lem:A-norm-bound}, we obtain  with probability $1 -\bigO(n^{-D})$ for all $k > n/2$,
\begin{align}
	\norm{\bA_k} \leq \norm{\bSigma^{\frac 1 2} \prn{\zeta \bI + \bX_{\lceil n/2 \rceil}^\sT \bX_{\lceil n/2 \rceil}}^{-1} \bSigma^{\frac 1 2}} = \frac{2}{n} \prn{1 + \frac{\bigO_{\constantx, D} \prn{\constantsig \sigma_{\lfloor \eta n\rfloor} \cdot \log n \log (\constantsig n)}}{\zeta}} \nonumber \, ,
\end{align}
for some $\eta = \eta(\constantx)$. Combine with the trivial bound $\|\bA_k\| \leq 1/ \zeta$, we conclude that 
$\|\bA_k\| \leq \gamma$ with probability  $1-\bigO(n^{-D})$, provided we take
\begin{align*}
	\gamma = \min \brc{\frac{2}{n} \prn{1 + \frac{\bigO_{\constantx, D} \prn{\constantsig \sigma_{\lfloor \eta n\rfloor} \cdot \log n \log (\constantsig n)}}{\zeta}} + \frac{2}{\muefct(\zeta, \mu)} , \frac{1}{\zeta}} \, .
\end{align*}
The proof is completed by noting that
\begin{align*}
		p_{k,k}(T_F, T_F, \bQ) - p_{k+1, k}(T_E, T_F, \bQ) =
		\prob(E_k(\bQ)^c;H_k)= \bigO(n^{-D}) \, .
	\end{align*}

\subsection{Proof of Lemma~\ref{lem:from-E-to-F}} \label{proof:from-E-to-F}
We begin by noticing that 
\begin{align}
	p_{k,k}(T_E, T_E, \bQ) - p_{k, k}(T_F, T_E, \bQ)& = 
	\prob\big(T_E(\bQ)\ge k,T_E(\bI)\ge k;F_{k-1}^c(\bQ)\big)\, .  \label{eq:E-to-F-goal}
\end{align}
We therefore need to control $|\Riter_{k-1}(\bQ) - \Riter_0(\bQ)|$ and 
$|\Siter_{k-1}(\bQ) - \Riter_0(\bQ)|$, as well as $|\mu_k - \wb{\mu}_k|$ and 
$\norm{\bB_{k-1}}$.

\paragraph{Part I: Decomposing into martingale part and bias part} Recall the calculations for Eq.~\eqref{eq:R-difference}, we have
\begin{align}
	\Riter_i(\bQ) - \Riter_{i-1}(\bQ) & =  - \frac{\Tr \prn{\bQ\bB_{i-1} \bz_i\bz_i^\sT \bB_{i-1}}}{1 + \bz_i^\sT \bB_{i-1} \bz_i} + \frac{\Tr \prn{\bQ \bB_{i-1} \bA_{i-1}}}{1 + \Siter_{i-1}(\bI)} \nonumber 
\end{align}
Define the stopping time
\begin{align*}
	\wb{T} = T_E(\bQ) \wedge T_E(\bI) \, ,
\end{align*} 
and on the event $\{T_E(\bQ) \geq k, T_E(\bI) \geq k\} = \{\wb{T} \geq k\}$, it holds
\begin{align*}
	\Riter_{k-1}(\bQ) - \Riter_0(\bQ) = \prn{\Riter_{k-1}(\bQ) - R_0(\bQ)} \ind \brc{\wb{T} \geq k} = \sum_{i=1}^{k-1} \prn{\Riter_i(\bQ) - \Riter_{i-1}(\bQ)} \ind\brc{\wb{T} \geq i+1} \, .
\end{align*}
For each of the summand, we can decompose it into two parts---the martingale difference part $D_i(\bQ, \wb{T})$ and a bias part $B_i(\bQ, \wb{T})$---to be specific, we can write
\begin{align*}
	\prn{\Riter_i(\bQ) - \Riter_{i-1}(\bQ)} \ind\brc{\wb{T} \geq i+1} = D_{i}(\bQ, \wb{T} ) + B_{i}(\bQ, \wb{T} ) \, ,
\end{align*}
where by setting $G_i := F_{i}(\bQ) \cap F_{i}(\bI) \in \mathcal{F}_i$ and $S_i := \brc{\wb{T} = i} \cap G_{i-1} \in \mathcal{F}_i $, the explicit forms of $D_i$ and $B_i$ are
(recall that $\E_i(\;\cdot\;):=\E(\;\cdot\;|\cF_i)$):
\begin{align}
	D_i(\bQ, \wb{T} ) & : = -\frac{\Tr \prn{\bQ\bB_{i-1} \bz_i \bz_i^\sT \bB_{i-1}}}{1 + \bz_i^\sT \bB_{i-1} \bz_i} \ind\brc{\wb{T}  \geq i+1} - \frac{\Tr\prn{\bQ\bB_{i-1}^2}}{1 + \Siter_{i-1}(\bI)} \ind\prn{S_i} \nonumber \\
	&\qquad + \Ep_{i-1} \brk{\frac{\Tr \prn{\bQ\bB_{i-1} \bz_i \bz_i^\sT \bB_{i-1}}}{1 + \bz_i^\sT \bB_{i-1} \bz_i} \ind\brc{\wb{T}  \geq i+1} + \frac{\Tr\prn{\bQ\bB_{i-1}^2}}{1 + \Siter_{i-1}(\bI)} \ind\prn{S_i}} \, , \label{eq:D-i}
\end{align}
and
\begin{align}
	B_i(\bQ, \wb{T} ) & :=  \frac{\Tr\prn{\bQ\bB_{i-1} \bA_{i-1}}}{1 + \Siter_{i-1}(\bI)} \ind\brc{\wb{T}  \geq i+1}  +  \frac{\Tr\prn{\bQ\bB_{i-1}^2}}{1 + \Siter_{i-1}(\bI)} \ind\prn{S_i}\nonumber \\
	& \qquad -  \Ep_{i-1} \brk{\frac{\Tr \prn{\bQ\bB_{i-1} \bz_i \bz_i^\sT \bB_{i-1}}}{1 + \bz_i^\sT \bB_{i-1} \bz_i} \ind\brc{\wb{T}  \geq i+1} + \frac{\Tr\prn{\bQ\bB_{i-1}^2}}{1 + \Siter_{i-1}(\bI)} \ind\prn{S_i}}  \, .  \label{eq:B-i}
\end{align}
Since $\wb{T}$ is a stopping time, one can easily have that $D_i(\bQ, \wb{T})$ is a martingale difference sequence for $i=0,1,\cdots,n$.
(We note in passing that the above decomposition is similar but does not coincide with the standard Doob 
decomposition. In particular $B_i(\bQ, \wb{T})$ is not measurable on $\cF_{i-1}$. We find the present 
decomposition more convenient.)

\paragraph{Part II: Controlling the martingale part} We will show $D_i(\bQ, \wb{T})$ is bounded and thus by concentration 
inequality for bounded martingale differences, we can obtain an upper bound for the sum of 
the $D_i(\bQ, \wb{T})$'s. To this end, we use the fact that if for some $m_{i-1} \in \mathcal{F}_{i-1}$
\begin{align*}
	\left|D_i(\bQ, \wb{T})  - m_{i-1}\right| \leq M \, ,
\end{align*}
then $ \left|D_i(\bQ, \wb{T})\right| = \left|D_i(\bQ, \wb{T}) - \Ep \brk{D_i(\bQ, \wb{T}) \mid \mathcal{F}_{i-1}}\right| \leq 2M$. Substitute the following $m_{i-1} \in \mathcal{F}_{i-1}$
\begin{align*}
	m_{i-1} 	&=  -\frac{\Tr\prn{\bQ\bB_{i-1}^2}}{1 + \Siter_{i-1}(\bI)} \ind\brc{\wb{T}  \geq  i} \ind \prn{G_{i-1}} + \Ep_{i-1} \brk{\frac{\Tr \prn{\bQ\bB_{i-1} \bz_i \bz_i^\sT \bB_{i-1}}}{1 + \bz_{i}^\sT \bB_{i-1} \bz_i} \ind\brc{\wb{T}  \geq i+1} + \frac{\Tr\prn{\bQ\bB_{i-1}^2}}{1 + \Siter_{i-1}(\bI)} \ind\prn{S_i} } 
\end{align*}
into the previous display, we have
\begin{align*}
	&\left| D_i(\bQ, \wb{T}) - m_{i-1}\right| \nonumber \\
	& \stackrel{\mathrm{(i)}}{=} \left|\frac{\Tr \prn{\bQ\bB_{i-1} \bz_i\bz_i^\sT \bB_{i-1}}}{1 + \bz_i^\sT \bB_{i-1} \bz_i}- \frac{\Tr\prn{\bQ\bB_{i-1}^2}}{1 + \Siter_{i-1}(\bI)} \right|  \ind\brc{\wb{T}  \geq i+1} \nonumber \\
	&\leq \left|\frac{ \bz_i^\sT \bB_{i-1}\bQ\bB_{i-1} \bz_i -\Tr\prn{\bQ\bB_{i-1}^2} }{1 + \bz_i^\sT \bB_{i-1} \bz_i} \right|  \ind\brc{\wb{T}  \geq i+1}  + \left|\frac{ \Tr\prn{\bQ\bB_{i-1}^2} \cdot \prn{\bz_i^\sT \bB_{i-1} \bz_i -\Siter_{i-1}(\bI) } }{\prn{1 + \bz_i^\sT \bB_{i-1} \bz_i} \prn{1 + \Siter_{i-1}(\bI)} } \right|  \ind\brc{\wb{T}  \geq i+1} \, ,
\end{align*}
where in (i) we use $\{\wb{T} \geq i+1\} \subset F_{i-1}(\bQ) \cap F_{i-1}(\bI) = G_{i-1} $, and therefore
\begin{align*}
\ind\brc{\wb{T}  \geq  i} \ind \prn{G_{i-1}} = \ind\brc{\wb{T}  =  i} \ind \prn{G_{i-1}} + \ind\brc{\wb{T} \geq  i + 1} \ind \prn{G_{i-1}} = \ind\prn{S_i} + \ind\brc{\wb{T} \geq  i + 1} \, .
\end{align*}
Recalling our assumptions for $\alpha_1$, we observe that on the event $\{\wb{T} \geq i+1\} \subset E_i(\bQ)$, 
\begin{align}
	\left|\bz_i^\sT \bB_{i-1} \bz_i - \Siter_{i-1}(\bI)\right| \leq \alpha_1 \leq \frac 1 4 \Riter_0(\bI) \, , \qquad \left|\bz_i^\sT \bB_{i-1}\bQ\bB_{i-1} \bz_i -\Tr\prn{\bQ\bB_{i-1}^2}\right|  \leq \alpha_2 \, , \label{eq:bias-main-ineq-1}
\end{align}
and on the event $\{\wb{T} \geq i+1\} \subset F_{i-1}(\bI)$ and $\{\wb{T} \geq i+1\} \subset F_{i-1}(\bQ)$ by assumptions on $\beta_1$,  
\begin{align}
	\left|\Siter_{i-1}(\bI) - \Riter_0(\bI) \right| \leq \beta_1 \leq \frac 1 4 \Riter_0(\bI) \, , \qquad  \left|\Siter_{i-1}(\bQ) - \Riter_0(\bQ) \right| \leq \beta_1 \leq \frac 1 4 \Riter_0(\bQ) \, , \label{eq:bias-main-ineq-2}
\end{align}
and finally on the event $\{\wb{T} \geq i+1\}  \subset G_{i-1}$ it holds
\begin{align}
	\Tr\prn{\bQ\bB_{i-1}^2} \leq \norm{\bB_{i-1}} \cdot \Tr \prn{\bQ \bB_{i-1}} \leq \gamma \Siter_{i-1}(\bQ) = \bigO \prn{\gamma \Riter_0(\bQ)}\, . \label{eq:bias-main-ineq-3}
\end{align}
Putting together bounds in Eqs.~\eqref{eq:bias-main-ineq-1}, \eqref{eq:bias-main-ineq-2}, \eqref{eq:bias-main-ineq-3} and making use of the fact that $\brc{\wb{T} \geq i+1} \subset E_i(\bQ) \cap F_{i-1}(\bI) \cap F_{i-1}(\bQ)$ yield
\begin{align*}
	\left| D_i(\bQ, \wb{T}) - m_{i-1}\right| & \leq  \left|\frac{ \alpha_2 }{1 + \frac{1}{2} \Riter_0(\bI)} \right|    + \left|\frac{\bigO(\gamma \Riter_0(\bQ))  \cdot \alpha_1 }{\prn{1 + \frac{1}{2} \Riter_0(\bI)} \prn{1 + \frac{3}{4} \Riter_0(\bI)} } \right| = \bigO \prn{\frac{\alpha_1 \gamma \Riter_0(\bQ) + \alpha_2 (1 + \Riter_0(\bI))}{1 + \Riter_0(\bI)^2}} \, .
\end{align*}
Then we can apply Azuma-Hoeffding inequality and obtain
\begin{align*}
	\max_{k\le n}\left|\sum_{i=1}^{k} D_i(\bQ, \wb{T})\right| = \bigO_{D} \prn{ \sqrt{n \log n} \cdot \frac{\alpha_1 \gamma \Riter_0(\bQ) + \alpha_2 (1 + \Riter_0(\bI))}{1 + \Riter_0(\bI)^2}} \, ,
\end{align*}
with probability $1 - \bigO(n^{-D})$.
\paragraph{Part III: Controlling the bias part} Now we proceed to bound the bias part $|B_i(\bQ, \wb{T})|$ in Eq.~\eqref{eq:B-i}. We can write an upper bound
\begin{align*}
	&\left|B_i(\bQ, \wb{T}) \right|
	 \leq \underbrace{\left| \frac{\Tr\prn{\bQ\bB_{i-1} \bA_{i-1}}}{1 + \Siter_{i-1}(\bI)} -  \frac{\Tr\prn{\bQ\bB_{i-1}^2}}{1 + \Siter_{i-1}(\bI)} \right| \ind\brc{\wb{T}  \geq i+1}}_{\mathrm{(I)}} \nonumber \\
	& \qquad + \underbrace{\left|\frac{\Tr\prn{\bQ\bB_{i-1}^2}}{1 + \Siter_{i-1}(\bI)} \ind\brc{\wb{T} \geq i} \ind \prn{G_{i-1}} -  \Ep_{i-1} \brk{\frac{\Tr \prn{\bQ\bB_{i-1} \bz_i \bz_i^\sT \bB_{i-1}}}{1 + \bz_{i}^\sT \bB_{i-1} \bz_i} \ind\brc{\wb{T}  \geq i+1} + \frac{\Tr\prn{\bQ\bB_{i-1}^2}}{1 + \Siter_{i-1}(\bI)} \ind\prn{S_i} }\right|}_{\mathrm{(II)}} \, .
\end{align*}
Using the fact that
\begin{align*}
	\bB_{i-1} - \bA_{i-1} = (\mu_{i-1} - \mu_i) \bB_{i-1} \bA_{i-1} = \frac{\bB_{i-1} \bA_{i-1}}{1 + \Siter_i(\bI)} \, ,
\end{align*}
we have
\begin{align*}
	\mathrm{(I)} & \leq \left| \frac{\Tr\prn{\bQ\bB_{i-1} \prn{\bA_{i-1} - \bB_{i-1}}}}{1 + \Siter_{i-1}(\bI)}\right| \ind\brc{\wb{T}  \geq i+1} = \left| \frac{\Tr\prn{\bQ\bB_{i-1}^2 \bA_{i-1}}}{\prn{1 + \Siter_{i-1}(\bI)}^2}\right| \ind\brc{\wb{T}  \geq i+1}  \, .
\end{align*}
Upper bounding the term $\Tr(\bQ\bB_{i-1}^2 \bA_{i-1})$ requires more careful treatment. Note that $\bB_{i-1}$ and $\bA_{i-1}$ commute, as follows from the observation that
\begin{align*}
	& \bB_{i-1} \bA_{i-1}  \nonumber \\
	& = (1 + \Siter_i(\bI)) \cdot \bSigma^{\frac 1 2} \brc{\prn{\zeta \bI + \mu_i \bSigma + \bX_{i-1}^\sT \bX_{i-1}}^{-1} \cdot \prn{\mu_{i-1} - \mu_i} \bSigma \cdot \prn{\zeta \bI + \mu_{i-1} \bSigma + \bX_{i-1}^\sT \bX_{i-1}}^{-1}} \bSigma^{\frac 1 2} \nonumber \\
	& =  (1 + \Siter_i(\bI))  \prn{\bB_{i-1} - \bA_{i-1}} \nonumber \\
	& =  (1 + \Siter_i(\bI)) \cdot \bSigma^{\frac 1 2} \brc{\prn{\zeta \bI + \mu_{i-1} \bSigma + \bX_{i-1}^\sT \bX_{i-1}}^{-1} \cdot \prn{\mu_{i-1} - \mu_i} \bSigma \cdot \prn{\zeta \bI + \mu_{i} \bSigma + \bX_{i-1}^\sT \bX_{i-1}}^{-1}} \bSigma^{\frac 1 2} \nonumber \\
	& = \bA_{i-1} \bB_{i-1}  \, .
\end{align*}
Since $\bA_{i-1}$ and $\bB_{i-1}$ are both p.s.d.\ compact self-adjoint operators in Hilbert space, commutativity implies they can be simultaneously orthogonally diagonalized, which further implies that $\bA_{i-1}^{\frac{1}{2}}$ and $\bB_{i-1}^{\frac{1}{2}}$ also commute. Therefore
\begin{align*}
	\Tr\prn{\bQ\bB_{i-1}^2 \bA_{i-1}} & = \Tr\prn{\bQ^{\frac 1 2} \bA_{i-1}^{\frac 1 2} \bB_{i-1}^2 \bA_{i-1}^{\frac 1 2} \bQ^{\frac 1 2}}  \leq \norm{\bB_{i-1}}^2 \cdot \Tr \prn{\bQ^{\frac 1 2} \bA_{i-1}\bQ^{\frac 1 2}} = \norm{\bB_{i-1}}^2 \cdot \Riter_{i-1}(\bQ) \, ,
\end{align*}
and thus
\begin{align}
	\mathrm{(I)}	& \leq\frac{\norm{\bB_{i-1}}^2 \cdot \Riter_{i-1}(\bQ) }{\prn{1 + \Siter_{i-1}(\bI)}^2} \ind\brc{\wb{T}  \geq i+1}  = \bigO \prn{ \frac{ \gamma^2 \Riter_0(\bQ)}{1 + \Riter_0(\bI)^2} }\, , \label{eq:B-difference-mid-0}
\end{align} 
where in the last inequality we use Eq.~\eqref{eq:bias-main-ineq-2} on the event $\{\wb{T} \geq i+1\}  \subset F_{i-1}(\bI)$, while $\{\wb{T} \geq i+1\}  \subset F_{i-1}(\bQ)$ also implies
\begin{align*}
	|\Riter_{i-1}(\bQ) - \Riter_0(\bQ)| \leq \beta_1 \leq \frac{1}{4} \Riter_0(\bQ) \, .
\end{align*}

Next to bound (II), we make use of the fact that
\begin{align*}
	\frac{\Tr\prn{\bQ\bB_{i-1}^2}}{1 + \Siter_{i-1}(\bI)} \ind\brc{\wb{T} \geq i} \ind \prn{G_{i-1}} & = \Ep_{i-1} \brk{\frac{\Tr\prn{\bQ\bB_{i-1} \bz_i \bz_i^\sT \bB_{i-1}}}{1 + \Siter_{i-1}(\bI)}  \ind\brc{\wb{T} \geq i} \ind \prn{ G_{i-1}}} \, ,
\end{align*}
and therefore
\begin{align}
	&\mathrm{(II)} \nonumber \\
	&= \Bigg|\Ep_{i-1} \Bigg[\frac{\Tr\prn{\bQ\bB_{i-1} \bz_i \bz_i^\sT \bB_{i-1}}}{1 + \Siter_{i-1}(\bI)} \ind\brc{\wb{T} \geq i} \ind \prn{G_{i-1}} - \frac{\Tr \prn{\bQ\bB_{i-1} \bz_i \bz_i^\sT \bB_{i-1}}}{1 + \bz_{i}^\sT \bB_{i-1} \bz_i} \ind\brc{\wb{T}  \geq i+1} \nonumber \\
	& \qquad - \frac{\Tr\prn{\bQ\bB_{i-1}^2}}{1 + \Siter_{i-1}(\bI)} \ind\prn{S_i} \Bigg] \Bigg| \nonumber \\
	& = \left|\Ep_{i-1} \brk{\frac{ \bz_i^\sT \bB_{i-1} \bQ\bB_{i-1} \bz_i \cdot \prn{\bz_{i}^\sT \bB_{i-1} \bz_i - \Siter_{i-1}(\bI)}}{\prn{1 + \Siter_{i-1}(\bI)}\prn{1 + \bz_{i}^\sT \bB_{i-1} \bz_i}} \ind\brc{\wb{T}  \geq i + 1} + \frac{\bz_i^\sT \bB_{i-1} \bQ\bB_{i-1}\bz_i - \Tr\prn{\bQ\bB_{i-1}^2}}{1 + \Siter_{i-1}(\bI)} \ind\prn{S_i}} \right| \nonumber \\
	&\leq \nonumber \\
	& \underbrace{\left|\Ep_{i-1} \brk{ \left(\frac{ \bz_i^\sT \bB_{i-1} \bQ\bB_{i-1}\bz_i \cdot \prn{\bz_{i}^\sT \bB_{i-1} \bz_i - \Siter_{i-1}(\bI)}}{\prn{1 + \Siter_{i-1}(\bI)}\prn{1 + \bz_{i}^\sT \bB_{i-1} \bz_i}} - \frac{\Tr \prn{\bQ \bB_{i-1}^2} \cdot \prn{\bz_{i}^\sT \bB_{i-1} \bz_i - \Siter_{i-1}(\bI)}}{\prn{1 + \Siter_{i-1}(\bI)}^2} \right) \ind\brc{\wb{T}  \geq i + 1} }  \right|}_{\mathrm{(III)}} \nonumber \\
	& \qquad + \underbrace{\left|\Ep_{i-1} \brk{\prn{ \frac{\Tr \prn{\bQ \bB_{i-1}^2} \cdot \prn{\bz_{i}^\sT \bB_{i-1} \bz_i - \Siter_{i-1}(\bI)}}{\prn{1 + \Siter_{i-1}(\bI)}^2} + \frac{\bz_i^\sT \bB_{i-1} \bQ\bB_{i-1}\bz_i - \Tr\prn{\bQ\bB_{i-1}^2}}{1 + \Siter_{i-1}(\bI)}} \ind\prn{S_i} } \right|}_{\mathrm{(IV)}} \, , \label{eq:B-difference-mid-2}
\end{align}
where in the last inequality we use that
\begin{align*}
	&\Ep_{i-1} \brk{\frac{\Tr \prn{\bQ \bB_{i-1}^2} \cdot \prn{\bz_{i}^\sT \bB_{i-1} \bz_i - \Siter_{i-1}(\bI)}}{\prn{1 + \Siter_{i-1}(\bI)}^2} \ind\brc{\wb{T} \geq i} \cap G_{i-1}} \nonumber \\
	&= \frac{\Tr \prn{\bQ \bB_{i-1}^2}}{\prn{1 + \Siter_{i-1}(\bI)}^2} \ind\big(\brc{\wb{T} \geq i} \cap G_{i-1}\big) \cdot \Ep_{i-1} \brk{\bz_{i}^\sT \bB_{i-1} \bz_i - \Siter_{i-1}(\bI)} = 0 \, .
\end{align*}

To control (III), we note that
\begin{align*}
	& \frac{ \bz_i^\sT \bB_{i-1} \bQ\bB_{i-1}\bz_i \cdot \prn{\bz_{i}^\sT \bB_{i-1} \bz_i - \Siter_{i-1}(\bI)}}{\prn{1 + \Siter_{i-1}(\bI)}\prn{1 + \bz_{i}^\sT \bB_{i-1} \bz_i}} - \frac{\Tr \prn{\bQ \bB_{i-1}^2} \cdot \prn{\bz_{i}^\sT \bB_{i-1} \bz_i - \Siter_{i-1}(\bI)}}{\prn{1 + \Siter_{i-1}(\bI)}^2}  \nonumber \\
	& = \frac{\brc{ \bz_i^\sT \bB_{i-1} \bQ\bB_{i-1}\bz_i \cdot \prn{1 + \Siter_{i-1}(\bI)} - \Tr \prn{\bQ \bB_{i-1}^2} \cdot \prn{1 + \bz_{i}^\sT \bB_{i-1} \bz_i}}\cdot \prn{\bz_{i}^\sT \bB_{i-1} \bz_i - \Siter_{i-1}(\bI)}}{\prn{1 + \Siter_{i-1}(\bI)}^2\prn{1 + \bz_{i}^\sT \bB_{i-1} \bz_i}} \nonumber \\
	& = \frac{1}{\prn{1 + \Siter_{i-1}(\bI)}^2\prn{1 + \bz_{i}^\sT \bB_{i-1} \bz_i}} \cdot \bigg\{ \prn{\bz_i^\sT \bB_{i-1} \bQ\bB_{i-1} \bz_i -\Tr \prn{\bQ \bB_{i-1}^2}} \cdot \prn{1 + \Siter_{i-1}(\bI)} \nonumber \\
	& \qquad -\Tr \prn{\bQ \bB_{i-1}^2} \cdot \prn{\bz_{i}^\sT \bB_{i-1} \bz_i - \Siter_{i-1}(\bI)} \bigg\} \cdot \prn{\bz_{i}^\sT \bB_{i-1} \bz_i - \Siter_{i-1}(\bI)} \, .
\end{align*}
We again make use of the bounds in Eqs.~\eqref{eq:bias-main-ineq-1}, \eqref{eq:bias-main-ineq-2} and~\eqref{eq:bias-main-ineq-3} on the event $\{\wb{T} \geq i+1\}$, which implies
\begin{align}
	\mathrm{(III)} & \leq \left|\Ep_{i-1} \brk{  \frac{\prn{ \alpha_2 \cdot \prn{1 + \Siter_{i-1}(\bI)} + \Tr \prn{\bQ \bB_{i-1}^2} \cdot \alpha_1 }\cdot \alpha_1}{\prn{1 + \Siter_{i-1}(\bI)}^2\prn{1 + \bz_{i}^\sT \bB_{i-1} \bz_i}} \ind\brc{\wb{T}  \geq i + 1}}  \right| \nonumber \\
	& = \bigO \prn{ \frac{\alpha_1 \alpha_2 \prn{1 + \Riter_0(\bI)} + \alpha_1^2 \gamma  \Riter_0(\bQ)}{1 + \Riter_0(\bI)^3}} \, . \label{eq:B-difference-mid-3}
\end{align}

Finally for term (IV) in Eq.~\eqref{eq:B-difference-mid-2}, we can control it by
\begin{align*}
	\mathrm{(IV)} & \leq \frac{ \bigO \prn{\gamma \Riter_0(\bQ)}}{1 + \Riter_0(\bI)^2} \cdot \Ep_{i-1} \brk{\left| \bz_{i}^\sT \bB_{i-1} \bz_i - \Siter_{i-1}(\bI) \right|  \ind\prn{S_i}} \nonumber \\
	& \qquad + \frac{1}{1 + \Riter_0(\bI)} \cdot \Ep_{i-1} \brk{\left|\bz_i^\sT \bB_{i-1} \bQ\bB_{i-1} \bz_i - \Tr\prn{\bQ\bB_{i-1}^2} \right|  \ind\prn{S_i} } \nonumber \, .
\end{align*}
Recall $S_i = \{\wb{T} = i\} \cap F_{i-1}(\bQ) \cap F_{i-1}(\bI)$, which implies $\{T_F(\bQ) \geq i, T_F(\bI) \geq i\}$ holds 
but at least one of $E_i(\bQ)$ and $E_i(\bI)$ doesn't hold. This allows us to invoke Lemma~\ref{lem:from-F-to-E} 
and conclude that $\Prb(S_i|\cF_{i-1}) = \bigO(n^{-D})$. Moreover, we can further deduce from Lemma~\ref{lem:from-F-to-E} that
\begin{align*}
	&\Ep_{i-1} \brk{\left| \bz_{i}^\sT \bB_{i-1} \bz_i - \Siter_{i-1}(\bI) \right|  \ind\prn{S_i}} \nonumber \\
	& \leq \alpha_1 \Prb(S_i|\cF_{i-1}) + \Ep_{i-1} \brk{\left| \bz_i^\sT \bB_{i-1} \bz_i - \Siter_{i-1}(\bI) \right| \ind\brc{\left| \bz_i^\sT \bB_{i-1} \bz_i - \Siter_{i-1}(\bI)\right| \geq  \alpha_1} } \ind \{T_F(\bQ) \geq i, T_F(\bI) \geq i\} \nonumber \\
	& = \bigO_{\constantx, D} \prn{n^{-D} \cdot \alpha_1} \, .
\end{align*}
Similarly, we also have
\begin{align*}
	\Ep_{i-1} \brk{\left|\bz_i^\sT \bB_{i-1} \bQ\bB_{i-1} \bz_i - \Tr\prn{\bQ\bB_{i-1}^2} \right|  \ind\prn{S_i} } & = \bigO_{\constantx, D} \prn{n^{-D} \cdot \alpha_2} \, .
\end{align*}
Combining the above displays, we obtain that
\begin{align*}
	\mathrm{(IV)} = \bigO_{\constantx, D} \prn{\frac{ n^{-D} \alpha_1 \gamma \Riter_0(\bQ)}{1 + \Riter_0(\bI)^2} 
	 + \frac{n^{-D} \alpha_2}{1 + \Riter_0(\bI)} }
\end{align*}
Now applying the assumption that 
\begin{align*}
	n^{-D} = \bigO \prn{ \frac{\alpha_1}{1 + \Riter_0(\bI)}} \, ,
\end{align*}
we obtain
\begin{align}
	\mathrm{(IV)} &  = \bigO_{\constantx, D} \prn{\frac{\alpha_1^2 \gamma \Riter_0(\bQ)}{1 + \Riter_0(\bI)^3} 
		+ \frac{\alpha_1 \alpha_2}{1 + \Riter_0(\bI)^2} } =\bigO_{\constantx, D} \prn{ \frac{\alpha_1 \alpha_2 \prn{1 + \Riter_0(\bI)} + \alpha_1^2 \gamma  \Riter_0(\bQ)}{1 + \Riter_0(\bI)^3}} \, . \label{eq:B-difference-mid-4}
\end{align}
Substitute Eqs.~\eqref{eq:B-difference-mid-3} and \eqref{eq:B-difference-mid-4} into Eq.~\eqref{eq:B-difference-mid-2} we have
\begin{align*}
	\mathrm{(II)} & = \bigO_{\constantx, D} \prn{ \frac{\alpha_1 \alpha_2 \prn{1 + \Riter_0(\bI)} + \alpha_1^2 \gamma  \Riter_0(\bQ)}{1 + \Riter_0(\bI)^3}} \, ,
\end{align*}
and together with Eq.~\eqref{eq:B-difference-mid-0} we obtain
\begin{align*}
	\left|B_i(\bQ, \wb{T}) \right| = \bigO_{\constantx, D} \prn{ \frac{\gamma^2 \Riter_0(\bQ) + \alpha_1 \alpha_2}{1 + \Riter_0(\bI)^2} + \frac{\alpha_1^2  \gamma \Riter_0(\bQ)}{1 + \Riter_0(\bI)^3}} \, .
\end{align*}
\paragraph{Part IV: Combining the results}
Hence, by combining results in part III and IV, we have with probability $1 - \bigO(n^{-D})$ that
\begin{align*}
	& \left| \prn{\Riter_{k-1}(\bQ) - \Riter_0(\bQ)} \ind\brc{\wb{T} \geq k} \right| \leq 	\left|\sum_{i=1}^{k-1} D_i(\bQ, \wb{T})\right| +\left| \sum_{i=1}^{k-1} B_i(\bQ, \wb{T}) \right| \nonumber  \\
	& \leq \bigO_{D} \prn{ \sqrt{n \log n} \cdot \frac{\alpha_1 \gamma \Riter_0(\bQ) + \alpha_2 (1 + \Riter_0(\bI))}{1 + \Riter_0(\bI)^2}} + n \cdot \bigO_{\constantx, D} \prn{ \frac{\gamma^2 \Riter_0(\bQ) + \alpha_1 \alpha_2}{1 + \Riter_0(\bI)^2} + \frac{\alpha_1^2  \gamma \Riter_0(\bQ)}{1 + \Riter_0(\bI)^3}} \nonumber \\
	& = \bigO_{\constantx, D} \prn{\sqrt{n \log n} \cdot \frac{\alpha_1 \gamma \Riter_0(\bQ) + \alpha_2 (1 + \Riter_0(\bI))}{1 + \Riter_0(\bI)^2} + n \cdot \brc{\frac{\gamma^2 \Riter_0(\bQ) + \alpha_1 \alpha_2}{1 + \Riter_0(\bI)^2} + \frac{\alpha_1^2  \gamma \Riter_0(\bQ)}{1 + \Riter_0(\bI)^3}}} \, . 
\end{align*}
We can first see $\normop{\bB_{k-1}} \leq \gamma$ holds with probability $1-\bigO(n^{-D})$, which follows via exactly the same argument as in Appendix~\ref{proof:from-F-to-E} for $\|\bA_k\| \leq \gamma$ by invoking Lemma~\ref{lem:A-norm-bound}.  Moreover, we have on $\{\wb{T} \geq k\}$ that
\begin{align*}
	\left|\Siter_{k-1}(\bQ) - \Riter_{k-1}(\bQ) \right| & = \left|\Tr \prn{\bQ \prn{\bB_{k-1} - \bA_{k-1}}}  \right| = \frac{\Tr \prn{\bQ \bB_{k-1} \bA_{k-1}}}{1 + \Siter_{k-1}(\bI)} \nonumber \\
	& \leq \frac{\norm{\bB_{k-1}} \Riter_{k-1}(\bQ)}{1 + \Siter_{k-1}(\bI)} \stackrel{\mathrm{(i)}}{\leq}  \frac{\gamma \Riter_{k-1}(\bQ)}{1 + \Riter_{k-1}(\bI)} = \bigO \prn{\frac{\gamma \Riter_0(\bQ)}{1 + \Riter_0(\bI)}} \, ,
\end{align*}
where in (i) we apply $\mu_{k-1} \geq \mu_k$ which indicates $\Riter_{k-1}(\bI) \leq \Siter_{k-1}(\bI)$.
Therefore, by setting a constant $\constant_\beta := \constant_\beta(\constantx, D)$ large enough and take
\begin{align*}
	\beta_1 & =  \constant_\beta  \prn{\sqrt{n \log n} \cdot \frac{\alpha_1 \gamma \Riter_0(\bQ) + \alpha_2 (1 + \Riter_0(\bI))}{1 + \Riter_0(\bI)^2} + n \cdot \brc{\frac{\gamma^2 \Riter_0(\bQ) + \alpha_1 \alpha_2}{1 + \Riter_0(\bI)^2} + \frac{\alpha_1^2  \gamma \Riter_0(\bQ)}{1 + \Riter_0(\bI)^3}} + \frac{\gamma \Riter_0(\bQ)}{1 + \Riter_0(\bI)}} \, , \\
	\beta_2 & = \frac{\constant_\beta n \beta_1}{1 + \Riter_0(\bI)^2} \, ,
\end{align*}
if this satisfies the assumption $\beta_1 \leq \Riter_0(\bI) /4$, we can conclude that with probability $1-\bigO(n^{-D})$,
 $$\left|(\Riter_{k-1}(\bQ) - \Riter_0(\bQ)) \ind \brc{\wb{T} \geq k} \right| \leq \beta_1 \, ,  \qquad 
 \left|(\Siter_{k-1}(\bQ) - \Riter_0(\bQ)) \ind \brc{\wb{T} \geq k} \right| \leq \beta_1 \, .$$ 
 
 Further, since
 \begin{align*}
 	\left|\prn{\mu_{k} - \wb{\mu}_k}\ind \brc{\wb{T} \geq k} \right| & = \left|\sum_{i=0}^{k-1} \prn{\frac{1}{1 + \Siter_i(\bI)} - \frac{1}{1 + \Riter_0(\bI)}} \ind \brc{\wb{T} \geq k} \right| \nonumber \\
 	& \leq \sum_{i=0}^{k-1}\frac{|\Siter_i(\bI) - \Riter_i(\bI)|}{(1 +\Siter_i(\bI) ) (1 + \Riter_i(\bI))} 
 	\ind \brc{\wb{T} \geq k}
 	= \bigO \prn{\frac{n \beta_1}{1 + \Riter_0(\bI)^2}} \, ,
 \end{align*}
taking $\constant_\beta$ large will guarantee $\left|\prn{\mu_{k} - \wb{\mu}_k}\ind \brc{\wb{T} \geq k} \right| \leq \beta_2$. Combining the above displays, we see on the event $\{\wb{T} \geq k\}$, it holds with probability $1 - \bigO(n^{-D})$ that
\begin{align*}
	\max \{\left|\Riter_{k-1}(\bQ) - \Riter_0(\bQ) \right|, \left|\Siter_{k-1}(\bQ) - \Riter_0(\bQ) \right|\} \leq \beta_1 \, , \qquad  \left|\mu_k- \wb{\mu}_k  \right| \leq \beta_2 \, , \qquad \norm{\bB_{k-1}} \leq \gamma \, ,
\end{align*}
which is exactly the event $F_{k-1}(\bQ)$ (cf.~Eq.~\eqref{eq:def-F-i}). Substituting into Eq.~\eqref{eq:E-to-F-goal} completes the proof.

\subsection{Proof of Lemma~\ref{lem:R-fct-approximation}} \label{proof:R-fct-approximation}

As $\beta_2>\mu/2$, we cannot directly apply Lemmas~\ref{lem:from-F-to-E} and \ref{lem:from-E-to-F}.
We will instead use a perturbation argument, reducing ourselves to the case $\beta_2\le \mu/2$.
We will define a second sequence  $\mu_i'$ following the recursion Eq.~\eqref{eq:iteration-mu} but
with a different initialization $\mu_0' = \muefct(\zeta, \mu')$ with $\mu':=64\beta_2 > \mu$. We use the notations
\begin{align*}
	\bA_i' & := \bSigma^{\half} \prn{\zeta \bI + \mu_i' \bSigma + \bX_i^\sT \bX_i}^{-1} \bSigma^{\half} \, ,  \\ 
	\bB_i' & := \bSigma^{\half} \prn{\zeta \bI + \mu_{i+1}' \bSigma + \bX_i^\sT \bX_i}^{-1} \bSigma^{\half} \, ,
\end{align*}
and also denote by $\Riter_i'(\bQ) = \Rfct_i(\zeta, \mu_i'; \bQ) := \Tr(\bQ \bA_i')$. For this second iteration, we 
define a parameter tuple $\Delta' = (\alpha_1', \alpha_2', \beta_1', \beta_2', \gamma)$ defined in Lemmas~\ref{lem:from-F-to-E} and \ref{lem:from-E-to-F} as
\begin{subequations}
\begin{align}
	\gamma' &= \min \brc{\frac{2}{n} \prn{1 + \frac{\constant_\gamma \constantsig \sigma_{\lfloor \eta n\rfloor} \cdot \log n \log (\constantsig n)}{\zeta}} + \frac{2}{\muefct(\zeta, \mu')} , \frac{1}{\zeta}} \, ,  \label{eq:gamma-prime}\\
	\alpha_1' &=  \constant_\alpha \log n \cdot \sqrt{\gamma' \Riter_0'(\bI)}\, , \label{eq:alpha-1-prime} \\
	\alpha_2' &=  \constant_\alpha \log n \cdot \sqrt{{\gamma'}^3 \Riter_0'(\bQ)} \, ,  \label{eq:alpha-2-prime} \\
	\beta_1' & = \constant_\beta \Bigg(\sqrt{n \log n} \cdot \frac{\alpha_1' \gamma' \Riter_0'(\bQ) + \alpha_2' (1 + \Riter_0'(\bI))}{1 + \Riter_0'(\bI)^2} + n \cdot \brc{\frac{{\gamma'}^2 \Riter_0'(\bQ) + \alpha_1' \alpha_2'}{1 + \Riter_0'(\bI)^2} + \frac{{\alpha_1'}^2  \gamma' \Riter_0'(\bQ)}{1 + \Riter_0'(\bI)^3}}  \nonumber \\
	& \qquad + \frac{\gamma' \Riter_0'(\bQ)}{1 + \Riter_0'(\bI)}\Bigg)  \, , \label{eq:beta-1-prime} \\
	\beta_2' & = \frac{\constant_\beta n \beta_1'}{1 + \Riter_0'(\bI)^2}\, . \label{eq:beta-2-prime}
\end{align}
\end{subequations}
We want to show $\alpha_1' \leq \Riter_0'(\bI)/4$, $\beta_1' \leq \Riter_0'(\bQ)/4$, $\beta_2' \leq \mu'/2$ and $n^{-D} = \bigO \prn{ \alpha_1' / \prn{1 + \Riter_0'(\bI)}}$ so that Lemmas~\ref{lem:from-F-to-E} and \ref{lem:from-E-to-F} are valid for $\zeta, \mu'$ and $\Delta'$. To prove this claim, we need the following result bounding the perturbation of $\muefct$.

\begin{lemma} \label{lem:muefct-derivative-bound}
	For any $\zeta >0$ and $\mu \geq 0$,
	\begin{align*}
		0\le \frac{\partial \muefct(\zeta, \mu)}{\partial \mu} \leq 1 + \Rfct_0(\zeta, \muefct(\zeta, \mu); \bI) \, .
	\end{align*}
\end{lemma}
\begin{proof}
	Taking derivatives w.r.t.\ $\mu$ on both sides of
	\begin{align*}
		\mu = \muefct - \frac{n}{1 + \Tr(\bSigma(\zeta \bI + \muefct \bSigma)^{-1})}\, ,
	\end{align*}
	we have
	\begin{align*}
		1 = \frac{\partial \muefct}{\partial \mu} \cdot \prn{1 - \frac{(\muefct - \mu) \Tr \prn{\bSigma^2 (\zeta \bI + \muefct \bSigma)^{-2}}}{1 + \Tr \prn{\bSigma (\zeta \bI + \muefct \bSigma)^{-1}}}} \, .
	\end{align*}
	Further
	\begin{align*}
		1 - \frac{(\muefct - \mu) \Tr \prn{\bSigma^2 (\zeta \bI + \muefct \bSigma)^{-2}}}{1 + \Tr \prn{\bSigma (\zeta \bI + \muefct \bSigma)^{-1}}} & \geq 1 - \frac{\muefct \Tr \prn{\bSigma^2 (\zeta \bI + \muefct \bSigma)^{-2}}}{1 + \Tr \prn{\bSigma (\zeta \bI + \muefct \bSigma)^{-1}}} = \frac{1 + \zeta \Tr \prn{\bSigma (\zeta \bI + \muefct \bSigma)^{-2}}}{1 + \Tr \prn{\bSigma (\zeta \bI + \muefct \bSigma)^{-1}}} \nonumber \\
		& \geq \frac{1}{1 + \Rfct_0(\zeta, \muefct(\zeta, \mu); \bI)} \, ,
	\end{align*}
	which gives the desired bounds $0\le \partial \muefct /\partial \mu \leq 1 + \Rfct_0(\zeta, \muefct(\zeta, \mu); \bI)$. 
\end{proof}

Recalling that $\muefct(\zeta, \mu)$ is increasing in $\mu$ and therefore $ \Rfct_0(\zeta, \muefct(\zeta, \mu); \bI)$ is decreasing in $\mu$, as a direct consequence of Lemma~\ref{lem:muefct-derivative-bound} we have
\begin{align}
	0 & \leq \muefct(\zeta, \mu') - \muefct(\zeta, \mu) = \int_{\mu}^{\mu'} \frac{\partial \muefct}{\partial \nu} 
	(\zeta,\nu)\, \de \nu  \nonumber \\
	& \leq  \int_{\mu}^{\mu'} (1 +  \Rfct_0(\zeta, \muefct(\zeta, \nu); \bI)) \de \nu \leq (\mu' - \mu)(1 + \Rfct_0(\zeta, \muefct(\zeta, \mu); \bI)) \leq \mu' (1 + \Riter_0(\bI)) \, , \label{eq:mu-efct-diff-bound}
\end{align}
and further
\begin{align*}
	0 \leq \Riter_0(\bQ) - \Riter_0'(\bQ) = (\muefct(\zeta, \mu') - \muefct(\zeta, \mu)) \Tr (\bQ \bA_0 \bA_0') \leq  \gamma \mu' (1 + \Riter_0(\bI)) \Riter_0'(\bQ) \, ,
\end{align*}
where in the last inequality we used $\|\bA_0\| \leq \min\{ 1/\muefct(\zeta, \mu), 1/\zeta\} \leq \gamma$. 
Substituting $\mu' = 64 \beta_2$ and using the condition $\gamma \beta_2(1 + \Riter_0(\bI)) \leq 1/64$, it then follows that
\begin{align*}
	\frac{1}{2} \Riter_0(\bQ) \leq \Riter_0'(\bQ) \leq \Riter_0(\bQ) \, , \qquad \forall \, \, \text{p.s.d.} \, \bQ \, .
\end{align*}
Using the last inequalities in Eqs.~\eqref{eq:gamma-prime} to \eqref{eq:alpha-2-prime}, it follows immediately that
\begin{align*}
	\gamma' \leq \gamma\, , \qquad \alpha_1' \leq \alpha_1 \, , \qquad \alpha_2' \leq \alpha_2 \, .
\end{align*}
We then first see that $\alpha_1 \leq \Riter_0(\bI)/8$ implies $\alpha_1' \leq \alpha_1 \leq \Riter_0'(\bI)/4$. For $\beta_1'$ and $\beta_2'$, using $1 + \Riter_0'(\bI)^k \geq 2^{-k} (1 + \Riter_0(\bI))$ with $k=1,2,3$,
 we can deduce from Eqs.~\eqref{eq:beta-1-prime} and \eqref{eq:beta-2-prime} that
\begin{align*}
	\beta_1' \leq 8 \beta_1 \, , \qquad \beta_2' \leq \frac{4\constant_\beta n \beta_1'}{1 + \Riter_0(\bI)^2} \leq 32 \beta_2 \, .
\end{align*}
The last inequality verifies $\beta_2' \leq 32 \beta_2 = \mu' / 2$. The condition $\beta_1 \leq \Riter_0(\bQ) / 64$
 implies $\beta_1' \leq 8 \beta_1 \leq \Riter_0(\bQ) / 8 \leq \Riter_0'(\bQ) / 4$.
 
 Finally we need to show $n^{-D} = \bigO \prn{ \alpha_1' / \prn{1 + \Riter_0'(\bI)}}$. From Eq.~\eqref{eq:mu-efct-diff-bound}, we can obtain that
 \begin{align*}
 	\muefct(\zeta, \mu') \leq \muefct(\zeta, \mu) + 64 \beta_2(1+ \Riter_0(\bI)) \leq \muefct(\zeta, \mu) + \frac{1}{\gamma} \, .
 \end{align*}
Recalling that $\gamma \leq 2/\muefct(\zeta, \mu)$, we then know $\muefct(\zeta, \mu') \leq 3 / \gamma$ and thus
\begin{align*}
	\gamma' \geq \min \brc{ \frac{2}{\muefct(\zeta, \mu')}, \frac{1}{\zeta}} \geq \frac{2}{3} \gamma \, .
\end{align*}
Together with $\Riter_0(\bI) = \bigTht(\Riter_0'(\bI))$, we then show $\alpha_1 = \bigO(\alpha_1')$ and further that $n^{-D} = \bigO \prn{ \alpha_1 / \prn{1 + \Riter_0(\bI)}} = \bigO \prn{ \alpha_1' / \prn{1 + \Riter_0'(\bI)}}$. Hence, we can apply Lemmas~\ref{lem:from-F-to-E} and \ref{lem:from-E-to-F}, and by Eq.~\eqref{eq:muLarger}
\begin{align}
	\left|\Rfct_n(\zeta, \mu'; \bQ) - \Rfct_0(\zeta, \muefct(\zeta, \mu'); \bQ) \right| = \bigO (\gamma' \beta_2' \Riter_0'(\bQ) + \beta_1') = \bigO (\gamma \beta_2 \Riter_0(\bQ) + \beta_1) \, . \label{eq:triangular-1}
\end{align}
In order to finish the perturbation argument, we bound
\begin{align}
	\left|\Rfct_n(\zeta, \mu; \bQ) - \Rfct_n(\zeta, \mu'; \bQ)\right| & \leq \left|(\mu' - \mu) \cdot \Tr \prn{\bQ \bSigma^{\half} (\zeta \bI + \mu \bSigma + \bX^\sT \bX)^{-1} \bSigma  (\zeta \bI + \mu' \bSigma + \bX^\sT \bX)^{-1}} \right| \nonumber \\
	& \leq \mu' \norm{\bSigma^{\half} (\zeta \bI + \bX^\sT \bX)^{-1} \bSigma} \Rfct_n(\zeta, \mu'; \bQ) \nonumber \\
	& \stackrel{\mathrm{(i)}}{=} \bigO \prn{\gamma \beta_2 \prn{ \Rfct_0(\zeta, \muefct(\zeta, \mu'); \bQ) +  \left|\Rfct_n(\zeta, \mu'; \bQ) -  \Rfct_0(\zeta, \muefct(\zeta, \mu'); \bQ)\right|}} \nonumber \\
	& = \bigO \prn{\gamma \beta_2  \prn{\Riter_0(\bQ) + \gamma \beta_2 \Riter_0(\bQ) + \beta_1}} \nonumber \\
	& = \bigO \prn{\gamma \beta_2  \Riter_0(\bQ)} \, ,  \label{eq:triangular-2}
\end{align}
where in (i) we apply Lemma~\ref{lem:A-norm-bound} and in the last line we use $\beta_1 = \bigO(\Riter_0(\bQ))$ and $\gamma \beta_2 = \bigO((1 + \Riter_0(\bI))^{-1}) = \bigO(1)$. Similarly, invoke Lemma~\ref{lem:muefct-derivative-bound} and we have
\begin{align}
	\left|\Rfct_0(\zeta, \muefct(\zeta, \mu'); \bQ) - \Rfct_0(\zeta, \muefct(\zeta, \mu); \bQ)\right| & \leq (\muefct(\zeta, \mu) - \muefct(\zeta, \mu')) \gamma \Rfct_0(\zeta, \muefct(\zeta, \mu); \bQ) \nonumber \\
	& \leq \mu'(1 + \Riter_0(\bI)) \gamma \Riter_0(\bQ) \nonumber \\
	& = \bigO \prn{\gamma \beta_2 (1 + \Riter_0(\bI)) \Riter_0(\bQ) } \, . \label{eq:triangular-3}
\end{align}
By triangular inequality, we deduce from Eqs.~\eqref{eq:triangular-1}, \eqref{eq:triangular-2} and \eqref{eq:triangular-3} that
\begin{align*}
	& |\Rfct_n(\zeta, \mu; \bQ) - \Rfct_0(\zeta, \muefct(\zeta, \mu); \bQ)| \nonumber \\
	& = \bigO \prn{ \gamma \beta_2 \Riter_0(\bQ) + \beta_1 } +  \bigO \prn{\gamma \beta_2 \Riter_0(\bQ)} +  \bigO \prn{\gamma \beta_2 (1 + \Riter_0(\bI)) \Riter_0(\bQ) } \nonumber \\
	& = \bigO \prn{\gamma \beta_2 \prn{1 + \Riter_0(\bI)} \Riter_0(\bQ) + \beta_1 } \, .
\end{align*}

\subsection{Proof of Corollary~\ref{cor:R-approximation-simplified}} \label{proof:R-approximation-simplified}
We first derive upper bounds for the parameter $\Delta = (\alpha_1, \alpha_2, \beta_1, \beta_2, \gamma)$ in 
Theorem~\ref{thm:R-approximation}. Since $\muefct(\zeta, 0) \leq \muefct(\zeta, \mu) \leq (1 - \const/2)^{-1} \muefct(\zeta, 0)$ when $\zeta = n\lambda$, we have
\begin{align*}
	\frac{n}{1 + \Riter_0(\bI)} \leq \muefct(\zeta, \mu) \leq \frac{1}{1 - \const/2} \muefct(\zeta, 0) = \frac{1}{1 - \const/2} \cdot \frac{n}{1 + \Rfct_0(\zeta, \muefct(\zeta, 0); \bI)} \leq \frac{1}{1 - \const/2} \cdot \frac{n}{1 + \Riter_0(\bI)} \, .
\end{align*}
On the other hand, by Eq.~\eqref{asmp:lambda} and the fact that at $\zeta = n \lambda$, $\muefct(\zeta, 0) = \zeta/ \lambdaefct$, we know
\begin{align*}
	\frac{1}{1 + \const^{-1}} \leq \const \leq \frac{\zeta}{n \lambdaefct} \leq \frac{1}{1 + \Riter_0(\bI)} \leq \frac{1}{1 - \const/2}\frac{\zeta}{n \lambdaefct} = \frac{1 - \const}{1 - \const/2} \leq \frac{1}{1 + \const/2} \, ,
\end{align*}
which implies $\const/2 \leq \Riter_0(\bI) \leq \const^{-1}$. 
Generalizing the definition of Eq.~\eqref{eq:RhoLambda} to $\mu>0$ and arbitrary $\bQ$, we let 
$\rho := \Riter_0(\bQ) / \Riter_0(\bI)  \in (0, 1]$.

\paragraph{Upper bound for $\gamma$} First we notice that $\constantsig \geq n$ by Assumption~\ref{asmp:data-dstrb} and therefore
\begin{align*}
	\log n \log (\constantsig n) \leq \log (\constantsig) \cdot \log (\constantsig^2) = \bigO \prn{\log^2 (\constantsig)}\, ,
\end{align*}
which yields
\begin{align*}
	\gamma  = \bigO \prn{\frac{2}{n} \brc{1 + \frac{\bigO_{\constantx, D} \prn{\sigma_{\lfloor \eta n\rfloor} \constantsig  \log^2 (\constantsig)}}{\zeta}} + \frac{2}{\muefct(\zeta, \mu)} } \, .
\end{align*}
Since $\muefct(\zeta, \mu) \geq \muefct(\zeta, 0) \geq n \const$, we can write
\begin{align}
	\gamma  = \frac{1}{n \const} \cdot \bigO_{\constantx, D} \prn{1 + \frac{\sigma_{\lfloor \eta n\rfloor} \constantsig  \log^2 (\constantsig)}{\zeta}}  =  \bigO_{\constantx, D} \prn{ \frac{\chi_n(\zeta)}{n \const}}\, . \label{eq:gamma-bound}
\end{align}

\paragraph{Upper bounds for $\alpha_1$ and $\alpha_2$} We know that $\Riter_0(\bI) \leq \const^{-1}$ and 
$\Riter_0(\bQ) = \rho \Riter_0(\bI)$. As a consequence, we have
\begin{align}
	\alpha_1 & = \bigO_{\constantx, D} \prn{ \log n \cdot \sqrt{\frac{\gamma}{\const}}}\, , \qquad \alpha_2 = \bigO_{\constantx, D} \prn{\log n \cdot \gamma \sqrt{\frac{\gamma \rho}{\const}}} \, . \label{eq:alpha-bound}
\end{align}

\paragraph{Upper bounds for $\beta_1$ and $\beta_2$} Using the bounds in the previous displays, we can write
\begin{align*}
\beta_1 & = \constant_\beta \prn{\sqrt{n \log n} \cdot \frac{\alpha_1 \gamma \Riter_0(\bQ) + \alpha_2 (1 + \Riter_0(\bI))}{1 + \Riter_0(\bI)^2} + n \cdot \brc{\frac{\gamma^2 \Riter_0(\bQ) + \alpha_1 \alpha_2}{1 + \Riter_0(\bI)^2} + \frac{\alpha_1^2  \gamma \Riter_0(\bQ)}{1 + \Riter_0(\bI)^3}} + \frac{\gamma \Riter_0(\bQ)}{1 + \Riter_0(\bI)}}  \nonumber \\
& = \bigO_{ \constantx, D}\prn{\sqrt{n \log n} \cdot \prn{\alpha_1 \gamma \rho + \alpha_2} + n \cdot \brc{\prn{\gamma^2 \rho  + \alpha_1 \alpha_2} + \alpha_1^2  \gamma \rho } + \gamma \rho} \nonumber \\
& = \bigO_{ \constantx, D}\prn{\sqrt{n \log n} \cdot \bigO_{\constantx, D} \prn{ \log n \cdot \gamma \sqrt{\frac{\gamma \rho}{\const}}} + n \cdot \bigO_{ \constantx, D} \prn{\log^2 n \cdot \gamma^2 \sqrt{\frac{\rho}{\const}}} + \gamma \rho} \nonumber \\
& = \bigO_{\constantx, D} \prn{ \sqrt{n (\log n)^3 \gamma^3} + n (\log n)^2 \gamma^2 + \gamma \sqrt{\const}} \cdot \sqrt{\frac{\rho}{\const}} \, .
\end{align*}
Substituting in Eq.~\eqref{eq:gamma-bound}, we can further bound
\begin{align}
	\beta_1 & =  \bigO_{\constantx, D} \prn{\frac{\sqrt{\rho} \chi_n(\lambda)^2 \log^2 n}{n \const^{2.5}}}   \, . \label{eq:beta-1-bound}
\end{align}
As for $\beta_2$, we simply bound it by
\begin{align}
	\beta_2 = \bigO_{\constantx, D} \prn{n \beta_1} = \bigO_{\constantx, D} \prn{\frac{\sqrt{\rho} \chi_n(\lambda)^2 \log^2 n}{\const^{2.5}} } \, . \label{eq:beta-2-bound}
\end{align}

\paragraph{Upper bound for resolvent approximation} Recall the approximation bound we have in Theorem~\ref{thm:R-approximation},
\begin{align*}
	|\Rfct_n(\zeta, \mu; \bQ) - \Rfct_0(\zeta, \muefct(\zeta, \mu); \bQ)|
	& = \bigO \prn{\gamma \beta_2 \prn{1 +   \Rfct_0(\zeta, \muefct(\zeta, \mu); \bI)}  \Rfct_0(\zeta, \muefct(\zeta, \mu); \bQ) + \beta_1 } \nonumber \\
	& = \bigO \prn{\gamma \beta_2 \const^{-1}  \Rfct_0(\zeta, \muefct(\zeta, \mu); \bQ) + \beta_1 } \, , 
\end{align*}
because $\Rfct_0(\zeta, \muefct(\zeta, \mu); \bI) = \Riter_0(\bI) \leq \const^{-1}$ at $\zeta = n\lambda$.
And thus
\begin{align*}
	& |\Rfct_n(\zeta, \mu; \bQ) - \Rfct_0(\zeta, \muefct(\zeta, \mu); \bQ)| = \bigO \prn{ \gamma \beta_2 \const^{-1} \cdot \rho \const^{-1} + \beta_1} \nonumber \\
	& = \bigO_{\constantx, D} \prn{\frac{\sqrt{\rho^3} \chi_n(\lambda)^3 \log^2 n}{n \const^{5.5}}} + \bigO_{\constantx, D} \prn{\frac{\sqrt{\rho} \chi_n(\lambda)^2 \log^2 n}{n \const^{2.5}}}   \nonumber \\
	& = \bigO_{\constantx, D} \prn{\frac{\sqrt{\rho} \chi_n (\lambda)^3 \log^2 n}{n \const^{5.5}}}\, .
\end{align*}
As $\Riter_0(\bQ) = \rho \Riter_0(\bI) \geq \rho \const / 2$, we can also write
\begin{align*}
	& |\Rfct_n(\zeta, \mu; \bQ) - \Rfct_0(\zeta, \muefct(\zeta, \mu); \bQ)|  	 = \bigO_{ \constantx, D} \prn{\frac{\chi_n (\lambda)^3 \log^2 n}{n \sqrt{\rho} \cdot \const^{6.5}} } \Rfct_0(\zeta, \muefct(\zeta, \mu); \bQ) \, .
\end{align*}
\paragraph{Simplifying the conditions} Finally we conclude the proof by simplifying the conditions 
$\alpha_1 \leq \Riter_0(\bI)/8$, $\beta_1 \leq \Riter_0(\bQ)/64$, $\gamma \beta_2(1 + \Riter_0(\bI)) \leq 1/64$ and $n^{-D} = \bigO(\alpha_1/(1 + \Riter_0(\bI)))$. As $\Riter_0(\bI) \geq \const/2$, by Eq.~\eqref{eq:alpha-bound} it is sufficient to have the first condition once
\begin{align*}
	\frac{\chi_n(\lambda) \log^2 n}{\const^4}  \leq \constant n \, ,
\end{align*}
for some sufficiently small constant $\constant = \constant(\constantx, D)$. Recall that $\Riter_0(\bQ) \geq \rho \const/2$. Therefore, by Eq.~\eqref{eq:beta-1-bound}, the second requirement can be deduced from
\begin{align*}
	\frac{\chi_n(\lambda)^2 \log^2 n}{ \const^{3.5}}  \leq \constant' n \sqrt{\rho}\, ,
\end{align*}
for some sufficiently small constant $\constant' = \constant'(\constantx, D)$. By Eqs.~\eqref{eq:gamma-bound} and \eqref{eq:beta-2-bound}, we can derive $\gamma \beta_2(1 + \Riter_0(\bI)) \leq 1/64$ from 
\begin{align*}
	 \frac{\chi_n(\lambda)^3 \log^2 n}{\const^{4.5}}   \leq \constant'' n / \sqrt{\rho} \, , 
\end{align*}
for some constant $\constant'' = \constant''(\constantx, D)$. For the last condition, we need a lower bound for $\alpha_1 = \constant_\alpha \log n \cdot \sqrt{\gamma \Riter_0(\bI)}$. As $\zeta = n\lambda$ and  $\muefct(\zeta, \mu) \leq (1  - \const/2)^{-1} \muefct(\zeta, 0) \leq (1  - \const/2)^{-1} n \leq 2n$, it follows that
\begin{align*}
	\gamma & = \min \brc{\frac{2}{n} \prn{1 + \frac{\constant_\gamma\constantsig \sigma_{\lfloor \eta n\rfloor} \cdot \log n \log (\constantsig n)}{\zeta}} + \frac{2}{\muefct(\zeta, \mu)} , \frac{1}{\zeta}} \nonumber \\
	& = \bigOmg \prn{\min \brc{\frac{1}{\muefct(\zeta, \mu)}, \frac{1}{\zeta}} } = \bigOmg \prn{\min \brc{\frac{1}{n}, \frac{1}{\zeta}}} \, .
\end{align*}
With $\const/2 \leq \Riter_0(\bI) \leq \const^{-1}$, we obtain
\begin{align*}
\frac{\alpha_1}{1 +\Riter_0(\bI)} = \bigOmg \prn{ \sqrt{\frac{\const^3 \log^2 n}{\max \brc{n, n\lambda}}}} \, .
\end{align*}
It is then sufficient to have
\begin{align*}
	n^{-D} = \bigO \prn{ \sqrt{\frac{\const^3 \log^2 n}{n\max \brc{1, \lambda}}}} \, .
\end{align*}
	\section{Proof of Proposition~\ref{propo:Nu}} \label{proof:Nu}
By Eq.~\eqref{eq:lambda-fixed-point}, we have
\begin{align*}
	\nu = \frac{\lambda}{ \lambdaefct(\lambda)} = 1 - \frac{1}{n} \Tr \prn{\bSigma(\bSigma + \lambdaefct(\lambda) \bI)^{-1}}.
\end{align*}
Let $\lambdaefct(\lambda) = C_1' \sigma_{2n} = C_2' \sigma_{\constantsig^{-1}(n/2)}$. On the one hand, we have
\begin{align*}
	\frac{\lambda}{ \lambdaefct(\lambda)} & \geq \frac{\sigma_{k_{\#}(n)}}{C' C_2' \sigma_{\constantsig^{-1}(n/2)}} \geq \frac{\sigma_{k_{\#}(2n)}}{C' C_2' \sigma_{\sigma_{k_{\#}(n/2)}}} \geq \frac{1}{C C'C_2'}, \\
	\frac{\lambda}{ \lambdaefct(\lambda)} & \leq  \frac{C'\sigma_{k_{\#}(n)}}{C_1' \sigma_{2n}} \leq  \frac{C'\sigma_{k_{\#}(n/2)}}{C_1' \sigma_{k_{\#}(2n)}} \leq \frac{CC'}{C_1'}.
\end{align*}
On the other hand,
\begin{align*}
	\frac{1}{n} \Tr \prn{\bSigma(\bSigma + \lambdaefct(\lambda) \bI)^{-1}} & = \frac{1}{n}\sum_{k=1}^d \frac{\sigma_k}{\sigma_k + C_1' \sigma_{2n} } \geq \frac{1}{n}\sum_{k=1}^{2n} \frac{\sigma_k}{\sigma_k + C_1' \sigma_{2n} } \geq  \frac{2}{1 + C_1'} \, , \\
		\frac{1}{n} \Tr \prn{\bSigma(\bSigma + \lambdaefct(\lambda) \bI)^{-1}} & = \frac{1}{n}\sum_{k=1}^d \frac{\sigma_k}{\sigma_k + C_2' \sigma_{\constantsig^{-1}(n/2)} } \leq \frac{1}{n} \prn{k_{\#}(n/2) + \frac{1}{C_2'}\sum_{k=k_{\#}(n/2)}^d \frac{\sigma_k}{\sigma_{k_{\#}(n/2)}} }  \nonumber \\
		& \leq \frac{1}{n} \prn{\frac{n}{2} + \frac{1}{C_2'} \frac{n}{2}} = \frac{1 + C_2'}{2C_2'} \, .
\end{align*}
Combining both set of inequalities yield
\begin{align*}
	\frac{1}{CC'C_2'} & \leq \frac{\lambda}{ \lambdaefct(\lambda)} = 1 - \frac{1}{n} \Tr \prn{\bSigma(\bSigma + \lambdaefct(\lambda) \bI)^{-1}} \leq 1 - \frac{2}{1 + C_1'} = \frac{C_1' - 1}{C_1' + 1} \, , \\
	\frac{CC'}{C_1'} & \geq \frac{\lambda}{ \lambdaefct(\lambda)} = 1 - \frac{1}{n} \Tr \prn{\bSigma(\bSigma + \lambdaefct(\lambda) \bI)^{-1}} \geq 1 - \frac{1 + C_2'}{2C_2'} = \frac{C_2' - 1}{2C_2'} \, .
\end{align*}
We then use the native bound $C_1' \leq C_2'$ to obtain
\begin{align*}
	\frac{1}{CC'C_1'} \leq  \frac{C_1' - 1}{C_1' + 1} \, , \qquad \frac{CC'}{C_2'}  \geq \frac{C_2' - 1}{2C_2'} \, ,
\end{align*}
which implies $C_1' = \bigO_{C, C'}(1)$ and $C_2' = \bigOmg_{C, C'}(1)$. With
\begin{align*}
	\frac{1}{CC'C_1'} \leq 	\frac{1}{CC'C_2'} \leq \nu \leq \frac{C_1' - 1}{C_1' + 1},
\end{align*}
we conclude the proof.

\section{Proof of Proposition~\ref{prop:control-lambda-zero}} \label{proof:control-lambda-zero}
Note that $\Tr \prn{\bSigma (\bSigma + \zeta \bI)^{-1}}$ is a decreasing function in $\zeta$, it suffices to show
\begin{align*}
	\Tr \prn{\bSigma (\bSigma + \sigma_{2n} \bI)^{-1}} \geq n \geq  \Tr \prn{\bSigma (\bSigma + \sigma_{\constantsig^{-1}(n/2)} \bI)^{-1}} \, .
\end{align*}
The left hand side follows from
\begin{align*}
	\Tr \prn{\bSigma (\bSigma + \sigma_{2n} \bI)^{-1}} & = \sum_{k=1}^d \frac{\sigma_k}{\sigma_k + \sigma_{2n}} \geq  \sum_{k=1}^{2n} \frac{\sigma_k}{\sigma_k + \sigma_{2n}} \geq \sum_{k=1}^{2n} \frac{\sigma_k}{\sigma_k + \sigma_k}  \nonumber \\
	& = 2n \cdot \frac{1}{2} = n \, , 
\end{align*}
and similarly the right hand side is a consequence of 
\begin{align*}
	\Tr \prn{\bSigma (\bSigma +\sigma_{\constantsig^{-1}(n/2)} \bI)^{-1}} & = \sum_{k=1}^d \frac{\sigma_k}{\sigma_k + \sigma_{\constantsig^{-1}(n/2)}} \leq  \constantsig^{-1}(n/2) + \sum_{k=\constantsig^{-1}(n/2)}^d \frac{\sigma_k}{\sigma_{\constantsig^{-1}(n/2)}} \nonumber \\
	& \leq \frac{n}{2} + \frac{n}{2} = n \, ,
\end{align*}
where we use the fact that $\constantsig(n) \geq n$ and hence $\constantsig^{-1}(n) \leq n$. 
	\section{Proof of Theorem~\ref{thm:main-Bis}}  \label{proof:main-Bis}
Under the `non-negligible regularization' regime, we can upper bound
\begin{align*}
	\chi_n (\lambda) & = 1 + \frac{\sigma_{\lfloor \eta n\rfloor} \constantsig  \log^2 (\constantsig)}{n\lambda} \leq 1 + \frac{\sigma_{\lfloor \eta n\rfloor} \tconstantsig(n)}{n \lambdaefct(\lambda) / C}  = 1 + \frac{\sigma_{\lfloor \eta n\rfloor} \tconstantsig(n)}{n \lambda_0(\nu(1-n)) / C}  \, ,
\end{align*}
and by Proposition~\ref{prop:control-lambda-zero}, $\lambda_0(n(1-\nu)) \geq \sigma_{2n(1-\nu)} \geq \sigma_{2n}$, implying
\begin{align*}
	\chi_n (\lambda) \leq 1 + \frac{\sigma_{\lfloor \eta n\rfloor} \tconstantsig(n)}{n \sigma_{2n} / C}
\end{align*}
By $\tconstantsig(n)\le (\sigma_{2n}/\sigma_{\lfloor\eta n\rfloor})n^{4/3}(\log n)^{-2/3-\epsilon}$, it then holds that
\begin{align*}
	\chi_n (\lambda) \leq 1 + \frac{C'n^{1/3}}{(\log n)^{2/3 + \epsilon}} = \bigO \prn{n^{1/3}(\log n)^{-2/3-\epsilon}} \, , 
\end{align*}
and consequently the condition $\chi_n(\lambda)^3 \log^2 n \leq \constant n \const^{4.5}$ is met in Theorem~\ref{thm:main}. Further $\const = \min (\nu, 1 - \nu) = 1/C$, taking $D=11, k=100$ in Theorem~\ref{thm:main}, it holds
\begin{align*}
	n^{-2D + 1} = \bigO \prn{ \sqrt{\frac{\const^3 \log^2 n}{n\max \brc{1, \lambda}}}} \, , \qquad \left|\var_\bX(\lambda) - \VAR_n(\lambda)\right| &  
		= \bigO\prn{\frac{1}{n^{0.99}}
			\left(\frac{\tconstantsig(n)\sigma_{\lfloor\eta n\rfloor }}
			{n \sigma_{2n}}\right)^3 }\cdot \VAR_n(\lambda) \, .
\end{align*}
To apply the bias approximation result, we verify the conditions. Firstly,
\begin{align*}
	\rho(\lambda) & = \frac{\Rfct_0(\lambdaefct, 1; \btheta \btheta^\sT/\norm{\btheta}^2)}{\Rfct_0(\lambdaefct,1; \bI)} = \frac{\<\boldbeta, (\lambdaefct \bI +  \bSigma)^{-1} \boldbeta\> / \norm{\boldbeta}_{\bSigma^{-1}}^2}{\Tr \prn{\bSigma(\lambdaefct \bI +  \bSigma)^{-1}}} \nonumber \\
	& =\frac{\<\boldbeta, (\lambdaefct \bI +  \bSigma)^{-1} \boldbeta\> / \norm{\boldbeta}_{\bSigma^{-1}}^2}{n(1 - \nu)} =\frac{\<\boldbeta, (\lambda_0(n(1-\nu)) \bI +  \bSigma)^{-1} \boldbeta\> / \norm{\boldbeta}_{\bSigma^{-1}}^2}{n(1 - \nu)}.
\end{align*}
Apply Proposition~\ref{prop:control-lambda-zero} and recall $\nu \in [1/C, 1- 1/C]$, we know $\rho(\lambda) = \bigOmg(n^{-1})$. Therefore given $\tconstantsig(n)\le (\sigma_{2n}/\sigma_{\lfloor\eta n\rfloor})n^{7/6}(\log n)^{-2/3-\epsilon}$,
\begin{align*}
	\chi_n (\lambda)^3 \log^2 n = \bigO \prn{n^{1/2}(\log n)^{-3\epsilon}} = \bigO\prn{\constant n \const^{4.5} \sqrt{\rho(\lambda)}}\, .
\end{align*}
The other condition $\lambda k n^{1-\frac{1}{k}} \leq n \const /2 $ holds evidently. Again taking $D=11, k =100$ and substitute into Eq.~\eqref{eq:MainBiasApprox}, we complete the proof for the bias approximation.
	\section{Proof of Theorem~\ref{thm:main-ridgeless}} \label{proof:main-ridgeless}
To apply triangle inequalities
\begin{equation} 
	\label{eq:V-B-approx-triangle}
	\begin{aligned}
	\left|\var_\bX(0) - \VAR_n(0)\right| &\leq \left|\var_\bX(\lambda) - \VAR_n(\lambda) \right| + \left|\var_\bX(0) - \var_\bX(\lambda) \right| + \left|\VAR_n(0) - \VAR_n(\lambda) \right| \, , \\
	\left|\bias_\bX(0) - \BIAS_n(0)\right| &\leq \left|\bias_\bX(\lambda) - \BIAS_n(\lambda) \right| + \left|\bias_\bX(0) - \bias_\bX(\lambda) \right| + \left|\BIAS_n(0) - \BIAS_n(\lambda) \right| \, ,
	\end{aligned}
\end{equation}
we define $\lambda = \const  \lambdaefct$ and bound each term separately. By homogeneity, we will assume $\norm{\btheta} = 1$ throughout the proof.

\paragraph{Part I: Bounding \texorpdfstring{$\left|\var_\bX(0) - \var_\bX(\lambda) \right|$}{TEXT} and \texorpdfstring{$\left|\bias_\bX(0) - \bias_\bX(\lambda) \right|$}{TEXT}} Assume $\bX^\sT \bX$ has rank $r$ with eigendecomposition $\bX^\sT \bX = \bU \bD \bU^\sT$ where $\bU \in \real^{d \times r}$ has orthonormal columns and $\bD$ is a diagonal matrix with entries $s_1 \geq \cdots \geq s_r > 0$. Note that $s_r = n \smin$.

For the variance term, by the elementary inequality $|1/x - x/(x+\zeta)^2| \leq 2 \zeta/ x^2$ for all $x, \zeta >0$, we have by Eq.~\eqref{eq:var-X}
\begin{align}
	\left|\var_\bX(0) - \var_\bX(\lambda) \right| & = \left|\vareps^2 \Tr \prn{\bSigma \bU \prn{\bD^{-1} - \bD (\bD + n\lambda \bI)^{-2}} \bU^\sT} \right| \nonumber \\
	& \leq \vareps^2 \Tr \prn{\frac{2 n\lambda}{s_r} \cdot \bSigma \bU \bD^{-1} \bU^\sT} = \frac{2\const \lambdaefct(\lambda)}{\smin} \cdot \var_\bX(0) \, ,  \label{eq:var-X-bound-mid-1}
\end{align}
where in the last equality we use $\lambda = \const \lambdaefct(\lambda)$ and $s_r = n \smin$. The next lemma bounds the difference between $\lambdaefct(0)$ and $\lambdaefct(\lambda)$.
\begin{lemma} \label{lem:lambdaefct-local-bound}
	Under the assumptions of Theorem~\ref{thm:main-ridgeless}, for $\lambda$ such that $\lambda = \const  \lambdaefct(\lambda)$ it holds that
	\begin{align*}
		\lambdaefct(0) \leq \lambdaefct(\lambda) \leq \prn{1 + \frac{2 \const}{\constantlambdaefct}}\lambdaefct(0) \leq 2 \lambdaefct(0) \, .
	\end{align*}
\end{lemma}
\begin{proof}
	Since
	\begin{align*}
		n\lambda = \lambdaefct \cdot \prn{n - \Tr \prn{\bSigma(\bSigma + \lambdaefct \bI)^{-1}}} \, ,
	\end{align*}
	we can compute that by change of variable $\zeta = n \lambda$,
	\begin{align*}
		\frac{\partial \zeta}{\partial \lambdaefct} & = \lambdaefct \cdot \Tr \prn{\bSigma(\bSigma + \lambdaefct \bI)^{-2}} + n - \Tr \prn{\bSigma(\bSigma + \lambdaefct \bI)^{-1}} = n - \Tr \prn{\bSigma^2(\bSigma + \lambdaefct(\lambda) \bI)^{-2}} \nonumber \\
		& \geq n - \Tr \prn{\bSigma^2(\bSigma + \lambdaefct(0) \bI)^{-2}}  \geq \constantlambdaefct n \, ,
	\end{align*}
	and thus
	\begin{align*}
		\lambdaefct(\lambda) & = \lambdaefct(0) + \int_0^\lambda \frac{\partial \lambdaefct(\lambda)}{\partial \lambda} \de \lambda = \lambdaefct(0) + \int_0^{n \lambda} \frac{\partial \lambdaefct(\zeta)}{\partial \zeta} \de \zeta \nonumber \\
		&  \leq \lambdaefct(0) + \frac{n \lambda}{\constantlambdaefct n} = \lambdaefct(0) + \frac{\lambda}{\constantlambdaefct  \lambdaefct(\lambda)} \cdot \lambdaefct(\lambda)  = \lambdaefct(0) + \frac{\const}{\constantlambdaefct} \cdot \lambdaefct(\lambda) \, .
	\end{align*}
	Rearranging terms, using $\const \leq \constantlambdaefct^2/8 \leq \constantlambdaefct/2$ and the fact that $(1-x)^{-1} \leq 1 +2x$ for $0 \leq x \leq 1/2$ conclude the proof.
\end{proof} 
Returning to the bound of the variance term, we can thus further derive the upper bound
\begin{align*}
	\left|\var_\bX(0) - \var_\bX(\lambda) \right| & \leq  \frac{4\const \lambdaefct(0)}{\smin} \cdot \var_\bX(0) \, .
\end{align*}Using the fact that $\const \leq \smin / (8 \lambdaefct(0))$, we further have
\begin{align}
	\left|\var_\bX(0) - \var_\bX(\lambda) \right| & \leq \prn{1 - \frac{4\const \lambdaefct(0)}{\smin}}^{-1} \cdot \frac{4\const \lambdaefct(0)}{\smin} \cdot \var_\bX(\lambda)  \leq \frac{8\const \lambdaefct(0)}{\smin} \cdot \var_\bX(\lambda) \, .\label{eq:V-approx-lambda-1}
\end{align}

Now we look at the bias term. From Eq.~\eqref{eq:bias-X}, we first have
\begin{align*}
	\bias_\bX(0) & = \lim_{\zeta \downarrow 0} \zeta^2 \Tr \prn{\boldbeta \boldbeta^\sT (\bX^\sT \bX + \zeta \bI)^{-1} \bSigma (\bX^\sT \bX + \zeta \bI)^{-1}} \nonumber \\
	& = \lim_{\zeta \downarrow 0}  \norm{ (\bX^\sT \bX / \zeta + \bI)^{-1} \boldbeta}_{\bSigma}^2 \nonumber \\
	& =\lim_{\zeta \downarrow 0}  \norm{ \prn{\bI + \bU \prn{(\bD / \zeta + \bI)^{-1} - \bI } \bU^\sT} \boldbeta}_{\bSigma}^2   = \norm{ \prn{\bI - \bU  \bU^\sT} \boldbeta}_{\bSigma}^2 \, .
\end{align*}
By triangle inequality, it thus follows
\begin{align*}
	\left|\bias_\bX(0)^{\half} - \bias_\bX(\lambda)^{\half} \right| & \leq \norm{ \bU (\bD / n\lambda + \bI)^{-1} \bU^\sT  \boldbeta}_{\bSigma} \leq \frac{n\lambda}{n\lambda + s_r} \norm{\boldbeta} = \frac{\const \lambdaefct(\lambda)}{\const \lambdaefct(\lambda) + \smin} \norm{\boldbeta} \nonumber \\
	& \leq \frac{2\const \lambdaefct(0)}{2\const \lambdaefct(0) + \smin} \norm{\boldbeta} \, ,
\end{align*}
where in the last line we invoke Lemma~\ref{lem:lambdaefct-local-bound} and use $\lambdaefct(\lambda) \leq 2 \lambdaefct(0)$. 

Additionally with  $\bias_\bX(0) \leq \norm{\boldbeta}^2$ and
\begin{align*}
	& \bias_\bX(\lambda) = n^2\lambda^2 \Tr \prn{\boldbeta \boldbeta^\sT (\bX^\sT \bX + n\lambda \bI)^{-1} \bSigma (\bX^\sT \bX + n\lambda \bI)^{-1}} \nonumber \\
	&  \leq \Tr \prn{\boldbeta \boldbeta^\sT (\bX^\sT \bX / n\lambda + \bI)^{-2}} \leq \Tr \prn{\boldbeta \boldbeta^\sT} = \norm{\boldbeta}^2 \, ,
\end{align*}
we conclude that
\begin{align}
	& \left|\bias_\bX(0) - \bias_\bX(\lambda) \right| \nonumber \\
	& = \left|\bias_\bX(0)^{\half} + \bias_\bX(\lambda)^{\half} \right| \cdot \left|\bias_\bX(0)^{\half} - \bias_\bX(\lambda)^{\half} \right| \leq \frac{4\const \lambdaefct(0)  \norm{\boldbeta}^2}{2\const \lambdaefct(0) + \smin} \leq \frac{4\const \lambdaefct(0)  \norm{\boldbeta}^2}{\smin} \, . \label{eq:B-approx-lambda-1-a}
\end{align}

We obtain an alternative upper bound for $\left|\bias_\bX(0) - \bias_\bX(\lambda) \right|$ in the following way. Note that
\begin{align*}
	& \left|\bias_\bX(0)^{\half} - \bias_\bX(\lambda)^{\half} \right|  \leq \norm{ \bU (\bD / n\lambda + \bI)^{-1} \bU^\sT  \boldbeta}_{\bSigma} = n\lambda \norm{ \bSigma^{1/2} (\bX^\sT \bX + n\lambda \bI)^{-1} \bX^\sT (\bX \bX^\sT)^{-1} \bX \bSigma^{1/2} \btheta} \\
	& = n\lambda \sqrt{\btheta^\sT \bSigma^{1/2} \bX^\sT (\bX \bX^\sT)^{-1} \bX (\bX^\sT \bX + n\lambda \bI)^{-1} \bSigma (\bX^\sT \bX + n\lambda \bI)^{-1} \bX^\sT (\bX \bX^\sT)^{-1} \bX \bSigma^{1/2} \btheta}  \\
	& \leq n\lambda \sqrt{\norm{(\bX^\sT \bX + n\lambda \bI)^{-\half}  \bSigma (\bX^\sT \bX + n\lambda \bI)^{-\half} }} \cdot \sqrt{\btheta^\sT \bSigma^{1/2} (\bX^\sT \bX + n\lambda \bI)^{-1}  \bSigma^{1/2} \btheta} \\
	& = n\lambda \sqrt{\norm{\bSigma^{\half} (\bX^\sT \bX + n\lambda \bI)^{-1}  \bSigma^{\half} }} \cdot \sqrt{\btheta^\sT \bSigma^{1/2} (\bX^\sT \bX + n\lambda \bI)^{-1}  \bSigma^{1/2} \btheta} \, .
\end{align*}
We next apply Lemma~\ref{lem:A-norm-bound}, which implies that with probability $1- \bigO(n^{-D})$
\begin{align*}
	\left|\bias_\bX(0)^{\half} - \bias_\bX(\lambda)^{\half} \right| & \leq \sqrt{\frac{1}{n} \bigO_{\constantx, D}(n\lambda \chi_n'(\const))} \cdot \sqrt{\frac{1}{n} \bigO_{\constantx, D}(n\lambda \chi_n'(\const))\norm{\btheta_{\leq n}}^2 +  2 \norm{\boldbeta_{>n}}^2} \nonumber \\
	& = \sqrt{ \bigO_{\constantx, D}(\const^2 \lambdaefct(0)^2 \chi_n'(\const)^2)\norm{\btheta_{\leq n}}^2 +  \bigO_{\constantx, D}(\const \lambdaefct(0) \chi_n'(\const)) \norm{\boldbeta_{>n}}^2} \, .
\end{align*}
Using the same argument, we can also bound
\begin{align*}
	\bias_\bX(\lambda) & = n^2\lambda^2 \Tr \prn{\boldbeta \boldbeta^\sT (\bX^\sT \bX + n\lambda \bI)^{-1} \bSigma (\bX^\sT \bX + n\lambda \bI)^{-1}} \nonumber \\
	&  \leq n^2\lambda^2 \norm{\bSigma^{1/2} (\bX^\sT \bX + n\lambda \bI)^{-1}  \bSigma^{1/2}} \cdot \btheta^\sT \bSigma^{1/2} (\bX^\sT \bX + n\lambda \bI)^{-1}  \bSigma^{1/2} \btheta \nonumber \\
	& = \bigO_{\constantx, D}(\const^2 \lambdaefct(0)^2 \chi_n'(\const)^2)\norm{\btheta_{\leq n}}^2 +  \bigO_{\constantx, D}(\const \lambdaefct(0) \chi_n'(\const)) \norm{\boldbeta_{>n}}^2 \, ,
\end{align*}
and we can therefore conclude that
\begin{align}
	\left|\bias_\bX(0) - \bias_\bX(\lambda) \right| & = \left|\bias_\bX(0)^{\half} + \bias_\bX(\lambda)^{\half} \right| \cdot \left|\bias_\bX(0)^{\half} - \bias_\bX(\lambda)^{\half} \right| \nonumber \\
	& \leq \prn{\left|\bias_\bX(0)^{\half} - \bias_\bX(\lambda)^{\half} \right| + 2 \bias_\bX(\lambda)^{\half}} \cdot \left|\bias_\bX(0)^{\half} - \bias_\bX(\lambda)^{\half} \right| \nonumber \\
	& = \bigO_{\constantx, D}(\const^2 \lambdaefct(0)^2 \chi_n'(\const)^2)\norm{\btheta_{\leq n}}^2 +  \bigO_{\constantx, D}(\const \lambdaefct(0) \chi_n'(\const)) \norm{\boldbeta_{>n}}^2  \, . \label{eq:B-approx-lambda-1-b}
\end{align}
Combining Eqs.~\eqref{eq:B-approx-lambda-1-a} and \eqref{eq:B-approx-lambda-1-b}, we finally have
\begin{align}
	& \left|\bias_\bX(0) - \bias_\bX(\lambda) \right| & \nonumber \\
	&=  \min \brc{\bigO\prn{\frac{\const \lambdaefct(0)  \norm{\boldbeta}^2}{\smin}}, \bigO_{\constantx, D}(\const^2 \lambdaefct(0)^2 \chi_n'(\const)^2)\norm{\btheta_{\leq n}}^2 +  \bigO_{\constantx, D}(\const \lambdaefct(0) \chi_n'(\const)) \norm{\boldbeta_{>n}}^2} \, .  \label{eq:B-approx-lambda-1}
\end{align}

\paragraph{Part II: Bounding $\left|\VAR_n(0) - \VAR_n(\lambda) \right|$ and $\left|\BIAS_n(0) - \BIAS_n(\lambda) \right|$} Note that
\begin{align*}
	0 \geq \frac{\partial \Tr \prn{\bSigma^2(\bSigma + \lambdaefct \bI)^{-2}}}{\partial \lambdaefct} = -2 \Tr \prn{\bSigma^2(\bSigma + \lambdaefct \bI)^{-3}} \geq -\frac{2}{\lambdaefct(0)} \Tr \prn{\bSigma^2(\bSigma + \lambdaefct(0) \bI)^{-2}} \, ,
\end{align*}
we can apply Lemma~\ref{lem:lambdaefct-local-bound} and obtain
\begin{align*}
	\Tr \prn{\bSigma^2(\bSigma + \lambdaefct(0) \bI)^{-2}} & \geq \Tr \prn{\bSigma^2(\bSigma + \lambdaefct(\lambda) \bI)^{-2}} \nonumber \\
	& \geq \Tr \prn{\bSigma^2(\bSigma + \lambdaefct(0) \bI)^{-2}} -\frac{2\prn{\lambdaefct(\lambda) - \lambdaefct(0)}}{\lambdaefct(0)} \Tr \prn{\bSigma^2(\bSigma + \lambdaefct(0) \bI)^{-2}} \nonumber \\
	& \geq \prn{1 - \frac{4 \const}{\constantlambdaefct}}  \cdot \Tr \prn{\bSigma^2(\bSigma + \lambdaefct(0) \bI)^{-2}} \, .
\end{align*}
We then have
\begin{align*}
	\VAR_n(0) \geq \VAR_n(\lambda) & = \frac{n - \Tr \prn{\bSigma^2(\bSigma + \lambdaefct(0) \bI)^{-2}}}{n - \Tr \prn{\bSigma^2(\bSigma + \lambdaefct(\lambda) \bI)^{-2}}} \cdot \frac{\Tr \prn{\bSigma^2(\bSigma + \lambdaefct(\lambda) \bI)^{-2}}}{\Tr \prn{\bSigma^2(\bSigma + \lambdaefct(0) \bI)^{-2}}} \cdot \VAR_n(0) \nonumber \\
	& \stackrel{\mathrm{(i)}}{\geq} \frac{\constantlambdaefct n}{\constantlambdaefct n + \frac{4\const}{\constantlambdaefct} \cdot \Tr \prn{\bSigma^2(\bSigma + \lambdaefct(0) \bI)^{-2}}} \cdot \prn{1 - \frac{4 \const}{\constantlambdaefct}} \cdot \VAR_n(0) \nonumber \\
	& \stackrel{\mathrm{(ii)}}{\geq} \frac{\constantlambdaefct^2}{\constantlambdaefct^2 + 4\const} \cdot \prn{1 - \frac{4\const}{\constantlambdaefct}} \cdot \VAR_n(0) \, ,
\end{align*}
where we use $n - \Tr \prn{\bSigma^2(\bSigma + \lambdaefct(0) \bI)^{-2}} \geq \constantlambdaefct n$ in (i) and $n \geq \Tr \prn{\bSigma^2(\bSigma + \lambdaefct(0) \bI)^{-2}} $ in (ii). By the elementary inequality $1-(1-a)(1-b) \leq a + b$ for all $0 \leq a, b \leq 1$, we can thus derive that
\begin{align}
	\left|\VAR_n(0) - \VAR_n(\lambda) \right| & \leq \brc{1 - \prn{1 - \frac{4 \const}{\constantlambdaefct^2 + 4\const}} \cdot \prn{1 - \frac{4\const}{\constantlambdaefct}}  }  \cdot \VAR_n(0) \nonumber \\
	& \leq \prn{\frac{4 \const}{\constantlambdaefct^2 + 4 \const} + \frac{4 \const}{\constantlambdaefct}} \cdot \VAR_n(0) \leq \frac{8 \const}{\constantlambdaefct^2} \cdot \VAR_n(0) \, . \label{eq:V-approx-lambda-2}
\end{align}

For the bias term, we first similarly derive
\begin{align*}
	\boldbeta^\sT \prn{\bSigma + \lambdaefct(0) \bI}^{-2} \bSigma \boldbeta \geq \boldbeta^\sT \prn{\bSigma + \lambdaefct(\lambda) \bI}^{-2} \bSigma \boldbeta \geq \prn{1 - \frac{4 \const}{\constantlambdaefct}}  \cdot \boldbeta^\sT \prn{\bSigma + \lambdaefct(0) \bI}^{-2} \bSigma \boldbeta \, .
\end{align*}
Note that
\begin{align*}
	\left|\frac{\BIAS_n(\lambda)}{\BIAS_n(0)} - 1 \right| & = \left|\frac{n - \Tr \prn{\bSigma^2(\bSigma + \lambdaefct(0) \bI)^{-2}}}{n - \Tr \prn{\bSigma^2(\bSigma + \lambdaefct(\lambda) \bI)^{-2}}} \cdot \frac{\boldbeta^\sT \prn{\bSigma + \lambdaefct(\lambda) \bI}^{-2} \bSigma \boldbeta}{\boldbeta^\sT \prn{\bSigma + \lambdaefct(0) \bI}^{-2} \bSigma \boldbeta} \cdot \frac{\lambdaefct(\lambda)^2}{\lambdaefct(0)^2} - 1\right|\nonumber \\
	& \leq \max \brc{1 - \frac{n - \Tr \prn{\bSigma^2(\bSigma + \lambdaefct(0) \bI)^{-2}}}{n - \Tr \prn{\bSigma^2(\bSigma + \lambdaefct(\lambda) \bI)^{-2}}} \cdot \frac{\boldbeta^\sT \prn{\bSigma + \lambdaefct(\lambda) \bI}^{-2} \bSigma \boldbeta}{\boldbeta^\sT \prn{\bSigma + \lambdaefct(0) \bI}^{-2} \bSigma \boldbeta}; \frac{\lambdaefct(\lambda)^2}{\lambdaefct(0)^2} - 1} \ , ,
\end{align*}
From the previous calculations for the variance term, we know
\begin{align*}
	1 - \frac{n - \Tr \prn{\bSigma^2(\bSigma + \lambdaefct(0) \bI)^{-2}}}{n - \Tr \prn{\bSigma^2(\bSigma + \lambdaefct(\lambda) \bI)^{-2}}} \cdot \frac{\boldbeta^\sT \prn{\bSigma + \lambdaefct(\lambda) \bI}^{-2} \bSigma \boldbeta}{\boldbeta^\sT \prn{\bSigma + \lambdaefct(0) \bI}^{-2} \bSigma \boldbeta} \leq \frac{8 \const}{\constantlambdaefct^2} \, ,
\end{align*}
and by Lemma~\ref{lem:lambdaefct-local-bound} we have
\begin{align*}
	\frac{\lambdaefct(\lambda)^2}{\lambdaefct(0)^2} - 1 & \leq \prn{1 + \frac{2 \const}{\constantlambdaefct}}^2 - 1 \leq \frac{4 \const}{\constantlambdaefct} + \frac{2 \const}{\constantlambdaefct} \cdot \frac{2 \const}{\constantlambdaefct} \leq \frac{6 \const}{\constantlambdaefct} \, .
\end{align*}
In the last inequality, recall $\const \leq \constantlambdaefct^2 /8 \leq \constantlambdaefct / 2$. Putting together, we have error of the bias term bounded by
\begin{align}
	\left|\BIAS_n(0) - \BIAS_n(\lambda) \right| & \leq \frac{8 \const}{\constantlambdaefct^2} \cdot \BIAS_n(0) \, . \label{eq:B-approx-lambda-2}
\end{align}

\paragraph{Part III: Variance approximation when $\lambda = 0$} Recalling that
$\lambda = \const  \lambdaefct(\lambda)$, we want to invoke Theorem~\ref{thm:main} to bound $|\var_\bX(\lambda) - \VAR_n(\lambda)|$. Note that by Lemma~\ref{lem:lambdaefct-local-bound} it holds $\lambdaefct(\lambda) = \bigTht(\lambdaefct(0))$ and thus
\begin{align*}
	\chi_n(\lambda) & =  1 + \frac{\sigma_{\lfloor \eta n\rfloor} \constantsig  \log^2 (\constantsig)}{n\lambda} =   1 + \frac{\sigma_{\lfloor \eta n\rfloor} \constantsig  \log^2 (\constantsig)}{\const n \lambdaefct(\lambda)} = \bigTht \prn{1 + \frac{\sigma_{\lfloor \eta n\rfloor} \constantsig  \log^2 (\constantsig)}{\const n \lambdaefct(0)} } \nonumber \\
	& = \bigTht(\chi_n'(\const)) \, .
\end{align*}
 Hence the conditions hold for Theorem~\ref{thm:main} by taking $\constant_1 = \bigTht(\constant)$, and we have for some constant $\constant' := \constant'(k, \constantx, D) > 0$,
	\begin{align*}
	\left|\var_\bX(\lambda) - \VAR_n(\lambda)\right| &  
	\leq \constant' \cdot \frac{\chi_n' (\const)^3 \log^2 n}{n^{1- \frac{1}{k}} \const^{9.5}} \cdot \VAR_n(\lambda) \, .
\end{align*}
Substituting the above display and Eqs.~\eqref{eq:V-approx-lambda-1}, \eqref{eq:V-approx-lambda-2} into Eq.~\eqref{eq:V-B-approx-triangle} yields
\begin{align*}
	\left|\var_\bX(0) - \VAR_n(0)\right| &\leq \constant' \cdot \frac{\chi_n' (\const)^3 \log^2 n}{n^{1- \frac{1}{k}} \const^{9.5}} \cdot \VAR_n(\lambda) +\frac{8\const \lambdaefct(0)}{\smin} \cdot \var_\bX(\lambda)  + \frac{8 \const}{\constantlambdaefct^2} \cdot \VAR_n(0) \nonumber \\
	& \leq \constant' \cdot \frac{\chi_n' (\const)^3 \log^2 n}{n^{1- \frac{1}{k}} \const^{9.5}} \cdot \VAR_n(\lambda) +\frac{8\const \lambdaefct(0)}{\smin} \cdot \prn{1+\constant' \cdot \frac{\chi_n' (\const)^3 \log^2 n}{n^{1- \frac{1}{k}} \const^{9.5}}} \VAR_n(\lambda)  + \frac{8 \const}{\constantlambdaefct^2} \cdot \VAR_n(0) \nonumber \\
	& = \brc{\prn{1 + \frac{8\const \lambdaefct(0)}{\smin}} \prn{1+\constant' \cdot \frac{\chi_n' (\const)^3 \log^2 n}{n^{1- \frac{1}{k}} \const^{9.5}}} - 1} \cdot \VAR_n(\lambda) + \frac{8 \const}{\constantlambdaefct^2} \cdot \VAR_n(0) \nonumber \\
	& \leq \brc{\prn{1 + \frac{8\const \lambdaefct(0)}{\smin}} \prn{1+\constant' \cdot \frac{\chi_n' (\const)^3 \log^2 n}{n^{1- \frac{1}{k}} \const^{9.5}}} - 1} \cdot \prn{1 + \frac{8 \const}{\constantlambdaefct^2}} \VAR_n(0) + \frac{8 \const}{\constantlambdaefct^2} \cdot \VAR_n(0)  \nonumber \\
	& \leq \brc{\prn{1 + \frac{8\const \lambdaefct(0)}{\smin}} \prn{1+\constant' \cdot \frac{\chi_n' (\const)^3 \log^2 n}{n^{1- \frac{1}{k}} \const^{9.5}}}  \prn{1 + \frac{8 \const}{\constantlambdaefct^2}}  - 1} \cdot \VAR_n(0) \nonumber \\
	& \leq \prn{\exp \prn{\frac{8\const \lambdaefct(0)}{\smin} + \frac{8 \const}{\constantlambdaefct^2} + \constant' \cdot \frac{\chi_n' (\const)^3 \log^2 n}{n^{1- \frac{1}{k}} \const^{9.5}}} - 1} \cdot \VAR_n(0) \, .
\end{align*}
Since $\const \leq \smin / (8 \lambdaefct(0))$ and $\const \leq \constantlambdaefct^2 /8$, if we additionally assume
\begin{align*}
	\frac{\chi_n' (\const)^3 \log^2 n}{n^{1- \frac{1}{k}} \const^{9.5}}  \leq \frac{1}{\constant'} \, ,
\end{align*}
we can then conclude that 
\begin{align*}
	\exp \prn{\frac{8\const \lambdaefct(0)}{\smin} + \frac{8 \const}{\constantlambdaefct^2} + \constant' \cdot \frac{\chi_n' (\const)^3 \log^2 n}{n^{1- \frac{1}{k}} \const^{9.5}}} - 1 = \bigO \prn{\frac{8\const \lambdaefct(0)}{\smin} + \frac{8 \const}{\constantlambdaefct^2} + \constant' \cdot \frac{\chi_n' (\const)^3 \log^2 n}{n^{1- \frac{1}{k}} \const^{9.5}}} \, ,
\end{align*}
and 
\begin{align*}
	\left|\var_\bX(0) - \VAR_n(0)\right| & = \bigO_{k, \constantx, D} \prn{ \const \cdot \prn{\frac{ \lambdaefct(0)}{\smin} + \frac{1}{\constantlambdaefct^2}} + \frac{\chi_n' (\const)^3 \log^2 n}{n^{1- \frac{1}{k}} \const^{9.5}}} \cdot \VAR_n(0) \, .
\end{align*}
We meet this assumption by setting $\constant_2 = 1/\constant'$.

\paragraph{Part IV: Bias approximation when $\lambda = 0$} To apply Theorem~\ref{thm:main} when 
$\lambda = \const \lambdaefct(\lambda)$, we first note that the condition $\lambda kn^{1-\frac{1}{k}} \leq n\const/2$
 is equivalent to $\lambdaefct(\lambda) k n^{-\frac{1}{k}} \leq 1/2 $, and by Lemma~\ref{lem:lambdaefct-local-bound} 
 it suffices to have $\lambdaefct(0) k n^{-\frac{1}{k}} \leq 1/4$, which holds by assumption. 
 Since we know $\chi_n'(\const) = \bigTht(\chi_n(\lambda))$ from the previous part of the proof, 
 we only need to additionally verify that $\lambdaefct(0) = \bigTht(\lambdaefct(\lambda))$ and $\rho(0) = \bigTht(\rho(\lambda))$. 
 The first relation is a direct consequence of Lemma~\ref{lem:lambdaefct-local-bound}, and for the second claim we observe that
\begin{align*}
	\rho(\lambda) & =  \frac{\Rfct_0(\lambdaefct(\lambda), 1; \btheta \btheta^\sT)}{\Rfct_0(\lambdaefct(\lambda), 1; \bI)}
	 = \frac{\Tr \prn{\bSigma^\half \btheta \btheta^\sT \bSigma^\half \prn{\bSigma + \lambdaefct(\lambda) \bI}^{-1}}}{\Tr \prn{\bSigma\prn{\bSigma + \lambdaefct(\lambda) \bI}^{-1}}}  \, .
\end{align*}
As for any p.s.d.\ $\bQ$,
\begin{align*}
	0 \geq \frac{\partial \Rfct_0(\lambdaefct, 1; \bQ)}{\partial \lambdaefct} = - \Tr \prn{\bSigma^\half \bQ \bSigma^\half \prn{\bSigma + \lambdaefct \bI}^{-2}} \geq - \frac{1}{\lambdaefct(0)} \Tr \prn{\bSigma^\half \bQ \bSigma^\half \prn{\bSigma + \lambdaefct(0) \bI}^{-1}} \, ,
\end{align*}
we have
\begin{align*}
	\Rfct_0(\lambdaefct(0), 1; \bQ) \geq \Rfct_0(\lambdaefct(\lambda), 1; \bQ) \geq \Rfct_0(\lambdaefct(0), 1; \bQ) - \frac{\lambdaefct(\lambda) - \lambdaefct(0)}{\lambdaefct(0)} \cdot \Rfct_0(\lambdaefct(0), 1; \bQ) \, .
\end{align*}
Therefore by Lemma~\ref{lem:lambdaefct-local-bound} and $\const \leq \constantlambdaefct^2/ 8$, we can obtain
\begin{align*}
	\left|\frac{\Rfct_0(\lambdaefct(\lambda), 1; \bQ)}{\Rfct_0(\lambdaefct(0), 1; \bQ)} - 1 \right| & \leq \left|\frac{\lambdaefct(\lambda)}{\lambdaefct(0)} - 1 \right| \leq \frac{2 \const}{\constantlambdaefct} \leq \frac{\constantlambdaefct}{4}  \leq \frac{1}{4} \, ,
\end{align*}
which implies $\Rfct_0(\lambdaefct(\lambda), 1; \bQ) / \Rfct_0(\lambdaefct(0), 1; \bQ) = \bigTht(1)$ and therefore $\rho(0) = \bigTht(\rho(\lambda))$. Now we are able to invoke Theorem~\ref{thm:main-ridgeless}, yielding for some constant $\constant' := \constant'(k, \constantx, D) > 0$,
\begin{align*}
	\left|\bias_\bX(\lambda) - \BIAS_n(\lambda)\right| & 
	\leq \constant' \cdot \prn{\frac{\lambdaefct(0)^{k+1}}{n \const^3} + \frac{\chi_n' (\const)^3 \log^2 n}{\sqrt{\rho(0)} n^{1- \frac{1}{k} } \const^{8.5}}}  \cdot \BIAS_n(\lambda)\, .
\end{align*}

Now we can substitute the above bound and Eq.~\eqref{eq:B-approx-lambda-2} into Eq.~\eqref{eq:V-B-approx-triangle},
\begin{align*}
	& \left|\bias_\bX(0) - \BIAS_n(0)\right|\nonumber \\
	 &\leq \constant' \cdot \prn{\frac{\lambdaefct(0)^{k+1}}{n \const^3} + \frac{\chi_n' (\const)^3 \log^2 n}{\sqrt{\rho(0)} n^{1- \frac{1}{k} } \const^{8.5}}}  \cdot \BIAS_n(\lambda) +\left|\bias_\bX(0) - \bias_\bX(\lambda)\right| + \frac{8 \const}{\constantlambdaefct^2} \cdot \BIAS_n(0) \nonumber \\
	& \leq \brc{\prn{1 + \constant' \cdot \prn{\frac{\lambdaefct(0)^{k+1}}{n \const^3} + \frac{\chi_n' (\const)^3 \log^2 n}{\sqrt{\rho(0)} n^{1- \frac{1}{k} } \const^{8.5}}} } \prn{1 + \frac{8 \const}{\constantlambdaefct^2}} - 1}\cdot \BIAS_n(0)  + \left|\bias_\bX(0) - \bias_\bX(\lambda)\right| \nonumber \\
	& \leq \prn{\exp \prn{ \constant' \cdot \prn{\frac{\lambdaefct(0)^{k+1}}{n \const^3} + \frac{\chi_n' (\const)^3 \log^2 n}{\sqrt{\rho(0)} n^{1- \frac{1}{k} } \const^{8.5}}}  +\frac{8 \const}{\constantlambdaefct^2} } - 1 } \cdot \BIAS_n(0) + \left|\bias_\bX(0) - \bias_\bX(\lambda)\right| \, .
\end{align*}

Similar to previous calculations for the variance approximation, setting $\constant_3 = 1/\constant'$ and thus
\begin{align*}
	\frac{\lambdaefct(0)^{k+1}}{n \const^3} + \frac{\chi_n' (\const)^3 \log^2 n}{\sqrt{\rho(0)} n^{1- \frac{1}{k} } \const^{8.5}} \leq \frac{1}{\constant'} \, .
\end{align*}
Substituting in Eq.~\eqref{eq:B-approx-lambda-1}, it then holds that
\begin{align*}
	& \left|\bias_\bX(0) - \BIAS_n(0)\right| \nonumber \\
	& = \bigO_{k, \constantx, D} \prn{ \frac{\const}{\constantlambdaefct^2} + \frac{\lambdaefct(0)^{k+1}}{n \const^3}  +  \frac{\chi_n' (\const)^3 \log^2 n}{\sqrt{\rho(0)} n^{1- \frac{1}{k}} \const^{8.5}}} \cdot \BIAS_n(0) +  \left|\bias_\bX(0) - \bias_\bX(\lambda)\right| \nonumber \\
	& = \bigO_{k, \constantx, D} \prn{ \frac{\const}{\constantlambdaefct^2} + \frac{\lambdaefct(0)^{k+1}}{n \const^3}  +  \frac{\chi_n' (\const)^3 \log^2 n}{\sqrt{\rho(0)} n^{1- \frac{1}{k}} \const^{8.5}}} \cdot \BIAS_n(0) \nonumber \\
	& \qquad + \min \brc{\bigO\prn{\frac{\const \lambdaefct(0)  \norm{\boldbeta}^2}{\smin}}, \bigO_{\constantx, D}(\const^2 \lambdaefct(0)^2 \chi_n'(\const)^2)\norm{\btheta_{\leq n}}^2 +  \bigO_{\constantx, D}(\const \lambdaefct(0) \chi_n'(\const)) \norm{\boldbeta_{>n}}^2} \, .
\end{align*}

\paragraph{Part V: Lower bounding the minimum eigenvalue $\smin$} To obtain the first bound, we apply known results on the minimum eigenvalue of sample covariance matrices with sub-Gaussian entries~\cite{bai2008limit, rudelson2009smallest}. Thus when $n = \bigOmg_{\constantx, \eps, D}(1)$, with probability $1 - \bigO(n^{-D+1})$ we have $\smin = \bigOmg_{\constantx, \eps,  D}(\sigma_d)$.

To obtain the other lower bound for $\smin$, we without loss of generality assume $n=\infty$ and use Cauchy interlacing theorem which implies
\begin{align*}
	\smin & \geq \lambda_n \prn{\frac{\bX^\sT \bX}{n}} = \lambda_n \prn{\frac{\bSigma^\half \bZ^\sT \bZ \bSigma^\half}{n}} \geq \lambda_n \prn{\frac{\proj_k \bSigma^\half \bZ^\sT \bZ \bSigma^\half \proj_k}{n}} \nonumber \\
	&  = \lambda_n \prn{\frac{\bZ \proj_k \bSigma \proj_k \bZ^\sT}{n}} \, .
\end{align*}
where $\proj_k$ is the projection to the space spanned by the top $k$ eigenvectors. Let $k \geq n$, we further have
\begin{align*}
	\smin \geq \sigma_k \cdot \lambda_n \prn{\frac{\bZ \bV_k \bV_k^\sT \bZ^\sT}{n}} \, ,
\end{align*}
where $\proj_k = \bV_k \bV_k^\sT$ and $\bV_k = \begin{bmatrix} \bv_1 & \cdots & \bv_k \end{bmatrix} \in \real^{d \times k}$ with $\bv_i$ being the $i$-th eigenvector of $\bSigma$. Since $\bZ \bV_k$ is a $n \times k$ random matrix with i.i.d.\ isotropic and sub-Gaussian rows. When $k \geq n$, by~\cite[Thm.~5.58, generalized version in Sec.~5.7]{vershynin2010introduction}, we have
\begin{align*}
	\lambda_n \prn{\frac{\bZ \bV_k \bV_k^\sT \bZ^\sT}{n}} & \geq \prn{(1 - \zeta)\sqrt{\frac{k}{n}} - \bigO_{\constantx}(1) - \bigO_{\constantx,D}\prn{\sqrt{\frac{\log n}{n}}} }^2 \, ,
\end{align*}
with probability at least $1 - \bigO(n^{-D+1})$, where $\zeta$ is the random variable 
\begin{align*}
	\zeta:= \max_{1 \leq i \leq n} \left|\frac{\norm{\bV_k^\sT \bz_i}^2}{k} - 1\right| \, .
\end{align*}
By Hanson-Wright in Lemma~\ref{lem:hanson-wright} and similar to the argument in Eq.~\eqref{eq:hanson-wright-for-z-top-k}, we have
\begin{align*}
	\Prb\prn{\left|\frac{\norm{\bV_k^\sT \bz_i}^2}{k} - 1\right| \geq t} = 2 \exp \prn{-\bigOmg_{\constantx} \prn{k \cdot \min \brc{t^2, t}}} \, .
\end{align*}
Given the above sharp concentration of $\zeta$, we can therefore conclude by taking $k = \lfloor\constant(\constantx) n\rfloor$ for some $\constant > 0$, and $n = \bigOmg_{\constantx, D}(1)$, we have with probability $1 - \bigO(n^{-D+1})$ that
\begin{align*}
	\lambda_n \prn{\frac{\bZ \bV_k \bV_k^\sT \bZ^\sT}{n}} \geq 1\, ,
\end{align*}
and therefore $\smin \geq \sigma_k = \sigma_{\lfloor\constant(\constantx) n\rfloor} $.
	\section{Proof of Theorem~\ref{thm:main-ridgeless-underparameterized}} \label{proof:main-ridgeless-underparameterized}
We follow the same proof strategy in Appendix~\ref{proof:main-ridgeless} for the overparameterized regime, taking 
$\lambda = \constu $. 
We state and prove the following more general result first.
\begin{theorem} \label{thm:main-ridgeless-underparameterized-raw}
	Suppose Assumption~\ref{asmp:data-dstrb} holds with $n>d$, and further assume
	\begin{align*}
			\nu \:= \min\Big(\frac{d}{n},1-\frac{d}{n}\Big)\in (0,1)\, .
		\end{align*}
		For any positive integers $k$ and $D$, there exist constants 
		 $\eta = \eta(\constantx)>0$
	 $\constant_1 = \constant_1(\constantx, D) > 0$,
	   $\constant_2 = \constant_2(k,\constantx, D) > 0$,  such that the following hold.
	   
	 If $\bX$ has rank $d$ and $\smin$ is the minimum eigenvalue of the sample
	 covariance $\bX^\sT \bX/n$, then the following hold:

	\begin{enumerate}
			\item \textbf{Variance approximation.} Let $\constu$ be such that  
		\begin{align*}
				\constu \leq  \constantlambdaefct^2 \sigma_d/4 \, , \qquad  n^{-2D + 1} =\bigO \prn{ \sqrt{\frac{\constantlambdaefct^3 \log^2 n}{n\max \brc{1, \eps}}}} \, ,\\
				\qquad \chi_n(\constu n)^3 \log^2n \leq \constant_1 n \constantlambdaefct^{4.5} \, , \qquad  \chi_n (\constu n)^3 \log^2 n \leq  \constant_2n^{1- \frac{1}{k}} \constantlambdaefct^{9.5} \, .
			\end{align*}
		Then, on the event $\{\smin \geq 2 \constu \}$, with probability $1-\bigO_k(n^{-D+1})$:
			\begin{align*}
					\left|\var_\bX(0) - \VAR_n(0)\right| & = \bigO_{k, \constantx, D} \prn{ \constu \cdot \prn{\frac{1}{\smin} + \frac{1}{\constantlambdaefct^2 \sigma_d}} + \frac{\chi_n (\constu n)^3 \log^2 n}{n^{1- \frac{1}{k}} \constantlambdaefct^{9.5}}} \cdot \VAR_n(0) \, .
				\end{align*}
		
			\item \textbf{Bias approximation.} $\bias_\bX(0) = \BIAS_n(0) = 0$ (this holds deterministically on the event
			$\rank(\bX) = d$).
		\end{enumerate}
\end{theorem}

\paragraph{Part I: Bounding \texorpdfstring{$\left|\var_\bX(0) - \var_\bX(\lambda) \right|$}{TEXT}} As we assume $\bX^\sT \bX$ has rank $d$, we can write its eigendecomposition $\bX^\sT \bX = \bU \bD \bU^\sT$ with $\bU \in \real^{d \times d}$ an orthogonal matrix and $\bD$ is a diagonal matrix with entries $s_1 \geq \cdots \geq s_d > 0$. In this case, $s_d = n \smin$.
Substitute $\lambda = \constu $ into Eq.~\eqref{eq:var-X-bound-mid-1} instead of $\lambda = \const  \lambdaefct(\lambda)$,
 we have
\begin{align}
\left|\var_\bX(0) - \var_\bX(\lambda) \right| & \leq \frac{2\constu}{\smin} \cdot \var_\bX(0) \, .\label{eq:V-approx-lambda-under-1}
\end{align}

\paragraph{Part II: Bounding $\left|\VAR_n(0) - \VAR_n(\lambda) \right|$}  Similar to the overparameterized case, we can control the growth of $\lambdaefct(\lambda)$ by
\begin{lemma} \label{lem:lambdaefct-local-bound-underparameterized}
	Under the assumptions of Theorem~\ref{thm:main-ridgeless-underparameterized}, for $\lambda$ such that $\lambda = \constu  $ it holds that
	\begin{align*}
		0 = \lambdaefct(0) \leq \lambdaefct(\lambda) \leq \frac{\constu}{\constantlambdaefct} \, .
	\end{align*}
\end{lemma}
\begin{proof}
	By the proof of Lemma~\ref{lem:lambdaefct-local-bound}, we have for $\zeta = n \lambda$,
	\begin{align*}
		\frac{\partial \zeta}{\partial \lambdaefct} & = n - \Tr \prn{\bSigma^2(\bSigma + \lambdaefct(\lambda) \bI)^{-2}}  \geq n - d \geq \constantlambdaefct n \, ,
	\end{align*}
	and thus
	\begin{align*}
		0 \leq \lambdaefct(0) \leq \lambdaefct(\lambda) & =  \int_0^{n \lambda} \frac{\partial \lambdaefct(\zeta)}{\partial \zeta} \de \zeta  \leq \frac{n\lambda}{\constantlambdaefct n} = \frac{\constu}{\constantlambdaefct}  \, .
	\end{align*}
\end{proof}  
In this case, note that
\begin{align*}
	0 \geq \frac{\partial \Tr \prn{\bSigma^2(\bSigma + \lambdaefct \bI)^{-2}}}{\partial \lambdaefct} = -2 \Tr \prn{\bSigma^2(\bSigma + \lambdaefct \bI)^{-3}} \geq -\frac{2}{\sigma_d} \Tr \prn{\bSigma^2(\bSigma + \lambdaefct(0) \bI)^{-2}} = - \frac{2d}{\sigma_d} \, ,
\end{align*}
we can apply Lemma~\ref{lem:lambdaefct-local-bound-underparameterized} and obtain for $\lambdaefct(0) = 0$,
\begin{align*}
	d = \Tr \prn{\bSigma^2(\bSigma + \lambdaefct(0) \bI)^{-2}} & \geq \Tr \prn{\bSigma^2(\bSigma + \lambdaefct(\lambda) \bI)^{-2}}  \geq \prn{1 - \frac{2 \constu}{\constantlambdaefct \sigma_d}}  \cdot d \, .
\end{align*}
We then have
\begin{align*}
	\VAR_n(0) \geq \VAR_n(\lambda) & = \frac{n - \Tr \prn{\bSigma^2(\bSigma + \lambdaefct(0) \bI)^{-2}}}{n - \Tr \prn{\bSigma^2(\bSigma + \lambdaefct(\lambda) \bI)^{-2}}} \cdot \frac{\Tr \prn{\bSigma^2(\bSigma + \lambdaefct(\lambda) \bI)^{-2}}}{\Tr \prn{\bSigma^2(\bSigma + \lambdaefct(0) \bI)^{-2}}} \cdot \VAR_n(0) \nonumber \\
	& = \frac{n -d}{n - \Tr \prn{\bSigma^2(\bSigma + \lambdaefct(\lambda) \bI)^{-2}}} \cdot \frac{\Tr \prn{\bSigma^2(\bSigma + \lambdaefct(\lambda) \bI)^{-2}}}{d} \cdot \VAR_n(0) \nonumber \\
	& \geq \frac{\constantlambdaefct n}{\constantlambdaefct n + \frac{2\constu}{\constantlambdaefct \sigma_d} \cdot d} \cdot \prn{1 - \frac{2 \constu}{\constantlambdaefct \sigma_d}} \cdot \VAR_n(0) \nonumber \\
	& \geq \frac{\constantlambdaefct^2 \sigma_d}{\constantlambdaefct^2 \sigma_d + 2\constu} \cdot \prn{1 - \frac{2\constu}{\constantlambdaefct \sigma_d}} \cdot \VAR_n(0) \, ,
\end{align*}
where in the last line we use $n \geq d$. Again by the elementary inequality $1-(1-a)(1-b) \leq a + b$ for all $0 \leq a, b \leq 1$,
\begin{align}
	\left|\VAR_n(0) - \VAR_n(\lambda) \right| & \leq \brc{1 - \prn{1 - \frac{2 \constu}{\constantlambdaefct^2 \sigma_d + 2\constu}} \cdot \prn{1 - \frac{2\constu}{\constantlambdaefct \sigma_d}}  }  \cdot \VAR_n(0) \nonumber \\
	& \leq \prn{\frac{2 \constu}{\constantlambdaefct^2 \sigma_d + 2 \constu} + \frac{2 \constu}{\constantlambdaefct \sigma_d}} \cdot \VAR_n(0) \leq \frac{4 \constu}{\constantlambdaefct^2 \sigma_d} \cdot \VAR_n(0) \, . \label{eq:V-approx-lambda-under-2}
\end{align}

\paragraph{Part III: Variance approximation} Taking $\lambda = \constu$, we want to invoke Theorem~\ref{thm:main} to bound $|\var_\bX(\lambda) - \VAR_n(\lambda)|$. Using Lemma~\ref{lem:lambdaefct-local-bound-underparameterized}, we know
\begin{align*}
	\constantlambdaefct \leq 1 - \frac{d}{n} \leq  1 - \frac{1}{n}\Tr \prn{\bSigma(\bSigma + \lambdaefct(\lambda) \bI)^{-1}} \leq 1 - \frac{d}{n} \cdot \frac{\sigma_d}{\sigma_d + \constu/ \constantlambdaefct} \leq 1 - \frac{\constantlambdaefct^2 \sigma_d}{\constantlambdaefct \sigma_d + \constu} \, .
\end{align*}
Since by assumption $\constu \leq \constantlambdaefct^2 \sigma_d / 4 \leq \constantlambdaefct \sigma_d$, Eq.~\eqref{asmp:lambda}  holds with $\kappa=\constantlambdaefct/2$,
because
\begin{align*}
	\constantlambdaefct \leq \frac{\lambda}{ \lambdaefct(\lambda)} & = 1 - \frac{1}{n}\Tr \prn{\bSigma(\bSigma + \lambdaefct(\lambda) \bI)^{-1}} \leq 1- \frac{\constantlambdaefct}{2} \, .
\end{align*}
Thus by Theorem~\ref{thm:main}, we have for some constant $\constant' := \constant'(k, \constantx, D) > 0$,
\begin{align*}
	\left|\var_\bX(\lambda) - \VAR_n(\lambda)\right| &  
	\leq \constant' \cdot \frac{\chi_n (\constu n )^3 \log^2 n}{n^{1- \frac{1}{k}} \constantlambdaefct^{9.5}} \cdot \VAR_n(\lambda) \, .
\end{align*}
Combining the above display with Eqs.~\eqref{eq:V-approx-lambda-under-1}, \eqref{eq:V-approx-lambda-under-2} yields
\begin{align*}
	\left|\var_\bX(0) - \VAR_n(0)\right| &\leq \constant' \cdot \frac{\chi_n (\constu n )^3 \log^2 n}{n^{1- \frac{1}{k}} \constantlambdaefct^{9.5}} \cdot \VAR_n(\lambda)  +\frac{2\constu}{\smin} \cdot \var_\bX(\lambda)  + \frac{4\constu}{\constantlambdaefct^2 \sigma_d} \cdot \VAR_n(0) \nonumber \\
	& \leq  \brc{\prn{1 + \frac{2\constu}{\smin}} \prn{1+\constant' \cdot \frac{\chi_n (\constu n)^3 \log^2 n}{n^{1- \frac{1}{k}} \constantlambdaefct^{9.5}}} - 1} \cdot \VAR_n(\lambda) + \frac{4 \constu}{\constantlambdaefct^2 \sigma_d} \cdot \VAR_n(0) \nonumber \\
	& \leq \brc{\prn{1 + \frac{2\constu}{\smin}} \prn{1+\constant' \cdot \frac{\chi_n (\constu n)^3 \log^2 n}{n^{1- \frac{1}{k}} \constantlambdaefct^{9.5}}}  \prn{1 + \frac{4\constu}{\constantlambdaefct^2 \sigma_d}}  - 1} \cdot \VAR_n(0) \nonumber \\
	& \leq \prn{\exp \prn{\frac{2\constu}{\smin} + \frac{4\constu}{\constantlambdaefct^2 \sigma_d} +\constant' \cdot \frac{\chi_n (\constu n)^3 \log^2 n}{n^{1- \frac{1}{k}} \constantlambdaefct^{9.5}}} - 1} \cdot \VAR_n(0) \, .
\end{align*}
Since $\constu \leq \smin / 2$ and $\constu \leq \constantlambdaefct^2 \sigma_d /4$, if we additionally assume
\begin{align*}
	\frac{\chi_n (\constu n)^3 \log^2 n}{n^{1- \frac{1}{k}} \constantlambdaefct^{9.5}}  \leq \frac{1}{\constant'} \, ,
\end{align*}
we can then conclude that 
\begin{align*}
	\exp \prn{\frac{2\constu}{\smin} + \frac{4\constu}{\constantlambdaefct^2 \sigma_d} +\constant' \cdot \frac{\chi_n (\constu n)^3 \log^2 n}{n^{1- \frac{1}{k}} \constantlambdaefct^{9.5}}} - 1 = \bigO \prn{\frac{\constu}{\smin} + \frac{\constu}{\constantlambdaefct^2 \sigma_d} +\constant' \cdot \frac{\chi_n (\constu n)^3 \log^2 n}{n^{1- \frac{1}{k}} \constantlambdaefct^{9.5}}} \, ,
\end{align*}
and the proof is complete with $\constant_2 = 1/\constant'$. 

\paragraph{Part IV: Applying Theorem~\ref{thm:main-ridgeless-underparameterized-raw}}
Finally, we apply Theorem~\ref{thm:main-ridgeless-underparameterized-raw} and obtain the statement of Theorem~\ref{thm:main-ridgeless-underparameterized}. We first notice that in the underparameterized regime with $\zeta = n \lambda$,
\begin{align*}
	\var_\bX(0) &= \lim_{\zeta \downarrow 0} \vareps^2 \Tr \prn{\bSigma \bX^\sT \bX (\bX^\sT \bX + \zeta \bI)^{-2}} =  \lim_{\zeta \downarrow 0} \vareps^2 \Tr \prn{\bSigma (\bX^\sT \bX)^{-1} } =  \vareps^2 \Tr \prn{ (\bZ^\sT \bZ)^{-1} } \, , \\
	\VAR_n(0) & = \lim_{\zeta \downarrow 0} \frac{\vareps^2 \Tr \prn{\bSigma^2 (\bSigma + \lambdaefct \bI)^{-2}}}{n -\Tr \prn{\bSigma^2 (\bSigma + \lambdaefct \bI)^{-2}} } = \frac{\vareps^2 d}{n -d } \, .
\end{align*}
Both $\var_\bX(0)$ and $\VAR_n(0)$ do not depend on the spectrum $\bSigma$. We can therefore without loss of generality assume $\bSigma = \bI$. Next, we identify that
\begin{align*}
	\constantlambdaefct = 
	1 - \frac{1}{n}\Tr \prn{\bSigma^2(\bSigma + \lambdaefct(0) \bI )^{-2}} = 1 - \frac{d}{n} \geq \nu.
\end{align*}
Setting $\constu = n^{-\delta} \nu^2$ in Theorem~\ref{thm:main-ridgeless-underparameterized-raw} for some $\delta > 0$ to be determined, the conditions 
\begin{align*}
	\constu \leq  \constantlambdaefct^2 /4 \, , \qquad  n^{-2D + 1} =\bigO \prn{ \sqrt{\frac{\constantlambdaefct^3 \log^2 n}{n\max \brc{1, \eps}}}} \, 
\end{align*}
hold true provided that
\begin{align*}
	n^{-\delta} \nu^2 \leq  \nu^2 /4 \, , \qquad  n^{-2D + 1} =\bigO \prn{ \sqrt{\frac{\nu^3 \log^2 n}{n\max \brc{1, n^{-\delta} \nu^2 }}}} \, . 
\end{align*}
Those are apparent as long as $\delta > 0$ when $n = \bigOmg_{\delta}(1)$. To verify the remaining conditions, note that
\begin{align*}
	\chi_n (\constu n) = 1 + \frac{\constantsig(\constu n)  \log^2 (\constantsig(\constu n))}{\constu n} \leq  1 + \frac{ n  \log^2 n}{n^{1-\delta} \nu^2} = \bigO \prn{\frac{n^\delta \log^2 n }{\nu^2 }}.
\end{align*}
Let $\delta = \prn{\frac{1}{4}+\frac{\epsilon}{3}} \prn{1 - \frac{1}{k}}$, it then follows 
\begin{align*}
	\chi_n(\constu n)^3 \log^2n \leq \constant_1 n \constantlambdaefct^{4.5} \, , \qquad  \chi_n (\constu n)^3 \log^2 n \leq  \constant_2n^{1- \frac{1}{k}} \constantlambdaefct^{9.5} \, .
\end{align*}
The event $\{\smin \geq 2 \eps\}$ holds with $1 - \bigO(n^{-D+1})$ by Eq.~\eqref{eq:SigmaMin_LB} in Theorem~\ref{thm:main-ridgeless}, and we can finally conclude that
\begin{align*}
	\left|\var_\bX(0) - \VAR_n(0)\right| & = \bigO_{k, \constantx, D} \prn{ \constu \cdot \prn{\frac{1}{\smin} + \frac{1}{\constantlambdaefct^2 }} + \frac{\chi_n (\constu n)^3 \log^2 n}{n^{1- \frac{1}{k}} \constantlambdaefct^{9.5}}} \cdot \VAR_n(0)  \\
	& = \bigO_{k, \constantx, D} \prn{ n^{-\delta} \nu^2 \cdot \frac{1}{\nu^2} + \frac{n^{3\delta} \log^8 n}{n^{1- \frac{1}{k}} \nu^6 \cdot \nu^{9.5}} } \cdot \VAR_n(0) \nonumber \\
	& = \bigO_{k, \constantx,  D} \prn{ \frac{\log^8 n}{n^{\prn{\frac{1}{4} - \epsilon} \prn{1- \frac{1}{k}}} \nu^{15.5} }  } \cdot \VAR_n(0)
\end{align*}
	\section{Proofs for proportional regime} \label{sec:proportional}

\subsection{Proof of Proposition~\ref{prop:proportional-ridge}} \label{proof:proportional-ridge}

To apply Theorem~\ref{thm:main}, we first provide upper bounds for $\constantsig(n)$ and $\const$ implying 
that Assumptions~\ref{asmp:data-dstrb} and Eq.~\eqref{asmp:lambda} hold.
Throughout we use the shorthand $\lambdaprop = \lambda/n\in [1/M,M]$.

\begin{lemma} \label{lem:proportional-ridge-assumption-bounds}
	Under Assumption~\ref{asmp:Sigma-proportional} and $\lambda =  \lambdaprop$, Assumptions~\ref{asmp:data-dstrb} 
	and \eqref{asmp:lambda} hold for
	\begin{align*}
		\constantsig(n) & = \bigO_{M}(n) \, , \\
		\const & = \bigOmg_{M}(1) \, .
	\end{align*}
	For such $\constantsig$ and $\const$, $\chi_n(\lambda)  = \bigO_{\lambdaprop, M}(\log^2 n)$.
\end{lemma}
\begin{proof}
	By Assumption~\ref{asmp:Sigma-proportional} we know $d \leq Mn$ and therefore for any $1 \leq k \leq \min \brc{n, d}$, 
	\begin{align*}
		\sum_{l=k}^d \sigma_l \leq d \sigma_k \leq M n \sigma_k =: \constantsig \sigma_k \, .
	\end{align*}
	Using $\lambda  = \lambdaprop$ into Eq.~\eqref{eq:lambda-fixed-point}, we have
	\begin{align*}
		1 -\frac{\lambdaprop}{\lambdaefct} = \frac{1}{n} \Tr \prn{\bSigma(\bSigma + \lambdaefct \bI)^{-1}} \, . 
	\end{align*}
	which implies
	\begin{align*}
		 1 -\frac{\lambdaprop}{\lambdaefct} \leq \frac{d}{n} \cdot \frac{1}{1 + \lambdaefct} \leq \frac{M}{\lambdaefct}  \, .
	\end{align*}
	This implies $\lambdaefct \leq \lambdaprop + M$ and therefore
	\begin{align*}
	 1 - \frac{\lambda}{\lambdaefct} = 1 - \frac{\lambdaprop}{\lambdaefct} \leq 1 - \frac{\lambdaprop}{\lambdaprop + M} \, .
	\end{align*}
	On the other hand,
	\begin{align*}
		1 - \frac{\lambda}{\lambdaefct} = 1 - \frac{\lambdaprop}{\lambdaefct} \geq \frac{d}{n} \cdot \frac{\sigma_d}{\sigma_d + \lambdaefct} \geq \frac{1}{M + M^2\lambdaefct} \geq \frac{1}{M + M^2 \lambdaprop + M^3} \, ,
	\end{align*}
	and therefore we have
	\begin{align*}
		\const : = \min \brc{\frac{\lambdaprop}{\lambdaprop + M}, \frac{1}{M + M^2 \lambdaprop + M^3}} = \bigOmg_{M}(1) \, .
	\end{align*}
	
	Finally, we can bound $\chi_n(\lambda)$ as $\constantsig = \bigO_{M}(n)$, and thus
	\begin{align*}
		\chi_n(\lambda) & = 1 + \frac{\sigma_{\lfloor \eta n\rfloor} \constantsig  \log^2   (\constantsig)}{n \lambdaprop} = \bigO_{M} \prn{1 + \frac{\log^2 n}{\lambdaprop}} = \bigO_{M} \prn{\log^2 n} \, .
	\end{align*}
\end{proof}

For any unit vector $\bu \in \real^d$, since
\begin{align}
	\Ffct_0(n\lambda, \muefct(n\lambda, 0); \bu \bu^\sT) & =n \lambda \Tr \prn{ \bu \bu^\sT \bSigma (n\lambda \bI + \muefct(n\lambda, 0) \bSigma)^{-1}} \nonumber \\
	& \geq \frac{n\lambda}{n\lambda M + \muefct(n\lambda, 0)} \Tr \prn{\bu \bu^\sT }= \frac{n\lambda}{d(n\lambda M + \muefct(n\lambda, 0))} \Tr (\bI) \nonumber \\
	& \geq \frac{n\lambda + \muefct(n\lambda, 0)}{d(n\lambda M + \muefct(n\lambda, 0))} \cdot n\lambda \Tr \prn{\bSigma(n\lambda \bI + \muefct(n\lambda, 0) \bSigma)^{-1}} \nonumber \\
	& \geq \frac{1}{dM}	\Ffct_0(n\lambda, \muefct(n\lambda, 0); \bI) \geq \frac{1}{n M^2}\Ffct_0(n\lambda, \muefct(n\lambda, 0); \bI)  \, , \label{eq:rho-lambda-bound-proportional}
\end{align}
we have $\rho(\lambda) = \bigOmg_{M}(n^{-1})$. Together with Lemma~\ref{lem:proportional-ridge-assumption-bounds}, 
since  $n = \bigOmg_{M, \constantx, D}(1)$, the following conditions 
in Theorem~\ref{thm:main} hold
\begin{align*}
	\chi_n(\lambda)^3 \log^2 n \leq \constant n \const^{4.5} \min \brc{ 1,  \sqrt{\rho(\lambda)}} \, , \qquad  n^{-2D + 1} =  \bigO \prn{ \sqrt{\frac{\const^3 \log^2 n}{n\max \brc{1, \lambda}}}} \, .
\end{align*}
Additionally, $\lambda k n^{1-\frac{1}{k}} \leq n \const /2 $ is equivalent to $\lambdaprop k n^{-\frac{1}{k}} \leq \const / 2$, which holds for $n = \bigOmg_{k, M}(1)$. Finally, by using 
$\lambdaefct(\lambda) \leq \lambdaprop+M = \bigO_{M}(1)$, as shown above, we can conclude from Theorem~\ref{thm:main}  and Lemma~\ref{lem:proportional-ridge-assumption-bounds} that,
for $n = \bigOmg_{k, M, \constantx, D}(1)$, with probability $1 - \bigO_{k}(n^{-D+1})$,
\begin{align*}
	\left|\var_\bX(\lambda) - \VAR_n(\lambda)\right| &  = \bigO_{k, \constantx, D} \prn{\frac{\chi_n (\lambda)^3 \log^2 n}{n^{1- \frac{1}{k}} \const^{9.5}}} \cdot \VAR_n(\lambda) = \bigO_{k,  M, \constantx, D} \prn{\frac{\log^8 n}{n^{1 - \frac{1}{k}}}} \cdot \VAR_n(\lambda) \, , \\
	\left|\bias_\bX(\lambda) - \BIAS_n(\lambda)\right| & 	= \bigO_{k, \constantx, D} \prn{\frac{\lambdaefct(\lambda)^{k+1}}{n \const^3} + \frac{\chi_n (\lambda)^3 \log^2 n}{\sqrt{\rho(\lambda)} n^{1- \frac{1}{k} } \const^{8.5}}}  \cdot \BIAS_n(\lambda) = \bigO_{k, M, \constantx, D} \prn{\frac{\log^8 n}{n^{\half - \frac{1}{k}}}}  \cdot \BIAS_n(\lambda) \, .
\end{align*}
The proof is complete.

\subsection{Proof of Proposition~\ref{prop:proportional-ridgeless}} \label{proof:proportional-ridgeless}

\paragraph{Overparameterized regime} When $d /n \geq 1 + M^{-1}$, by
\begin{align*}
	n = \Tr \prn{\bSigma(\bSigma + \lambdaefct(0)\bI)^{-1}} \geq \frac{d }{1 + M \lambdaefct(0)} \, ,
\end{align*}
we can deduce that $\lambdaefct(0) \geq M^{-1} \cdot \prn{d/n - 1} \geq M^{-2}$. Hence, 
\begin{align*}
	n - \Tr \prn{\bSigma^2(\bSigma + \lambdaefct(0)\bI)^{-2}} \geq n - \frac{1}{1 + \lambdaefct(0)} \cdot \Tr \prn{\bSigma(\bSigma + \lambdaefct(0)\bI)^{-1}} \geq \frac{\lambdaefct(0)}{1 + \lambdaefct(0)} \cdot n \geq  \frac{1}{M^2 + 1} \cdot n \, ,
\end{align*}
and therefore, in Theorem~\ref{thm:main-ridgeless} we can take $\constantlambdaefct \ge 1/(M^2 + 1) = \bigTht_M(1)$. By Eq.~\eqref{eq:rho-lambda-bound-proportional} we know $\rho(0) = \Omega_M(n^{-1})$. By~\cite{bai2008limit, rudelson2009smallest}, we know when $n = \bigOmg_{\constantx, M, D}(1)$, with probability $1 - \bigO(n^{-D+1})$ we have $\smin = \bigOmg_{M, \constantx, D}(1)$. Substituting $\lambdaefct(0) = \bigOmg_{M}(1)$ and $\constantsig(n) = \bigO_M(n)$ (c.f.~Lemma~\ref{lem:proportional-ridge-assumption-bounds}) into $\chi_n'(\const)$, we get for $\const = \bigO(1)$,
\begin{align*}
	\chi_n'(\const) = \bigO_M \prn{\frac{\log^2 n}{\const}} \, .
\end{align*}

Thus, by taking $\const = n^{-1/14}$, the conditions below hold for $n = \bigOmg_{k, M, \constantx, D}(1)$ given $k \geq 15$,
\begin{align*}
	& \const \leq \min \brc{\smin / (8 \lambdaefct(0)), \constantlambdaefct^2/8} \, , \qquad  n^{-2D + 1} =  \bigO \prn{ \sqrt{\frac{\const^3 \log^2 n}{n\max \brc{1, \lambda}}}}\, , \nonumber \\
	&  \chi_n' (\const)^3 \log^2 n \leq  \constant_2n^{1- \frac{1}{k}} \const^{9.5} \, ,
\end{align*}
and by taking $\const = n^{-1/28}$, the following additional conditions hold when $n = \bigOmg_{k, M, \constantx, D}(1)$ given $k \geq 29$,
\begin{align*}
	\chi_n'(\const)^3 \log^2n \leq \constant_1 n \const^{4.5} \min \brc{1, \sqrt{\rho(0)}}\, , \qquad  \frac{\lambdaefct(0)^{k+1}}{n \const^3}  +  \frac{\chi_n' (\const)^3 \log^2 n}{\sqrt{\rho(0)} n^{1- \frac{1}{k}} \const^{8.5}} \leq \constant_3 \, .
\end{align*}
We can then invoke Theorem~\ref{thm:main-ridgeless} by taking $\const = n^{-1/14}$ for variance approximation and $n^{-1/28}$ for bias approximation. Therefore, we can conclude that for $k \geq 29$,
\begin{align*}
	\left|\var_\bX(0) - \VAR_n(0)\right| & = \bigO_{k, M, \constantx, D} \prn{ n^{-1/14} + \frac{\log^8 n}{n^{1.5/14- \frac{1}{k}} }} \cdot \VAR_n(0) \, , \nonumber \\
		\left|\bias_\bX(0) - \BIAS_n(0)\right| & = \bigO_{k, M, \constantx, D} \prn{ n^{-1/28} + n^{-25/28}  +  \frac{\log^8 n}{ n^{2.5/28- \frac{1}{k}} }} \cdot \BIAS_n(0) +  \bigO \prn{ \norm{\boldbeta}^2 n^{-1/28}} \, .
\end{align*}
Use again $\lambdaefct(0) = \bigOmg_M(1)$ and $\constantlambdaefct = \bigOmg_M(1)$, we know
\begin{align*}
	 \BIAS_n(0) & = \frac{\boldbeta^\sT \prn{\bSigma / \lambdaefct(0)+ \bI}^{-2} \bSigma \boldbeta}{1 - n^{-1} \Tr \prn{\bSigma^2 (\bSigma +\lambdaefct(0) \bI)^{-2}}} = \bigOmg_{M}(\norm{\boldbeta}^2) \, .
\end{align*}
We conclude the proof by fixing $k \geq 57$, and thus
\begin{align*}
	\left|\var_\bX(0) - \VAR_n(0)\right| & = \bigO_{ M, \constantx, D} \prn{ n^{-1/14}} \cdot \VAR_n(0) \, , \nonumber \\
	\left|\bias_\bX(0) - \BIAS_n(0)\right| & = \bigO_{M, \constantx, D} \prn{n^{-1/28}} \cdot \BIAS_n(0)\, .
\end{align*}
\paragraph{Underparameterized regime} Suppose $M^{-1} \leq d/ n \leq 1 - M^{-1}$, we can invoke
 Theorem~\ref{thm:main-ridgeless-underparameterized} with $\constantlambdaefct = M^{-1}$. By~\cite{bai2008limit}
  we have $\smin = \bigOmg_{M, \constantx, D}(1)$. Also as we can take $\constantsig(n) = n$ in this case, we have
\begin{align*}
	\chi_n(\constu n) \leq 1 + \frac{n \log^2 n}{\constu n} = \bigO \prn{\frac{\log^2 n}{\const}},
\end{align*}
and therefore the conditions below hold for $n = \bigOmg_{k, M, \constantx, D}(1)$ by taking $\constu = n^{-1/4}$ when $k \geq 5$,
\begin{align*}
	&\constu \leq \min \brc{\smin / 2, \constantlambdaefct^2 \sigma_d/4} \, , \qquad \chi_n(\constu n)^3 \log^2n \leq \constant_1 n \constantlambdaefct^{4.5} \, , \qquad  n^{-2D + 1} = \bigO \prn{ \sqrt{\frac{\constantlambdaefct^3 \log^2 n}{n\max \brc{1, \lambda}}}} \, ,  \nonumber \\
	& \chi_n (\constu n)^3 \log^2 n \leq  \constant_2n^{1- \frac{1}{k}} \constantlambdaefct^{9.5} \, .
\end{align*}
We thus have
\begin{align*}
	\left|\var_\bX(0) - \VAR_n(0)\right| & = \bigO_{k, \constantx, D} \prn{ n^{-1/4} \cdot \prn{\frac{1}{\smin} + \frac{1}{\constantlambdaefct^2 \sigma_d}} + \frac{\log^8 n}{n^{\frac{1}{4}- \frac{1}{k}} \constantlambdaefct^{9.5}}} \cdot \VAR_n(0) \, .
\end{align*}
By fixing $k > 20$, we know for all $n = \bigOmg_{ M, \constantx, D}(1)$,
\begin{align*}
	\left|\var_\bX(0) - \VAR_n(0)\right| & = \bigO_{M, \constantx, D} \prn{\frac{1}{n^{\frac{1}{5}}}} \cdot \VAR_n(0) \, .
\end{align*}

	\section{Proofs for polynomially varying spectrum regime} \label{sec:bounded-varying}

\subsection{Proof of Proposition~\ref{prop:bounded-varying-ridge}} \label{proof:bounded-varying-ridge}

Throughout this proof, we will use the shorthand $\lambdabv:=\lambda/ \lambdaefct(0)$.
We begin by controlling $\lambdaefct(0)$. Since
\begin{equation} \label{eq:lambdaefct-zero-lower-bound}
\begin{aligned}
	\Tr \prn{\bSigma(\bSigma + \lambdaefct(0)\bI)^{-1}} & = n \, , \\
	\Tr \prn{\bSigma(\bSigma + \sigma_{2n} \bI)^{-1}} & \geq \sum_{i=1}^{2n} \frac{\sigma_i}{\sigma_i + \sigma_{2n}}  \geq n \, ,
\end{aligned}
\end{equation}
we know that $\lambdaefct(0) \geq \sigma_{2n}$ and therefore $\psi(\delta) \lambdaefct(0) \geq \psi(\delta) \sigma_{2n} \geq \sigma_{\lfloor 2\delta n \rfloor}$ for any $\delta \in (0, 1]$. We then have
\begin{align*}
	\Tr \prn{\bSigma(\bSigma + \psi(\delta)\lambdaefct(0)\bI)^{-1}} & \leq  2 \delta n + \sum_{i=\lfloor 2 \delta n \rfloor + 1}^{\infty} \frac{\sigma_i}{\sigma_i +\psi(\delta) \lambdaefct(0)} \nonumber \\
	& \leq 2 \delta n + \frac{\sigma_{\lfloor 2 \delta n \rfloor} + \lambdaefct(0)}{\sigma_{\lfloor 2 \delta n \rfloor} + \psi(\delta) \lambdaefct(0)} 	\cdot \Tr \prn{\bSigma(\bSigma + \lambdaefct(0)\bI)^{-1}} \nonumber \\
	& \leq  \prn{\frac{1}{2} + 2 \delta  + \frac{1}{\psi(\delta)} } \cdot n \, ,
\end{align*}
where in the last inequality we use $\psi(\delta) \lambdaefct(0) \geq \sigma_{\lfloor 2 \delta n \rfloor}$. Further
\begin{align*}
	n - \frac{n\lambda}{\psi(\delta) \lambdaefct(0)} & = \prn{1 - \frac{\lambdabv}{\psi(\delta)}} \cdot n \, ,
\end{align*}
and therefore, using the previous inequality, we conclude the following. If $\delta>0$ is such that
\begin{align*}
1 - \frac{\lambdabv}{\psi(\delta)} \ge \frac{1}{2} + 2 \delta  + \frac{1}{\psi(\delta)} \, ,
\end{align*}
then $\lambdaefct(\lambda)\le \psi(\delta) \lambdaefct(0)$.
Let $\delta_0=\delta_0(M,\psi)$ be defined follows
\begin{align*}
\delta_0:=\sup\Big\{  \delta\in(0,1/2):\;	2 \delta + \frac{1 + M}{\psi(\delta)} \le \frac{1}{2} \Big\}\, .
\end{align*}
Then $\lambdaefct(0) \leq \lambdaefct(\lambda) \leq \psi(\delta_0) \lambdaefct(0)$. Hence
\begin{align*}
	 \lambdabv = \frac{\lambda}{\lambdaefct(0)} \geq \frac{\lambda}{\lambdaefct(\lambda)} & = \frac{\lambdaefct(0) \lambdabv}{\lambdaefct(\lambda)} \geq \frac{\lambdabv}{\psi(\delta_0)} \, .
\end{align*}
Combined with $\const = \min \prn{\frac{\lambda}{\lambdaefct(\lambda)}, 1 - \frac{\lambda}{ \lambdaefct(\lambda)}} \in [1/M, 1-1/M] $ we have $\lambdaefct(\lambda) = \bigTht_{\psi, M}(\lambdaefct(0))$ and $\lambdabv =  \bigTht_{\psi, M} \prn{\frac{\lambda}{ \lambdaefct(\lambda)}} =  \bigTht_{\psi, M}(1)$.

To verify the conditions of Theorem~\ref{thm:main}, we first assume $\constantsig = \bigO(n^{1 + \gamma})$ for
 $ \gamma \in[0, 1/3)$,
\begin{align*}
	\chi_n (\lambdaefct(0) \lambdabv) & = 1 + \frac{\sigma_{\lfloor \eta n\rfloor} \constantsig  \log^2 (\constantsig)}{n \lambdaefct(0) \lambdabv} \leq 1 + \frac{\sigma_{\lfloor \eta n\rfloor} \constantsig  \log^2 (\constantsig)}{n \sigma_{2n} \lambdabv} \leq 1 + \frac{\psi(\eta/4) \constantsig  \log^2 (\constantsig)}{n  \lambdabv} \nonumber \\
	& = \bigO_{M, \psi, \constantx} \prn{\frac{\constantsig \log^2 n}{n}} = \bigO_{M, \psi, \constantx} \prn{n^{\gamma } \log^2 n}\, ,
\end{align*}
and with $\const = \bigOmg_{M}(1)$, the conditions
\begin{align*}
	\chi_n( \lambdaefct(0) \lambdabv)^3 \log^2 n \leq \constant n \const^{4.5} \, , \qquad  n^{-2D + 1} = \bigO \prn{ \sqrt{\frac{\const^3 \log^2 n}{n\max \brc{1, \lambda}}}} \, ,
\end{align*}
hold if $n = \bigOmg_{M, \psi, \gamma, \constantx, D}(1)$. We then can apply Theorem~\ref{thm:main} to
approximate the variance. Given any positive integer $k$, if $n = \bigOmg_{k, M, \psi, \gamma, \constantx, D}(1)$, it 
holds with probability $1 - \bigO_k(n^{-D+1})$ that
\begin{align*}
	\left|\var_\bX(\lambda) - \VAR_n(\lambda)\right| &  
	= \bigO_{k, M, \psi, \constantx, D} \prn{\frac{(\constantsig/n)^3 \log^8 n}{n^{1- \frac{1}{k}} }} \cdot \VAR_n(\lambda) \, .
\end{align*}

If additionally  $\constantsig = \bigO_{M, \psi, \constantx}(n^{1 + \gamma} (\rho(\lambda))^{1/6})$, we have 
$$\chi_n(\lambda)^3 \log^2 n \leq \constant n \const^{4.5} \sqrt{\rho(\lambda)} \, ,$$ when 
$n = \bigOmg_{M, \psi, \gamma, \constantx, D}(1)$. The condition $\lambda k n^{1-\frac{1}{k}} \leq n \const /2$ is equivalent to $\lambdaefct(0) \lambdabv kn^{-1/k} \leq \const / 2$, which holds when $n = \bigOmg_{k, M, \psi}(1)$ since we have assumed $\lambdaefct(0) = \bigO(1)$. Therefore, we can appeal to the bias approximation result in Theorem~\ref{thm:main}, yielding
\begin{align*}
	\left|\bias_\bX(\lambda) - \BIAS_n(\lambda)\right| & 
	= \bigO_{k, M, \psi, \constantx, D} \prn{\frac{(\constantsig / n)^3 \log^8 n}{\sqrt{\rho(\lambda)} n^{1- \frac{1}{k} }}}  \cdot \BIAS_n(\lambda)\, .
\end{align*}

\subsection{Proof of Proposition~\ref{prop:bounded-varying-ridgeless}} \label{proof:bounded-varying-ridgeless}
We provide the following bounds for the quantities in Theorem~\ref{thm:main-ridgeless}.
\begin{lemma}  \label{lem:bounded-varying-ridgeless-assumption-bounds}
	Under the same Assumptions of Theorem~\ref{thm:main-ridgeless}, we can take
	\begin{align*}
		\constantlambdaefct = \bigOmg_{\psi}(1) \, ,
	\end{align*}
	when $n = \bigOmg(1)$. For $\const = \bigO(1)$, we have 
	\begin{align*}
			\chi_n'(\const) = \bigO_{\psi} \prn{\frac{ \log^{2 + \bigO(1)} n}{\const}} \, .
	\end{align*} 
	In addition, $\smin = \bigOmg_{\psi, \constantx}(\sigma_n)$ with probability $1 - \bigO(n^{-D+1})$ .
\end{lemma}
\begin{proof}
	Since
	\begin{align*}
		n - \Tr \prn{\bSigma^2 \prn{\bSigma + \lambdaefct(0) \bI}^{-2}} & \geq n - \sum_{i=1}^{\lfloor n / 2 \rfloor} \frac{\sigma_i}{\sigma_i + \lambdaefct(0)} - \frac{\sigma_{\lfloor n/2 \rfloor}}{\sigma_{\lfloor n/2 \rfloor} + \lambdaefct(0)} \sum_{i=\lfloor n / 2 \rfloor + 1}^{\infty} \frac{\sigma_i}{\sigma_i + \lambdaefct(0)} \nonumber \\
		& =  \frac{\lambdaefct(0)}{\sigma_{\lfloor n/2 \rfloor} + \lambdaefct(0)} \sum_{i=\lfloor n / 2 \rfloor + 1}^{\infty} \frac{\sigma_i}{\sigma_i + \lambdaefct(0)} \, , 
	\end{align*}
	where in the last line we use $\Tr \prn{\bSigma(\bSigma + \lambdaefct(0) \bI)^{-1}} = n$. Since
	\begin{align*}
		\sum_{i=\lfloor n / 2 \rfloor + 1}^{\infty} \frac{\sigma_i}{\sigma_i + \lambdaefct(0)} & = n - \sum_{i=1}^{\lfloor n / 2 \rfloor} \frac{\sigma_i}{\sigma_i + \lambdaefct(0)} \geq n - \frac{n}{2} = \frac{n}{2} \, ,
	\end{align*}
	and $\lambdaefct(0) \geq \sigma_{2n}$ from Eq.~\eqref{eq:lambdaefct-zero-lower-bound}, we know
	\begin{align*}
		n - \Tr \prn{\bSigma^2 \prn{\bSigma + \lambdaefct(0) \bI}^{-2}} & \geq \frac{\sigma_{2n}}{\sigma_{\lfloor n/2 \rfloor} + \sigma_{2n}} \cdot \frac{n}{2}  \geq \frac{1}{2\psi(1/4) + 2}  \cdot n \, .
	\end{align*}
	We can hence take $\constantlambdaefct := \prn{2\psi(1/4) + 2}^{-1} = \bigOmg_{\psi}(1)$.
	
	Substituting $\lambdaefct(0) \geq \sigma_{2n}$ and $\constantsig = \bigO(n^{1 + \gamma})$ for some $1 \leq \gamma < 1/2$ into Eq.~\eqref{eq:def-chi-n-prime}, we have for $\const = \bigO(1)$,
	\begin{align*} 
		\chi_n'(\const) & = \bigO \prn{ \frac{\sigma_{\lfloor \eta n\rfloor} \constantsig  \log^2 n}{\const n \sigma_{2n}}} = \bigO_{\psi} \prn{\frac{\log^{2 + \bigO(1)} n}{\const}} \, .
	\end{align*} 
	
	We then finally apply Eq.~\eqref{eq:SigmaMin_LB} in Theorem~\ref{thm:main-ridgeless} and conclude that by taking $k = \lfloor\constant(\constantx) n\rfloor$ for some $\constant > 0$, and $n = \bigOmg_{\constantx, D}(1)$, $\smin \geq \sigma_k \geq \sigma_n / \psi(\constant)$ by Assumption~\ref{asmp:Sigma-bounded-varying}.

\end{proof}

By Lemma~\ref{lem:bounded-varying-ridgeless-assumption-bounds} and the assumption $\lambdaefct(0) / \sigma_n = \bigO(\log^{\bigO(1)} n)$, we know by taking $\const = n^{-1/14}$, the conditions below hold for $n = \bigOmg_{k, \psi, \constantx, D}(1)$ whenever $k \geq 15$,
\begin{align*}
	&\const \leq \min \brc{\smin / (8 \lambdaefct(0)), \constantlambdaefct^2/8} \, , \qquad \chi_n'(\const)^3 \log^2n \leq \constant_1 n \const^{4.5} \, , \qquad  n^{-2D + 1} = \bigO \prn{ \sqrt{\frac{\const^3 \log^2 n}{n\max \brc{1, \lambda}}}}\, , \nonumber \\
	&\chi_n' (\const)^3 \log^2 n \leq  \constant_2n^{1- \frac{1}{k}} \const^{9.5} \, .
\end{align*}
Therefore by the variance approximation in Theorem~\ref{thm:main-ridgeless}, it holds
\begin{align*}
	\left|\var_\bX(0) - \VAR_n(0)\right| & = \bigO_{k, \psi, \constantx, D} \prn{ n^{-1/14} \log^{\bigO(1)}(n) + \frac{\log^{8 + \bigO(1)} n}{n^{1.5/14- \frac{1}{k}} }} \cdot \VAR_n(0) \, .
\end{align*}
Fixing $k \geq 29$, we have with probability $1 - \bigO(n^{-D+1})$, we know as $n = \bigOmg_{ \psi, \constantx, D}(1)$,
\begin{align*}
	\left|\var_\bX(0) - \VAR_n(0)\right| & = \bigO_{\psi, \constantx, D} \prn{ n^{-1/15}} \cdot \VAR_n(0) \, .
\end{align*}

For the bias approximation with the assumption $\rho(0) = \bigOmg(n^{-2 + \gamma})$, the following additional conditions hold by taking $\const = n^{-\gamma/28}$ when $n = \bigOmg_{k, \psi, \constantx, D}(1)$ given $k \geq 29/\gamma$,
\begin{align*}
	\chi_n'(\const)^3 \log^2n \leq \constant_1 n \const^{4.5} \min \brc{1, \sqrt{\rho(0)}}\, , \qquad  \frac{\lambdaefct(0)^{k+1}}{n \const^3}  +  \frac{\chi_n' (\const)^3 \log^2 n}{\sqrt{\rho(0)} n^{1- \frac{1}{k}} \const^{8.5}} \leq \constant_3 \, .
\end{align*}
In verifying the second condition above, we use $\lambdaefct(0) =  \bigO(\sigma_n\log^{\bigO(1)} n ) = \bigO(\sigma_n\log^{\bigO(1)} n)$. By Lemma~\ref{lem:bounded-varying-ridgeless-assumption-bounds}, we also have
\begin{align*}
	\const \lambdaefct(0) \chi_n'(\const) = \bigO_{\psi} (\sigma_n \log^{2 + \bigO(1)}) \, .
\end{align*} 
We can then write out the bias approximation result applying Theorem~\ref{thm:main-ridgeless}
\begin{align*}
	\left|\bias_\bX(0) - \BIAS_n(0)\right| & = \bigO_{k, \psi, \constantx, D} \prn{ n^{-\gamma/28} + n^{-25\gamma/28} \log^{\bigO(k)}n  +  \frac{\log^{8 + \bigO(1)} n}{ n^{2.5\gamma /28- \frac{1}{k}} }} \cdot \BIAS_n(0) \nonumber \\
	& \qquad + \bigO_{\psi, \constantx, D} \prn{\sigma_n^2 \log^{4 + \bigO(1)} \norm{\btheta_{\leq n}}^2 + \sigma_n \log^{2 + \bigO(1)} \norm{\boldbeta_{> n}}^2} \, .
\end{align*}
Fixing $k \geq 57/\gamma$, we conclude the proof with
\begin{align*}
	\left|\bias_\bX(0) - \BIAS_n(0)\right| & = \bigO_{\psi, \constantx, D} \prn{ n^{-\gamma/29} } \cdot \BIAS_n(0) + \log^{\bigO(1)} \cdot  \bigO_{\psi, \constantx, D} \prn{\sigma_n^2 \norm{\btheta_{\leq n}}^2 + \sigma_n  \norm{\boldbeta_{> n}}^2}   \, .
\end{align*}

\section{Proof of Theorem~\ref{thm:bounded-varying-ridge-check-assumptions}} \label{proof:bounded-varying-ridge-check-assumptions}
Define the following increasing function in $t$,
\begin{align*}
	f_n(t; \lambda) = 1 - \frac{\lambda}{ t \sigma_n} - \frac{1}{n} \Tr \prn{\bSigma (\bSigma + t  \sigma_n \bI)^{-1}} \, .
\end{align*}
\paragraph{Case I: regularly varying spectrum when $\alpha > 1$}  In the first case, we set $\lambda = \nu  \sigma_n$. For any $t > 0$, we can compute that
\begin{align*}
	f_n(t; \lambda) = 1 - \frac{\nu}{t} - \frac{1}{n} \Tr \prn{\bSigma (\bSigma + t  \sigma_n \bI)^{-1}} \, .
\end{align*}

We will first show $\constantsig(n) = \bigO_{\bSigma}(n)$ and $\lambdaefct = \bigTht_{\nu}( \lambda )$, and then we can invoke Proposition~\ref{prop:bounded-varying-ridge} for variance approximation. For simplicity, we will suppress the dependence on sequences $\{a_i\}$ and $\{b_i\}$ in the big-O and big-$\Omega$ notations. For instance, we will just write for all $n = \Omega_{\alpha}(1)$, $|b_n| \leq \alpha$. 

We first upper bound $\constantsig$. Note that
\begin{align*}
	\sum_{l=k}^d \sigma_l & = \sum_{l=k}^\infty l^{-\alpha} a_l \exp \brc{\sum_{j=1}^l b_j / j} = \sigma_k \cdot \sum_{l=k}^\infty \prn{\frac{l}{k}}^{-\alpha} \cdot \frac{a_l}{a_k} \cdot \exp \brc{\sum_{j=k+1}^l b_j/j} \, .
\end{align*}
As $a_l$ converges to a positive limit, we have $a_l/a_k = \bigO(1)$. For $k = \bigOmg_\alpha(1)$ such that $|b_l| \leq \alpha/2$ for all $l \geq k$, we can further derive that
\begin{align*}
	\sum_{l=k}^d \sigma_l & \leq \sigma_k \cdot \bigO(1) \cdot \sum_{l=k}^\infty \prn{\frac{l}{k}}^{-\alpha} \cdot \exp \brc{\frac{\alpha}{2}\sum_{j=k+1}^l j^{-1}} \nonumber \\
	& \leq \sigma_k \cdot \bigO(1) \cdot \sum_{l=k}^\infty \prn{\frac{l}{k}}^{-\alpha} \cdot \exp \brc{\frac{\alpha}{2}\int_k^l \frac{\de t}{t}} \nonumber \\
	& = \sigma_k \cdot \bigO(1) \cdot \sum_{l=k}^\infty \prn{\frac{l}{k}}^{\alpha/2-\alpha} = k\sigma_k \cdot \bigO(1) \cdot \sum_{l=k}^\infty \frac{1}{k}\prn{\frac{l}{k}}^{-\alpha/2} \nonumber \\
	& \leq k\sigma_k \cdot \bigO(1) \cdot \prn{ \frac{1}{k} + \int_1^\infty t^{ - \alpha/2} \de t} \nonumber \\
	& = \bigO_{\alpha}(k \sigma_k ) \, .
\end{align*}
This implies for all $n = \Omega_\alpha(1)$, we can take $\constantsig(n) = \bigO_{\alpha}(n)$. Next we show $\lambdaefct = \bigTht_\nu(\lambda)$. Note that
\begin{align}
	& \lim_{n \to \infty} \frac{1}{n} \Tr \prn{\bSigma (\bSigma + t \sigma_n \bI)^{-1}} \nonumber \\
	& = \lim_{n \to \infty} \frac{1}{n} \sum_{l=1}^\infty  \frac{\sigma_l}{\sigma_l + t \sigma_n} = \lim_{M \to \infty} \lim_{n \to \infty} \frac{1}{n} \sum_{M^{-1} n \leq l \leq Mn} \frac{\sigma_l}{\sigma_l + t \sigma_n} \nonumber \\
	& = \lim_{M \to \infty} \lim_{n \to \infty} \frac{1}{n} \sum_{M^{-1} n \leq l \leq Mn} \frac{1}{1 + t (l/n)^{\alpha}} = \int_0^\infty \frac{1}{1 + t x^\alpha} \de x = t^{-1/\alpha} \cdot \frac{1}{\alpha}\int_0^\infty \frac{u^{1/\alpha - 1}}{1 + u} \de u  \nonumber \\
	& = t^{-1/\alpha} \cdot \frac{\mathsf{Beta}(1/\alpha, 1 - 1/\alpha)}{\alpha} = t^{-1/\alpha} \cdot \frac{\Gamma(1/\alpha)\Gamma(1 - 1/\alpha)}{\alpha \Gamma(1)}  \nonumber \\ & \stackrel{\mathrm{(i)}}{=} t^{-1/\alpha} \frac{\pi/\alpha}{\sin(\pi/\alpha)} \, , \label{eq:regular-spec-1-mid-1}
\end{align}
where in (i) we use the reflection formula for $\Gamma$ function. Recall that we define $\c_{\star} = \c_{\star}(\nu)$ as the unique solution of
\begin{align*}
	1 = \nu \c_{\star}^{-1} + \frac{\pi/\alpha}{\sin(\pi / \alpha)} \c_{\star}^{-1/\alpha} \, ,
\end{align*}
it then follows from the above displays that
\begin{align*}
	\lim_{n \to \infty} f_n(\c_{\star}; \lambda) & = 1 - \nu \c_{\star}^{-1} - \lim_{n \to \infty} \frac{1}{n} \Tr \prn{\bSigma (\bSigma + \c_{\star} \sigma_n \bI)^{-1}} \nonumber \\
	& = 1 - \nu \c_{\star}^{-1} -\frac{\pi/\alpha}{\sin(\pi / \alpha)} \c_{\star}^{-1/\alpha} = 0 \, .
\end{align*}
By the definition of $\lambdaefct$ in~\eqref{eq:lambda-fixed-point}, we can write $f_n(\lambdaefct/\sigma_n; \lambda) = 0$. Combining with the above limit, we can then conclude that
\begin{align*}
	\lambdaefct = \c_{\star} \sigma_n (1 + o_n(1)) \, .
\end{align*}
Substituting into Eq.~\eqref{eq:RhoLambda}, we further have
\begin{align*}
	\rho(\lambda) & = \frac{\btheta^\sT \bSigma^{\half} (\bSigma + \lambdaefct \bI)^{-1} \bSigma^{\half} \btheta}{\norm{\btheta}^2 \Tr \prn{\bSigma (\bSigma + \lambdaefct \bI)^{-1}}} = \frac{\boldbeta^\sT (\bSigma + \lambdaefct \bI)^{-1} \boldbeta}{\norm{\boldbeta}_{\bSigma^{-1}}^2 (n - n\lambda / \lambdaefct)} \nonumber \\
	& =  \frac{\sum_{l=1}^\infty \prn{\sigma_l + \lambdaefct}^{-1} \< \boldbeta, \bv_l\>^2}{n (1 - \nu \c_{\star}^{-1}) \sum_{l=1}^\infty \sigma_l^{-1} \< \boldbeta, \bv_l\>^2} \cdot (1 + o_n(1)) \nonumber \\
	& = \frac{\sum_{l=1}^\infty \sigma_n \prn{\sigma_l + \lambdaefct}^{-1} \< \boldbeta, \bv_l\>^2}{n (1 - \nu \c_{\star}^{-1}) \sum_{l=1}^\infty \sigma_n \sigma_l^{-1} \< \boldbeta, \bv_l\>^2}  \cdot (1 + o_n(1)) \nonumber \\
	& = \frac{\sum_{l=1}^\infty (l/n)^\alpha \prn{1 + \c_{\star} (l/n)^\alpha}^{-1} \< \boldbeta, \bv_l\>^2}{n (1 - \nu \c_{\star}^{-1}) \sum_{l=1}^\infty (l/n)^\alpha \< \boldbeta, \bv_l\>^2}  \cdot (1 + o_n(1)) \nonumber \\
	& = \frac{\int_0^\infty x^\alpha \prn{1 + \c_{\star}x^\alpha}^{-1} \,  \de F_{\boldbeta}(x) }{n (1 - \nu \c_{\star}^{-1}) \int_0^\infty x^\alpha \,  \de F_{\boldbeta}(x)}  \cdot (1 + o_n(1)) \, .
\end{align*}
Therefore, under the additional condition for some $0 < \theta \leq 1$ that
\begin{align*}
	\int_0^\infty x^\alpha\,  \de F_{\boldbeta}(x) = \bigO \prn{n^{1- \theta} \int_0^\infty x^\alpha \prn{1 + \c_{\star}x^\alpha}^{-1} \,  \de F_{\boldbeta}(x)} \, ,
\end{align*}
we have $\rho(\lambda) = \bigOmg(n^{-2+ \theta})$. By choosing $\gamma = (1-\theta)/3$, we can invoke Proposition~\ref{prop:bounded-varying-ridge}. Choosing a sufficiently large $k$ yields $\var_\bX(\lambda) = \VAR_n(\lambda)(1 + o_n(1))$ and $\bias_\bX(\lambda) = \BIAS_n(\lambda)(1 + o_n(1))$. 

In the next step, we derive explicit asymptotic formulas for $\VAR_n$ and $\BIAS_n$. Similar to the previous calculations in Eq.~\eqref{eq:regular-spec-1-mid-1}, we can compute that
\begin{align*}
	& \lim_{n \to \infty} \frac{1}{n} \Tr \prn{\bSigma^2 (\bSigma + t \sigma_n \bI)^{-2}} \nonumber \\
	& = \int_0^\infty \frac{1}{(1 + t x^\alpha)^2} \de x = t^{-1/\alpha} \cdot \frac{1}{\alpha}\int_0^\infty \frac{u^{1/\alpha - 1}}{(1 + u)^2} \de u  \nonumber \\
	& = t^{-1/\alpha} \cdot \frac{\mathsf{Beta}(1/\alpha, 2 - 1/\alpha)}{\alpha} = t^{-1/\alpha} \cdot \frac{\Gamma(1/\alpha)\Gamma(2 -  1/\alpha)}{\alpha \Gamma(2)}  =  t^{-1/\alpha} \cdot \frac{\Gamma(1/\alpha)\Gamma(1 - 1/\alpha)}{\alpha} \cdot \prn{1 - \frac{1}{\alpha}} \nonumber \\
	& =  t^{-1/\alpha} \frac{\pi/\alpha}{\sin(\pi/\alpha)} \cdot \prn{1 - \frac{1}{\alpha}} \, .
\end{align*}
and further $n^{-1} \Tr(\bSigma^2 (\bSigma + \lambdaefct \bI)^{-2}) \to (1- \nu \c_{\star}^{-1})(1-\alpha^{-1})$. This then gives the variance
\begin{align*}
	\VAR_n(\lambda) = \frac{\tau^2 n^{-1}\Tr(\bSigma^2 (\bSigma + \lambdaefct \bI)^{-2})}{1 - n^{-1}\Tr(\bSigma^2 (\bSigma + \lambdaefct \bI)^{-2})} = \frac{\tau^2  (1- \nu \c_{\star}^{-1})(\alpha - 1)}{1 + \nu \c_{\star}^{-1} (\alpha - 1)}   (1 +o_n(1)) \, .
\end{align*}
For the bias term, we can similarly write
\begin{align*}
	\lambdaefct^2 \<\boldbeta, \prn{\bSigma + \lambdaefct \bI}^{-2} \bSigma \boldbeta\> & = (1+o_n(1)) \cdot \sigma_n \c_\star^2  \sum_{l=1}^\infty \frac{\sigma_l \sigma_n}{(\sigma_l + \lambdaefct)^2} \<\boldbeta,\bv_l\>^2 \nonumber \\
	& = (1+o_n(1)) \cdot \sigma_n \c_\star^2 \sum_{l=1}^\infty \frac{(l/n)^\alpha}{(1 + \c_\star (l/n)^\alpha)^2} \<\boldbeta,\bv_l\>^2 \nonumber \\
	& = \sigma_n \c_{\star}^2  \int_{0}^{\infty}\frac{x^{\alpha}}
	{(1+\c_{\star}x^{\alpha})^2}\, \de F_{\boldbeta}(x)\, \big(1+o_n(1)\big) \, .
\end{align*}
Together with $ \Tr(\bSigma^2 (\bSigma + \lambdaefct \bI)^{-2}) = n(1- \nu \c_{\star}^{-1})(1-\alpha^{-1})  (1+o_n(1))$, we conclude the proof for this case.

\paragraph{Case II: regularly varying spectrum when $\alpha = 1$} Setting $\lambda = \nu \sigma_n \log n $. For any $t > 0$, we can compute that
\begin{align*}
	f_n(t; \lambda) = 1 - \frac{\nu \log n}{t} - \frac{1}{n} \Tr \prn{\bSigma (\bSigma + t  \sigma_n \bI)^{-1}} \, .
\end{align*}

We first verify Assumption~\ref{asmp:data-dstrb} holds.
With $a_i = \bigO(1)$ bounded and $\alpha' > 1$, we indeed have that $\Tr(\bSigma) < \infty$ as
\begin{align*}
	\Tr(\bSigma) = \sum_{l=1}^\infty \frac{a_l}{l (1 + \log l)^{\alpha'}} = \bigO \prn{1 + \int_1^\infty \frac{M}{t(1 + \log t)^{\alpha'}} \de t} = \bigO\prn{-\left. \frac{(1 + \log t)^{1 - \alpha'}}{\alpha' - 1}\right|_{t = 1}^\infty} = \bigO\prn{\frac{1}{\alpha' - 1}} \, .
\end{align*}
Since the sequence $\{a_i\}$ converge to a positive limit, we have for $k = \bigOmg(1)$
\begin{align*}
	\sum_{l=k}^d \sigma_l & = \bigTht\prn{\sum_{l=k}^\infty \frac{1}{l (1 + \log l)^{\alpha'}}} = \bigTht\prn{ \sigma_k + \int_k^\infty \frac{1}{t(1 + \log t)^{\alpha'}} \de t} = \bigTht\prn{\sigma_k + \frac{(1 + \log k)^{1- \alpha'}}{\alpha' - 1}} \nonumber \\
	& = \bigTht_{\alpha'}(k \log k \sigma_k) \, ,
\end{align*}
and therefore, we can take $\constantsig(n) = \bigTht_{\alpha'}(n \log n)$. We proceed to compute $\lambdaefct$. Taking any $t > 0$,
\begin{align*}
	& \lim_{n \to \infty} \frac{1}{n} \Tr \prn{\bSigma (\bSigma +  t \sigma_n \log n \bI)^{-1}} \nonumber \\
	& = \lim_{n \to \infty} \frac{1}{n} \sum_{l=1}^\infty  \frac{\sigma_l}{\sigma_l + t \sigma_n \log n} = \lim_{M \to \infty} \lim_{n \to \infty} \frac{1}{n} \sum_{l \geq M^{-1} n} \frac{\sigma_l}{\sigma_l + t \sigma_n \log n} \nonumber \\
	& = \lim_{M \to \infty} \lim_{n \to \infty} \frac{1}{n} \sum_{l \geq M^{-1} n} \frac{\sigma_l}{t \sigma_n \log n} = \lim_{M \to \infty} \lim_{n \to \infty} \frac{(1 + \log n - \log M)^{1 - \alpha'}}{t (\alpha' - 1)(\log n)^{1 - \alpha'}} \nonumber \\
	& = \frac{1}{t(\alpha ' - 1)} \, .
\end{align*}
Recalling that $\c_\star$ solves
\begin{align*}
	1 = \nu \c_{\star}^{-1} + (\alpha'-1)^{-1} \c_{\star}^{-1} \, ,
\end{align*}
we then have
\begin{align*}
	\lim_{n \to \infty} f_n(\c_{\star} \log n; \lambda) & = 1 - \nu \c_{\star}^{-1} - \lim_{n \to \infty} \frac{1}{n} \Tr \prn{\bSigma (\bSigma + \c_{\star} \sigma_n \log n \bI)^{-1}} \nonumber \\
	& = 1 - \nu \c_{\star}^{-1} - (\alpha'-1)^{-1} \c_{\star}^{-1} = 0 \, ,
\end{align*}
and consequently by Eq.~\eqref{eq:lambda-fixed-point},
\begin{align*}
	\lambdaefct = \c_{\star} \sigma_n \log n ( 1 + o_n(1)) \, .
\end{align*}
Taking the above display into Eq.~\eqref{eq:RhoLambda}, we get
\begin{align*}
	\rho(\lambda) & = \frac{\btheta^\sT \bSigma^{\half} (\bSigma + \lambdaefct \bI)^{-1} \bSigma^{\half} \btheta}{\norm{\btheta}^2 \Tr \prn{\bSigma (\bSigma + \lambdaefct \bI)^{-1}}}  =  \frac{\sum_{l=1}^\infty \prn{\sigma_l + \lambdaefct}^{-1} \< \boldbeta, \bv_l\>^2}{n (1 - \nu \c_{\star}^{-1}) \sum_{l=1}^\infty \sigma_l^{-1} \< \boldbeta, \bv_l\>^2} \cdot (1 + o_n(1)) \nonumber \\
	& = \frac{\sum_{l=1}^\infty \sigma_n \prn{\sigma_l + \c_{\star} \sigma_n \log n}^{-1} \< \boldbeta, \bv_l\>^2}{n (1 - \nu \c_{\star}^{-1}) \sum_{l=1}^\infty \sigma_n \sigma_l^{-1} \< \boldbeta, \bv_l\>^2}  \cdot (1 + o_n(1)) \nonumber \\
	& = \frac{\sum_{l=1}^\infty l/n \cdot \prn{1 + l \cdot (\c_{\star} \log n / n)}^{-1} \< \boldbeta, \bv_l\>^2}{n (1 - \nu \c_{\star}^{-1}) \sum_{l=1}^\infty l/n \cdot  \< \boldbeta, \bv_l\>^2}  \cdot (1 + o_n(1)) \nonumber \\
	& = \frac{\sum_{l=1}^\infty l \cdot ( \log n / n) \cdot \prn{1 + l \cdot (\c_{\star} \log n / n)}^{-1} \< \boldbeta, \bv_l\>^2}{n (1 - \nu \c_{\star}^{-1}) \sum_{l=1}^\infty l \cdot ( \log n / n) \cdot  \< \boldbeta, \bv_l\>^2}  \cdot (1 + o_n(1)) \nonumber \\
	& = \frac{\int_0^\infty x \prn{1 + \c_{\star}x}^{-1} \,  \de F_{\boldbeta}(x) }{n (1 - \nu \c_{\star}^{-1}) \int_0^\infty x \,  \de F_{\boldbeta}(x)}  \cdot (1 + o_n(1)) \, ,
\end{align*}
where in the last line we use $F_\boldbeta(x) = \sum_{k=1}^{\lfloor (n/\log n) x\rfloor }\<\boldbeta,\bv_k\>^2$. Thus we can have $\rho(\lambda) = \Omega(n^{-2+\theta})$ provided the condition
\begin{align*}
	\int_0^\infty x \,  \de F_{\boldbeta}(x) = \bigO \prn{n^{1-\theta} \int_0^\infty x \prn{1 + \c_{\star}x}^{-1} \,  \de F_{\boldbeta}(x)} \, .
\end{align*}
Setting $\gamma = (1-\theta)/3$, we can invoke Proposition~\ref{prop:bounded-varying-ridge} and obtain $\var_\bX(\lambda) = \VAR_n(\lambda)(1 + o_n(1))$, $\bias_\bX(\lambda) = \BIAS_n(\lambda)(1 + o_n(1))$.

For the variance $\VAR_n(\lambda)$, we note
\begin{align*}
	& \lim_{n \to \infty} \frac{\log n}{n} \Tr \prn{\bSigma^2 (\bSigma +  t \sigma_n \log n \bI)^{-2}} \nonumber \\
	& = \lim_{n \to \infty} \frac{\log n}{n} \sum_{l=1}^\infty  \frac{\sigma_l^2}{(\sigma_l + t \sigma_n \log n)^2} = \lim_{M \to \infty} \lim_{n \to \infty} \frac{\log n}{n} \sum_{M^{-1}n / \log n \leq l \leq M n / \log n} \frac{\sigma_l^2}{(\sigma_l + t \sigma_n \log n)^2}  \nonumber \\
	& = \lim_{M \to \infty} \lim_{n \to \infty} \frac{\log n}{n} \sum_{M^{-1}n / \log n \leq l \leq M n / \log n} \frac{1}{\prn{1 + tl \log n/n}^2} = \int_0^\infty \frac{1}{(1 + tx)^2} \de x  = \frac{1}{t} \, .
\end{align*}

Substituting in $\lambdaefct$, we thus have $n^{-1} \Tr \prn{\bSigma^2 (\bSigma +  \lambdaefct \bI)^{-2}} = (1+o_n(1))/(\c_{\star} \log n)$, which further implies that
\begin{align*}
	\VAR_n(0) & = \frac{\tau^2 n^{-1}\Tr(\bSigma^2 (\bSigma + \lambdaefct(0) \bI)^{-2})}{1 - n^{-1}\Tr(\bSigma^2 (\bSigma + \lambdaefct(0) \bI)^{-2})} = \frac{\vareps^2}{\c_\star \log n}\, \big(1+o_n(1)\big) \, .
\end{align*}
Finally for the bias, we have
\begin{align*}
	\lambdaefct^2 \<\boldbeta, \prn{\bSigma + \lambdaefct \bI}^{-2} \bSigma \boldbeta\> & = (1+o_n(1)) \cdot \c_{\star}^2  \sigma_n \log n  \sum_{l=1}^\infty \frac{\sigma_l \sigma_n \log n}{(\sigma_l + \lambdaefct)^2}  \<\boldbeta,\bv_l\>^2 \nonumber \\
	& = (1+o_n(1)) \cdot \c_{\star}^2  \sigma_n \log n   \sum_{l=1}^\infty \frac{(l \log n/n)}{(1 + \c_{\star} l \log n/n)^2} \<\boldbeta,\bv_l\>^2 \,  \nonumber \\
	& = \c_{\star}^2  \sigma_n \log n \int_0^\infty	\frac{x}{(1+\c_{\star} x)^2}\, \de F_{\boldbeta}(x)\, \big(1+o_n(1)\big)\, .
\end{align*}
Combining with $n^{-1} \Tr \prn{\bSigma^2 (\bSigma +  \lambdaefct \bI)^{-2}} = (1+o_n(1))/(\c_{\star} \log n)$, it holds that
\begin{align*}
	\BIAS_n(\lambda) & = \c_{\star}^2  \sigma_n \log n \int_0^\infty	\frac{x}{(1+\c_{\star} x)^2}\, \de F_{\boldbeta}(x)\, \big(1+o_n(1)\big)\, .
\end{align*}

\paragraph{Case III: a non-regularly varying spectrum}  Take $\lambda = \nu \sigma_n$. For any $t > 0$, we can compute that
\begin{align*}
	f_n(t; \lambda) = 1 - \frac{\nu }{t} - \frac{1}{n} \Tr \prn{\bSigma (\bSigma + t  \sigma_n \bI)^{-1}} \, .
\end{align*} 

If $\sigma_k = p^{-s}$, we can easily have $\sigma_{l} \leq p^{-r - s}$ if $q^r k \leq l < q^{r+1} k$. This immediately yields
\begin{align*}
	\sum_{l=k}^d \sigma_l \leq \sum_{r = 0}^\infty (q^{r+1}k- q^rk) \cdot p^{-r - s} \leq k\sigma_k \sum_{r=0}^\infty \frac{q^{r+1}}{p^r}  = \bigO_{p,q}(k \sigma_k)\, ,
\end{align*} 
as $q < p$ and the geometric sum converges. We can thus take $\constantsig(n) = \bigO_{p,q}(n)$. For $\lambdaefct$, using that
\begin{align*}
	\Tr \prn{\bSigma(\bSigma + t \sigma_n \bI)^{-1}} & = \sum_{l=0}^\infty \frac{(q^{l+1} - q^l) p^{-l}}{p^{-l} + t p^{-s_\star}} = \prn{q^{s_\star + 1} - q^{s_\star}} \cdot \sum_{l=0}^\infty \frac{q^{l-s_\star}}{1 + tp^{l-s^\star}} \, .
\end{align*}
Since $s^\star \to \infty$ as $n$ tends to infinity, we have
\begin{align*}
	\Tr \prn{\bSigma(\bSigma + t \sigma_n \bI)^{-1}} & = n \rho_\star^{-1} \cdot G_{p,q,1}(t) \big(1+o_n(1)\big) \, .
\end{align*} 
Hence
\begin{align*}
	\lim_{n \to \infty} f_n(t; \lambda) = 1 - \nu t^{-1} - \rho_\star^{-1} \cdot G_{p,q,1}(t) \, .
\end{align*}
While the right hand side is increasing in $t$ ranging in $(-\infty, 1)$. There exists a unique $\c_\star = \c_{\star}(\nu)$ solving
\begin{align*}
	\lim_{n \to \infty} f_n(\c_{\star}; \lambda) = 0 \, ,
\end{align*}
and substituting into Eq.~~\eqref{eq:lambda-fixed-point} yields
\begin{align*}
	\lambdaefct  = \c_{\star} \sigma_n (1+o_n(1)) \, .
\end{align*}	
Next we compute $\rho(\lambda)$ from Eq.~\eqref{eq:RhoLambda},
\begin{align*}
	\rho(\lambda) & = \frac{\btheta^\sT \bSigma^{\half} (\bSigma + \lambdaefct \bI)^{-1} \bSigma^{\half} \btheta}{\norm{\btheta}^2 \Tr \prn{\bSigma (\bSigma + \lambdaefct \bI)^{-1}}} =  \frac{\sum_{l=1}^\infty \prn{\sigma_l + \lambdaefct}^{-1} \< \boldbeta, \bv_l\>^2}{n (1 - \nu \c_{\star}^{-1}) \sum_{l=1}^\infty \sigma_l^{-1} \< \boldbeta, \bv_l\>^2} \cdot (1 + o_n(1)) \nonumber \\
	& = \frac{\sum_{l=1}^\infty \sigma_n \prn{\sigma_l + \lambdaefct}^{-1} \< \boldbeta, \bv_l\>^2}{n (1 - \nu \c_{\star}^{-1}) \sum_{l=1}^\infty \sigma_n \sigma_l^{-1} \< \boldbeta, \bv_l\>^2}  \cdot (1 + o_n(1)) \nonumber \\
	& = \frac{\sum_{s=0}^\infty p^{s-s_\star} /(1 + \c_{\star} p^{s-s_\star} ) \sum_{l=q^s}^{q^{s+1}-1} \< \boldbeta, \bv_l \>^2}{n (1 - \nu \c_{\star}^{-1}) \sum_{s=0}^\infty p^{s-s_\star}  \sum_{l=q^s}^{q^{s+1}-1} \< \boldbeta, \bv_l \>^2}  \cdot (1 + o_n(1)) \nonumber \\
	& = \frac{\int_0^\infty p^{x-s_\star} /(1 + \c_{\star} p^{x-s_\star} )  \,  \de F_{\boldbeta}(x) }{n (1 - \nu \c_{\star}^{-1}) \int_0^\infty p^{x-s_\star}  \,  \de F_{\boldbeta}(x)}  \cdot (1 + o_n(1)) \, .
\end{align*}
Given the ``rapid-decay'' property
\begin{align*}
	\int_0^\infty p^{x-s_\star}  \,  \de F_{\boldbeta}(x) = \bigO \prn{n^{1-\theta} \int_0^\infty p^{x-s_\star} (1 + \c_{\star} p^{x-s_\star} )^{-1}  \,  \de F_{\boldbeta}(x)} \, ,
\end{align*}
we have $\rho(\lambda) = \Omega(n^{-2+\theta})$ and Proposition~\ref{prop:bounded-varying-ridge} holds with $\gamma = (1-\theta)/3$, implying that $\var_\bX(\lambda) = \VAR_n(\lambda)(1 + o_n(1))$ and $\bias_\bX(\lambda) = \BIAS_n(\lambda)(1 + o_n(1))$.

To compute the effective variance $\VAR_n(\lambda)$, we first note that
\begin{align*}
	\Tr \prn{\bSigma^2(\bSigma + t \sigma_n \bI)^{-2}} = \prn{q^{s_\star + 1} - q^{s_\star}} \cdot \sum_{l=0}^\infty \frac{q^{l-s_\star}}{(1 + tp^{l-s^\star})^2} = n \rho_\star^{-1} \cdot G_{p,q,2}(t)\big(1+o_n(1)\big)  \, .
\end{align*}
Thus
\begin{align*}
	\VAR_n(\lambda) & = \frac{\tau^2 \Tr \prn{\bSigma^2(\bSigma + \lambdaefct \bI)^{-2}}}{n - \Tr \prn{\bSigma^2(\bSigma + \lambdaefct\bI)^{-2}}}  = \frac{G_{p,q,2}(\c_\star) \tau^2 }{\rho_\star - G_{p,q,2}(\c_\star)} \big(1+o_n(1)\big) \, .
\end{align*}
For the bias term, we have
\begin{align*}
	\lambdaefct^2 \<\boldbeta, \prn{\bSigma + \lambdaefct \bI}^{-2} \bSigma \boldbeta\> & = (1+o_n(1)) \c_{\star}^2 \sigma_n  \sum_{l=0}^\infty \frac{\sigma_l \sigma_n}{(\sigma_l + \c_{\star} \sigma_n )^2} \< \boldbeta, \bv_l \>^2 \nonumber \\
	& = (1+o_n(1)) \c_{\star}^2 \sigma_n  \sum_{s=0}^\infty \brc{\frac{p^{s-s_\star}}{(1 +  \c_{\star} p^{s-s_\star})^2} \cdot \sum_{l=q^s}^{q^{s+1}-1} \< \boldbeta, \bv_l \>^2} \nonumber \\
	& =  \c_{\star}^2 \sigma_n \int_0^\infty \frac{p^{x-s_\star}}{(1 +  \c_{\star} p^{x-s_\star})^2}\, \de F_{\boldbeta}(x)\, \big(1+o_n(1)\big) \, .
\end{align*}
We conclude the proof for $\BIAS_n(\lambda)$ by substituting in $ n^{-1} \Tr \prn{\bSigma^2(\bSigma + \lambdaefct\bI)^{-2}} =(1+o_n(1)) \rho_\star^{-1} G_{p,q,2} ( \c_{\star})  $.

\end{document}